\documentclass{amsart}
\usepackage[utf8]{inputenc}
\usepackage{fullpage}
\usepackage{parskip}
\usepackage{amsmath}
\usepackage{amsfonts}
\usepackage{amssymb}
\usepackage{mathrsfs}
\usepackage{tikz}
\usepackage{tikz-cd}
\usepackage{tikz-3dplot}
\usepackage{graphicx}
\usepackage{comment}
\usepackage{todonotes}
\usepackage{hyperref}
\usetikzlibrary{calc}
\usetikzlibrary{decorations.markings}

\newtheorem{theorem}{Theorem}[section]
\newtheorem{prop}[theorem]{Proposition}
\newtheorem{definition}[theorem]{Definition}

\newtheorem{lemma}[theorem]{Lemma}
\newtheorem{remark}[theorem]{Remark}

\newtheorem{conv}[theorem]{Convention}
\newtheorem{algorithm}[theorem]{Algorithm}

\newcommand{\N}{\mathbb{N}}
\newcommand{\Z}{\mathbb{Z}}

\newcommand{\R}{\mathbb{R}}
\newcommand{\C}{\mathbb{C}}
\newcommand{\G}{\mathbb{G}}
\newcommand{\CD}{\mathit{CD}}
\newcommand{\CDP}{\mathit{CDP}}
\newcommand{\Id}{\mathrm{Id}}
\newcommand{\SO}{\mathit{SO}}
\newcommand{\Pin}{\mathit{Pin}}

\tikzset{->-/.style={decoration={
  markings,
  mark=at position .5 with {\arrow{>}}},postaction={decorate}}}

\title{A Steenrod Square for Link Floer Homology}
\author{Yan Tao}
\date{}

\begin{document}

\begin{abstract}
Recently, Manolescu-Sarkar constructed a stable homotopy type for link Floer homology, which uses grid homology and accounts for all domains that do not pass through a specific square. We explicitly give the framings of the lower-dimensional moduli spaces of the Manolescu-Sarkar construction as well as the more general moduli spaces corresponding to the full grid. Though in the latter case the stable homotopy type is not known, the explicit framings are enough to construct a framed $1$-flow category, a construction by Lobb-Orson-Schütz which contains enough information to find the second Steenrod square. Finally, we find an algorithm for computing the second Steenrod square for all versions of grid homology coming from the full grid.\end{abstract}

\maketitle

\section{Introduction}

Link Floer homology, developed by \cite{OS3}, \cite{RA}, and \cite{OS4}, is an invariant of oriented links in three-manifolds which comes from Heegaard Floer homology, from \cite{OS1} and \cite{OS2}. There are two commonly studied versions of the link Floer chain complex: the ``hat" version $\widehat{\mathit{CFK}}$ and the ``minus" version $\mathit{CFK}^-$, as well as a ``plus" version which is the quotient of a chain complex $\mathit{CFK}^{\infty}$ by $\mathit{CFK}^-$, which is better suited for our purposes. Their homologies $\widehat{\mathit{HFK}}$, $\mathit{HFK}^-$, and $\mathit{HFK}^+$ are the versions of link Floer homology. Among their many applications are the categorification of the Alexander polynomial (see \cite{OS4}), detection of the Thurston norm of the knot complement (and in particular unknot detection, see \cite{OS6} and \cite{Ni}), detection of fibered knots (see \cite{Ghi} and \cite{Ni2}), and construction of a $\Z^{\infty}$ summand of the topologically slice subgroup of the knot concordance group (see \cite{DHST} See \cite{Man}, \cite{Hom}, \cite{OS8} for surveys of further applications.

\cite{MOS}, \cite{MOST}, and \cite{OSS} gave a combinatorial description of the link Floer chain complex for a link in $S^3$ using grid diagrams, known as \textit{grid homology}. A toroidal grid diagram is a $n \times n$ grid of squares, with the left and right edges identified and the top and bottom edges identified, together with markings $X$ and $O$, such that each row and column contains exactly one $X$ and one $O$. Given a grid diagram $\G$, drawing vertical segments from the $X$ to the $O$ in each column and horizontal segments---going under the vertical segments whenever they cross---from the $O$ to the $X$ in each row gives the diagram of an oriented link $L$; we say that $\G$ is a \textit{grid diagram for} $L$. Figure $\ref{figgrid}$ shows a $5 \times 5$ grid diagram for the trefoil.

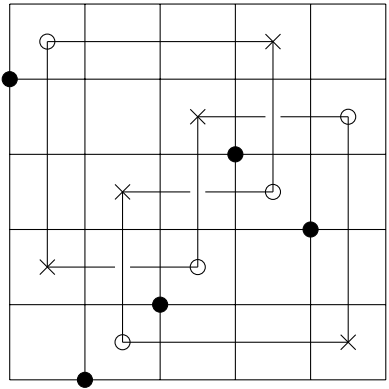
\begin{figure}\begin{tikzpicture}

\draw (0, 0) grid (5, 5);
\draw (0.5, 4.5) -- (3.5, 4.5);
\draw (0.5, 1.5) -- (0.5, 4.5);
\draw (1.5, 0.5) -- (1.5, 2.5);
\draw (1.5, 0.5) -- (4.5, 0.5);
\draw (4.5, 0.5) -- (4.5, 3.5);
\draw (3.5, 2.5) -- (3.5, 4.5);
\draw (2.5, 1.5) -- (2.5, 3.5);
\draw (0.5, 1.5) -- (1.4, 1.5);
\draw (2.5, 1.5) -- (1.6, 1.5);
\draw (1.5, 2.5) -- (2.4, 2.5);
\draw (3.5, 2.5) -- (2.6, 2.5);
\draw (2.5, 3.5) -- (3.4, 3.5);
\draw (4.5, 3.5) -- (3.6, 3.5);

\foreach \i\j in {0.5/1.5, 1.5/2.5, 2.5/3.5, 3.5/4.5, 4.5/0.5}{

\draw (\i - 0.1, \j + 0.1) -- (\i + 0.1, \j - 0.1);
\draw (\i - 0.1, \j - 0.1) -- (\i + 0.1, \j + 0.1);

}

\foreach \i\j in {0.5/4.5, 1.5/0.5, 2.5/1.5, 3.5/2.5, 4.5/3.5}{

\draw (\i, \j) circle (0.1);

}

\foreach \i\j in {0/4, 1/0, 2/1, 3/3, 4/2}{

\filldraw[black] (\i, \j) circle (0.1);

}

\end{tikzpicture}\caption{A $5 \times 5$ grid diagram for the trefoil, along with a generator of the grid chain complex drawn with \textbullet. Note that the generator is independent of the $X$ and $O$ markings.}\label{figgrid}\end{figure}

The full version $\mathit{GC}^+$ of the grid chain complex (corresponding to the plus version $\mathit{CFK}^+$) is a filtered chain complex over $\mathbb{F}[U_1, \dots, U_n]$, with one variable corresponding to each $O$ marking that domains of differentials may contain, and a filtration grading coming from the $X$ markings. For the hat version $\widehat{\mathit{GC}}$, we disallow those differentials whose domains contain a distinguished marking $O_{i_1}, \dots,  O_{i_l}$, where we distinguish one marking for each component of the link $L_1, \dots, L_l$; consequently $\widehat{\mathit{GC}}$ is a filtered chain complex over $\mathbb{F}[U_1, \dots, U_n]/(U_{i_1}, \dots, U_{i_l})$. As the terminology suggests, the homologies of their associated gradeds, $\mathit{GH}$ and $\widehat{\mathit{GH}}$, are isomorphic to $\mathit{HFK}^+$ and $\widehat{\mathit{HFK}}$, respectively---the reader is advised to check \cite{OSS} for full details.

Grid homology has been useful in a variety of applications in Heegaard Floer homology. \cite{MOT} and \cite{MO} obtain the Heegaard Floer invariants of 3- and 4-manifolds using grid diagrams, which give algorithmically computable descriptions. \cite{Sar} uses grid homology to give another proof of Milnor's conjecture on the slice genus of torus knots (originally due to Kronheimer-Mrowka via gauge theory, see \cite{KM}). \cite{OST}, \cite{NOT}, \cite{CN}, and \cite{KN} use a version of grid homology to prove results about Legendrian knots.

Cohen-Jones-Segal (\cite{CJS}) proposed the problem of lifting Floer homology to a spectrum. Since then, stable homotopy refinements of several homologies have been constructed: in Seiberg-Witten theory by \cite{ManSW}, \cite{SaSt}, \cite{KLS}; in symplectic geometry and Lagrangian Floer homology by \cite{AB1}, \cite{AB}, \cite{PS}; and in Khovanov homology by \cite{LS1}, \cite{SSS}. These stable homotopy refinements tend to have useful extra data---see for example \cite{LS}, \cite{Sch}, \cite{Adv}, and \cite{MMSW}.

Recently, Manolescu-Sarkar (\cite{MS}) constructed a stable homotopy refinement of grid homology with one $O$ marking disallowed. (For knots, this is $\widehat{\mathit{GC}}$.) The Manolescu-Sarkar construction relies on the Cohen-Jones-Segal construction, in which a link Floer framed flow category is first constructed. The moduli spaces of this framed flow category are framed by an inductive argument using obstruction theory. Unfortunately, the same construction fails for the full grid due to its obstructions being significantly more complicated (see \cite{YT} for a description of them), so its stable homotopy type is not known.

However, it is still possible to extract enough data to compute stable homotopy invariants such as Steenrod squares without constructing the full framed flow category. The Lobb-Orson-Schütz construction (\cite{LOS}) uses a \textit{framed 1-flow category} to find the second Steenrod square $Sq^2$. We build a framed $1$-flow category from the lower-dimensional moduli spaces of the link Floer framed flow category, whose obstructions to framing are well-understood (even for the full grid---see \cite{YT}).

In order to construct a framed $1$-flow category and compute $Sq^2$, we need the combinatorial data of the \textit{sign} and \textit{frame assignments}. The sign (respectively frame) assignment records the framing of the $0$- (respectively $1$-) dimensional moduli spaces. In both cases, there are two ways to frame the moduli space (up to homotopy), so we choose one of them as \textit{preferred}, and the sign and frame assignments will be $0$ for a moduli space that has the preferred frame and $1$ for a moduli space that does not. In order to specify a preferred framing, we detail the framings of the lower-dimensional moduli spaces more explicitly in this paper.

Sign assignments will coincide with sign assignments for domains with partitions. \cite{MS} introduce the complex of domains with partitions $\CDP_*$ (which we refer to as $\CDP_*^0$ to distinguish it from our $\CDP_*$ which allows domains passing through $O_1$---see Section 3 for the definitions of these complexes) to track bubbling in the boundaries of certain moduli spaces. Sign assignments for domains with partitions are known to exist---see \cite{MS} and \cite{YT} for details. Our first results concern the existence of frame assignments, which paves the way for computing $Sq^2$.

\begin{theorem}\label{mainthm}A frame assignment $f$ on $\CDP_*$ exists.\end{theorem}

Unfortunately, $\CDP_*$ (and even $\CDP_*^0$) are infinite-dimensional at each positive homological grading, as there is no limit to bubbling. It is currently unknown how to compute $f$ for these complexes. To compute $f$ (and by extension, the framed $1$-flow category and $Sq^2$), we pass to a finite subcomplex. In particular, we will compute $f$ for $\CD_*$, which is a different complex generated by domains with no bubbles (see Section 3), as their moduli spaces are the ones needed for $Sq^2$. Unfortunately, $\CD_*$ is not a subcomplex of $\CDP_*$, so we must first find a finite subcomplex of $\CDP_*$ containing $\CD_*$.

\begin{theorem}\label{mainthm2}There is an algorithm to determine $f$ for domains in $\CD_*$.\end{theorem}

Following some modifications to our moduli spaces, we are then able to use the Lobb-Orson-Schütz construction:

\begin{theorem}\label{mainthm3}There is a framed $1$-flow category for the full grid, coming from $\CDP_*$. Additionally, there is an algorithmically computable framed $1$-flow category for $\CD_*$, and hence an algorithmically computable $Sq^2$ on the grid chain complex (which uses only these domains).\end{theorem}

\begin{remark}In fact, there is a $\binom{n}{2} + 1$ parameter family of such frame assignments (and hence framed $1$-flow categories), given by the generators of $H_2(\CDP_*)$. Of these, $\binom{n}{2}$ parameters correspond to choices of certain preferred paths, so we will only consider the $1$-parameter family (that is, the two frame assignments) given by the final generator. Each frame assignment we obtain in this way is similarly algorithmic. See Propositions $\ref{unique}$ and $\ref{unique2}$ for the precise statements.\end{remark}

In the case of knots, our algorithm for computing $Sq^2$ specializes to computing $Sq^2$ for $\widehat{\mathit{CFK}}$, which is necessarily unique as the Manolescu-Sarkar construction is a stable homotopy refinement of $\widehat{\mathit{CFK}}$ in that case, so its $Sq^2$ comes from a space. \cite{MS} showed that every knot with 14 or fewer crossings except possibly $14n26580$ cannot have a nontrivial $Sq^2$, and conjectured that $14n26580$ cannot either. We show this conjecture, as well as some examples of $15$-crossing knots which have the possibility for $Sq^2$.

\begin{theorem}\label{14cross}No knots with $14$ crossings or fewer can have a nontrivial $Sq_2$. In particular, the potential example $14n26580$ of \cite{MS} cannot have a nontrivial $Sq_2$. \end{theorem}

\textbf{Organization of the paper.} In Section 2, we fix notation for grid homology and review the embeddings of the moduli spaces from the Manolescu-Sarkar construction. \\ In Section 3, we describe the boundaries of the lower-dimensional moduli spaces and introduce $\CDP_*$, the complex of positive domains with partitions. \\ In Section 4, we describe the embedded framed cobordism group and how to show whether a moduli space is frameable given the framing of its boundary. \\ In Section 5, we describe the most basic preferred paths, as well as how to compute some algebraic topology in the special orthogonal groups $\SO(A)$. \\ In Section 6, we show that sign assignments (for $\CDP_*$) give coherent framings for the $0$-dimensional moduli spaces (in other words, a sign assignment in the Lobb-Orson-Schütz sense). We also give the preferred framing of each $1$-dimensional moduli space. \\ In Section 7, we find the conditions for coherent framing of the $1$-dimensional moduli spaces, and prove Theorem $\ref{mainthm}$. \\ In Section 8, we specialize to domains without bubbles, where we can compute the frame assignment $f$ for domains in $\CD_*$ by solving a system of equations given by $\delta f$, proving Theorem $\ref{mainthm2}$ \\ In Section 9, we construct a framed $1$-flow category from the Manolescu-Sarkar moduli spaces for domains in $\CD_*$, proving Theorem $\ref{mainthm3}$. \\ In Section 10, we give an example computation for the $2 \times 2$ grid diagram for the unknot. \\ Finally, in Section 11, we prove Theorem $\ref{14cross}$ and show a potential example of a $15$-crossing knot with nontrivial $Sq^2$.

\textbf{Acknowledgements.} The author would like to thank Sucharit Sarkar for many helpful conversations throughout the writing of this paper. The author would also like to thank Dirk Schütz for helpful comments on a preprint version of this paper; Mike Hill for inspiring the ``preferred path" terminology; Geva Yashfe who suggested a key idea for evaluating the homotopy classes of loops in $\SO$; Robert Lipshitz, Ciprian Manolescu, and Zoltán Szabó for useful comments; and Marc Culler, Nathan Dunfield, Matthias Görner, Jeffrey Weeks, et al for developing SnapPy which enabled the computations in the final section. This work was supported by an NSF grant DMS-2136090.

\section{Grid Moduli Spaces}

Definitions related to grid diagrams are summarized below. For details, see \cite{MOS, MOST, OSS}.

\begin{itemize}

\item An index $n$ grid diagram $\G$ is a torus together with $n$ $\alpha$-circles (drawn horizontally) and $n$ $\beta$-circles (drawn vertically). The complements of the $\alpha$ (respectively, $\beta$) circles are called the horizontal (respectively, vertical) annuli---the complements of the $\alpha$ and $\beta$ circles are called the square regions.

\item Each vertical and horizontal annulus contains exactly one $X$ and $O$ marking, which are labelled $X_1, \dots, X_n$ and $O_1, \dots, O_n$.

\item The horizontal (respectively, vertical) annuli can be labeled by which $O$-marking they pass through---write $H_i$ (respectively, $V_i$) for the horizontal (respectively, vertical) annulus passing through $O_i$.

\item Given a fixed planar drawing of the grid, we can also label the the $\alpha$ circles $\alpha_1, \dots, \alpha_n$ from bottom to top, and the $\beta$ circles $\beta_1, \dots, \beta_n$ from left to right. The annuli can also be labelled by which sets of $\alpha$ or $\beta$ circles they lie between---write $H_{(i)}$ (respectively, $V_{(i)}$) for the horizontal annulus between $\alpha_i$ and $\alpha_{i + 1}$ (respectively, vertical annulus between $\beta_i$ and $\beta_{i + 1}$). Note that $H_{(n)}$ and $V_{(n)}$ lie between $\alpha_n$ and $\alpha_1$, and $\beta_n$ and $\beta_1$, respectively.

\item A generator is an unordered $n$-tuples of points such that each $\alpha$ and $\beta$ circle contains exactly one. Generators can equivalently be viewed a $\Z$-linear combination of $n$ points, or alternatively as permutations---for a permutation $\sigma \in S_n$ the generator $x^{\sigma}$ is the unique generator with a point at each $\alpha_{\sigma(i)} \cap \beta_{i}$. For instance, Figure $\ref{figgrid}$ shows a generator, consisting of the bulleted points \textbullet.

\item A domain is a $\Z$-linear combination of square regions with the property that $\partial D \cap \alpha = y - x$ for some generators $x, y$. We say that $D$ is a domain from $x$ to $y$, and write $D \in \mathscr{D}(x, y)$. $D$ is said to be positive if none of the coefficients are negative, in which case we would write $D \in \mathscr{D}^+(x, y)$.

\item Given a domain $D$, we denote its coefficient in the square with the $O_i$ marking $O_i(D)$, and its coefficient in the square with the $X_i$ marking $X_i(D)$. We write $\mathbb{O}(D) := (O_1(D), \dots, O_n(D))$ and $\mathbb{X}(D) := (X_1(D), \dots, X_n(D))$.

\item Given $D \in \mathscr{D}(x, y), E \in \mathscr{D}(y, z)$, we get a domain $D * E \in \mathscr{D}(x, z)$ by adding $D$ and $E$ as $2$-chains.

\item The constant domain from a generator $x$ to itself is the domain $c_x \in \mathscr{D}(x, x)$ whose coefficients are zero in every square region.

\item For every domain $D$, there is an associated integer $\mu(D)$ called its Maslov index, which satisfies: \begin{itemize}\item $\mu(D * E) = \mu(D) + \mu(E)$ \item For a positive domain $D$, $\mu(D) \geq 0$, with equality if and only if $D$ is some constant domain. \item For $D \in \mathscr{D}^+(x, y)$, $\mu(D) = 1$ if and only if $D$ is a rectangle: that is, its bottom left and top right corners are coordinates of $x$, its bottom right and top left corners are coordinates of $y$, and the other $n - 2$ coordinates of $x$ and $y$ agree and do not lie in $D$. \item $\mu(D) = k$ if and only if $D$ can be decomposed (not necessarily uniquely) into $k$ rectangles $D = R_1 * \cdots * R_k$.\end{itemize}

\item Generators have two well-defined integer gradings called the Maslov grading $M(x)$ and the Alexander grading $A(x)$, such that for any domain $D \in \mathcal{D}(x, y)$,
\begin{align*}&M(x) - M(y) = \mu(D) - 2 \lvert\mathbb{O}(D) \rvert \\&A(x) - A(y) = \lvert \mathbb{X}(D) \rvert - \lvert\mathbb{O}(D) \rvert\end{align*}

\end{itemize}

Link Floer homology can be described with grid diagrams (this version is typically called \textit{grid homology})---the differential counts the zero-dimensional moduli space of rectangles from a generator $x$ to a generator $y$, either mod $2$ or with sign given a sign assignment for rectangles (see \cite{MOST}, \cite{GA}). As a result, the Maslov grading of a generator is its homological grading.

To compute the Steenrod square $Sq^2$, we will also need to understand the one- and two-dimensional moduli spaces, which arise from index $2$ and $3$ domains, respectively, as well as bubbling. In general, we will construct moduli spaces $\mathcal{M}(D, \vec{N}, \vec{\lambda})$, where $D$ is a positive domain, $\vec{N} \in \mathbb{N}^n$, and $\vec{\lambda} = (\lambda_1, \dots, \lambda_n)$ where $\lambda_j$ is an ordered partition of $N_j$. $\mathcal{M}(D, \vec{N}, \vec{\lambda})$ will be a model for the compactified moduli space of pseudo-holomorphic strips in $\text{Sym}^n(T^2)$ relative to $\mathbb{T}_{\alpha} = \alpha_1 \times \dots \times \alpha_n$ and $\mathbb{T}_{\beta} = \beta_1 \times \dots \times \beta_n$ so that
\begin{itemize}
\item The strips have domain $D$.
\item Each strip is equipped with $\lvert \vec{N} \rvert := \sum_{j = 1}^{n} N_j$ marked points on the boundaries in groups of $N_j$. Each group of $N_j$ consists of points where an $\alpha$ or $\beta$ disk degeneration with domain $H_j$ or $V_j$ respectively has bubbled off.
\item For each $j$, the $N_j$ marked points are partitioned according to $\lambda_j$. Points in the same part of $\lambda_j$ occur at the same height on the strip.
\end{itemize}

Provided either $D$ or $\vec{N}$ is nontrivial, the dimension of $\mathcal{M}(D, \vec{N}, \vec{\lambda})$ is given by $\mu(D) - 1 + \lvert \vec{l(\lambda)} \rvert$, where $\lvert l(\vec{\lambda}) \rvert := \sum_{j = 1}^{n} l(\lambda_j)$ is the total length of $\vec{\lambda}$. See \cite[Section 8]{MS} for details.

In general, $\mathcal{M}(D, \vec{N}, \vec{\lambda})$ is a Whitney stratified space. We construct $\mathcal{M}(D, \vec{N}, \vec{\lambda})$ to have the strata
\begin{align*}\mathcal{M}(D^1, \vec{N^1} + \mathbb{O}(E^1) + \mathbb{O}(F^1), \vec{\lambda^1}) \times \dots \times \mathcal{M}(D^r, \vec{N^r} + \mathbb{O}(E^r) + \mathbb{O}(F^r), \vec{\lambda^r})\end{align*}
where $r \geq 1$, each $E^i$ is a sum of rows, each $F^i$ is a sum of columns,
\begin{align*}&D = \sum\limits_{i = 1}^{r} (D^i + E^i + F^i), \\&\vec{N} = \sum\limits_{i = 1}^{r} \vec{N^i}, \text{ and for each i,} \\& \vec{\lambda^i} = (\lambda_1^i, \dots, \lambda_n^i) \text{ where } \lambda_i^j \text{ is an ordered partition of } (N_j^i + O_j(E^i) + O_j(F^i)), \text{ and } \\&\text{for each } j, \lambda_j \text{ is the concatenation of } \eta_j^1, \dots, \eta_j^r, \text{ where } \eta_j^i \text{ is a coarsening of } \lambda_j^i\end{align*}
The $D^i$ correspond to trajectory breaking, the $E^i$ to horizontal bubbling ($\alpha$-boundary degenerations), the $F^i$ to vertical bubbling ($\beta$-boundary degenerations), and the coarsenings of the partitions correspond to joining groups of bubbles of the same type. See \cite[Section 9]{MS} for details.

The local models for the stratification are detailed in \cite[Section 7]{MS}, and we recap the important parts below. First, let $\tilde{Z}_N = \text{Sym}^N(\C)$, where we identify $\tilde{Z}_N \cong \C^N$ via the elementary symmetric polynomials in $z_1 = x_1 + iy_1, \dots, z_N = x_N + iy_N$:
\begin{align*}
s_1 = \sum\limits_{j = 1}^{N} z_j, s_2 = \sum\limits_{j < k} z_j z_k, \dots, s_N = \prod\limits_{j = 1}^{N} z_j
\end{align*}
Specifically, we take the coordinates $\text{Re}(s_1), \text{Im}(s_1), \text{Re}(s_2), \dots, \text{Im}(s_N)$ on $\tilde{Z}_n$. Let $Z_N = \tilde{Z}_N/\R$, where the quotient is identified with the subspace $\{ \text{Re}(s_1) = \sum_{j = 1}^{N} x_j = 0 \}$. $Z_N$ is a stratified space with strata
\begin{align*}Z(p^-, p^0, p^+; \lambda) \text{ where } N = p^- + p^0 + p^+ \text{ and } \lambda \text{ is an ordered partition of } p^0.\end{align*}
$Z(p^-, p^0, p^+; \lambda)$ consists of multisets $\{ z_1, \dots, z_N \}$ where $p^-$ of the $y_j = \text{Im}(z_j)$ are negative, $p^0$ of them are zero, and $p^+$ of them are positive, and the real parts of the $p^0$ coordinates with zero imaginary part are partitioned according to $\lambda$. It will also often be convenient to work with $\tilde{Z}_N$ directly, which is itself stratified with strata
\begin{align*}\tilde{Z}(p^-, p^0, p^+; \lambda) := \R \times Z(p^-, p^0, p^+; \lambda).\end{align*}
Notably, in order to frame $Z(p^-, p^0, p^+; \lambda)$, it suffices to frame $\tilde{Z}(p^-, p^0, p^+; \lambda)$.

Consider any stratum $\tilde{Z}(p^-, p^0, p^+; \lambda)$, where $\lambda = (\lambda_1, \dots, \lambda_m)$, a point $z$ in the stratum, and fix $\epsilon > 0$. We relabel the $p^0$ coordinates with zero imaginary part $z_{1, 1}, \dots, z_{1, \lambda_1}, z_{2, 1}, \dots, z_{2, \lambda_2}, \dots, z_{m, 1}, \dots, z_{m, \lambda_m}$; that is, their real parts satisfy
\begin{align*}x_{1, 1} = \dots = x_{1, \lambda_1} < x_{2, 1} = \dots = x_{2, \lambda_2} < \dots < x_{m, 1} = \dots = x_{m, \lambda_m}\end{align*}
We then pick $z' \in \tilde{Z}(p^-, p^0, p^+; (1, 1, \dots, 1))$ which is $\epsilon$-close to $z$; specifically, we leave unchanged the $N - p^0$ coordinates with nonzero imaginary parts, and in the $p^0$ coordinates with zero imaginary parts we change only the real parts by spacing them by $\epsilon$:
\begin{align*}x_{j, l}' = x_{j, 1} + \frac{(2l - \lambda_j - 1)\epsilon}{2 \lambda_j}\end{align*}
so that $z' \rightarrow z$ as $\epsilon \rightarrow 0$. We pick coordinates \textit{tailored to} $z'$ near $z'$ as follows. Pick disjoint open neighborhoods $U_{j,l}$ of $z'_{j, l}$, which have coordinates $u_{j, l} + iv_{j, l}$. On $\prod_{l = 1}^{\lambda_j} U_{j, l}$, we put the coordinates
\begin{align*}v_{j, 1}, \Delta_{j, 1}, v_{j, 2}, \Delta_{j, 2}, \dots, v_{j, \lambda_j}, \sum\limits_{l = 1}^{\lambda_j} u_{j, l}\end{align*}
where $\Delta_{j, l} := u_{j, l+1} - u_{j, l}$. For the remaining $N - p^0$ coordinates, we use the typical symmetric polynomial coordinates.

The infinitesimals of the tailored coordinates form a basis for the tangent bundle of a neighborhood of $z'$, but their linear independence is not generally preserved by sending $\epsilon \rightarrow 0$. Instead, we choose an isotopy from $z'$ to $z$ that preserves this basis, and the standard frame of $\tilde{Z}(p^-, p^0, p^+; \lambda)$ comes from the normal part of this basis.
\begin{definition}
The standard frame for the normal bundle to $\tilde{Z}(p^-, p^0, p^+; \lambda)$ in $\tilde{Z}_N$ (and equivalently $Z(p^-, p^0, p^+; \lambda)$ in $Z_N$) is given by the images under the above isotopy to
\begin{align*}\delta v_{1, 1}, -\delta \Delta_{1, 1}, \delta v_{1, 2}, &-\delta \Delta_{1, 2}, \dots, \delta v_{1, \lambda_1} \\& \dots \\\delta v_{m, 1}, -\delta \Delta_{m, 1}, \delta v_{m, 2}, &-\delta \Delta_{m, 2}, \dots, \delta v_{m, \lambda_m}
\end{align*}
(While the choice of frame appears to depend on choices of $\epsilon, z'$, and the isotopy, the different choices are canonically isotopic. See \cite[Definition 7.5]{MS})\end{definition}

More generally, we consider the spaces $Z_{\vec{N}} = (\text{Sym}^{N_1}(\C) \times \dots \times \text{Sym}^{N_n}(\C))/\R$ where we quotient by the diagonal action of $\R$. This naturally decomposes into strata
\begin{align*}Z(\vec{p}^-, \vec{p}^0, \vec{p}^+; \vec{\lambda}) \text{ where } \vec{N} = \vec{p}^- + \vec{p}^0 + \vec{p}^+ \text{ and } \lambda_j \text{ is an ordered partition of } p^0_j \text{ for each } 1 \leq j \leq n.\end{align*}
where for each $Z(p^-, p^0, p^+; \lambda)$ we have $p_j^-, p_j^0, p_j^+$ negative, zero, and positive imaginary parts in each $\text{Sym}^{N_j}(\C)$ coordinates.

\begin{definition}\label{defstdframe}The standard frame for the normal bundle to a stratum $Z(\vec{p}^-, \vec{p}^0, \vec{p}^+; \vec{\lambda})$ in $Z_{\vec{N}}$ is the concatenation of the standard frames for each $Z(p_j^-, p_j^0, p_j^+; \lambda_j)$ in $Z_{N_j}$.\end{definition}

Now we describe the local models for each moduli space. Note that a moduli space $\mathcal{M}(D, \vec{N}, \vec{\lambda})$ is itself a lower dimensional stratum of some (not generally unique) moduli space $\mathcal{M}(\tilde{D}, \vec{0}, \vec{0})$, where we must have $\tilde{D} = D + \tilde{E} + \tilde{F}$, where $\tilde{E}$ is a sum of rows, $\tilde{F}$ a sum of columns, and $\mathbb{O}(\tilde{E}) + \mathbb{O}(\tilde{F}) = \vec{N}$. So it will suffice to consider moduli spaces of the form $\mathcal{M}(\tilde{D}, \vec{0}, \vec{0})$.

In general, to ensure that the moduli space $\mathcal{M}(D, \vec{N}, \vec{\lambda})$ is neatly embedded as a stratum of $\mathcal{M}(\tilde{D}, \vec{0}, \vec{0})$, both are embedded in
\begin{align*}\mathbb{E}_{l}^{d} := \R \times \R_+ \times \R \times \dots \times \R_+ \times \R \cong \R_+^l \times \R^{d(l+1)}\end{align*}
where $d$ is a sufficiently large even integer (so that the above isomorphism is orientation-preserving) and $l = \mu(\tilde{D}) - 1 + 2\lvert \vec{N} \rvert$ (called the thick dimension of $\mathcal{M}(D, \vec{N}, \vec{\lambda})$; see \cite[Section 10]{MS} for details). The local model for the stratum
\begin{align*}\mathcal{M}(D^1, \vec{N^1} + \mathbb{O}(E^1) + \mathbb{O}(F^1), \vec{\lambda^1}) \times \dots \times \mathcal{M}(D^r, \vec{N^r} + \mathbb{O}(E^r) + \mathbb{O}(F^r), \vec{\lambda^r})\end{align*}
is the same as the local model for
\begin{align*}Z(0, \mathbb{O}(E^1) + \mathbb{O}(F^1), 0; \vec{\lambda}^1) \times \dots \times Z(0, \mathbb{O}(E^r) + \mathbb{O}(F^r), 0; \vec{\lambda}^r) \times \{ 0 \} \times \R^a\end{align*}
in
\begin{align*}Z(\mathbb{O}(E^1) + \mathbb{O}(\tilde{E}^1), 0, \mathbb{O}(F^1) + \mathbb{O}(\tilde{F}^1)) \times \dots \times Z(\mathbb{O}(E^r) + \mathbb{O}(\tilde{E}^r), 0, \mathbb{O}(F^r) + \mathbb{O}(\tilde{F}^r)) \times \R_+^{r - 1} \times R^a \cong \mathbb{E}_{l}^{d}\end{align*}
whose standard frame is given by the concatenation of the standard frame of Definition $\ref{defstdframe}$ with the positive unit vectors in each extra $\R_+$ and $\R$ factor.

\begin{definition}\label{intframes}The first $l$ frames in a stratum as above are called the internal frames. The remaining $d(l+1)$ are called the external frames.\end{definition}

\section{The Complex of Domains with Partitions}

\cite[Section 9]{MS} shows that $\mathcal{M}(D, \vec{N}, \vec{\lambda})$ has a single codimension zero stratum, which we call its interior, and whose complement we call its boundary and denote $\partial\mathcal{M}(D, \vec{N}, \vec{\lambda})$. They also show that $\partial\mathcal{M}(D, \vec{N}, \vec{\lambda})$ is the union of the codimension $1$ strata, so we recap the classification of the codimension 1 strata.

\begin{definition} The following changes to an ordered partition describe parts of the codimension 1 strata---see \cite[Definitions 4.1, 4.2, 4.3]{MS} for more details.

\begin{itemize}

\item A unit enlargement (at position $k$) increases $N$ by $1$ and adds a $1$ to the tuple $\lambda$ (at position $k$). The set of unit enlargements of $\lambda$ is denoted $\text{UE}(\lambda)$.

\item An elementary coarsening (at position $k$) replaces both terms $\lambda_k$ and $\lambda_{k + 1}$ with one term $\lambda_k + \lambda_{k + 1}$. The set of elementary coarsenings of $\lambda$ is denoted $\text{EC}(\lambda)$.

\item An initial reduction removes $\lambda_1$ (and decreases $N$ by $\lambda_1$), and a final reduction removes $\lambda_m$ (and decreases $N$ by $\lambda_m$). The set of initial reductions (respectively, final reductions) of $\lambda$ is denoted $\text{IR}(\lambda)$ (respectively, $\text{FR}(\lambda)$), where we consider both sets empty if $N = 0$.

\end{itemize}

\label{defpart}\end{definition}

Type I codimension 1 strata correspond to trajectory breaking $D$ into two non-constant domains:
\begin{align}\label{eqn1}\mathcal{M}(D^1, \vec{N^1}, \vec{\lambda^1}) \times \mathcal{M}(D^2, \vec{N^2}, \vec{\lambda^2})\end{align}
where $r = 2, E^i = F^i = 0$, and there is no coarsening of the partition (that is, $\lambda_j$ is the concatenation of $\lambda_1$ and $\lambda_2$ for each $j$).

Type II codimension 1 strata correspond to no trajectory breaking (so $r = 1$) and a single boundary degeneration:
\begin{align}\label{eqn2}\mathcal{M}(D^1, \vec{N} + \vec{e_j}, \vec{\lambda'})\end{align}
where $D = D^1 + H_j$ or $D = D^1 + V_j$, and $\vec{\lambda'} = (\lambda_1, \dots, \lambda_{j - 1}, \lambda_j', \lambda_{j  + 1}, \dots, \lambda_n)$ where $\lambda_j' \in \text{UE}(\lambda_j)$ (and $\vec{e_j}$ denotes the $j^{th}$ unit vector in $\N^n$).

Type III codimension 1 strata correspond to no trajectory breaking ($r = 1$), no boundary degenerations, and a single elementary coarsening:
\begin{align}\label{eqn3}\mathcal{M}(D, \vec{N}, \vec{\lambda'})\end{align}
where $\vec{\lambda'} = (\lambda_1, \dots, \lambda_{j - 1}, \lambda_j', \lambda_{j  + 1}, \dots, \lambda_n)$ where $\lambda_j' \in \text{EC}(\lambda_j)$.

Type IV codimension 1 strata correspond to trajectory breaking, similarly to type I, but where one of the domains is constant:
\begin{align}\label{eqn4}\mathcal{M}(c_x, \vec{N^1}, \vec{\lambda^1}) \times \mathcal{M}(D, \vec{N^2}, \vec{\lambda^2}) \text{ or } \mathcal{M}(D, \vec{N^1}, \vec{\lambda^1}) \times \mathcal{M}(c_y, \vec{N^2}, \vec{\lambda^2})\end{align}
where $D$ is a domain from $x$ to $y$, $r = 2, E^i = F^i = 0$, and there is no coarsening of the partitions, like for the type I strata.

Of the codimension 1 strata, we are specifically interested in the codimension 1 strata of type II, type III, and of type I or IV where one of the factors is zero-dimensional. For type I strata, a factor is zero-dimensional if it consists of a rectangle with an empty partition:
\begin{align}\label{eqn5}\mathcal{M}(D^1, \vec{N}, \vec{\lambda}) \times \mathcal{M}(R, \vec{0}, \vec{0}) \text{ or } \mathcal{M}(R, \vec{0}, \vec{0}) \times \mathcal{M}(D^2, \vec{N}, \vec{\lambda})\end{align}
For type IV strata, the constant domain factor is zero-dimensional if it contains a single length $1$ partition:
\begin{align}\label{eqn6}\mathcal{M}(c_x, N^1 \vec{e_j}, (N_1)_j) \times \mathcal{M}(D, \vec{N^2}, \vec{\lambda^2}) \text{ or } \mathcal{M}(D, \vec{N^1}, \vec{\lambda^1}) \times \mathcal{M}(c_y, N^2\vec{e_j}, (N^2)_j)\end{align}
where $N_1, N_2$ are positive integers and we use the convention that
\begin{conv}\label{partconv}
$(\lambda)_i$ denotes the vector consisting of the partition $\lambda$ at position $i$ and empty partitions everywhere else.
\end{conv}

\cite{MS} introduces the complex $\CDP_*$ of domains with partitions whose differential corresponds to the boundary of the corresponding moduli space. We give a brief description of $\CDP_*$ below; see \cite[Section 4]{MS} and \cite[Sections 4-5]{YT} for more details. 

\begin{definition}The complex of positive domains with partitions $\CDP_* = \CDP_*(\G; \Z/2)$ is freely generated by triples of the form $D, \vec{N}, \vec{\lambda}$, where
\begin{itemize}

\item $D \in \mathscr{D}^+(x, y)$ is a positive domain.

\item $\vec{N} \in \N^{n}$ is an $n$-tuple of nonnegative integers, $\vec{N} = (N_1, \dots, N_n)$.

\item $\vec{\lambda} = (\lambda_1, \dots, \lambda_n)$ is an $n$-tuple of ordered partitions, where $\lambda_j = (\lambda_{j, 1}, \dots, \lambda_{j, m_j})$ is an ordered partition of $N_j$.

\end{itemize}
The grading of $(D, \vec{N}, \vec{\lambda})$ is given by the Maslov index of $D$ plus $\lvert l(\vec{\lambda}) \rvert$ (the dimension of the corresponding moduli space). The differential is given by the sum of the following four terms.
\begin{itemize}

\item Type I terms, given by taking out a rectangle from $D$, corresponding to the codimension 1 strata as in Equation $(\ref{eqn5})$.

\item Type II terms, given by taking out a vertical or horizontal annulus passing through $O_j$ from $D$ and performing a unit enlargement to $\lambda_j$, corresponding to the codimension 1 strata as in Equation $(\ref{eqn2})$.

\item Type III terms, given by an elementary coarsening of one of the partitions $\lambda_j$, corresponding to the codimension 1 strata as in Equation $(\ref{eqn3})$.

\item Type IV terms, given by taking the initial or final reduction of one of the partitions $\lambda_j$, corresponding to the codimension 1 strata as in Equation $(\ref{eqn6})$.

\end{itemize}

Precisely, we can write $\partial = \partial_1 + \partial_2 + \partial_3 + \partial_4$ where

\begin{align*}\partial_1(D, \vec{N}, \vec{\lambda}) &= \sum\limits_{R * E = D} (E, \vec{N}, \vec{\lambda}) + \sum\limits_{E * R = D} (E, \vec{N}, \vec{\lambda}) \\ \partial_2(D, \vec{N}, \vec{\lambda}) &= \sum\limits_{j = 1}^{n} \sum\limits_{D = E * H_j \text{ or } E * V_j} \sum\limits_{\lambda_j' \in \text{UE}(\lambda_j)} (E, \vec{N} + \vec{e}_j, \vec{\lambda}') \\ \partial_3(D, \vec{N}, \vec{\lambda}) &= \sum\limits_{j = 1}^{n} \sum\limits_{\lambda_j' \in \text{EC}(\lambda_j)} (D, \vec{N}, \vec{\lambda}') \\ \partial_4(D, \vec{N}, \vec{\lambda}) &= \sum\limits_{j = 1}^{n} \sum\limits_{\lambda_j' \in \text{IR}(\lambda_j)} (D, \vec{N} - \lambda_{j, 1}\vec{e}_j, \vec{\lambda}') + \sum\limits_{j = 1}^{n} \sum\limits_{\lambda_j' \in \text{FR}(\lambda_j)} (D, \vec{N} - \lambda_{j, m_j}\vec{e}_j, \vec{\lambda}')\end{align*}
(Here $R$ always denotes a rectangle.)

\label{defcdp}\end{definition}

$(\CDP_*, \partial)$ is a chain complex ($\partial^2 = 0$). Though we have only defined it in $\Z/2$ coefficients, a similar definition with some signs (see \cite[Section 4]{MS} and \cite[Section 5]{YT} for specifics) makes it a chain complex over $\Z$ as well. Refer to \cite[Lemma 4.6]{MS} for the proof of this more general fact. We also now define the related complexes:

\begin{definition}\begin{itemize}

\item $\CDP_*^0$ is freely generated by the same generators as $\CDP_*$ and has the same differential, except with $O_1$ blocked---that is, where we exclude Type I differentials with $R$ passing through $O_1$.

\item $\CD_*$ is freely generated by the positive domains with no bubbling ($\vec{N} = 0$). Its differential is the Type I differential of $\CDP_*$.

\end{itemize}\end{definition}

$\CDP_*^0$ is a subcomplex of $\CDP_*$, and is the version of $\CDP_*$ studied by \cite{MS}. On the other hand, $\CD_*$ is not a subcomplex of $\CDP_*$ (as some domains without bubbling have Type II differentials). However, $\CD_*$ is a chain complex in its own right (see \cite[Section 3]{MS}) for a proof of this fact) that can also be extended over $\Z$ coefficients.

The following classification of the generators of $\CDP_0, \CDP_1, \CDP_2, \CDP_3$ from \cite{YT} will be helpful. It will also be convenient to name the different types of domains.

\begin{enumerate}

\item[(0)] $\CDP_0$ is generated by the constant domains with no partitions $(c_{x}, 0, 0)$ for some generator $x$.

\item[(1)] $\CDP_1$ is generated by rectangles with no partitions $(R, 0, 0)$, which we will call \textbf{Type 1.1} triples, as well as triples of the form $(c_x, N\vec{e}_j, (N)_j)$ for a constaint domain $c_x$, which we will call \textbf{Type 1.0} triples.

\item[(2)] $\CDP_2$ is generated by Maslov index $2$ domains with no partition $(D, 0, 0)$, which we will call \textbf{Type 2.2} triples, triples of the form $(R, N\vec{e}_j, (N)_j)$ for a rectangle $R$, which we will call \textbf{Type 2.1} triples, or a constant domain with partitions of total length $2$, which we will call \textbf{Type 2.0} triples. As Type 2.2 and 2.0 triples can have different terms appear in their differentials, we divide them into subtypes accordingly:

\begin{itemize}

\item Type 2.2 triples where $D$ is not an annulus have only type I differentials, and will be called \textbf{Type 2.2a} triples.

\item Type 2.2 triples where $D$ is an annulus have a type I and a type II differential, and will be called \textbf{Type 2.2b} triples.

\item Type 2.0 triples of the form $(c_x, (N + M)\vec{e}_j, (N, M)_j)$ have a type III and IV differential, and will be called \textbf{Type 2.0a} triples.

\item Type 2.0 triples of the form $(c_x, N\vec{e}_j + M \vec{e}_k, ((N)_j, (M)_k))$ (where $j \neq k$) have only type IV differentials, and will be called \textbf{Type 2.0b} triples.

\end{itemize}

\item[(3)] Finally, $\CDP_3$ is generated by Maslov index $3$ domains with no partition (\textbf{Type 3.3}), Maslov index $2$ domains with a partition of the form $(D, N\vec{e}_j, (N)_j)$ (\textbf{Type 3.2}), rectangles with a partition of total length $2$ (\textbf{Type 3.1}), and constant domains with partitions of total length $3$ (\textbf{Type 3.0}). The subtypes are the following

\begin{itemize}

\item Type 3.3 triples where $D$ does not contain an annulus have only type I differentials, and will be called \textbf{Type 3.3a} triples.

\item Type 3.3 triples where $D$ contains an annulus $A = R * S$, and $D$ can be written $D = A * R$, have two type I differentials and one type II differential, and will be called \textbf{Type 3.3b} triples.

\item Type 3.3 triples where $D$ contains an annulus $A = R * S$, and $D \neq A * R$, have three type I differentials and one type II differential, and will be called \textbf{Type 3.3c} triples.

\item Type 3.2 triples $(D, N\vec{e}_j, (N)_j)$ where $D$ is not an annulus have only type I and IV differentials, and will be called \textbf{Type 3.2a} triples.

\item Type 3.2 triples $(D, N\vec{e}_j, (N)_j)$ where $D$ is an annulus but is neither $H_j$ nor $V_j$ have type I, II, and IV differentials, and will be called \textbf{Type 3.2b} triples.

\item Type 3.2 triples $(D, N\vec{e}_j, (N)_j)$ where $D$ is either $H_j$ nor $V_j$ have type I, II, and IV differentials as well as a type III term in $\partial^2$, and will be called \textbf{Type 3.2c} triples.

\item Type 3.1 triples of the form $(R, (N + M)\vec{e}_j, (N, M)_j)$ have type I, III, and IV differential, and will be called \textbf{Type 3.1a} triples.

\item Type 3.1 triples of the form $(R, N\vec{e}_j + M \vec{e}_k, ((N)_j, (M)_k))$ (where $j \neq k$) have only type I and IV differentials, and will be called \textbf{Type 3.1b} triples.

\item Type 3.0 triples of the form $(c_x, (N_j + M_j + P_j) \vec{e}_j, (N_j, M_j, P_j)_j)$ have type III and IV differentials (even in $\partial^2$), and will be called \textbf{Type 3.0a} triples.

\item Type 3.0 triples of the form $(c_x, (N_j + M_j) \vec{e}_j + N_k \vec{e}_k, ((N_j, M_j)_j, (N_k)_k)$ for $j, k$ distinct have type III and IV differentials, and will be called \textbf{Type 3.0b} triples.

\item Type 3.0 triples of the form $(c_x, N_j \vec{e}_j + N_k \vec{e}_k + N_l \vec{e}_l, ((N_j)_j, (N_k)_k, (N_l)_l)$ for $j, k, l$ distinct have only type IV differentials, and will be called \textbf{Type 3.0c} triples.

\end{itemize}\end{enumerate}

Since $\CDP_*$ captures the boundary of the corresponding moduli space, its homology computation is useful:

\begin{theorem}(Theorem 1.4 of \cite{YT}) We have that
\begin{enumerate}

\item[(0)] $H_0(\CDP_*; \Z/2)$ is isomorphic to $\Z/2$, and is generated by $(c_{\Id}, 0, 0)$.

\item[(1)] $H_1(\CDP_*; \Z/2)$ is isomorphic to $(\Z/2)^{n}$, and is generated by $(c_{\Id}, \vec{e_j}, (1))$ for each $1 \leq j \leq n$.

\item[(2)] $H_2(\CDP_*; \Z/2)$ is isomorphic to $(\Z/2)^{\binom{n}{2} + 1}$, and is generated by $(c_{\Id}, \vec{e}_j + \vec{e}_k, ((1), (1)))$ for $1 \leq j < k \leq n$ and a linear combination $U'$ of domains as given in \cite{YT}.

\item[(3)] $H_3(\CDP_*; \Z/2)$ is isomorphic to $(\Z/2)^{\binom{n}{3} + n}$, and is generated by $(c_{\Id}, \vec{e}_j + \vec{e}_k + \vec{e}_l, ((1), (1), (1)))$ for $1 \leq j < k < l \leq n$ and $n$ additional generators $U_j'$ obtained from $U'$ as given in \cite{YT}.

\end{enumerate}\label{thm2}\end{theorem}

\begin{proof}See \cite{YT} for details.\end{proof}

One computation that follows from the homology is the existence of a \textit{sign assignment}, which gives a way to extend $\CDP_*$ over $\Z$ coefficients.

\begin{definition}\label{defsa}Given $s_j \in \Z/2$ for each $1 \leq j \leq n$, a sign assignment $s$ is a $\Z/2$-valued $1$-cochain on $\CDP_*$ such that for rectangles $R, S, R', S'$,
\begin{itemize}
\item (Square Rule) If $R * S = R' * S'$ is not an annulus, $s(R, 0, 0) + s(S, 0, 0) + s(R', 0, 0) + s(S', 0, 0) = 1 \pmod{2}$
\item (Annuli) If $R * S$ is a vertical annulus, $s(R, 0, 0) + s(S, 0, 0) = 1 \pmod{2}$, and if $R * S$ is a horizontal annulus, $s(R, 0, 0) + s(S, 0, 0) = 0 \pmod{2}$
\item (Bubbles) For any generator $x$ and any positive integer $N$, $s(c_x, N\vec{e_j}, (N)_j) = N s_j$\end{itemize}\end{definition}

\begin{theorem}(Theorem 1.5 of \cite{YT}) A sign assignment $s$ on $\CDP_*$ exists, and is unique up to gauge transformations and the values of $s_j$.\label{thm1}\end{theorem}

See \cite[Section 5]{YT} for further details of the proof. The essence of the proof is that the Square and Annuli rules, viewed as a $2$-cochain $T$ on the index $2$ domains, is a $2$-cocycle, which evaluates to zero on $H_2(\CDP_*)$. So $T$ must be the coboundary of our sign assignment $s$, which is unique up to $H^1(\CDP_*)$, which is generated by $(c_x, \vec{e_j}, (1)_j)$.

\section{The Framed Cobordism Group}

To computationally describe which boundaries are framed null-cobordant, we use the embedded framed cobordism groups $\tilde{\Omega}_{fr}^{k}$ described by \cite{MS}. We give a brief description of $\tilde{\Omega}_{fr}^{k}$ below.

\begin{definition}
\begin{itemize}

\item For $A \geq 2k + 3$, $\tilde{\Omega}_{fr, A}^{k}$ is the set of equivalence classes of closed $k$-manifolds $M$ embedded in $\R^A$ together with a vector field $\vec{v}$ along $M$ transverse to $\mathit{TM}$ and a framing of the orthogonal complement of $\mathit{TM} \oplus \langle \vec{v} \rangle$.

\item The equivalence relation is given as follows. $(M_1, \vec{v}_1) \sim (M_2, \vec{v}_2)$ if there is an embedded framed cobordism in $\R^A$ which starts in the direction of $\vec{v_1}$ and ends in the direction of $\vec{v}_2$. (If $M_1 \cap M_2 \neq \emptyset$, we first replace $M_2$ with a generic translate $M_2'$ that does not intersect $M_1$.)

\item We define the map $\tilde{\Omega}_{fr, A}^{k} \rightarrow \tilde{\Omega}_{fr, A+1}^{k}$ by embedding $\R^A \rightarrow \R^A \times \{ 0 \} \subset \R^{A+1}$ and adding the unit vector in the $(A+1)^{st}$ $\R$ direction to the end of the frame.

\item Finally, we define $\tilde{\Omega}_{fr}^{k}$ as the colimit of $\tilde{\Omega}_{fr, A}^{k}$ under these maps.

\end{itemize}
\end{definition}

\begin{conv}\label{frameconv}In this paper, we refer to an ordered basis of the orthogonal complement a frame, while typically it refers to an orthonormal basis thereof. Since any such ordered basis can be made orthonormal using the Gram-Schmidt process, we will use this equivalent (for our purposes) definition for convenience.\end{conv}

The (abelian) group structure of $\tilde{\Omega}_{fr, A}^{k}$ is as follows. Addition is given by disjoint union (after a generic translation if necessary), the identity by the empty submanifold, and inverse by reversing $\vec{v}$. The maps $\tilde{\Omega}_{fr, A}^{k} \rightarrow \tilde{\Omega}_{fr, A+1}^{k}$ are group morphisms, so the colimit can also be taken in the category of groups. (See \cite[Section 11]{MS} for further details.)

The classical framed cobordism group is defined very similarly, as follows.

\begin{definition}
\begin{itemize}

\item $\Omega_{fr, A}^{k}$ is the set of equivalence classes of closed $k$-manifolds $M$ embedded in $\R^A$ together with a framing of its normal bundle.

\item The equivalence relation $M_1 \sim M_2$ is given by framed cobordisms connecting $M_1$ to $M_2$ (again, up to a translation if necessary).

\item Addition is similarly given by disjoint union.

\item The maps $\Omega_{fr, A}^{k} \rightarrow \Omega_{fr, A+1}^{k}$ are given by embedding $\R^A \rightarrow \R^A \times \{ 0 \} \subset \R^{A+1}$. The group $\Omega_{fr}^{k}$ is the colimit of $\Omega_{fr, A}^{k}$ under these maps.

\end{itemize}
\end{definition}

Since $\tilde{\Omega}_{fr, A}^{k}$ is just $\Omega_{fr, A}^{k}$ but where we distinguish a vector $\vec{v}$ in the frame, we can define a map 
\begin{align}\label{forgetmap}\tilde{\Omega}_{fr, A}^{k} \rightarrow \Omega_{fr, A}^{k}\end{align}
by forgetting the distinguished vector. Specifically, given $(M, \vec{v}) \in \tilde{\Omega}_{fr, A}^{k}$ with frame $[v_1, \dots, v_{A-k-1}]$, we map it to $M$ with the frame $[(-1)^{k} \vec{v}, v_1, \dots, v_{A-k-1}]$. By \cite[Proposition 11.5]{MS}, this map is a group isomorphism.

The boundary of a $1$-dimensional moduli space, an interval, consists of two points, each with a distinguished normal direction $\vec{v}$ pointing in the direction of the interval. Therefore the boundary consists of two elements of $\Tilde{\Omega}_{fr,A}^0$, which must be opposite elements if the interval is to be frameable. By the isomorphism $(\ref{forgetmap})$, to prove that a 1-dimensional moduli space is frameable, it suffices to show that its boundary consists of two oppositely framed points in $\Omega_{fr,A}^0$, which is more convenient for our purposes. As in \cite{LS} and \cite{MS}, we will use a sign assignment to frame all of the zero-dimensional moduli spaces---see Section $6$ for the construction.

Similarly, we will record the framings of the $1$-dimensional moduli spaces (intervals) by a new $\CDP_*$-cochain called the \textit{frame assignment}---see Section $7$ for the construction. The $2$-dimensional boundaries consist of intervals forming a loop, which also come with a distinguished normal direction $\vec{v}$ pointing in the direction of where the $2$-dimensional moduli space would be attached. While we could similarly make computations in $\Tilde{\Omega}_{fr,A}^1$, it will be more convenient to use the following correspondence:

\begin{definition}\label{corresppath}Given a simple framed loop $(M, \vec{v}) \in \tilde{\Omega}_{fr,A}^1$, at each point $p \in M$, let $\vec{v}_p$ be the distinguished vector and $[v_1, \dots, v_{A-2}]$ be the frame, which we view as an orthonormal basis of the orthogonal complement of $T_pM \oplus \langle v_p \rangle$ by Convention $\ref{frameconv}$. We add $\vec{v}_p$ to the front of this frame to get $[\vec{v}_p, v_1, \dots, v_{A - 2}]$, and complete it to a positive orthonormal basis $[\overline{e}, \vec{v}_p, v_1, \dots, v_{A - 2}]$ for $\R^A$. Viewing this orthonormal basis as the orthogonal matrix with columns $\overline{e}, \vec{v}_p, v_1, \dots, v_{A-2}$, we have an element of $\SO(A)$. Finally, varying $p$ along $M$ produces a simple loop in $\SO(A)$, which is the loop corresponding to $(M, \vec{v})$ in $\SO(A)$.\end{definition}

\begin{lemma}\label{lem4}Let $A \geq 3$. A simple framed loop in $\mathbb{R}^A$ is framed null-cobordant if and only if the corresponding loop in $\SO(A)$ is not null-homotopic.\end{lemma}

\begin{proof}Note that this is similar to \cite{LS}. For $A \geq 3$, $\pi_1(\SO(A)) \cong \Z/2$, so every simple loop in $\SO(A)$ corresponds to one of two loops. One such loop is a circle in $\R^2$ with the product framing of the outward frame with the standard frame of $\R^{A-2}$, and the other loop is the same circle with a full twist of the other frame (see below diagrams).

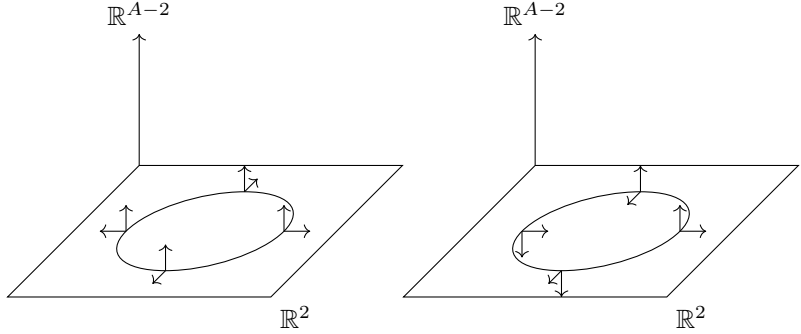
\begin{center}\begin{tikzpicture}[x={(1,0)}, y={(0,1)}, z={(-0.5, -0.5)}, scale=3.5]

\draw (0, 0, 0) -- (1, 0, 0) -- (1, 0, 1) -- (0, 0, 1) -- (0, 0, 0);
\draw[->] (0, 0, 0) -- (0, 0.5, 0) node[anchor=south]{$\R^{A - 2}$};
\node[anchor=north west] at (1, 0, 1) {$\mathbb{R}^2$};

\draw (0.5, 0, 0.5) [y={(0,0,1)}] circle (0.3);

\draw[->] (0.8, 0, 0.5) -- (0.9, 0, 0.5);
\draw[->] (0.8, 0, 0.5) -- (0.8, 0.1, 0.5);
\draw[->] (0.5, 0, 0.8) -- (0.5, 0, 0.9);
\draw[->] (0.5, 0, 0.8) -- (0.5, 0.1, 0.8);
\draw[->] (0.2, 0, 0.5) -- (0.1, 0, 0.5);
\draw[->] (0.2, 0, 0.5) -- (0.2, 0.1, 0.5);
\draw[->] (0.5, 0, 0.2) -- (0.5, 0, 0.1);
\draw[->] (0.5, 0, 0.2) -- (0.5, 0.1, 0.2);

\end{tikzpicture}\begin{tikzpicture}[x={(1,0)}, y={(0,1)}, z={(-0.5, -0.5)}, scale=3.5]

\draw (0, 0, 0) -- (1, 0, 0) -- (1, 0, 1) -- (0, 0, 1) -- (0, 0, 0);
\draw[->] (0, 0, 0) -- (0, 0.5, 0) node[anchor=south]{$\R^{A - 2}$};
\node[anchor=north west] at (1, 0, 1) {$\mathbb{R}^2$};

\draw (0.5, 0, 0.5) [y={(0,0,1)}] circle (0.3);

\draw[->] (0.8, 0, 0.5) -- (0.9, 0, 0.5);
\draw[->] (0.8, 0, 0.5) -- (0.8, 0.1, 0.5);
\draw[->] (0.5, 0, 0.8) -- (0.5, 0, 0.9);
\draw[->] (0.5, 0, 0.8) -- (0.5, -0.1, 0.8);
\draw[->] (0.2, 0, 0.5) -- (0.3, 0, 0.5);
\draw[->] (0.2, 0, 0.5) -- (0.2, -0.1, 0.5);
\draw[->] (0.5, 0, 0.2) -- (0.5, 0, 0.3);
\draw[->] (0.5, 0, 0.2) -- (0.5, 0.1, 0.2);

\end{tikzpicture}\end{center}

We can see that the first loop is framed null-cobordant and not null-homotopic as an element of $\SO(A)$.\end{proof}

\section{Describing Framed Paths in SO}

Fix $A \geq 3$, so that $\pi_1(\SO(A)) \simeq \Z/2$. There are two homotopy classes of paths in $\SO(A)$ from a given point to another, which together form the nontrivial loop. We will distinguish them by calling one of them \textit{preferred}. We list the following preferred paths between points that will appear later.

We fix the following convention. Let $e_1, \dots, e_{A}$ be the unit vectors such that $[e_1, \dots, e_A]$ is the standard positive frame on $\R^A$.

\begin{definition}Let $v_1 \in \{e_1, -e_1\}$ and $v_2 \in \{e_2, -e_2\}$. The short preferred path between the points $[v_1, v_2]$ and $[-v_2, v_1]$ in $\SO(2)$ is the rotation by $\pi/2$.

The short preferred path between $[v_1, \dots, v_{k-1}, v_k, v_{k+1}, \dots, v_{l-1}, v_l, v_{l+1}, \dots, v_A]$ and \\$[v_1, \dots, v_{k-1}, -v_l, v_{k+1}, \dots, v_{l-1}, v_k, v_{l+1}, \dots, v_A]$ in $\SO(A)$ is the rotation by $\pi/2$ in the $(v_k, v_l)$-plane.\end{definition}

\begin{definition}Let $v_1 \in \{e_1, -e_1\}$ and $v_2 \in \{e_2, -e_2\}$. The short preferred path between the points $[v_1, v_2]$ and $[v_2, -v_1]$ in $\SO(2)$ is the rotation by $\pi/2$.

The short preferred path between $[v_1, \dots, v_{k-1}, v_k, v_{k+1}, \dots, v_{l-1}, v_l, v_{l+1}, \dots, v_A]$ and \\$[v_1, \dots, v_{k-1}, v_l, v_{k+1}, \dots, v_{l-1}, -v_k, v_{l+1}, \dots, v_A]$ in $\SO(A)$ is the rotation by $\pi/2$ in the $(v_k, v_l)$-plane.\end{definition}

A loop in $\SO(A)$ is nullhomotopic if it lifts to a loop in the universal cover $\mathit{Spin}(A)$. For convenience we will instead consider the more general lifts of loops in $O(A)$ to its universal cover $\Pin(A)$.

We use the representation of $\Pin(A)$ as units in the Clifford algebra, for which we omit the details; the reader is advised to see \cite{LM}, \cite{WOIT}, and \cite{CLIFFALG} for details. The Clifford algebra is the quotient of the tensor algebra generated by the unit vectors $e_1, \dots, e_A$ of $\R^A$ by the relations $e_i^2 = 1$ and $e_i e_j = -e_j e_i$ for $i \neq j$.

For any unit vector $v \in \R^A$, let $M_v \in O(A)$ be the reflection which maps $v \mapsto -v$ and fixes its orthogonal complement. Then $M_v$ lifts to the elements $\pm v \in \Pin(A)$ as elements of the Clifford algebra. The short preferred path in the $(e_i, e_j)-$ plane which sends $e_i \rightarrow e_j$ and $e_j \rightarrow -e_i$ can be parametrized as
\begin{align*}M_{\cos\theta e_i + \sin \theta e_j} M_{e_i} \text{ for } 0 \leq \theta \leq \frac{\pi}{4}\end{align*}
since the endpoint, a rotation by $\frac{\pi}{2}$, can be written as the composition of reflections $M_{1/\sqrt{2}(e_i + e_j)}M_{e_i}$ as in Figure $\ref{figrot}$. Lifting this path to $\Pin(A)$ starting at $\Id$, the endpoint lifts to $\frac{1}{\sqrt{2}}(1 - e_ie_j) \in \Pin(A)$ as an element of the Clifford algebra.

\begin{figure}
\begin{tikzpicture}[scale=3]

\begin{scope}[xshift = 0, yshift = 0]
\draw (0, 0) rectangle (1, 1);
\draw[color=red] (-0.25, 0) -- (1.25, 0);
\node at (0, 0) {
\begin{tikzpicture}

\draw[fill=white] (0,0) circle (0.125);

\node at (0,0) {\tiny $1$};

\end{tikzpicture}
};
\node at (1, 0) {
\begin{tikzpicture}

\draw[fill=white] (0,0) circle (0.125);

\node at (0,0) {\tiny $2$};

\end{tikzpicture}
};
\node at (1, 1) {
\begin{tikzpicture}

\draw[fill=white] (0,0) circle (0.125);

\node at (0,0) {\tiny $3$};

\end{tikzpicture}
};
\node at (0, 1) {
\begin{tikzpicture}

\draw[fill=white] (0,0) circle (0.125);

\node at (0,0) {\tiny $4$};

\end{tikzpicture}
};
\node[anchor=north] at (0.5, 0) {$e_1$};
\node[anchor=east] at (0, 0.5) {$e_2$};
\end{scope}

\begin{scope}[xshift = 1.5cm, yshift = -1cm]
\draw (0, 0) rectangle (1, 1);
\draw[color=red] (-0.25, 0.75) -- (0.25, 1.25);
\node at (0, 0) {
\begin{tikzpicture}

\draw[fill=white] (0,0) circle (0.125);

\node at (0,0) {\tiny $4$};

\end{tikzpicture}
};
\node at (1, 0) {
\begin{tikzpicture}

\draw[fill=white] (0,0) circle (0.125);

\node at (0,0) {\tiny $3$};

\end{tikzpicture}
};
\node at (1, 1) {
\begin{tikzpicture}

\draw[fill=white] (0,0) circle (0.125);

\node at (0,0) {\tiny $2$};

\end{tikzpicture}
};
\node at (0, 1) {
\begin{tikzpicture}

\draw[fill=white] (0,0) circle (0.125);

\node at (0,0) {\tiny $1$};

\end{tikzpicture}
};
\node[anchor=south] at (0.5, 1) {$e_1$};
\node[anchor=east] at (0, 0.5) {$-e_2$};
\end{scope}

\begin{scope}[xshift = 3cm, yshift = 0]
\draw (0, 0) rectangle (1, 1);
\node at (0, 0) {
\begin{tikzpicture}

\draw[fill=white] (0,0) circle (0.125);

\node at (0,0) {\tiny $4$};

\end{tikzpicture}
};
\node at (1, 0) {
\begin{tikzpicture}

\draw[fill=white] (0,0) circle (0.125);

\node at (0,0) {\tiny $1$};

\end{tikzpicture}
};
\node at (1, 1) {
\begin{tikzpicture}

\draw[fill=white] (0,0) circle (0.125);

\node at (0,0) {\tiny $2$};

\end{tikzpicture}
};
\node at (0, 1) {
\begin{tikzpicture}

\draw[fill=white] (0,0) circle (0.125);

\node at (0,0) {\tiny $3$};

\end{tikzpicture}
};
\node[anchor=north] at (0.5, 0) {$-e_1$};
\node[anchor=west] at (1, 0.5) {$e_2$};
\end{scope}

\end{tikzpicture}
\caption{The composition $M_{1/\sqrt{2}(e_1 + e_2)}M_{e_1}$ equals the rotation from $e_1$ to $e_2$. Here vertex 1 is the origin throughout.}\label{figrot}\end{figure}
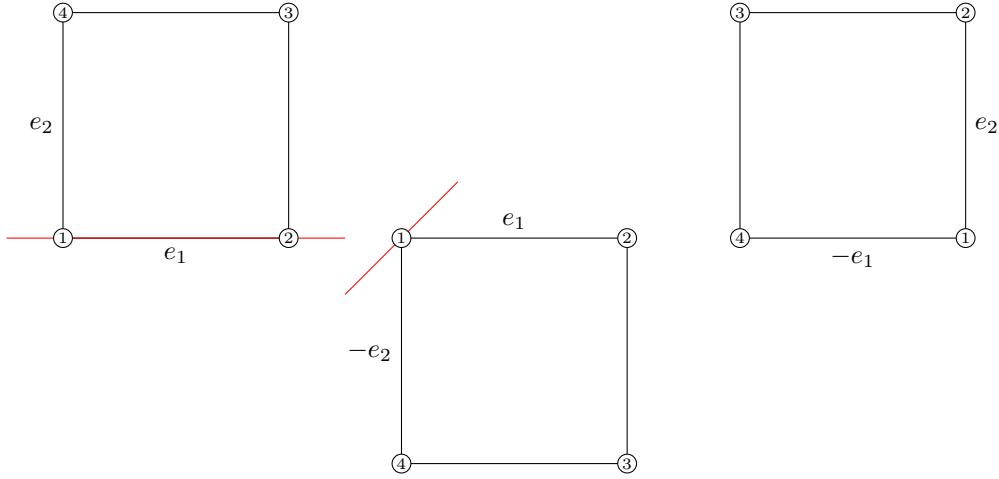

Finally, the endpoint of a loop of short preferred paths lifts to a product of the corresponding $\frac{1}{\sqrt{2}}(1 - e_ie_j)$, which is either $\pm 1$, and is $1$ if and only if the loop is nullhomotopic. We therefore have the following algorithm for computing the homotopy class of a given loop of short preferred paths:

\begin{algorithm}\label{multalgo} \begin{itemize}

\item Multiply, say, the first two terms of the lift of the loop.

\item Use the relations $e_i^2 = 1$ and $e_i e_j = -e_j e_i$ to write the resulting word in the Clifford algebra in terms of sorted words $e_{i_1}\dots e_{i_k}$ where $i_1 < \dots < i_k$.

\item Repeat with the next term until all terms are multiplied.

\item If the final result is positive, the loop is nullhomotopic, and if it is negative, the loop is not.

\end{itemize}

\end{algorithm}

In practice, we will often omit the $1/\sqrt{2}$ coefficients for convenience of computation, as they do not change the signs---a value of $-4$, for example, while not a valid element of $\mathit{Pin}(A)$, unambiguously represents the element $-1$. We keep the $1/\sqrt{2}$ coefficients in manual calculations for illustrative purposes, but not where implemented in a computer program (see https://github.com/ytao783/PinAlg for some examples).

It will be convenient to work out the homotopy classes of the following specific paths in $\SO(3), \SO(4)$. Since the inclusions $\SO(3), \SO(4) \hookrightarrow \SO(A)$ induce isomorphisms of the fundamental group, we can apply these results to loops in $\SO(A)$ constant in all but three or four coordinates.

\begin{lemma}\label{lem3}

A rotation by $2\pi$ in a plane is not nullhomotopic in $\SO(3)$.\end{lemma}

\begin{proof}This is a well-known result, but we give the following proof to illustrate the utility of the $\Pin$ presentation.

By post-composition with some orientation-preserving isotopy of $\R^3$, we may assume that the plane of rotation is the $(e_1e_2)-$plane, that the base of the loop is $[e_1, e_2, e_3]$, and that the loop has the form
\begin{center}\begin{tikzcd}
{[e_1, e_2, e_3]} \arrow[rr, no head] \arrow[dd, no head] &  & {[-e_2, e_1, e_3]} \arrow[dd, no head] \\
                                                     &  &                                   \\
{[e_2, -e_1, e_3]} \arrow[rr, no head]                    &  & {[-e_1, -e_2, e_3]}                     
\end{tikzcd}\end{center}
where each edge is a short preferred path. The corresponding $\Pin(3)$ element is given by
\begin{align*}
\left(\frac{1}{\sqrt{2}}(1 - e_1e_2)\right)^4 = \left(\frac{1}{2} - e_1 e_2 + \frac{1}{2} e_1 e_2 e_1 e_2\right)^2 = (-e_1 e_2)^2 = e_1 e_2 e_1 e_2 = -1
\end{align*}
so the loop is homotopically nontrivial.\end{proof}

\begin{lemma}\label{lem5}Suppose that all edges of the following loop in $\SO(4)$ are short preferred paths.
\begin{center}\begin{tikzcd}
{[e_1, e_2, e_3, e_4]} \arrow[dd, no head] \arrow[rr, no head] &  & {[-e_2, e_1, e_3, e_4]}                      \\
                                                               &  &                                              \\
{[e_1, e_2, -e_4, e_3]} \arrow[rr, no head]                    &  & {[-e_2, e_1, -e_4, e_3]} \arrow[uu, no head]
\end{tikzcd}\end{center}
Then the loop corresponds to the element $0 \in \pi_1(\SO(A)) \cong \Z/2$; that is, it is nullhomotopic.\end{lemma}

\begin{proof}
Let $s, t \in [0, 1]$ and consider the homotopy of paths where the vertical segments are rotations by $s \pi/2$ in the $(e_1, e_2)$-plane and the horizontal segments are rotations by $t \pi/2$ in the $(e_3, e_4)$-plane. The given loop is the $s = t = 1$ side and the $s = t = 0$ side is the identity.\end{proof}

While it will be possible to describe all of our framed paths using short preferred paths, it will be convenient to combine two short preferred paths in the same coordinate plane into one rotation by $\pi$. There are two different ways to rotate by $\pi$ in a coordinate plane since their concatenation is a nontrivial loop by Lemma $\ref{lem3}$, so we must choose a preferred one.

\begin{definition}Let $v_1 \in \{e_1, -e_1\}$ and $v_2 \in \{e_2, -e_2\}$. The long preferred path with respect to $e_1$ between the points $[v_1, v_2]$ and $[-v_1, -v_2]$ in $\SO(2)$ is the concatenation of the short preferred paths from both points to $[-e_2, e_1]$; in other words, it is the rotation by $\pi$ where the second vector equals $e_1$ halfway through.

The long preferred path with respect to $e_n$ between the points $[v_1, \dots, v_k, \dots, v_l, \dots, v_A]$ and \\$[v_1, \dots, -v_k, \dots, -v_l, \dots, v_A]$ in $\SO(A)$, where $v_k = \pm e_n$, is the concatenation of the preferred paths from both points to $[v_1, \dots, \mp v_l, \dots, e_n, \dots, v_A]$.\end{definition}

In particular, the long preferred path from $[v_1, \dots, e_n, \dots, e_m, \dots, v_A]$ to $[v_1, \dots, -e_n, \dots, -e_m, \dots, v_A]$ with respect to $e_n$ lifts to 
\begin{align*}\left(\frac{1}{\sqrt{2}}(1 - e_me_n)\right)^2 = \frac{1}{2} + \frac{1}{2}e_me_ne_me_n - e_me_n = e_n e_m
\end{align*}and that the long preferred path from $[v_1, \dots, e_n, \dots, -e_m, \dots, v_A]$ to $[v_1, \dots, -e_n, \dots, e_m, \dots, v_A]$ with respect to $e_n$ lifts to 
\begin{align*}\left(\frac{1}{\sqrt{2}}(1 - e_ne_m)\right)^2 = \frac{1}{2} + \frac{1}{2}e_ne_me_ne_m - e_ne_m = e_m e_n.\end{align*}

\begin{lemma}\label{lem1}Let $x \in \Z/2$, $v_1 \in \{ e_1, -e_1 \}$, and $v_2 \in \{ e_2, -e_2 \}$. Suppose that all horizontal edges of the following loops in $\SO(3)$ are short preferred paths, and all vertical edges are long preferred paths with respect to the first direction---that is, if $v_1$ appears first, it is with respect to $e_1$, and if $v_2$ appears first, it is with respect to $e_2$.
\begin{center}\begin{tikzcd}
{[\pm v_1, v_2, (-1)^x e_3]} \arrow[rr, no head] \arrow[dd, no head] &  & {[\mp v_2, v_1, (-1)^x e_3]} \arrow[dd, no head] \\
                                                                 &  &                                              \\
{[\pm (-1)^{x} v_1, v_2, e_3]} \arrow[rr, no head]                            &  & {[\mp (-1)^{x} v_2, v_1, e_3]}                           
\end{tikzcd} and \begin{tikzcd}
{[\pm v_1, v_2, (-1)^x e_3]} \arrow[rr, no head] \arrow[dd, no head] &  & {[v_2, \mp v_1, (-1)^x e_3]} \arrow[dd, no head] \\
                                                                 &  &                                              \\
{[\pm (-1)^{x}v_1, v_2, e_3]} \arrow[rr, no head]                            &  & {[(-1)^{x} v_2, \mp v_1, e_3]}                           
\end{tikzcd}\end{center}
Then both loops correspond to the element $x \in \pi_1(\SO(3)) \cong \Z/2$.\end{lemma}

\begin{proof}In the case when $x = 0$, the vertical edges are the identity paths and the horizontal edges are the same path, so both loops are nullhomotopic by contracting the horizontal path.

In the case where $x = 1, v_1 = e_1, v_2 = e_2$, the corresponding $\Pin(3)$ element of the left loop is given by
\begin{align*}
&\left(\frac{1}{\sqrt{2}}(1 - e_2e_1)\right)(e_3 e_2)\left(\frac{1}{\sqrt{2}}(1 - e_1e_2)\right)(e_1e_3) \\&= \left(\frac{1}{\sqrt{2}}(e_3 e_2 + e_1e_3)\right)\left(\frac{1}{\sqrt{2}}(e_1 e_3 -e_2e_3)\right) \\&= \frac{1}{2}(e_3 e_2 e_1 e_3 - e_3 e_2 e_2 e_3 + e_1 e_3 e_1 e_3 - e_1 e_3 e_2 e_3) = -1
\end{align*}
so the loop is homotopically nontrivial. The other cases follow with a similar computation.\end{proof}

\begin{lemma}\label{lem2}Let $x, y \in \Z/2$. Suppose that all edges of the following loop in $\SO(3)$ are long preferred paths with respect to $e_1$.
\begin{center}\begin{tikzcd}
{[(-1)^{x + y} e_1, (-1)^x e_2, (-1)^y e_3)]} \arrow[rr, no head] \arrow[dd, no head] &  & {[(-1)^{y} e_1, e_2, (-1)^y e_3)]} \arrow[dd, no head] \\
                                                                             &  &                                                   \\
{[(-1)^x e_1, (-1)^x e_2, e_3)]} \arrow[rr, no head]                            &  & {[e_1, e_2, e_3)]}                           
\end{tikzcd}\end{center}
Then the loop corresponds to the element $xy \in \pi_1(\SO(3)) \cong \Z/2$.\end{lemma}

\begin{proof}In the case when $x = 0$, the horizontal edges are the identity paths and the vertical edges are the same path, so the loop is nullhomotopic by contracting the vertical path. Likewise, when $y = 0$, the vertical edges are identity and the loop is nullhomotopic by contracting the horizontal path.

In the case where $x = y = 1$, the corresponding $\Pin(3)$ element is given by
\begin{align*}
(e_2 e_1)(e_3 e_1)(e_1 e_2)(e_1 e_3) = -1
\end{align*}
so the loop is homotopically nontrivial.

Alternatively, each preferred path in this diagram is exactly the standard path given by \cite{LS}, with the exception of the leftmost path in the case that $x, y = 1 \pmod{2}$. In that case, the leftmost path is the opposite of the \cite{LS}-standard path. Since the \cite{LS}-standard paths form a nullhomotopic loop in $\SO(3)$, this loop is nullhomotopic if and only if $x, y$ are not both $1$.\end{proof}

\section{0- and 1-dimensional Moduli Spaces}

We now explicitly construct the $0$-, $1$-, and $2$-dimensional moduli spaces, which are needed to calculate $Sq^2$. This section will be devoted solely to the $0$- and $1$-dimensional moduli spaces.

First, we recap the implications of the local model. As the differential of $\CDP_*$ captures the boundary of each moduli space by construction, we consider the generators of $\CDP_1$, which must correspond to the $0$-dimensional moduli spaces. The moduli spaces of type 1.1 triples $(R, \vec{0}, \vec{0})$ are embedded in $\R^d$, with $d$ external frames and no internal frames.

The moduli spaces of type 1.0 triples $(c_x, N\vec{e}_j, (N)_j)$ are embedded in $\R_+^{2N-1} \times \R^{2Nd}$ and have $2N - 1$ internal frames and $2Nd$ external frames. The internal frames correspond to horizontal or vertical bubbling from higher strata. The point itself is the stratum $Z(0, N, 0; (N))$, which has the standard framing $\delta v_1, -\delta \Delta_1, \delta v_2, \dots, -\delta \Delta_{N-1}, \delta v_n$ since $\lambda$ is a length $1$ partition. Each $\delta v_j$ represents the direction in which the $j^{th}$ bubble has merged as the vertical annulus $V_j$, while each $-\delta v_j$ represents the direction in which the $j^{th}$ bubble has merged as the horizontal annulus $H_j$.

We use our sign assignment on $\CDP_*$ to fix the framing of each $0$-dimensional moduli space.

\begin{conv}We denote the unit vectors in $\mathbb{E}_{l}^{d} = \R_+^l \times \R^{d(l+1)}$ by $f_1, \dots, f_l, e_1, \dots, e_{d(l+1)}$.\end{conv}

\begin{definition}Given a sign assignment $s$
\begin{itemize}
\item The moduli space of a Type 1.1 triple $\mathcal{M}(R, \vec{0}, \vec{0})$ is framed according to $s$ if its framing is \\$[(-1)^{s(R, 0, 0)} e_1, e_2, \dots, e_d]$.

\item The moduli space of a Type 1.0 triple $\mathcal{M}(c_x, N\vec{e}_j, (N)_j)$ is framed according to $s$ if its framing is $[f_1, \dots, f_{2N-1}, (-1)^{s(c_x, N\vec{e}_j, (N)_j)} e_1, e_2, \dots, e_{2Nd}]$.
\end{itemize}

(Note that in both cases, the sign is placed on the first external frame vector.)\end{definition}

Before the $1$-dimensional moduli spaces can be framed, the framings of their boundaries must be understood. Most boundary points are the products of $0$-dimensional moduli spaces in the same strata, but in the cases of two groups of bubbles of the same type merging or an annulus bubbling, one endpoint is a $0$-dimensional moduli space in a lower stratum. For ease of comparison, in these cases we consider a \textit{smoothening} $\partial'\mathcal{M}(D, \vec{N}, \vec{\lambda})$ of the $1$-dimensional boundary by pushing off from that endpoint in the $\vec{v}$ direction while preserving frames.

Given a framing of each point, the framing of their product is forced by the local model. Their internal frame is the product of the internal frames, but where the internal frames of type $1$ bubbles (the ones that arise from the annuli $V_1$ and $H_1$) come first, type $2$ bubbles next, and so on. Their external frames is the product of the external frames, and goes after the internal frame as always.

\begin{lemma}\label{lemsame}Given a sign assignment $s$ with all $s_j = 0$, 
the framing where we frame every $0$-dimensional moduli space according to it is coherent.\end{lemma}

\begin{proof}It suffices to check the boundaries of the $1$-dimensional moduli spaces corresponding to the generators of $\CDP_2$---specifically, we check triples of types 2.0a, 2.0b, 2.1, 2.2a, and 2.2b. The remainder of this section is devoted to these cases.\end{proof}

\begin{conv}For sign assignments $s$ satisfying the conditions of Lemma $\ref{lemsame}$, we have that $s(c_x, N \vec{e}_j, (N)_j) = 0$ for all generators $x$ and positive integers $N$. So for such sign assignments, we abbreviate $s(R, \vec{0}, \vec{0})$ to $s(R)$.\end{conv}

Once a 1-dimensional moduli space is shown to be frameable, its frame corresponds to a path in $\SO(A)$ as in the previous section. One of these paths we will call preferred. For each case, we will also give the preferred path.

\textbf{Type 2.0a}: In this case, we consider
\begin{align*}\mathcal{M}(c_x, (N_1 + N_2) \vec{e}_j, (N_1, N_2)) \subseteq \mathbb{E}_{2N_1+2N_2-1}^{d}.\end{align*}
One endpoint of the moduli space is the point
\begin{align*}P = \mathcal{M}(c_x, N_1 \vec{e}_j, (N_1)) \times \mathcal{M}(c_x, N_1 \vec{e}_j, (N_1))\end{align*}
which is a point in $\Tilde{\Omega}_{fr,A}^0$ with
\begin{align*}&\vec{v} = f_{2N_1} \text{ (the direction where this moduli space is broken) and frame} \\&[f_1, \dots, f_{2N_1-1}, f_{2N_1+1}, \dots, f_{2N_1+2N_2-1}, e_1, \dots, e_{2N_1 d + 2N_2 d}]. \\& (2(N_1 + N_2) - 2 \text{ internal frames and } 2(N_1 + N_2)d \text{ external frames})\end{align*}
In $\Omega_{fr,A}^0$, $P$ has the frame
\begin{align*}[f_{2N_1}, f_1, \dots, f_{2N_1-1}, f_{2N_1+1}, \dots, f_{2N_1+2N_2-1}, e_1, \dots, e_{2N_1 d + 2N_2 d}]\end{align*}
and in $\SO((2N_1 + 2N_2 - 1)(d + 1) - 1)$, it corresponds to the point
\begin{align*}[-f_{2N_1}, f_1, \dots, f_{2N_1-1}, f_{2N_1+1}, \dots, f_{2N_1+2N_2-1}, e_1, \dots, e_{2N_1 d + 2N_2 d}].\end{align*}
The other endpoint is is the point 
\begin{align*}Q = \mathcal{M}(c_x, (N_1 + N_2) \vec{e}_j, (N_1 + N_2))\end{align*}
which is a point in $\Tilde{\Omega}_{fr,A}^0$ with
\begin{align*}&[f_1, \dots, f_{2N_1+2N_2-1}, e_1, \dots, e_{2N_1 d + 2N_2 d}]\\& (2(N_1 + N_2) -1 \text{ internal frames and } 2(N_1 + N_2)d \text{ external frames})\end{align*}
Note that $Q$ has one more internal frame than $P$, which is a result of it lying in the lower stratum $Z(0, N_1 + N_2, 0; (N_1 + N_2))$ than the stratum $Z(0, N_1 + N_2, 0; (N_1, N_2))$ where the rest of the moduli space lies. The local model for the moduli space is identified with $\R_+$ with strata $\{ 0 \}$ (lower) and $(0, \infty)$ (higher), so $Q' \in \partial' \mathcal{M}(c_x, (N_1 + N_2) \vec{e}_j, (N_1, N_2))$ is the translate of $Q$ towards $P$. Compared to $Q$, $Q'$ lacks the internal frame $f_{2N_1} = -\delta \Delta_{N_1}$ and instead has $\vec{v} = -f_{2N_1}$ (pointing towards $P$). So in $\Omega_{fr,A}^0$, $Q'$ has frame
\begin{align*}[-f_{2N_1}, f_1, \dots, f_{2N_1-1}, f_{2N_1+1}, \dots, f_{2N_1+2N_2-1}, e_1, \dots, e_{2N_1 d + 2N_2 d}]\end{align*}
and in $\SO((2N_1 + 2N_2 - 1)(d + 1) - 1)$, it corresponds to the point
\begin{align*}[-f_{2N_1}, f_1, \dots, f_{2N_1-1}, f_{2N_1+1}, \dots, f_{2N_1+2N_2-1}, e_1, \dots, e_{2N_1 d + 2N_2 d}].\end{align*}
We see immediately that $P$ and $Q'$ are oppositely framed in $\Omega_{fr,A}^0$, and that they correspond to the same point in $\SO((2N_1 + 2N_2 - 1)(d + 1) - 1)$. In this case, the preferred framing of the interval between $P$ to $Q'$ corresponds to the identity path at the frame of $P$. The preferred path from $Q'$ (or equivalently $P$) to $Q$ follows the short preferred paths $[v_1, v_2] \rightarrow [v_2, -v_1]$ with $v_1 = f_{2N_1}$ until $f_{2N_1}$ is in the correct coordinate.

\begin{figure}
\begin{tikzpicture}[scale=3]

\filldraw[color=gray] (0, 0) .. controls (-0.125, 0.25) .. (0, 0.5) .. controls (0.125, 0.25) .. (0, 0);
\filldraw[color=gray] (0, 0.5) .. controls (-0.125, 0.75) .. (0, 1) .. controls (0.125, 0.75) .. (0, 0.5);
\filldraw[color=gray] (2, 0) .. controls (1.75, 0.5) .. (2, 1) .. controls (2.25, 0.5) .. (2, 0);

\draw (0.1, 0.5) node {\textbullet};
\draw (0.1, 0.5) node[anchor=south west] {$P$};

\draw (0.1, 0.5) -- (1.75, 0.5);

\draw (1.75, 0.5) node {\textbullet};
\draw (1.5, 0.5) node {\textbullet};
\draw (1.5, 0.5) node[anchor=south west] {$Q'$};
\draw (1.75, 0.5) node[anchor=north east] {$Q$};

\draw (0.2, 0.75) circle (0.1);
\draw (0.2, 0.25) circle (0.1);
\draw (2.295, 0.5) circle (0.1);

\node at (0.2, 0.75) {\small $2_j$};
\node at (0.2, 0.25) {\small $1_j$};
\node at (2.295, 0.5) {\small $3_j$};

\draw[->] (1, 1) -- (1.45, 0.55);

\filldraw[color=gray] (1, 1) .. controls (0.875, 1.25) .. (1, 1.5) .. controls (1.125, 1.25) .. (1, 1);
\draw (1.14, 1.15) circle (0.065);
\draw (1.14, 1.35) circle (0.065);
\node at (1.14, 1.15) {\small $1_j$};
\node at (1.14, 1.35) {\small $2_j$};

\end{tikzpicture}
\begin{tikzpicture}[x={(1,0)}, y={(0,1)}, z={(-0.5, -0.5)}, scale=3.5]

\draw (0.5, 0, 0.5) node {\textbullet};
\draw (0.5, 0, 0.5) node[anchor=north east] {$P$};

\draw (0.5, 0, 0.5) -- (0.5, 0.75, 0.5);

\draw (0.5, 0.6, 0.5) node {\textbullet};
\draw (0.5, 0.75, 0.5) node {\textbullet};
\draw (0.5, 0.6, 0.5) node[anchor=north east] {$Q'$};
\draw (0.5, 0.75, 0.5) node[anchor=south east] {$Q$};

\draw[->] (0,0,0) -- (1,0,0) node[anchor=north]{$\R_{+(1)} \times \R_{+(2)} \times \R_{+(3)} \times \R_{+(5)}$};
\draw[->] (0,0,0) -- (0,1,0) node[anchor=south]{$\R_{+(4)}$};
\draw[->] (0,0,-1) -- (0,0,1) node[anchor=south east]{$\R^{6d}$};

\end{tikzpicture}
\caption{Type 2.0a, in the case where $N_1 = 2, N_2 = 1$. Here $N_j$ denotes a cluster of $N$ type $j$ bubbles, and $\R_{+(j)}$ denotes the $j^{th}$ $\R_+$ coordinate.}\label{fig1}\end{figure}
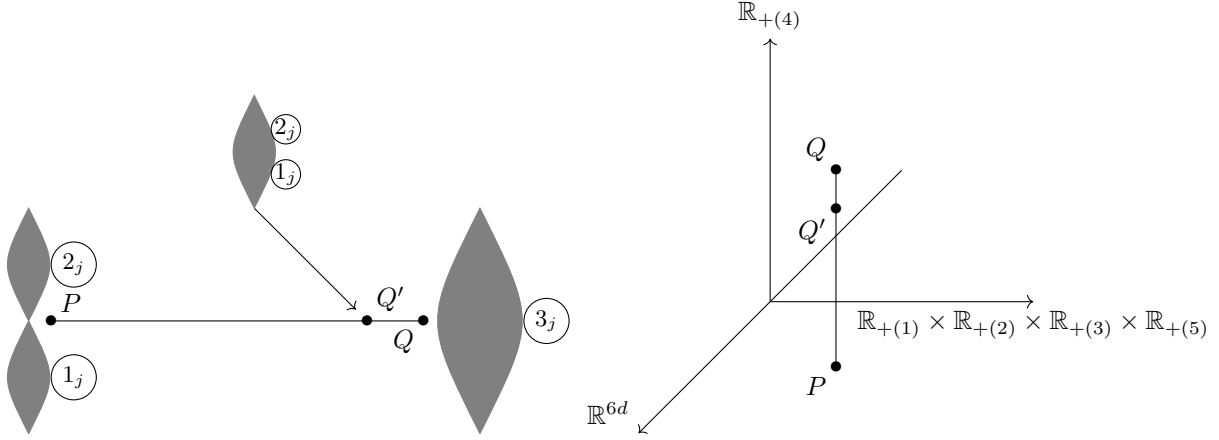

\textbf{Type 2.0b}: In this case, we consider
\begin{align*}\mathcal{M}(c_x, N_1 \vec{e}_j + N_2 \vec{e}_k, (N_1), (N_2)) \subseteq \mathbb{E}_{2N_1+2N_2-1}^{d}\end{align*}
where without loss of generality $j < k$. In this case, the endpoint
\begin{align*}P = \mathcal{M}(c_x, N_1 \vec{e}_j, (N_1)) \times \mathcal{M}(c_x, N_2 \vec{e}_k, (N_2))\end{align*}
is a point in $\Tilde{\Omega}_{fr,A}^0$ with
\begin{align*}&\vec{v} = f_{2N_1} \text{ and frame} \\&[f_1, \dots, f_{2N_1 - 1}, f_{2N_1 + 1}, \dots, f_{2N_1+2N_2-1}, e_1, \dots, e_{2N_1 d + 2N_2 d}]\\& (2(N_1 + N_2) - 2 \text{ internal frames and } 2(N_1 + N_2)d \text{ external frames})\end{align*}
In $\Omega_{fr,A}^0$, $P$ has the frame
\begin{align*}[f_{2N_1}, f_1, \dots, f_{2N_1 - 1}, f_{2N_1 + 1}, \dots, f_{2N_1+2N_2-1}, e_1, \dots, e_{2N_1 d + 2N_2 d}]\end{align*}
and in $\SO((2N_1 + 2N_2 - 1)(d + 1) - 1)$, $P$ corresponds to the point
\begin{align*}[-f_{2N_1}, f_1, \dots, f_{2N_1 - 1}, f_{2N_1 + 1}, \dots, f_{2N_1+2N_2-1}, e_1, \dots, e_{2N_1 d + 2N_2 d}]\end{align*}
The endpoint
\begin{align*}Q = \mathcal{M}(c_x, N_2 \vec{e}_k, (N_2)) \times \mathcal{M}(c_x, N_1 \vec{e}_j, (N_1))\end{align*}
is a point in $\Tilde{\Omega}_{fr,A}^0$ with
\begin{align*}&\vec{v} = f_{2N_2} \text{ and frame} \\&[f_{2N_2 + 1}, \dots, f_{2N_1 + 2N_2 - 1}, f_{1}, \dots, f_{2N_2-1}, e_1, \dots, e_{2N_1 d + 2N_2 d}]\\& (2(N_1 + N_2) - 2 \text{ internal frames and } 2(N_1 + N_2)d \text{ external frames})\end{align*}
In $\Omega_{fr,A}^0$, $Q$ has the frame
\begin{align*}[f_{2N_2}, f_{2N_2 + 1}, \dots, f_{2N_1 + 2N_2 - 1}, f_{1}, \dots, f_{2N_2-1}, e_1, \dots, e_{2N_1 d + 2N_2 d}].\end{align*}
and in $\SO((2N_1 + 2N_2 - 1)(d + 1) - 1)$, $Q$ corresponds to the point
\begin{align*}[f_{2N_2}, f_{2N_2 + 1}, \dots, f_{2N_1 + 2N_2 - 1}, f_{1}, \dots, f_{2N_2-1}, e_1, \dots, e_{2N_1 d + 2N_2 d}].\end{align*}
We can see that the $\Omega_{fr,A}^0$ frames are opposite since $P$ takes an odd number ($2N_1 - 1$) of switches to get to $[f_1, \dots, f_{2N_1+2N_1-1}, e_1, \dots, e_{2N_1 d + 2N_2 d}]$ and $Q$ takes an even number ($2N_2(2N_1 - 1)$) of switches to get to the same frame. The preferred framing is given as follows. Starting from $P$, we follow the short preferred paths $[v_1, v_2] \rightarrow [v_2, -v_1]$ with $v_1 = f_{2N_1}$ to reach the point
\begin{align*}[f_1, f_2 ,\dots, f_{2N_1+2N_2-1}, e_1, \dots, e_{2N_1 d + 2N_2 d}].\end{align*}
We then follow the short preferred paths $[v_1, v_2] \rightarrow [-v_2, v_1]$ with $v_2 = f_{2N_1+2N_-1}$ to obtain
\begin{align*}[f_{2N_1+2N_2-1}, f_1, f_2, \dots, f_{2N_1+2N_2-2}, e_1, \dots, e_{2N_1 d + 2N_2 d}]\end{align*}
followed by short preferred paths $[v_1, v_2] \rightarrow [-v_2, v_1]$ with $v_2 = f_{2N_1+2N_2-2}$ until that vector is in front, and then repeating the same with $v_2 = f_{2N_1+2N_2-3}$ and so on, until we do the same with $v_2 = f_{2N_2}$ and finally arrive at the frame of $Q$.

\begin{figure}
\begin{tikzpicture}[scale=4]

\filldraw[color=gray] (0, 0) .. controls (-0.125, 0.25) .. (0, 0.5) .. controls (0.125, 0.25) .. (0, 0);
\filldraw[color=gray] (0, 0.5) .. controls (-0.125, 0.75) .. (0, 1) .. controls (0.125, 0.75) .. (0, 0.5);
\filldraw[color=gray] (2, 0) .. controls (1.875, 0.25) .. (2, 0.5) .. controls (2.125, 0.25) .. (2, 0);
\filldraw[color=gray] (2, 0.5) .. controls (1.875, 0.75) .. (2, 1) .. controls (2.125, 0.75) .. (2, 0.5);

\draw (0.1, 0.5) node {\textbullet};
\draw (0.1, 0.5) node[anchor=south west] {$P$};

\draw (0.1, 0.5) -- (1.9, 0.5);

\draw (1.9, 0.5) node {\textbullet};
\draw (1.9, 0.5) node[anchor=north east] {$Q$};

\draw (0.2, 0.75) circle (0.1);
\draw (0.2, 0.25) circle (0.1);
\draw (2.2, 0.75) circle (0.1);
\draw (2.2, 0.25) circle (0.1);

\node at (0.2, 0.75) {\small $1_j$};
\node at (0.2, 0.25) {\small $2_k$};
\node at (2.2, 0.75) {\small $2_k$};
\node at (2.2, 0.25) {\small $1_j$};

\end{tikzpicture}
\begin{tikzpicture}[x={(1,0)}, y={(0,1)}, z={(-0.5, -0.5)}, scale=3.5]

\draw (0, 0, 0) node {\textbullet};
\draw (0, 0, 0) node[anchor=south east] {$\R^{6d}$};

\draw (0, 0.5, 0.5) node {\textbullet};
\draw (0, 0.5, 0.5) node[anchor=north east] {$P$};

\draw (0, 0.5, 0.5) .. controls (0.4, 0.4, 0.5) .. (0.5, 0, 0.5);

\draw (0.5, 0, 0.5) node {\textbullet};
\draw (0.5, 0, 0.5) node[anchor=north east] {$Q$};

\draw[->] (0,0,0) -- (1,0,0) node[anchor=west]{$\R_{+(4)}$};
\draw[->] (0,0,0) -- (0,1,0) node[anchor=south]{$\R_{+(2)}$};
\draw[->] (0,0,0) -- (0,0,1) node[anchor=north west]{$\R_{+(1)} \times \R_{+(3)} \times \R_{+(5)}$};

\end{tikzpicture}
\caption{Type 2.0b, in the case where $N_1 = 1, N_2 = 2$, with all conventions as before.}\label{fig2}\end{figure}
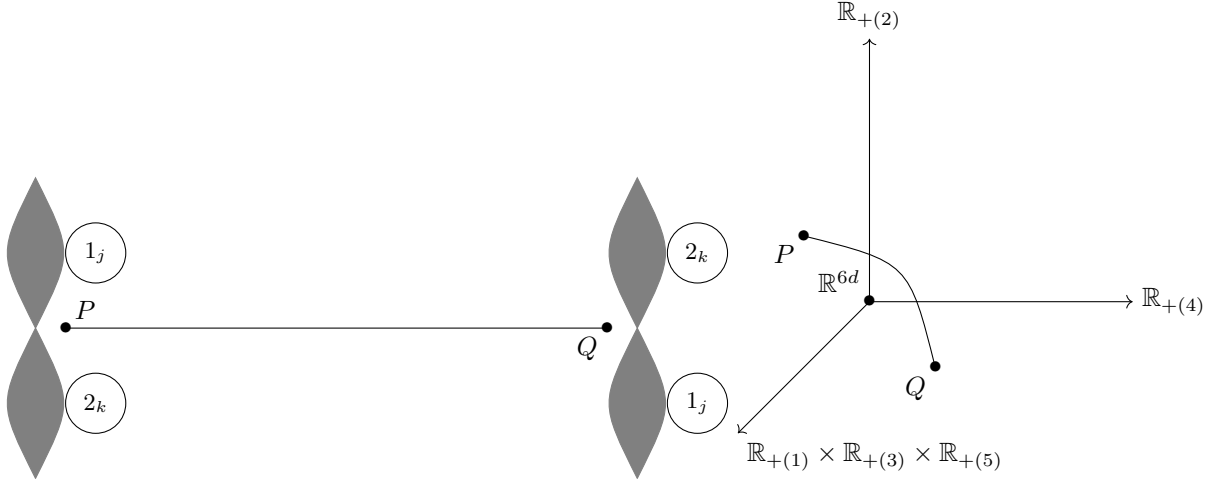

\textbf{Type 2.1}: $D$ is an index $1$ domain. In this case, $\vec{N} = N\vec{e}_j$ and $\lambda_j = (N)$, and the
\begin{align*}P = \mathcal{M}(D, 0, 0) \times \mathcal{M}(c_x, N \vec{e}_j, (N)) \subseteq \mathbb{E}_{2N}^{d}\end{align*}
endpoint is a point in $\Tilde{\Omega}_{fr,A}^0$ with
\begin{align*}&\vec{v} = f_{1}\text{ and frame} \\&[f_2, ..., f_{2N}, s(D) e_1, \dots, e_{2N d + d}]\\& (2N-1 \text{ internal frames and } 2(N+1)d \text{ external frames})\end{align*}
In $\Omega_{fr,A}^0$, $P$ has the frame
\begin{align*}[f_1, f_2, ..., f_{2N}, (-1)^{s(D)} e_1, \dots, e_{2N d + d}]\end{align*}
and in $\SO(2N(d + 1) - 1)$, $P$ corresponds to the point
\begin{align*}[(-1)^{s(D)}f_1, f_2, ..., f_{2N}, (-1)^{s(D)} e_1, \dots, e_{2N d + d}]\end{align*}
The endpoint
\begin{align*}Q = \mathcal{M}(c_x, N \vec{e}_j, (N)) \times \mathcal{M}(D, 0, 0)\end{align*}
is a point in $\Tilde{\Omega}_{fr,A}^0$ with
\begin{align*}&\vec{v} = f_{2N}\text{ and frame} \\&[f_1, ..., f_{2N - 1}, e_1, \dots, e_{2N d}, s(D) e_{2Nd + 1}, e_{2Nd + 2} \dots, e_{2Nd + d}].\\& (2N-1 \text{ internal frames and } 2(N+1)d \text{ external frames})\end{align*}
In $\Omega_{fr,A}^0$, $Q$ has the frame
\begin{align*}[f_{2N}, f_1, ..., f_{2N - 1}, e_1, \dots, e_{2N d}, (-1)^{s(D)} e_{2Nd + 1}, e_{2Nd + 2} \dots, e_{2Nd + d}].\end{align*}
and in $\SO(2N(d + 1) - 1)$, $Q$ corresponds to the point
\begin{align*}[(-1)^{s(D) + 1} f_{2N}, f_1, ..., f_{2N - 1}, e_1, \dots, e_{2N d}, (-1)^{s(D)} e_{2Nd + 1}, e_{2Nd + 2} \dots, e_{2Nd + d}].\end{align*}
We can see that the $\Omega_{fr,A}^0$ frames are opposite since they are related by odd number ($2N - 1$) of switches of the internal frames and an even number ($2$) of sign flips of external frames. The preferred framing is given as follows. Starting from $P$, we change the first external frame with the long preferred path with respect to the first direction $f_1$ to reach
\begin{align*}[f_1, f_2, ..., f_{2N}, e_1, \dots, e_{2N d + d}]\end{align*}
We then take the short preferred paths $[v_1, v_2] \rightarrow [-v_2, v_1]$ with $v_2 = f_{2N}$ to obtain
\begin{align*}[-f_{2N}, f_1 f_2, \dots, f_{2N - 1}, e_1, \dots, e_{2N d + d}].\end{align*}
Finally, we again use the long preferred path with respect to the first coordinate $f_{2N}$ to change the $(2Nd+1)^{st}$ external frame and reach
\begin{align*}[(-1)^{s(D) + 1} f_{2N}, f_1 f_2, \dots, f_{2N - 1}, e_1, \dots, e_{2N d}, (-1)^{s(D)} e_{2Nd + 1}, e_{2Nd + 2} \dots, e_{2Nd + d}]\end{align*}
which is the frame of $Q$.

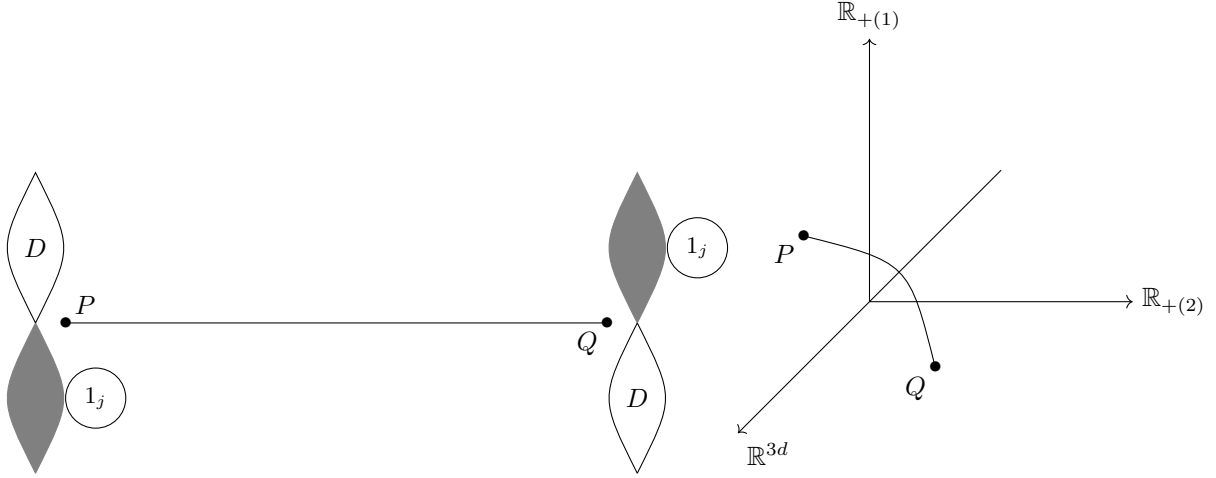
\begin{figure}
\begin{tikzpicture}[scale=4]

\filldraw[color=gray] (0, 0) .. controls (-0.125, 0.25) .. (0, 0.5) .. controls (0.125, 0.25) .. (0, 0);
\draw (0, 0.5) .. controls (-0.125, 0.75) .. (0, 1) .. controls (0.125, 0.75) .. (0, 0.5);
\draw (2, 0) .. controls (1.875, 0.25) .. (2, 0.5) .. controls (2.125, 0.25) .. (2, 0);
\filldraw[color=gray] (2, 0.5) .. controls (1.875, 0.75) .. (2, 1) .. controls (2.125, 0.75) .. (2, 0.5);

\draw (0.1, 0.5) node {\textbullet};
\draw (0.1, 0.5) node[anchor=south west] {$P$};

\draw (0.1, 0.5) -- (1.9, 0.5);

\draw (1.9, 0.5) node {\textbullet};
\draw (1.9, 0.5) node[anchor=north east] {$Q$};

\draw (0.2, 0.25) circle (0.1);
\draw (2.2, 0.75) circle (0.1);

\node at (0.2, 0.25) {\small $1_j$};
\node at (2.2, 0.75) {\small $1_j$};
\node at (0, 0.75) {$D$};
\node at (2, 0.25) {$D$};

\end{tikzpicture}
\begin{tikzpicture}[x={(1,0)}, y={(0,1)}, z={(-0.5, -0.5)}, scale=3.5]

\draw (0, 0.5, 0.5) node {\textbullet};
\draw (0, 0.5, 0.5) node[anchor=north east] {$P$};

\draw (0, 0.5, 0.5) .. controls (0.4, 0.4, 0.5) .. (0.5, 0, 0.5);

\draw (0.5, 0, 0.5) node {\textbullet};
\draw (0.5, 0, 0.5) node[anchor=north east] {$Q$};

\draw[->] (0,0,0) -- (1,0,0) node[anchor=west]{$\R_{+(2)}$};
\draw[->] (0,0,0) -- (0,1,0) node[anchor=south]{$\R_{+(1)}$};
\draw[->] (0,0,-1) -- (0,0,1) node[anchor=north west]{$\R^{3d}$};

\end{tikzpicture}
\caption{Type 2.1, in the case where $N = 1$, with all conventions as before.}\label{fig3}\end{figure}

In the last case, $D$ is an index $2$ domain. In this case $\vec{N} = 0$, so
\begin{align*}\mathcal{M}(D, 0, 0) \subseteq \mathbb{E}_1^d\end{align*}
and there are three further subcases:

\textbf{Type 2.2a}: $D$ is not an annulus. In this case, $D$ decomposes into rectangles in exactly two ways, $D = R * S = R' * S'$. The endpoint
\begin{align*}P = \mathcal{M}(R, 0, 0) \times \mathcal{M}(S, 0, 0)\end{align*}
is a point in $\Tilde{\Omega}_{fr,A}^0$ with
\begin{align*}&\vec{v} = f_{1}\text{ and frame} \\&[(-1)^{s(R)} e_1, e_2, \dots, e_d, (-1)^{s(S)} e_{d + 1}, e_{d + 2}, \dots, e_{2d}]
\\& (\text{No internal frames and } 2d \text{ external frames})\end{align*}
In $\Omega_{fr,A}^0$, $P$ has the frame
\begin{align*}[f_1, (-1)^{s(R)} e_1, e_2, \dots, e_d, (-1)^{s(S)} e_{d + 1}, e_{d + 2}, \dots, e_{2d}]\end{align*}
and in $\SO(2d + 1)$, $P$ corresponds to the point
\begin{align*}[(-1)^{s(R) + s(S)}f_1, (-1)^{s(R)} e_1, e_2, \dots, e_d, (-1)^{s(S)} e_{d + 1}, e_{d + 2}, \dots, e_{2d}]\end{align*}
The endpoint
\begin{align*}Q = \mathcal{M}(R', 0, 0) \times \mathcal{M}(S', 0, 0)\end{align*}
is a point in $\Tilde{\Omega}_{fr,A}^0$ with
\begin{align*}&\vec{v} = f_{1}\text{ and frame} \\&[(-1)^{s(R')} e_1, e_2, \dots, e_d, (-1)^{s(S')} e_{d + 1}, e_{d + 2}, \dots, e_{2d}].\\& (\text{No internal frames and } 2d \text{ external frames})\end{align*}
In $\Omega_{fr,A}^0$, $Q$ has the frame
\begin{align*}[f_1, (-1)^{s(R')} e_1, e_2, \dots, e_d, (-1)^{s(S')} e_{d + 1}, e_{d + 2}, \dots, e_{2d}].\end{align*}
and in $\SO(2d + 1)$, $Q$ corresponds to the point
\begin{align*}[(-1)^{s(R') + s(S')}f_1, (-1)^{s(R')} e_1, e_2, \dots, e_d, (-1)^{s(S')} e_{d + 1}, e_{d + 2}, \dots, e_{2d}].\end{align*}
We can see that the $\Omega_{fr,A}^0$ frames are opposite exactly when $s(R) + s(R') + s(S) + s(S') = 1 \pmod{2}$, in which case the frames are related by an odd number of sign switches. The preferred path is as follows. Near both endpoints, we change the $(d+1)^{st}$ frame to positive using the long preferred path with respect to $f_1$, then the first coordinate to positive using the long preferred path with respect to $f_1$ (so that in the middle, the frame is the positive $[f_1, e_1, \dots, e_{2d}]$).

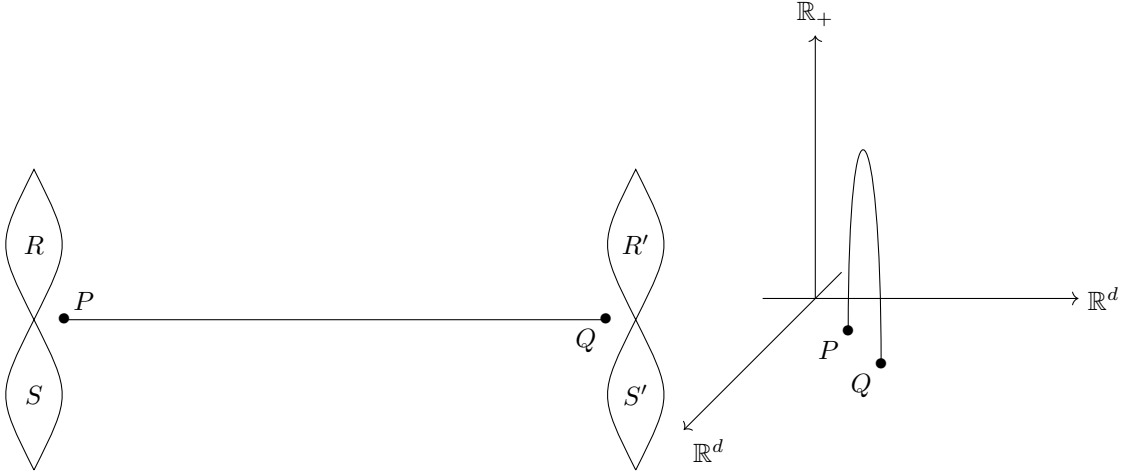
\begin{figure}
\begin{tikzpicture}[scale=4]

\draw (0, 0) .. controls (-0.125, 0.25) .. (0, 0.5) .. controls (0.125, 0.25) .. (0, 0);
\draw (0, 0.5) .. controls (-0.125, 0.75) .. (0, 1) .. controls (0.125, 0.75) .. (0, 0.5);
\draw (2, 0) .. controls (1.875, 0.25) .. (2, 0.5) .. controls (2.125, 0.25) .. (2, 0);
\draw (2, 0.5) .. controls (1.875, 0.75) .. (2, 1) .. controls (2.125, 0.75) .. (2, 0.5);

\draw (0.1, 0.5) node {\textbullet};
\draw (0.1, 0.5) node[anchor=south west] {$P$};

\draw (0.1, 0.5) -- (1.9, 0.5);

\draw (1.9, 0.5) node {\textbullet};
\draw (1.9, 0.5) node[anchor=north east] {$Q$};

\node at (0, 0.75) {$R$};
\node at (2, 0.25) {$S'$};
\node at (0, 0.25) {$S$};
\node at (2, 0.75) {$R'$};

\end{tikzpicture}
\begin{tikzpicture}[x={(1,0)}, y={(0,1)}, z={(-0.5, -0.5)}, scale=3.5]

\draw (0.25, 0, 0.25) node {\textbullet};
\draw (0.25, 0, 0.25) node[anchor=north east] {$P$};

\draw (0.25, 0, 0.25) .. controls (0.25, 1, 0.25) and (0.5, 1, 0.5) .. (0.5, 0, 0.5);

\draw (0.5, 0, 0.5) node {\textbullet};
\draw (0.5, 0, 0.5) node[anchor=north east] {$Q$};

\draw[->] (-0.2,0,0) -- (1,0,0) node[anchor=west]{$\R^d$};
\draw[->] (0,0,0) -- (0,1,0) node[anchor=south]{$\R_+$};
\draw[->] (0,0,-0.2) -- (0,0,1) node[anchor=north west]{$\R^d$};

\end{tikzpicture}
\caption{Type 2.2a.}\label{fig4}\end{figure}

\textbf{Type 2.2b}: $D = R * S$ is either the vertical annulus $V_j$ or the horizontal annulus $H_j$. We treat the cases separately, starting in the case when $D = V_j$. In this case, the endpoint
\begin{align*}P = \mathcal{M}(R, 0, 0) \times \mathcal{M}(S, 0, 0)\end{align*}
is a point in $\Tilde{\Omega}_{fr,A}^0$ with
\begin{align*}&\vec{v} = f_{1}\text{ and frame} \\&[(-1)^{s(R)} e_1, e_2, \dots, e_d, (-1)^{s(S)} e_{d + 1}, e_{d + 2}, \dots, e_{2d}]\\& (\text{No internal frames and } 2d \text{ external frames})\end{align*}
In $\Omega_{fr,A}^0$, $P$ has the frame
\begin{align*}[f_1, (-1)^{s(R)} e_1, e_2, \dots, e_d, (-1)^{s(S)} e_{d + 1}, e_{d + 2}, \dots, e_{2d}]\end{align*}
and in $\SO(2d + 1)$, $P$ corresponds to the point
\begin{align*}[(-1)^{s(R)+s(S)}f_1, (-1)^{s(R)} e_1, e_2, \dots, e_d, (-1)^{s(S)} e_{d + 1}, e_{d + 2}, \dots, e_{2d}]\end{align*}
The endpoint
\begin{align*}Q = \mathcal{M}(c_x, \vec{e}_j, (1))\end{align*}
is a point in $\Tilde{\Omega}_{fr,A}^0$ with
\begin{align*}\\&[f_1, e_1, \dots, e_{2d}].\\& (\text{One internal frame and } 2d \text{ external frames})\end{align*}
Note that $Q$ has one more internal frame than $P$, which is a result of it lying in the lower stratum $Z(0, 1, 0; (1))$ than the stratum $Z(0, 0, 1)$ where the rest of the moduli space lies. The local model for the moduli space is identified with $\R_+$ with strata $\{ 0 \}$ (lower) and $(0, \infty)$ (higher), so $Q' \in \partial' \mathcal{M}(V_j, 0, 0)$ is a translate of $Q$. Compared to $Q$, $Q'$ lacks the internal frame $f_1 = \delta v_1$ and instead has $\vec{v} = f_1$ (the direction of merging the bubble as a vertical annulus). In $\Omega_{fr,A}^0$, $Q'$ has the frame
\begin{align*}[f_1, e_1, \dots, e_{2d}].\end{align*}
and in $\SO(2d + 1)$, $Q'$ corresponds to the point
\begin{align*}[f_1, e_1, \dots, e_{2d}].\end{align*}
We can see that the $\Omega_{fr,A}^0$ frames are opposite exactly when $s(R) + s(S) = 1 \pmod{2}$, so these frames are related by exactly one sign switch and thus opposite. The preferred path is as follows. Starting from $P$, we change the $(d+1)^{st}$ frame to positive using the long preferred path with respect to $f_1$, then the first coordinate to positive using the long preferred path with respect to $f_1$. (Note that since $Q'$ is positively framed, this coincides with the preferred path of Type 2.2a.)

\begin{figure}
\begin{tikzpicture}[scale=3]

\draw (0, 0) .. controls (-0.125, 0.25) .. (0, 0.5) .. controls (0.125, 0.25) .. (0, 0);
\draw (0, 0.5) .. controls (-0.125, 0.75) .. (0, 1) .. controls (0.125, 0.75) .. (0, 0.5);
\filldraw[color=gray] (2, 0) .. controls (1.75, 0.5) .. (2, 1) .. controls (2.25, 0.5) .. (2, 0);

\draw (0.1, 0.5) node {\textbullet};
\draw (0.1, 0.5) node[anchor=south west] {$P$};

\draw (0.1, 0.5) -- (1.75, 0.5);

\draw (1.75, 0.5) node {\textbullet};
\draw (1.5, 0.5) node {\textbullet};
\draw (1.5, 0.5) node[anchor=south west] {$Q'$};
\draw (1.75, 0.5) node[anchor=north east] {$Q$};

\node at (0, 0.75) {$R$};
\node at (0, 0.25) {$S$};

\draw (2.295, 0.5) circle (0.1);
\node at (2.295, 0.5) {1};

\draw[->] (1, 1) -- (1.45, 0.55);

\draw (1, 1) .. controls (0.875, 1.25) .. (1, 1.5) .. controls (1.125, 1.25) .. (1, 1);
\node at (1, 1.25) {$\mathit{RS}$};

\end{tikzpicture}
\begin{tikzpicture}[x={(1,0)}, y={(0,1)}, z={(-0.5, -0.5)}, scale=3.5]

\draw (0.25, 0, 0.25) node {\textbullet};
\draw (0.25, 0, 0.25) node[anchor=north east] {$P$};

\draw (0.25, 0, 0.25) .. controls (0.25, 1, 0.25) and (0, 1, 0) .. (0, 0.8, 0);

\draw (0, 0.7, 0) node {\textbullet};
\draw (0, 0.8, 0) node {\textbullet};
\draw (0, 0.7, 0) node[anchor=north east] {$Q$};
\draw (0, 0.8, 0) node[anchor=south east] {$Q'$};

\draw[->] (-0.2,0,0) -- (1,0,0) node[anchor=west]{$\R^d$};
\draw[->] (0,0,0) -- (0,1,0) node[anchor=east]{$\R_+$};
\draw[->] (0,0,-0.2) -- (0,0,1) node[anchor=north west]{$\R^d$};

\end{tikzpicture}
\caption{Type 2.2b, in the case when $D$ is a vertical annulus.}\label{fig5}\end{figure}
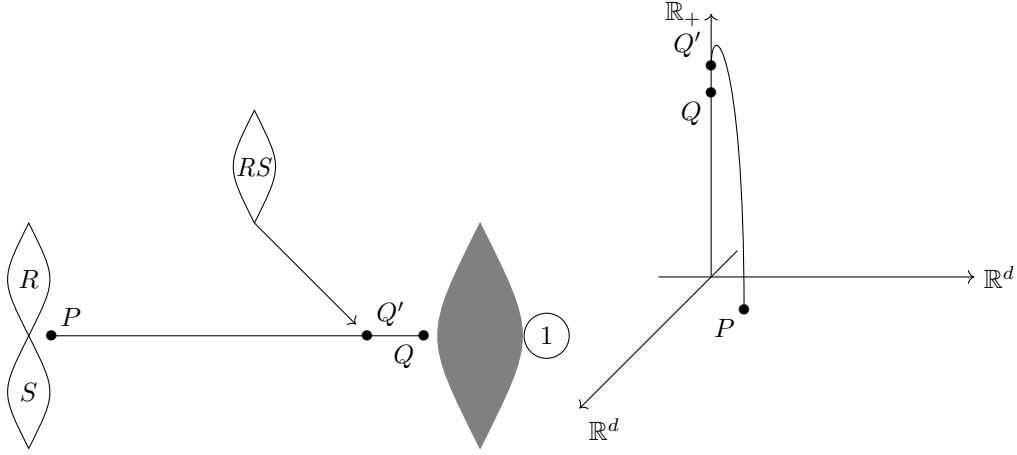

Finally, suppose $D = R * S$ is the horizontal annulus $H_j$. In this case, the endpoint
\begin{align*}P = \mathcal{M}(R, 0, 0) \times \mathcal{M}(S, 0, 0)\end{align*}
is a point in $\Tilde{\Omega}_{fr,A}^0$ with
\begin{align*}&\vec{v} = f_{1}\text{ and frame} \\&[(-1)^{s(R)} e_1, e_2, \dots, e_d, (-1)^{s(S)} e_{d + 1}, e_{d + 2}, \dots, e_{2d}]\\& (\text{No internal frames and } 2d \text{ external frames})\end{align*}
In $\Omega_{fr,A}^0$, $P$ has the frame
\begin{align*}[f_1, (-1)^{s(R)} e_1, e_2, \dots, e_d, (-1)^{s(S)} e_{d + 1}, e_{d + 2}, \dots, e_{2d}]\end{align*}
and in $\SO(2d + 1)$, $P$ corresponds to the point
\begin{align*}[(-1)^{s(R)+s(S)}f_1, (-1)^{s(R)} e_1, e_2, \dots, e_d, (-1)^{s(S)} e_{d + 1}, e_{d + 2}, \dots, e_{2d}]\end{align*}
The endpoint
\begin{align*}Q = \mathcal{M}(c_x, \vec{e}_j, (1))\end{align*}
has frame
\begin{align*}\\&[f_1, e_1, \dots, e_{2d}].\\& (\text{One internal frame and } 2d \text{ external frames})\end{align*}
Note that $Q$ has one more internal frame than $P$, which is a result of it lying in the lower stratum $Z(0, 1, 0; (1))$ than the stratum $Z(1, 0, 0)$ where the rest of the moduli space lies. The local model for the moduli space is identified with $\R_+$ with strata $\{ 0 \}$ (lower) and $(0, \infty)$ (higher), so $Q' \in \partial' \mathcal{M}(V_j, 0, 0)$ is a translate of $Q$. Compared to $Q$, $Q'$ lacks the internal frame $f_1 = \delta v_1$ and instead has $\vec{v} = -f_1$ (the direction of merging the bubble as a horizontal annulus). In $\Omega_{fr,A}^0$, $Q'$ has the frame
\begin{align*}[-f_1, e_1, \dots, e_{2d}].\end{align*}
and in $\SO(2d + 1)$, $Q'$ corresponds to the point
\begin{align*}[f_1, e_1, \dots, e_{2d}].\end{align*}
We can see that the $\Omega_{fr,A}^0$ frames are opposite exactly when $s(R) + s(S) = 0 \pmod{2}$, so that the frames are related by exactly one or three sign switches. The preferred path is as follows. Starting from $P$, we change the $(d+1)^{st}$ frame to positive using the long preferred path with respect to $f_1$, then the first coordinate to positive using the long preferred path with respect to $f_1$. (Note that since $Q'$ is positively externally framed, this coincides with the preferred path of Type 2.2a.)

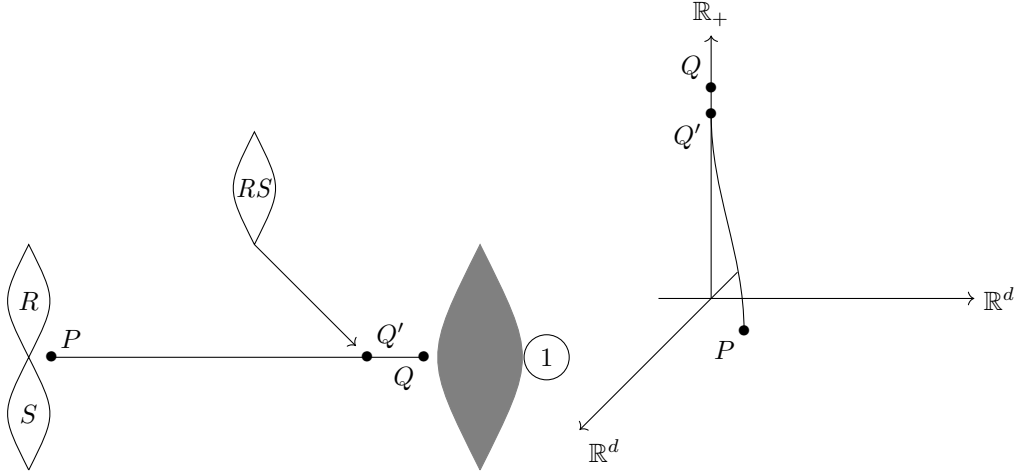
\begin{figure}
\begin{tikzpicture}[scale=3]

\draw (0, 0) .. controls (-0.125, 0.25) .. (0, 0.5) .. controls (0.125, 0.25) .. (0, 0);
\draw (0, 0.5) .. controls (-0.125, 0.75) .. (0, 1) .. controls (0.125, 0.75) .. (0, 0.5);
\filldraw[color=gray] (2, 0) .. controls (1.75, 0.5) .. (2, 1) .. controls (2.25, 0.5) .. (2, 0);

\draw (0.1, 0.5) node {\textbullet};
\draw (0.1, 0.5) node[anchor=south west] {$P$};

\draw (0.1, 0.5) -- (1.75, 0.5);

\draw (1.75, 0.5) node {\textbullet};
\draw (1.5, 0.5) node {\textbullet};
\draw (1.5, 0.5) node[anchor=south west] {$Q'$};
\draw (1.75, 0.5) node[anchor=north east] {$Q$};

\node at (0, 0.75) {$R$};
\node at (0, 0.25) {$S$};

\draw (2.295, 0.5) circle (0.1);
\node at (2.295, 0.5) {1};

\draw[->] (1, 1) -- (1.45, 0.55);

\draw (1, 1) .. controls (0.875, 1.25) .. (1, 1.5) .. controls (1.125, 1.25) .. (1, 1);
\node at (1, 1.25) {$\mathit{RS}$};

\end{tikzpicture}
\begin{tikzpicture}[x={(1,0)}, y={(0,1)}, z={(-0.5, -0.5)}, scale=3.5]

\draw (0.25, 0, 0.25) node {\textbullet};
\draw (0.25, 0, 0.25) node[anchor=north east] {$P$};

\draw (0.25, 0, 0.25) .. controls (0.25, 0.3, 0.25) and (0, 0.4, 0) .. (0, 0.7, 0);

\draw (0, 0.7, 0) node {\textbullet};
\draw (0, 0.8, 0) node {\textbullet};
\draw (0, 0.7, 0) node[anchor=north east] {$Q'$};
\draw (0, 0.8, 0) node[anchor=south east] {$Q$};

\draw[->] (-0.2,0,0) -- (1,0,0) node[anchor=west]{$\R^d$};
\draw[->] (0,0,0) -- (0,1,0) node[anchor=south]{$\R_+$};
\draw[->] (0,0,-0.2) -- (0,0,1) node[anchor=north west]{$\R^d$};

\end{tikzpicture}
\caption{Type 2.2b, in the case where $D$ is a horizontal annulus.}\label{fig6}\end{figure}

\section{The 2-dimensional Coherence Conditions}

To describe the $2$-dimensional boundaries, their frames must be understood. Again, we first consider a smoothening $\partial'\mathcal{M}(D, \vec{N}, \vec{\lambda})$ of each $2$-dimensional boundary by pushing off in the $\vec{v}$ direction to avoid the strata while preserving frames.

The local model guarantees that framing of a product of a point and an interval, or of an interval and a point is given by a similar product structure as in the framing of a point times a point. Their internal frame is the product of the internal frames, but where the internal frames of type $1$ bubbles come first, type $2$ bubbles next, and so on. Their external frames is the product of the external frames, and goes after the internal frame as always.

We record the framings of each part of $\partial'\mathcal{M}(D, \vec{N}, \vec{\lambda})$ with a frame assignment, which detects whether a path is preferred. More formally:
\begin{definition}
Given a framing of the $0$-dimensional moduli spaces, a frame assignment $f$ is a $2$-cochain on $(\CDP_*; \Z/2)$ such that $f(D, \vec{N}, \vec{\lambda}) = 0$ if and only if the moduli space $\mathcal{M}(D, \vec{N}, \vec{\lambda})$ is homotopic (relative endpoints) to the preferred framed path for its moduli space as in Section 6. 
\end{definition}

\begin{lemma}\label{framecond}(Frame Assignment Conditions) Given a sign assignment $s$ satisfying Lemma $\ref{lemsame}$ which all $0$-dimensional moduli spaces have been framed according to, a frame assignment $f$ must satisfy
\begin{enumerate}
\item[Type 3.3a] If $D$ is an index $3$ domain which contains no annulus,
\begin{align*}\delta f(D, 0, 0) = 1 + \sum\limits_{j = 1}^{k}(1 + s(C_j))\end{align*}
where $C_j$ are all possible rectangles that can show up at the end of a decomposition $D = A_k * B_l * C_j$.
\item[Type 3.3b-c] If $D$ is an index $3$ domain which contains an annulus, 
\begin{align*}\delta f(D, 0, 0) = 0.\end{align*}
\item[Type 3.2a-b] For all $1 \leq j \leq n$ and all index $2$ domains $D$ which are not either the horizontal annulus $H_j$ or the vertical annulus $V_j$,
\begin{align*}\delta f(D, N \vec{e_j}, (N)_j) = 0.\end{align*}
\item[Type 3.2c] For $D = H_j$ or $V_j$,
\begin{align*}\delta f(D, \vec{e_j}, (1)_j) = 0.\end{align*}
\item[Type 3.1a] For all index $1$ domains $D$ and all $1 \leq j \leq n$, 
\begin{align*}\delta f(D, 2\vec{e_j}, ((1, 1))) = s(D).\end{align*}
\item[Type 3.1b] For all index $1$ domains $D$ and all $1 \leq j < k \leq n$, 
\begin{align*}\delta f(D, \vec{e_j} + \vec{e_k}, ((1), (1))) = 1 + s(D).\end{align*}
\item[Type 3.0c] For all constant domains $c_x$ and all $1 \leq j < k < l \leq n$, 
\begin{align*}\delta f(c_x, N_1 \vec{e_j} + N_2 \vec{e_k} + N_3 \vec{e_l}, ((N_1), (N_2), (N_3))) = 0.\end{align*}
\end{enumerate}\end{lemma}

\begin{remark}Lemma $\ref{framecond}$ does not describe $\delta f$ for every generator of $\CDP_3$, but the above generators are sufficient to prove Theorem $\ref{mainthm}$. In the following Section 8, we describe how to compute $\delta f$ on the remaining generators.\end{remark}

\begin{proof}We check triples of types 3.3a, 3.3b, 3.3c, 3.2a, 3.2b, as well as types 3.2c, 3.1a, 3.1b, and 3.0c in the cases where they have one bubble each ($N = N_1 = N_2 = N_3 = 1$). By Lemma $\ref{lem4}$, it suffices to check the corresponding loops in $\SO(A)$. Much of the remainder of this section is devoted to these cases.

\textbf{Type 3a}: In this case, $D$ is an index $3$ domain, $\vec{\lambda}$ is empty, and $D$ does not contain an annulus, then the proof of \cite[Lemma 3.2]{YT} shows that $\partial \mathcal{M}(D, 0, 0)$ is some $2k$-gon. In fact, $k = 2, 3, 4$, with $k = 2$ only for the following domains

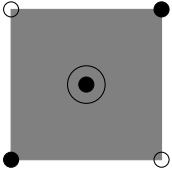
\begin{tikzpicture}

\filldraw[gray] (0, 0) rectangle (2, 2);
\filldraw[black] (1, 1) circle (0.1);
\draw (1, 1) circle (0.25);
\filldraw[black] (0, 0) circle (0.1);
\draw (0, 2) circle (0.1);
\filldraw[black] (2, 2) circle (0.1);
\draw (2, 0) circle (0.1);

\end{tikzpicture}

and $k = 4$ only for the following domains

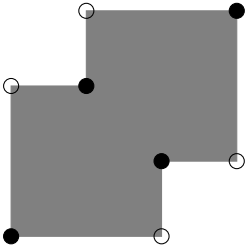
\begin{tikzpicture}

\filldraw[gray] (0, 0) -- (2, 0) -- (2, 1) -- (3, 1) -- (3, 3) -- (1, 3) -- (1, 2) -- (0, 2) -- (0, 0);
\filldraw[black] (0, 0) circle (0.1);
\draw (0, 2) circle (0.1);
\filldraw[black] (1, 2) circle (0.1);
\draw (1, 3) circle (0.1);
\filldraw[black] (2, 1) circle (0.1);
\draw (2, 0) circle (0.1);
\filldraw[black] (3, 3) circle (0.1);
\draw (3, 1) circle (0.1);

\end{tikzpicture}

In the most common case $k = 3$, Figure $\ref{fig7}$ shows the embedding of $\partial \mathcal{M}(D, 0, 0)$ and $\partial' \mathcal{M}(D, 0, 0)$.

\begin{figure}

\begin{tikzpicture}[x={(1,0)}, y={(0,1)}, z={(-0.5, -0.5)}, scale=7]

\draw[->] (0,0,0) -- (0.5,0,0) node[anchor=west]{$\R_{+(1)}$};
\draw[->] (0,0,0) -- (0,0.5,0) node[anchor=south]{$\R_{+(2)}$};
\draw[->] (0,0,-1) -- (0,0,1) node[anchor=north west]{$\R^{3d}$};

\coordinate (A) at (0, 0, -0.75);
\coordinate (B) at (0, 0, -0.45);
\coordinate (C) at (0, 0, -0.15);
\coordinate (D) at (0, 0, 0.15);
\coordinate (E) at (0, 0, 0.45);
\coordinate (F) at (0, 0, 0.75);

\draw (A) .. controls ($(A) + (0.2, 0, 0)$) and ($(B) + (0.2, 0, 0)$) .. (B);
\draw (B) .. controls ($(B) + (0, 0.2, 0)$) and ($(C) + (0, 0.2, 0)$) .. (C);
\draw (C) .. controls ($(C) + (0.2, 0, 0)$) and ($(D) + (0.2, 0, 0)$) .. (D);
\draw (D) .. controls ($(D) + (0, 0.2, 0)$) and ($(E) + (0, 0.2, 0)$) .. (E);
\draw (E) .. controls ($(E) + (0.2, 0, 0)$) and ($(F) + (0.2, 0, 0)$) .. (F);
\draw (A) .. controls ($(A) + (0, 0.5, 0)$) and ($(F) + (0, 0.5, 0)$) .. (F);

\coordinate (A1) at (0.1, 0.05, -0.75);
\coordinate (A2) at (0.05, 0.1, -0.75);
\coordinate (B1) at (0.1, 0.05, -0.45);
\coordinate (B2) at (0.05, 0.1, -0.45);
\coordinate (C1) at (0.1, 0.05, -0.15);
\coordinate (C2) at (0.05, 0.1, -0.15);
\coordinate (D1) at (0.1, 0.05, 0.15);
\coordinate (D2) at (0.05, 0.1, 0.15);
\coordinate (E1) at (0.1, 0.05, 0.45);
\coordinate (E2) at (0.05, 0.1, 0.45);
\coordinate (F1) at (0.1, 0.05, 0.75);
\coordinate (F2) at (0.05, 0.1, 0.75);

\node at (A1) {
\begin{tikzpicture}

\draw (0,0) circle (0.125);

\node at (0,0) {\tiny $1$};

\end{tikzpicture}
};
\node at (A2) {
\begin{tikzpicture}

\draw (0,0) circle (0.125);

\node at (0,0) {\tiny $2$};

\end{tikzpicture}
};
\node at (B1) {
\begin{tikzpicture}

\draw (0,0) circle (0.125);

\node at (0,0) {\tiny $3$};

\end{tikzpicture}
};
\node at (B2) {
\begin{tikzpicture}

\draw (0,0) circle (0.125);

\node at (0,0) {\tiny $4$};

\end{tikzpicture}
};
\node at (C1) {
\begin{tikzpicture}

\draw (0,0) circle (0.125);

\node at (0,0) {\tiny $5$};

\end{tikzpicture}
};
\node at (C2) {
\begin{tikzpicture}

\draw (0,0) circle (0.125);

\node at (0,0) {\tiny $6$};

\end{tikzpicture}
};
\node at (D1) {
\begin{tikzpicture}

\draw (0,0) circle (0.125);

\node at (0,0) {\tiny $7$};

\end{tikzpicture}
};
\node at (D2) {
\begin{tikzpicture}

\draw (0,0) circle (0.125);

\node at (0,0) {\tiny $8$};

\end{tikzpicture}
};
\node at (E1) {
\begin{tikzpicture}

\draw (0,0) circle (0.125);

\node at (0,0) {\tiny $9$};

\end{tikzpicture}
};
\node at (E2) {
\begin{tikzpicture}

\draw (0,0) circle (0.125);

\node at (0,0) {\tiny $10$};

\end{tikzpicture}
};
\node at (F1) {
\begin{tikzpicture}

\draw (0,0) circle (0.125);

\node at (0,0) {\tiny $11$};

\end{tikzpicture}
};
\node at (F2) {
\begin{tikzpicture}

\draw (0,0) circle (0.125);

\node at (0,0) {\tiny $12$};

\end{tikzpicture}
};

\draw[color=red] (A1) .. controls ($(A1) + (0.2, 0, 0)$) and ($(B1) + (0.2, 0, 0)$) .. (B1);
\draw[color=red] (B2) .. controls ($(B2) + (0, 0.2, 0)$) and ($(C2) + (0, 0.2, 0)$) .. (C2);
\draw[color=red] (C1) .. controls ($(C1) + (0.2, 0, 0)$) and ($(D1) + (0.2, 0, 0)$) .. (D1);
\draw[color=red] (D2) .. controls ($(D2) + (0, 0.2, 0)$) and ($(E2) + (0, 0.2, 0)$) .. (E2);
\draw[color=red] (E1) .. controls ($(E1) + (0.2, 0, 0)$) and ($(F1) + (0.2, 0, 0)$) .. (F1);
\draw[color=red] (A2) .. controls ($(A2) + (0, 0.5, 0)$) and ($(F2) + (0, 0.5, 0)$) .. (F2);

\draw[color=red] (A1) .. controls (0.05, 0.05, -0.75) .. (A2);
\draw[color=red] (B1) .. controls (0.05, 0.05, -0.45) .. (B2);
\draw[color=red] (C1) .. controls (0.05, 0.05, -0.15) .. (C2);
\draw[color=red] (D1) .. controls (0.05, 0.05, 0.15) .. (D2);
\draw[color=red] (E1) .. controls (0.05, 0.05, 0.45) .. (E2);
\draw[color=red] (F1) .. controls (0.05, 0.05, 0.75) .. (F2);

\node (broken) at (-0.5,0) {
\begin{tikzpicture}[scale=0.7]

\draw (0,0) to[out=20,in=-20] ++(0,1) to[out=-160,in=160] ++(0,-1);
\draw (0,1) to[out=20,in=-20] ++(0,1) to[out=-160,in=160] ++(0,-1);
\draw (0,2) to[out=20,in=-20] ++(0,1) to[out=-160,in=160] ++(0,-1);

\node at (0,0.5) {\tiny $C_k$};
\node at (0,1.5) {\tiny $B_j$};
\node at (0,2.5) {\tiny $A_i$};

\end{tikzpicture}
};

\draw[thin, densely dotted, ->] (broken) -- (A);
\draw[thin, densely dotted, ->] (broken) -- (B);
\draw[thin, densely dotted, ->] (broken) -- (C);
\draw[thin, densely dotted, ->] (broken) -- (D);
\draw[thin, densely dotted, ->] (broken) -- (E);
\draw[thin, densely dotted, ->] (broken) -- (F);

\end{tikzpicture}
\caption{Type 3.3a when $\partial \mathcal{M}(D, 0, 0)$ (black) is a hexagon. In this case, $\partial' \mathcal{M}(D, 0, 0)$ (red) is a $12$-gon embedded in the interior of $\mathbb{E}_2^d$ with the numbered vertices.}\label{fig7}\end{figure}
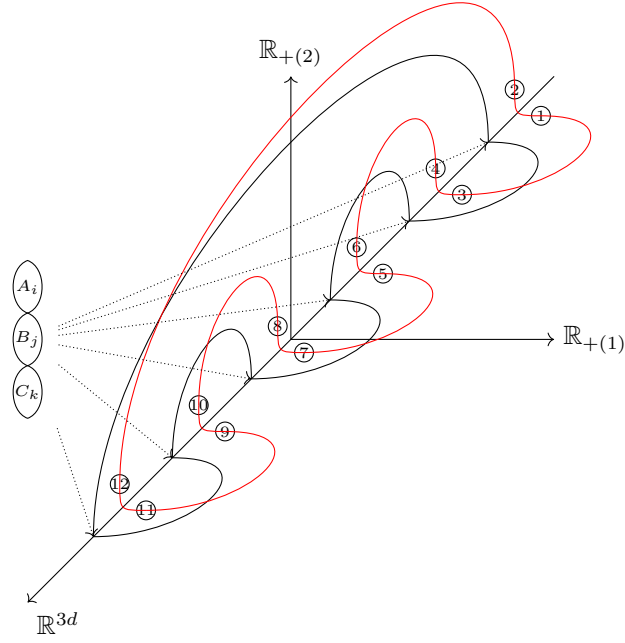

The vertices of $\partial' \mathcal{M}(D, 0, 0)$ (the red $12$-gon) shown in Figure $\ref{fig7}$ can be written as points in $\SO(3d + 2)$ that form a loop as above. Since most of the framing vectors are identical (the $(j + 2)^{nd}$ vector is always $e_j$ for $j \neq 1, 2, 3, d + 3, 2d + 3$, and always $\pm e_j$ for $j = 3, d+3, 2d+3$), it suffices to consider the coordinates that are possibly different. We fix the following conventions:

\begin{conv}Let $\sigma \in S_{m-1}$ be a permutation. If a point in $\SO(md+m-1)$ has a frame of the form
\begin{align*}&[(-1)^{\text{sgn}(\sigma) + r_1 + \dots + r_{m - 2} + s_1 + \dots + s_m} f_{\sigma(1)}, (-1)^{r_1}f_{\sigma(2)}, \dots, (-1)^{r_{m-2}} f_{\sigma(m - 1)}, \\&(-1)^{s_1} e_1, e_2, \dots, e_d, (-1)^{s_2} e_{d + 1}, e_{d + 2}, \dots, e_{md-d}, (-1)^{s_m} e_{md-d + 1}, e_{md-d + 2}, \dots, e_{md}]\end{align*}
we will abbreviate its frame as $(-1)^{r_1}f_{\sigma(2)}, \dots, (-1)^{r_{m-2}} f_{\sigma(m - 1)} s_1 \dots s_m$. That is, we report all but the first $\R_+$ frame and the signs of each $e_1, e_{d+1}, \dots, e_{md-m+1}$, which is enough to recover all the information of the original frame.\label{conv1}\end{conv}

Let $a_j$ be the sign of rectangle $A_j$, $b_j$ be the sign of rectangle $B_j$, and $c_j$ be the sign of rectangle $C_j$. We can now decompose all of the edges of $\partial' \mathcal{M}(D, 0, 0)$ into short and long preferred paths as shown in Figure $\ref{fig8}$. We also fix the following conventions for the figure.

\begin{conv}\label{conv2}
\begin{itemize}

\item All black paths are short preferred.

\item All red paths are either identity paths or long preferred paths changing the first $\R$ coordinate with respect to the first $\R_+$ coordinate.

\item All green paths are either identity paths or long preferred paths changing the $(d+1)^{st}$ $\R$ coordinate with respect to the first $\R_+$ coordinate.

\item All blue paths are either identity paths or long preferred paths changing the $(2d+1)^{st}$ $\R$ coordinate with respect to the first $\R_+$ coordinate.

\end{itemize}
\end{conv}

\begin{conv}\label{conv3}Long preferred paths in an $\R$ coordinate with respect to the first $\R_+$ coordinate have the following direction. If such a long preferred path in the $k^{th}$ $\R$ coordinate goes from vertex $A$ to vertex $B$, then vertex $B$'s $k^{th}$ vector must be positive $e_k$. (Note that if vertex $A$'s $k^{th}$ vector is also positive $e_k$, this path is the identity. We will still draw such an identity path with an arrow for consistency.)\end{conv}

\begin{figure}

\begin{tikzpicture}[scale=6]
\begin{scope}[x={(1cm, 0cm)}, y={(0cm, 1cm)}]

\coordinate (O) at (0, 0);
\coordinate (A) at (0:1);
\coordinate (B) at ($(A) + (90:0.447)$);
\coordinate (C) at ($(B) + (120:0.2235)$);
\coordinate (D) at ($(C) + (150:0.447)$);
\coordinate (E) at ($(D) + (240:0.175)$);
\coordinate (F) at ($(E) + (330:0.2235)$);

\node[anchor=west] at (B) {\small $f_2 a_1 b_2 c_2$};
\node[anchor=south west] at (C) {\small $f_1 a_1 b_2 c_2$};
\node[anchor=west] at (A) {\small $f_2 0 0 c_2$};
\node[anchor=west] at ($(A) ! 0.5 ! (B)$) {\small $f_2 a_1 0 c_2$};
\node[anchor=south west] at ($(C) ! 0.5 ! (D)$) {\small $f_1 a_1 b_2 0$};
\node[anchor=south west] at (D) {\small $f_1 a_1 0 0$};
\node at (F) {\small $f_2 a_1 b_2 0$};
\node at (E) {\small $f_2 a_1 0 0$};
\node at (O) {\small $f_2 000$};

\draw[->-,thick,color=blue] (A) -- (O);
\draw[->-,thick,color=red] ($(A) ! 0.5 ! (B)$) -- (A);
\draw[->-,thick,color=green] (B) -- ($(A) ! 0.5 ! (B)$);
\draw[thick] (B) -- (C);
\draw[->-,thick,color=blue] (C) -- ($(C) ! 0.5 ! (D)$);
\draw[->-,thick,color=green] ($(C) ! 0.5 ! (D)$) -- (D);
\draw[->-,thick,color=red] (E) -- (O);


\draw[thick] (E) -- (D);
\draw[->-,thick,color=green] (F) -- (E);
\draw[->-,thick,color=blue] (B) -- (F);
\draw[thick] (F) -- ($(C) ! 0.5 ! (D)$);
\draw[->-,thick,color=blue] ($(A) ! 0.5 ! (B)$) -- (E);

\node at (B) {
\begin{tikzpicture}

\draw[fill=white] (0,0) circle (0.125);

\node at (0,0) {\tiny $2$};

\end{tikzpicture}
};
\node at (C) {
\begin{tikzpicture}

\draw[fill=white] (0,0) circle (0.125);

\node at (0,0) {\tiny $1$};

\end{tikzpicture}
};

\end{scope}

\begin{scope}[x={(-0.5cm, 0.866cm)}, y={(0.866cm, 0.5cm)}]

\coordinate (O) at (0, 0);
\coordinate (A) at (0:1);
\coordinate (B) at ($(A) + (90:0.447)$);
\coordinate (C) at ($(B) + (120:0.2235)$);
\coordinate (D) at ($(C) + (150:0.447)$);
\coordinate (E) at ($(D) + (240:0.175)$);
\coordinate (F) at ($(E) + (330:0.2235)$);

\node[anchor=south east] at (B) {\small $f_2 a_1 b_1 c_1$};
\node[anchor=south west] at (C) {\small $f_1 a_1 b_1 c_1$};
\node[anchor=south east] at (A) {\small $f_2 0 0 c_1$};
\node[anchor=south east] at ($(A) ! 0.5 ! (B)$) {\small $f_2 a_1 0 c_1$};
\node[anchor=south west] at ($(C) ! 0.5 ! (D)$) {\small $f_1 a_1 b_1 0$};
\node at (F) {\small $f_2 a_1 b_1 0$};

\draw[->-,thick,color=blue] (A) -- (O);
\draw[->-,thick,color=red] ($(A) ! 0.5 ! (B)$) -- (A);
\draw[->-,thick,color=green] (B) -- ($(A) ! 0.5 ! (B)$);
\draw[thick] (B) -- (C);
\draw[->-,thick,color=blue] (C) -- ($(C) ! 0.5 ! (D)$);
\draw[->-,thick,color=green] ($(C) ! 0.5 ! (D)$) -- (D);
\draw[->-,thick,color=red] (E) -- (O);


\draw[thick] (E) -- (D);
\draw[->-,thick,color=green] (F) -- (E);
\draw[->-,thick,color=blue] (B) -- (F);
\draw[thick] (F) -- ($(C) ! 0.5 ! (D)$);
\draw[->-,thick,color=blue] ($(A) ! 0.5 ! (B)$) -- (E);

\node at (B) {
\begin{tikzpicture}

\draw[fill=white] (0,0) circle (0.125);

\node at (0,0) {\tiny $4$};

\end{tikzpicture}
};
\node at (C) {
\begin{tikzpicture}

\draw[fill=white] (0,0) circle (0.125);

\node at (0,0) {\tiny $3$};

\end{tikzpicture}
};

\end{scope}

\begin{scope}[x={(-0.5cm, 0.866cm)}, y={(-0.866cm, -0.5cm)}]

\coordinate (O) at (0, 0);
\coordinate (A) at (0:1);
\coordinate (B) at ($(A) + (90:0.447)$);
\coordinate (C) at ($(B) + (120:0.2235)$);
\coordinate (D) at ($(C) + (150:0.447)$);
\coordinate (E) at ($(D) + (240:0.175)$);
\coordinate (F) at ($(E) + (330:0.2235)$);

\node[anchor=south east] at (B) {\small $f_2 a_2 b_3 c_1$};
\node[anchor=south east] at (C) {\small $f_1 a_2 b_3 c_1$};
\node[anchor=south east] at ($(A) ! 0.5 ! (B)$) {\small $f_2 a_2 0 c_1$};
\node[anchor=east] at ($(C) ! 0.5 ! (D)$) {\small $f_1 a_2 b_3 0$};
\node[anchor=east] at (D) {\small $f_1 a_2 0 0$};
\node at (F) {\small $f_2 a_2 b_3 0$};
\node at (E) {\small $f_2 a_2 0 0$};

\draw[->-,thick,color=blue] (A) -- (O);
\draw[->-,thick,color=red] ($(A) ! 0.5 ! (B)$) -- (A);
\draw[->-,thick,color=green] (B) -- ($(A) ! 0.5 ! (B)$);
\draw[thick] (B) -- (C);
\draw[->-,thick,color=blue] (C) -- ($(C) ! 0.5 ! (D)$);
\draw[->-,thick,color=green] ($(C) ! 0.5 ! (D)$) -- (D);
\draw[->-,thick,color=red] (E) -- (O);


\draw[thick] (E) -- (D);
\draw[->-,thick,color=green] (F) -- (E);
\draw[->-,thick,color=blue] (B) -- (F);
\draw[thick] (F) -- ($(C) ! 0.5 ! (D)$);
\draw[->-,thick,color=blue] ($(A) ! 0.5 ! (B)$) -- (E);

\node at (B) {
\begin{tikzpicture}

\draw[fill=white] (0,0) circle (0.125);

\node at (0,0) {\tiny $6$};

\end{tikzpicture}
};
\node at (C) {
\begin{tikzpicture}

\draw[fill=white] (0,0) circle (0.125);

\node at (0,0) {\tiny $5$};

\end{tikzpicture}
};

\end{scope}

\begin{scope}[x={(-0.5cm, -0.866cm)}, y={(-0.866cm, 0.5cm)}]

\coordinate (O) at (0, 0);
\coordinate (A) at (0:1);
\coordinate (B) at ($(A) + (90:0.447)$);
\coordinate (C) at ($(B) + (120:0.2235)$);
\coordinate (D) at ($(C) + (150:0.447)$);
\coordinate (E) at ($(D) + (240:0.175)$);
\coordinate (F) at ($(E) + (330:0.2235)$);

\node[anchor=north east] at (B) {\small $f_2 a_2 b_4 c_3$};
\node[anchor=north east] at (C) {\small $f_1 a_2 b_4 c_3$};
\node[anchor=north east] at (A) {\small $f_2 0 0 c_3$};
\node[anchor=north east] at ($(A) ! 0.5 ! (B)$) {\small $f_2 a_2 0 c_3$};
\node[anchor=east] at ($(C) ! 0.5 ! (D)$) {\small $f_1 a_2 b_4 0$};
\node at (F) {\small $f_2 a_2 b_4 0$};

\draw[->-,thick,color=blue] (A) -- (O);
\draw[->-,thick,color=red] ($(A) ! 0.5 ! (B)$) -- (A);
\draw[->-,thick,color=green] (B) -- ($(A) ! 0.5 ! (B)$);
\draw[thick] (B) -- (C);
\draw[->-,thick,color=blue] (C) -- ($(C) ! 0.5 ! (D)$);
\draw[->-,thick,color=green] ($(C) ! 0.5 ! (D)$) -- (D);
\draw[->-,thick,color=red] (E) -- (O);


\draw[thick] (E) -- (D);
\draw[->-,thick,color=green] (F) -- (E);
\draw[->-,thick,color=blue] (B) -- (F);
\draw[thick] (F) -- ($(C) ! 0.5 ! (D)$);
\draw[->-,thick,color=blue] ($(A) ! 0.5 ! (B)$) -- (E);

\node at (B) {
\begin{tikzpicture}

\draw[fill=white] (0,0) circle (0.125);

\node at (0,0) {\tiny $8$};

\end{tikzpicture}
};
\node at (C) {
\begin{tikzpicture}

\draw[fill=white] (0,0) circle (0.125);

\node at (0,0) {\tiny $7$};

\end{tikzpicture}
};

\end{scope}

\begin{scope}[x={(-0.5cm, -0.866cm)}, y={(0.866cm, -0.5cm)}]

\coordinate (O) at (0, 0);
\coordinate (A) at (0:1);
\coordinate (B) at ($(A) + (90:0.447)$);
\coordinate (C) at ($(B) + (120:0.2235)$);
\coordinate (D) at ($(C) + (150:0.447)$);
\coordinate (E) at ($(D) + (240:0.175)$);
\coordinate (F) at ($(E) + (330:0.2235)$);

\node[anchor=north east] at (B) {\small $f_2 a_3 b_6 c_3$};
\node[anchor=north west] at (C) {\small $f_1 a_3 b_6 c_3$};
\node[anchor=north east] at ($(A) ! 0.5 ! (B)$) {\small $f_2 a_3 0 c_3$};
\node[anchor=north west] at ($(C) ! 0.5 ! (D)$) {\small $f_1 a_3 b_6 0$};
\node[anchor=north west] at (D) {\small $f_1 a_3 0 0$};
\node at (F) {\small $f_2 a_3 b_6 0$};
\node at (E) {\small $f_2 a_3 0 0$};

\draw[->-,thick,color=blue] (A) -- (O);
\draw[->-,thick,color=red] ($(A) ! 0.5 ! (B)$) -- (A);
\draw[->-,thick,color=green] (B) -- ($(A) ! 0.5 ! (B)$);
\draw[thick] (B) -- (C);
\draw[->-,thick,color=blue] (C) -- ($(C) ! 0.5 ! (D)$);
\draw[->-,thick,color=green] ($(C) ! 0.5 ! (D)$) -- (D);
\draw[->-,thick,color=red] (E) -- (O);


\draw[thick] (E) -- (D);
\draw[->-,thick,color=green] (F) -- (E);
\draw[->-,thick,color=blue] (B) -- (F);
\draw[thick] (F) -- ($(C) ! 0.5 ! (D)$);
\draw[->-,thick,color=blue] ($(A) ! 0.5 ! (B)$) -- (E);

\node at (B) {
\begin{tikzpicture}

\draw[fill=white] (0,0) circle (0.125);

\node at (0,0) {\tiny $10$};

\end{tikzpicture}
};
\node at (C) {
\begin{tikzpicture}

\draw[fill=white] (0,0) circle (0.125);

\node at (0,0) {\tiny $9$};

\end{tikzpicture}
};

\end{scope}

\begin{scope}[x={(1cm, 0cm)}, y={(0cm, -1cm)}]

\coordinate (O) at (0, 0);
\coordinate (A) at (0:1);
\coordinate (B) at ($(A) + (90:0.447)$);
\coordinate (C) at ($(B) + (120:0.2235)$);
\coordinate (D) at ($(C) + (150:0.447)$);
\coordinate (E) at ($(D) + (240:0.175)$);
\coordinate (F) at ($(E) + (330:0.2235)$);

\node[anchor=north west] at (B) {\small $f_2 a_3 b_5 c_2$};
\node[anchor=north west] at (C) {\small $f_1 a_3 b_5 c_2$};
\node[anchor=west] at ($(A) ! 0.5 ! (B)$) {\small $f_2 a_3 0 c_2$};
\node[anchor=south west] at ($(C) ! 0.5 ! (D)$) {\small $f_1 a_3 b_5 0$};
\node at (F) {\small $f_2 a_3 b_5 0$};

\draw[->-,thick,color=blue] (A) -- (O);
\draw[->-,thick,color=red] ($(A) ! 0.5 ! (B)$) -- (A);
\draw[->-,thick,color=green] (B) -- ($(A) ! 0.5 ! (B)$);
\draw[thick] (B) -- (C);
\draw[->-,thick,color=blue] (C) -- ($(C) ! 0.5 ! (D)$);
\draw[->-,thick,color=green] ($(C) ! 0.5 ! (D)$) -- (D);
\draw[->-,thick,color=red] (E) -- (O);


\draw[thick] (E) -- (D);
\draw[->-,thick,color=green] (F) -- (E);
\draw[->-,thick,color=blue] (B) -- (F);
\draw[thick] (F) -- ($(C) ! 0.5 ! (D)$);
\draw[->-,thick,color=blue] ($(A) ! 0.5 ! (B)$) -- (E);

\node at (B) {
\begin{tikzpicture}

\draw[fill=white] (0,0) circle (0.125);

\node at (0,0) {\tiny $12$};

\end{tikzpicture}
};
\node at (C) {
\begin{tikzpicture}

\draw[fill=white] (0,0) circle (0.125);

\node at (0,0) {\tiny $11$};

\end{tikzpicture}
};

\end{scope}


\node at (-0.1, 0.95) {\small $b_1 c_1$};
\node at (0.2, 0.95) {\small $c_1$};
\node at (0.35, 0.85) {\small $b_1$};
\node at (0.55, 0.75) {\small $b_3$};
\node at (0.8, 0.6) {\small $c_2$};
\node at (0.9, 0.4) {\small $b_3 c_2$};
\node at (0.9, -0.4) {\small $b_5 c_2$};
\node at (0.8, -0.6) {\small $c_2$};
\node at (0.55, -0.75) {\small $b_5$};
\node at (0.35, -0.85) {\small $b_6$};
\node at (0.2, -0.95) {\small $c_3$};
\node at (-0.1, -0.95) {\small $b_6 c_3$};
\node at (-0.8, -0.6) {\small $b_4 c_3$};
\node at (-0.925, -0.35) {\small $c_3$};
\node at (-0.925, -0.1) {\small $b_4$};
\node at (-0.925, 0.1) {\small $b_3$};
\node at (-0.925, 0.35) {\small $c_1$};
\node at (-0.8, 0.6) {\small $b_3 c_1$};
\node at (210:0.5) {\small $a_2 c_3$};
\node at (150:0.5) {\small $a_2 c_1$};
\node at (270:0.5) {\small $a_3 c_3$};
\node at (90:0.5) {\small $a_1 c_1$};
\node at (330:0.5) {\small $a_3 c_2$};
\node at (30:0.5) {\small $a_1 c_2$};

\end{tikzpicture}
\caption{Type 3.3a when $\partial \mathcal{M}(D, 0, 0)$ is a hexagon. The numbered vertices correspond to the numbered vertices of $\partial' \mathcal{M}(D, 0, 0)$ of Figure $\ref{fig7}$.}\label{fig8}\end{figure}
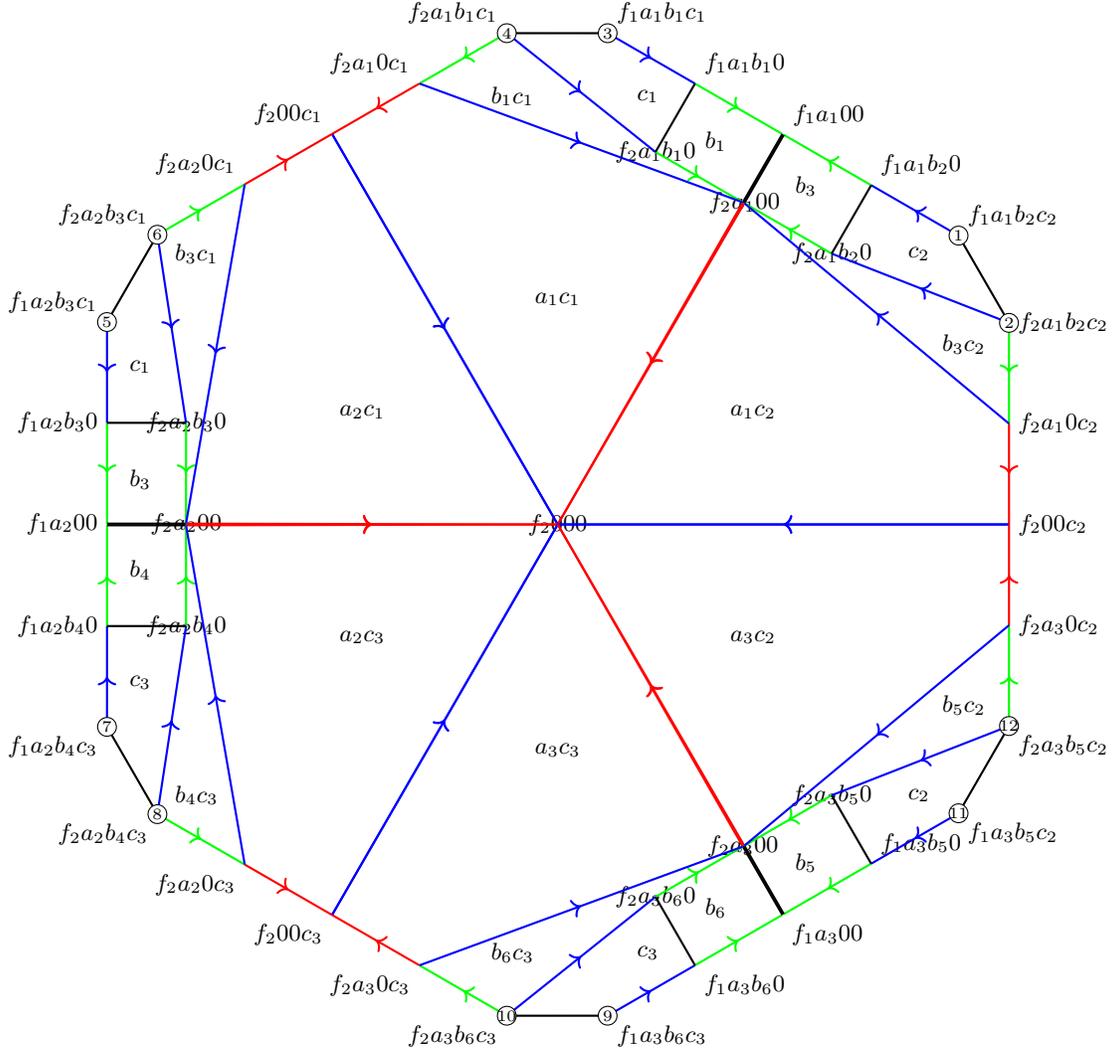

By Lemma $\ref{lem1}$, all black-green and black-blue quadrilaterals correspond to the element of $\pi_1(\SO(3d + 2))$ labelled in Figure $\ref{fig8}$, and by Lemma $\ref{lem2}$ all green-blue and red-green quadrilaterals also correspond to the labelled element of $\pi_1(\SO(3d + 2))$. So, when $\partial' \mathcal{M}(D, 0, 0)$ is a hexagon,
\begin{align*}\delta f(D, 0, 0) &= 1 + b_1 + b_2 + b_3 + b_4 + b_5 + b_6 + 2(c_1 + c_2 + c_3) \\&+c_1(a_1 + a_2 + b_1 + b_3) + c_2(a_1 + a_3 + b_2 + b_5) + c_3(a_2 + a_3 + b_4 + b_6) \\&=1 + b_1 + b_2 + b_3 + b_4 + b_5 + b_6 + 2(c_1 + c_2 + c_3) \\&+c_1 + c_2 + c_3 \text{ by the properties of sign assignments} \\&= 1 + c_1 + c_2 + c_3 \\&+ (b_1 + b_2 + c_1 + c_3) + (b_3 + b_4 + c_1 + c_3) + (b_5 + b_6 + c_2 + c_3) \\&= 1 + (1 + c_1) + (1 + c_2) + (1 + c_3) \pmod{2} \text{ again by the properties of sigh assignments}\end{align*}
In general when $\partial' \mathcal{M}(D, 0, 0)$ is a $2s$-gon, a similar decomposition yields
\begin{align*}\delta f(D, 0, 0) = 1 + \sum\limits_{j = 1}^{s} (1 + c_j)\end{align*}

\textbf{Type 3.3b}: In this case, the index $3$ domain $D$ contains an annulus $A = R * S$, and $D = A * R$. In this case, $\mathcal{M}(D, 0, 0)$ is the following triangle:

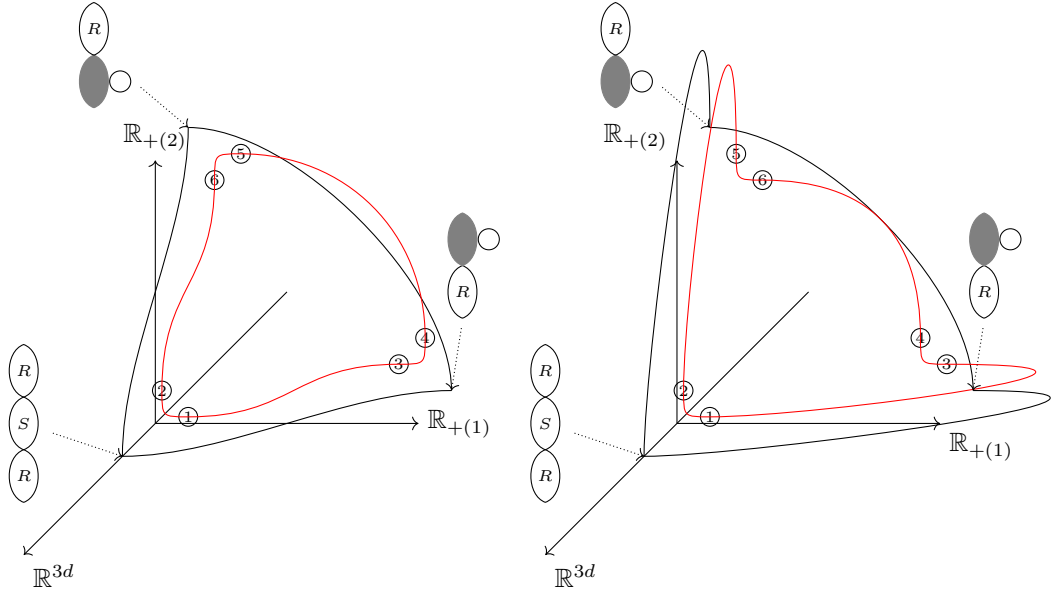
\begin{figure}

\begin{tikzpicture}[x={(1,0)}, y={(0,1)}, z={(-0.5, -0.5)}, scale=7]

\draw[->] (0,0,0) -- (0.5,0,0) node[anchor=west]{$\R_{+(1)}$};
\draw[->] (0,0,0) -- (0,0.5,0) node[anchor=south]{$\R_{+(2)}$};
\draw[->] (0,0,-0.5) -- (0,0,0.5) node[anchor=north west]{$\R^{3d}$};

\coordinate (A) at (0, 0, 0.125);
\coordinate (B) at (0.5, 0, -0.125);
\coordinate (C) at (0, 0.5, -0.125);

\draw (A) .. controls ($(A) + (0.2, 0, 0)$) and ($(B) + (-0.2, 0, 0)$) .. (B);
\draw (A) .. controls ($(A) + (0, 0.2, 0)$) and ($(C) + (0, -0.2, 0)$) .. (C);
\draw (B) .. controls ($(B) + (0, 0.2, 0)$) and ($(C) + (0.2, 0, 0)$) .. (C);

\coordinate (A1) at (0.1, 0.05, 0.075);
\coordinate (A2) at (0.05, 0.1, 0.075);
\coordinate (B1) at (0.4, 0.05, -0.125);
\coordinate (B2) at (0.45, 0.1, -0.125);
\coordinate (C1) at (0.1, 0.45, -0.125);
\coordinate (C2) at (0.05, 0.4, -0.125);

\node at (A1) {
\begin{tikzpicture}

\draw (0,0) circle (0.125);

\node at (0,0) {\tiny $1$};

\end{tikzpicture}
};
\node at (A2) {
\begin{tikzpicture}

\draw (0,0) circle (0.125);

\node at (0,0) {\tiny $2$};

\end{tikzpicture}
};
\node at (B1) {
\begin{tikzpicture}

\draw (0,0) circle (0.125);

\node at (0,0) {\tiny $3$};

\end{tikzpicture}
};
\node at (B2) {
\begin{tikzpicture}

\draw (0,0) circle (0.125);

\node at (0,0) {\tiny $4$};

\end{tikzpicture}
};
\node at (C1) {
\begin{tikzpicture}

\draw (0,0) circle (0.125);

\node at (0,0) {\tiny $5$};

\end{tikzpicture}
};
\node at (C2) {\begin{tikzpicture}

\draw (0,0) circle (0.125);

\node at (0,0) {\tiny $6$};

\end{tikzpicture}};

\draw[color=red] (A1) .. controls ($(A1) + (0.2, 0, 0)$) and ($(B1) + (-0.2, 0, 0)$) .. (B1);
\draw[color=red] (A2) .. controls ($(A2) + (0, 0.2, 0)$) and ($(C2) + (0, -0.2, 0)$) .. (C2);
\draw[color=red] (B2) .. controls ($(B2) + (0, 0.2, 0)$) and ($(C1) + (0.2, 0, 0)$) .. (C1);

\draw[color=red] (A1) .. controls (0.05, 0.05, 0.075) .. (A2);
\draw[color=red] (B1) .. controls (0.45, 0.05, -0.125) .. (B2);
\draw[color=red] (C1) .. controls (0.05, 0.45, -0.125) .. (C2);

\node (broken) at (-0.25,0) {
\begin{tikzpicture}[scale=0.7]

\draw (0,0) to[out=20,in=-20] ++(0,1) to[out=-160,in=160] ++(0,-1);
\draw (0,1) to[out=20,in=-20] ++(0,1) to[out=-160,in=160] ++(0,-1);
\draw (0,2) to[out=20,in=-20] ++(0,1) to[out=-160,in=160] ++(0,-1);

\node at (0,0.5) {\tiny $R$};
\node at (0,1.5) {\tiny $S$};
\node at (0,2.5) {\tiny $R$};

\end{tikzpicture}
};

\node (bubblefirst) at (0.6,0.3) {
\begin{tikzpicture}[scale=0.7]

\draw (0,0) to[out=20,in=-20] ++(0,1) to[out=-160,in=160] ++(0,-1);
\filldraw[gray] (0,1) to[out=20,in=-20] ++(0,1) to[out=-160,in=160] ++(0,-1);
\draw (0.5,1.5) circle (0.2);

\node at (0,0.5) {\tiny $R$};

\end{tikzpicture}
};

\node (bubblelast) at (-0.1,0.7) {
\begin{tikzpicture}[scale=0.7]

\filldraw[gray] (0,0) to[out=20,in=-20] ++(0,1) to[out=-160,in=160] ++(0,-1);
\draw (0,1) to[out=20,in=-20] ++(0,1) to[out=-160,in=160] ++(0,-1);
\draw (0.5,0.5) circle (0.2);

\node at (0,1.5) {\tiny $R$};

\end{tikzpicture}
};

\draw[thin, densely dotted, ->] (broken) -- (A);
\draw[thin, densely dotted, ->] (bubblefirst) -- (B);
\draw[thin, densely dotted, ->] (bubblelast) -- (C);

\end{tikzpicture} \begin{tikzpicture}[x={(1,0)}, y={(0,1)}, z={(-0.5, -0.5)}, scale=7]

\draw[->] (0,0,0) -- (0.5,0,0) node[anchor=north west]{$\R_{+(1)}$};
\draw[->] (0,0,0) -- (0,0.5,0) node[anchor=south east]{$\R_{+(2)}$};
\draw[->] (0,0,-0.5) -- (0,0,0.5) node[anchor=north west]{$\R^{3d}$};

\coordinate (A) at (0, 0, 0.125);
\coordinate (B) at (0.5, 0, -0.125);
\coordinate (C) at (0, 0.5, -0.125);

\draw (A) .. controls ($(A) + (0.2, 0, 0)$) and ($(B) + (0.5, 0, 0)$) .. (B);
\draw (A) .. controls ($(A) + (0, 0.2, 0)$) and ($(C) + (0, 0.5, 0)$) .. (C);
\draw (B) .. controls ($(B) + (0, 0.2, 0)$) and ($(C) + (0.2, 0, 0)$) .. (C);

\coordinate (A1) at (0.1, 0.05, 0.075);
\coordinate (A2) at (0.05, 0.1, 0.075);
\coordinate (B1) at (0.45, 0.05, -0.125);
\coordinate (B2) at (0.4, 0.1, -0.125);
\coordinate (C1) at (0.05, 0.45, -0.125);
\coordinate (C2) at (0.1, 0.4, -0.125);

\node at (A1) {
\begin{tikzpicture}

\draw (0,0) circle (0.125);

\node at (0,0) {\tiny $1$};

\end{tikzpicture}
};
\node at (A2) {
\begin{tikzpicture}

\draw (0,0) circle (0.125);

\node at (0,0) {\tiny $2$};

\end{tikzpicture}
};
\node at (B1) {
\begin{tikzpicture}

\draw (0,0) circle (0.125);

\node at (0,0) {\tiny $3$};

\end{tikzpicture}
};
\node at (B2) {
\begin{tikzpicture}

\draw (0,0) circle (0.125);

\node at (0,0) {\tiny $4$};

\end{tikzpicture}
};
\node at (C1) {
\begin{tikzpicture}

\draw (0,0) circle (0.125);

\node at (0,0) {\tiny $5$};

\end{tikzpicture}
};
\node at (C2) {
\begin{tikzpicture}

\draw (0,0) circle (0.125);

\node at (0,0) {\tiny $6$};

\end{tikzpicture}
};

\draw[color=red] (A1) .. controls ($(A1) + (0.2, 0, 0)$) and ($(B1) + (0.5, 0, 0)$) .. (B1);
\draw[color=red] (A2) .. controls ($(A2) + (0, 0.2, 0)$) and ($(C1) + (0, 0.5, 0)$) .. (C1);
\draw[color=red] (B2) .. controls ($(B2) + (0, 0.2, 0)$) and ($(C2) + (0.2, 0, 0)$) .. (C2);

\draw[color=red] (A1) .. controls (0.05, 0.05, 0.075) .. (A2);
\draw[color=red] (B1) .. controls (0.4, 0.05, -0.125) .. (B2);
\draw[color=red] (C1) .. controls (0.05, 0.4, -0.125) .. (C2);

\node (broken) at (-0.25,0) {
\begin{tikzpicture}[scale=0.7]

\draw (0,0) to[out=20,in=-20] ++(0,1) to[out=-160,in=160] ++(0,-1);
\draw (0,1) to[out=20,in=-20] ++(0,1) to[out=-160,in=160] ++(0,-1);
\draw (0,2) to[out=20,in=-20] ++(0,1) to[out=-160,in=160] ++(0,-1);

\node at (0,0.5) {\tiny $R$};
\node at (0,1.5) {\tiny $S$};
\node at (0,2.5) {\tiny $R$};

\end{tikzpicture}
};

\node (bubblefirst) at (0.6,0.3) {
\begin{tikzpicture}[scale=0.7]

\draw (0,0) to[out=20,in=-20] ++(0,1) to[out=-160,in=160] ++(0,-1);
\filldraw[gray] (0,1) to[out=20,in=-20] ++(0,1) to[out=-160,in=160] ++(0,-1);
\draw (0.5,1.5) circle (0.2);

\node at (0,0.5) {\tiny $R$};

\end{tikzpicture}
};

\node (bubblelast) at (-0.1,0.7) {
\begin{tikzpicture}[scale=0.7]

\filldraw[gray] (0,0) to[out=20,in=-20] ++(0,1) to[out=-160,in=160] ++(0,-1);
\draw (0,1) to[out=20,in=-20] ++(0,1) to[out=-160,in=160] ++(0,-1);
\draw (0.5,0.5) circle (0.2);

\node at (0,1.5) {\tiny $R$};

\end{tikzpicture}
};

\draw[thin, densely dotted, ->] (broken) -- (A);
\draw[thin, densely dotted, ->] (bubblefirst) -- (B);
\draw[thin, densely dotted, ->] (bubblelast) -- (C);

\end{tikzpicture}\caption{Type 3.3b, in the case when the annulus is horizontal (left) or vertical (right). In either case $\partial \mathcal{M}(D, 0, 0)$ is a triangle, and $\partial' \mathcal{M}(D, 0, 0)$ (red) is a hexagon embedded in the interior of $\mathbb{E}_2^d$ with the numbered vertices.}\label{fig9}\end{figure}

Let $r$ be the sign of rectangle $R$ and $s$ be the sign of rectangle $S$. Similarly to Type 3.3a, we can decompose $\partial' \mathcal{M}(D, 0, 0)$ into short and long preferred paths as shown in Figure $\ref{fig10}$, which follows Conventions $\ref{conv1}, \ref{conv2},$ and $\ref{conv3}$, and the following conventions:

\begin{conv}\label{conv4}\begin{itemize}

\item The points with $\pm f_1$ or $\pm f_2$ are plus if the annulus $D$ contains is vertical, and minus if the annulus $D$ contains is horizontal.

\item The dashed black paths are identity for the vertical annulus, and the long preferred path in the second $\R_+$ with respect to the first $\R_+$ for the horizontal annulus. They are directed according to Convention $\ref{conv3}$.
\end{itemize}\end{conv}

\begin{figure}

\begin{tikzpicture}[scale=6]


\node[anchor=west] at (2, 1.4) {\small $\pm f_2 r00$};
\node[anchor=west] at (2, 1.6) {\small $f_1 r00$};

\node[anchor=south] at (0.8, 2) {\small $f_2 rsr$};
\node[anchor=south] at (1.2, 2) {\small $f_1 rsr$};

\node[anchor=east] at (0, 1.4) {\small $\pm f_1 00r$};
\node[anchor=east] at (0, 1.6) {\small $f_2 00r$};

\node at (1.2, 1.6) {\small $f_1 000$};

\draw (0.8, 2) -- (1.2, 2);
\draw (0, 1.4) -- (0, 1.6);
\draw (2, 1.4) -- (2, 1.6);

\draw[->-,thick,color=green] (0.8, 2) -- (0.4, 1.8);
\draw[->-,thick,color=red] (0.4, 1.8) -- (0, 1.6);

\draw[->-,thick,color=blue] (1.2, 2) -- (1.6, 1.8);
\draw[->-,thick,color=green] (1.6, 1.8) -- (2, 1.6);

\draw[->-,thick,color=blue] (0, 1.4) -- (0.5, 1.4);
\draw[->-,thick,color=red] (2, 1.4) -- (1.5, 1.4);
\draw[thick] (0.5, 1.4) -- (1.5, 1.4);


\draw[->-,thick,color=green] (1.2, 2) -- (0.8, 1.8);
\draw[->-,thick,color=red] (0.8, 1.8) -- (0.4, 1.6);
\draw[thick] (0.4, 1.8) -- (0.8, 1.8);
\draw[thick] (0, 1.6) -- (0.4, 1.6);
\draw[thick,dashed,->-] (0, 1.4) -- (0.4, 1.6);


\draw[->-,thick,color=blue] (0.8, 1.8) -- (2, 1.6);
\draw[->-,thick,color=red] (2, 1.6) -- (1.2, 1.6);
\draw[->-,thick,color=blue] (0.4, 1.6) -- (1.2, 1.6);
\draw[thick,dashed,->-] (0.5, 1.4) -- (1.2, 1.6);
\draw[thick] (1.2, 1.6) -- (1.5, 1.4);

\node at (0.4, 1.6) {\small $f_1 00r$};
\node at (0.8, 1.8) {\small $f_1 r0r$};
\node[anchor=south east] at (0.4, 1.8) {\small $f_2 r0r$};
\node[anchor=south west] at (1.6, 1.8) {\small $f_1 rs0$};
\node[anchor=north] at (0.5, 1.4) {\small $\pm f_1 000$};
\node[anchor=north] at (1.5, 1.4) {\small $\pm f_2 000$};


\node at (1.3, 1.8) {\small $rs$};
\node at (0.8, 1.9) {\small $s$};
\node at (0.4, 1.7) {\small $r$};
\node at (0.8, 1.7) {\small $r^2 = r$};
\node at (0.1, 1.5) {\small $0$};
\node at (0.5, 1.5) {\small $0 \text{ or } r$};
\node at (1.1, 1.5) {\small $0 \text{ or } 1$};
\node at (1.7, 1.5) {\small $r$};

\node at (0.8, 2) {
\begin{tikzpicture}

\draw[fill=white] (0,0) circle (0.125);

\node at (0,0) {\tiny $1$};

\end{tikzpicture}
};
\node at (1.2, 2) {
\begin{tikzpicture}

\draw[fill=white] (0,0) circle (0.125);

\node at (0,0) {\tiny $2$};

\end{tikzpicture}
};
\node at (0, 1.4) {
\begin{tikzpicture}

\draw[fill=white] (0,0) circle (0.125);

\node at (0,0) {\tiny $4$};

\end{tikzpicture}
};
\node at (0, 1.6) {
\begin{tikzpicture}

\draw[fill=white] (0,0) circle (0.125);

\node at (0,0) {\tiny $3$};

\end{tikzpicture}
};
\node at (2, 1.4) {
\begin{tikzpicture}

\draw[fill=white] (0,0) circle (0.125);

\node at (0,0) {\tiny $5$};

\end{tikzpicture}
};
\node at (2, 1.6) {
\begin{tikzpicture}

\draw[fill=white] (0,0) circle (0.125);

\node at (0,0) {\tiny $6$};

\end{tikzpicture}
};

\end{tikzpicture}
\caption{Type 3.3b. The numbered vertices correspond to the numbered vertices of $\partial' \mathcal{M}(D, 0, 0)$ of Figure $\ref{fig9}$.}\label{fig10}\end{figure}

Similarly to Type 3.3a, we use Lemmas $\ref{lem1}, \ref{lem2}$ to compute the homotopy class of most quadrilaterals in Figure $\ref{fig10}$. If $A$ is a vertical annulus, the dashed lines are identity, so the dashed black-blue quadrilateral and both black triangles are clearly nullhomotopic. In this case, $r + s = 1 \pmod{2}$ by the definition of a sign assignment, so that
\begin{align*}\partial f = 1 + 3r + s(r + 1) = 1 + r + s^2 = 0 \pmod{2}\end{align*}
If $A$ is a horizontal annulus, by Lemma $\ref{lem2}$ the dashed black-blue quadrilateral's homotopy class is given by $r$, while the left and right dashed triangles are nullhomotopic and not nullhomotopic, respectively. (This can be checked with hand motions.) In this case, $r + s = 0 \pmod{2}$ by the definition of a sign assignment, so that
\begin{align*}\partial f = 1 + 1 + 4r + s(r + 1) = s(s + 1) = 0 \pmod{2}\end{align*}

\textbf{Type 3.3c}: In this case, $D$ contains an annulus $A = R * S$, and $D = A * T$ for some rectangle $T \neq R$. In this case, $\mathcal{M}(D, 0, 0)$ is the following pentagon:

\begin{figure}

\begin{tikzpicture}[x={(1,0)}, y={(0,1)}, z={(-0.5, -0.5)}, scale=5.5]

\draw[->] (0,0,0) -- (0.5,0,0) node[anchor=west]{$\R_{+(1)}$};
\draw[->] (0,0,0) -- (0,1,0) node[anchor=south]{$\R_{+(2)}$};
\draw[->] (0,0,-0.5) -- (0,0,1) node[anchor=north west]{$\R^{3d}$};

\coordinate (A) at (0, 0, 0.75);
\coordinate (B) at (0, 0, 0.45);
\coordinate (C) at (0, 0, 0.15);
\coordinate (D) at (0.5, 0, -0.1);
\coordinate (E) at (0, 0.8, -0.1);

\draw (A) .. controls ($(A) + (0.2, 0, 0)$) and ($(B) + (0.2, 0, 0)$) .. (B);
\draw (B) .. controls ($(B) + (0, 0.2, 0)$) and ($(C) + (0, 0.2, 0)$) .. (C);
\draw (C) .. controls ($(C) + (0.2, 0, 0)$) and ($(D) + (-0.2, 0, 0)$) .. (D);
\draw (A) .. controls ($(A) + (0, 0.2, 0)$) and ($(E) + (0, -0.2, 0)$) .. (E);
\draw (D) .. controls ($(D) + (0, 0.2, 0)$) and ($(E) + (0.2, 0, 0)$) .. (E);

\coordinate (A1) at (0.1, 0.05, 0.75);
\coordinate (A2) at (0.05, 0.1, 0.75);
\coordinate (B1) at (0.1, 0.05, 0.45);
\coordinate (B2) at (0.05, 0.1, 0.45);
\coordinate (C1) at (0.1, 0.05, 0.15);
\coordinate (C2) at (0.05, 0.1, 0.15);
\coordinate (D1) at (0.4, 0.05, -0.1);
\coordinate (D2) at (0.45, 0.1, -0.1);
\coordinate (E1) at (0.1, 0.75, -0.1);
\coordinate (E2) at (0.05, 0.7, -0.1);

\node at (A1) {
\begin{tikzpicture}

\draw (0,0) circle (0.125);

\node at (0,0) {\tiny $1$};

\end{tikzpicture}
};
\node at (A2) {
\begin{tikzpicture}

\draw (0,0) circle (0.125);

\node at (0,0) {\tiny $2$};

\end{tikzpicture}
};
\node at (B1) {
\begin{tikzpicture}

\draw (0,0) circle (0.125);

\node at (0,0) {\tiny $3$};

\end{tikzpicture}
};
\node at (B2) {
\begin{tikzpicture}

\draw (0,0) circle (0.125);

\node at (0,0) {\tiny $4$};

\end{tikzpicture}
};
\node at (C1) {
\begin{tikzpicture}

\draw (0,0) circle (0.125);

\node at (0,0) {\tiny $5$};

\end{tikzpicture}
};
\node at (C2) {
\begin{tikzpicture}

\draw (0,0) circle (0.125);

\node at (0,0) {\tiny $6$};

\end{tikzpicture}
};
\node at (D1) {
\begin{tikzpicture}

\draw (0,0) circle (0.125);

\node at (0,0) {\tiny $7$};

\end{tikzpicture}
};
\node at (D2) {
\begin{tikzpicture}

\draw (0,0) circle (0.125);

\node at (0,0) {\tiny $8$};

\end{tikzpicture}
};
\node at (E1) {
\begin{tikzpicture}

\draw (0,0) circle (0.125);

\node at (0,0) {\tiny $9$};

\end{tikzpicture}
};
\node at (E2) {
\begin{tikzpicture}

\draw (0,0) circle (0.125);

\node at (0,0) {\tiny $10$};

\end{tikzpicture}
};

\draw[color=red] (A1) .. controls ($(A1) + (0.2, 0, 0)$) and ($(B1) + (0.2, 0, 0)$) .. (B1);
\draw[color=red] (B2) .. controls ($(B2) + (0, 0.2, 0)$) and ($(C2) + (0, 0.2, 0)$) .. (C2);
\draw[color=red] (C1) .. controls ($(C1) + (0.2, 0, 0)$) and ($(D1) + (-0.2, 0, 0)$) .. (D1);
\draw[color=red] (A2) .. controls ($(A2) + (0, 0.2, 0)$) and ($(E2) + (0, -0.2, 0)$) .. (E2);
\draw[color=red] (D2) .. controls ($(D) + (0, 0.2, 0)$) and ($(E) + (0.2, 0, 0)$) .. (E1);

\draw[color=red] (A1) .. controls (0.05, 0.05, 0.75) .. (A2);
\draw[color=red] (B1) .. controls (0.05, 0.05, 0.45) .. (B2);
\draw[color=red] (C1) .. controls (0.05, 0.05, 0.15) .. (C2);
\draw[color=red] (D1) .. controls (0.45, 0.05, -0.1) .. (D2);
\draw[color=red] (E1) .. controls (0.05, 0.75, -0.1) .. (E2);

\node (broken) at (-0.4,0.5) {
\begin{tikzpicture}[scale=0.7]

\draw (0,0) to[out=20,in=-20] ++(0,1) to[out=-160,in=160] ++(0,-1);
\draw (0,1) to[out=20,in=-20] ++(0,1) to[out=-160,in=160] ++(0,-1);
\draw (0,2) to[out=20,in=-20] ++(0,1) to[out=-160,in=160] ++(0,-1);

\end{tikzpicture}
};

\node (bubblefirst) at (0.6,0.3) {
\begin{tikzpicture}[scale=0.7]

\draw (0,0) to[out=20,in=-20] ++(0,1) to[out=-160,in=160] ++(0,-1);
\filldraw[gray] (0,1) to[out=20,in=-20] ++(0,1) to[out=-160,in=160] ++(0,-1);
\draw (0.5,1.5) circle (0.2);

\node at (0,0.5) {\tiny $T$};

\end{tikzpicture}
};

\node (bubblelast) at (-0.2,0.9) {
\begin{tikzpicture}[scale=0.7]

\filldraw[gray] (0,0) to[out=20,in=-20] ++(0,1) to[out=-160,in=160] ++(0,-1);
\draw (0,1) to[out=20,in=-20] ++(0,1) to[out=-160,in=160] ++(0,-1);
\draw (0.5,0.5) circle (0.2);

\node at (0,1.5) {\tiny $T$};

\end{tikzpicture}
};

\draw[thin, densely dotted, ->] (broken) -- (A);
\draw[thin, densely dotted, ->] (broken) -- (B);
\draw[thin, densely dotted, ->] (broken) -- (C);
\draw[thin, densely dotted, ->] (bubblefirst) -- (D);
\draw[thin, densely dotted, ->] (bubblelast) -- (E);

\end{tikzpicture} \begin{tikzpicture}[x={(1,0)}, y={(0,1)}, z={(-0.5, -0.5)}, scale=5.5]

\draw[->] (0,0,0) -- (0.5,0,0) node[anchor=north west]{$\R_{+(1)}$};
\draw[->] (0,0,0) -- (0,1,0) node[anchor=south]{$\R_{+(2)}$};
\draw[->] (0,0,-0.5) -- (0,0,1) node[anchor=north west]{$\R^{3d}$};

\coordinate (A) at (0, 0, 0.75);
\coordinate (B) at (0, 0, 0.45);
\coordinate (C) at (0, 0, 0.15);
\coordinate (D) at (0.5, 0, -0.1);
\coordinate (E) at (0, 0.7, -0.1);

\draw (A) .. controls ($(A) + (0.2, 0, 0)$) and ($(B) + (0.2, 0, 0)$) .. (B);
\draw (B) .. controls ($(B) + (0, 0.2, 0)$) and ($(C) + (0, 0.2, 0)$) .. (C);
\draw (C) .. controls ($(C) + (0.2, 0, 0)$) and ($(D) + (0.5, 0, 0)$) .. (D);
\draw (A) .. controls ($(A) + (0, 0.2, 0)$) and ($(E) + (0, 0.5, 0)$) .. (E);
\draw (D) .. controls ($(D) + (0, 0.2, 0)$) and ($(E) + (0.2, 0, 0)$) .. (E);

\coordinate (A1) at (0.1, 0.05, 0.75);
\coordinate (A2) at (0.05, 0.1, 0.75);
\coordinate (B1) at (0.1, 0.05, 0.45);
\coordinate (B2) at (0.05, 0.1, 0.45);
\coordinate (C1) at (0.1, 0.05, 0.15);
\coordinate (C2) at (0.05, 0.1, 0.15);
\coordinate (D1) at (0.45, 0.05, -0.1);
\coordinate (D2) at (0.4, 0.1, -0.1);
\coordinate (E1) at (0.1, 0.6, -0.1);
\coordinate (E2) at (0.05, 0.65, -0.1);

\node at (A1) {
\begin{tikzpicture}

\draw (0,0) circle (0.125);

\node at (0,0) {\tiny $1$};

\end{tikzpicture}
};
\node at (A2) {
\begin{tikzpicture}

\draw (0,0) circle (0.125);

\node at (0,0) {\tiny $2$};

\end{tikzpicture}
};
\node at (B1) {
\begin{tikzpicture}

\draw (0,0) circle (0.125);

\node at (0,0) {\tiny $3$};

\end{tikzpicture}};
\node at (B2) {
\begin{tikzpicture}

\draw (0,0) circle (0.125);

\node at (0,0) {\tiny $4$};

\end{tikzpicture}
};
\node at (C1) {
\begin{tikzpicture}

\draw (0,0) circle (0.125);

\node at (0,0) {\tiny $5$};

\end{tikzpicture}
};
\node at (C2) {
\begin{tikzpicture}

\draw (0,0) circle (0.125);

\node at (0,0) {\tiny $6$};

\end{tikzpicture}
};
\node at (D1) {
\begin{tikzpicture}

\draw (0,0) circle (0.125);

\node at (0,0) {\tiny $7$};

\end{tikzpicture}
};
\node at (D2) {
\begin{tikzpicture}

\draw (0,0) circle (0.125);

\node at (0,0) {\tiny $8$};

\end{tikzpicture}
};
\node at (E1) {
\begin{tikzpicture}

\draw (0,0) circle (0.125);

\node at (0,0) {\tiny $9$};

\end{tikzpicture}
};
\node at (E2) {
\begin{tikzpicture}

\draw (0,0) circle (0.125);

\node at (0,0) {\tiny $10$};

\end{tikzpicture}
};

\draw[color=red] (A1) .. controls ($(A1) + (0.2, 0, 0)$) and ($(B1) + (0.2, 0, 0)$) .. (B1);
\draw[color=red] (B2) .. controls ($(B2) + (0, 0.2, 0)$) and ($(C2) + (0, 0.2, 0)$) .. (C2);
\draw[color=red] (C1) .. controls ($(C1) + (0.2, 0, 0)$) and ($(D1) + (0.5, 0, 0)$) .. (D1);
\draw[color=red] (A2) .. controls ($(A2) + (0, 0.2, 0)$) and ($(E2) + (0, 0.5, 0)$) .. (E2);
\draw[color=red] (D2) .. controls ($(D) + (0, 0.2, 0)$) and ($(E) + (0.2, 0, 0)$) .. (E1);

\draw[color=red] (A1) .. controls (0.05, 0.05, 0.75) .. (A2);
\draw[color=red] (B1) .. controls (0.05, 0.05, 0.45) .. (B2);
\draw[color=red] (C1) .. controls (0.05, 0.05, 0.15) .. (C2);
\draw[color=red] (D1) .. controls (0.4, 0.05, -0.1) .. (D2);
\draw[color=red] (E1) .. controls (0.05, 0.6, -0.1) .. (E2);

\node (broken) at (-0.4,0.5) {
\begin{tikzpicture}[scale=0.7]

\draw (0,0) to[out=20,in=-20] ++(0,1) to[out=-160,in=160] ++(0,-1);
\draw (0,1) to[out=20,in=-20] ++(0,1) to[out=-160,in=160] ++(0,-1);
\draw (0,2) to[out=20,in=-20] ++(0,1) to[out=-160,in=160] ++(0,-1);

\end{tikzpicture}
};

\node (bubblefirst) at (0.6,0.3) {
\begin{tikzpicture}[scale=0.7]

\draw (0,0) to[out=20,in=-20] ++(0,1) to[out=-160,in=160] ++(0,-1);
\filldraw[gray] (0,1) to[out=20,in=-20] ++(0,1) to[out=-160,in=160] ++(0,-1);
\draw (0.5,1.5) circle (0.2);

\node at (0,0.5) {\tiny $T$};

\end{tikzpicture}
};

\node (bubblelast) at (-0.2,0.9) {
\begin{tikzpicture}[scale=0.7]

\filldraw[gray] (0,0) to[out=20,in=-20] ++(0,1) to[out=-160,in=160] ++(0,-1);
\draw (0,1) to[out=20,in=-20] ++(0,1) to[out=-160,in=160] ++(0,-1);
\draw (0.5,0.5) circle (0.2);

\node at (0,1.5) {\tiny $T$};

\end{tikzpicture}
};

\draw[thin, densely dotted, ->] (broken) -- (A);
\draw[thin, densely dotted, ->] (broken) -- (B);
\draw[thin, densely dotted, ->] (broken) -- (C);
\draw[thin, densely dotted, ->] (bubblefirst) -- (D);
\draw[thin, densely dotted, ->] (bubblelast) -- (E);

\end{tikzpicture}
\caption{Type 3.3c, in the case when the annulus is horizontal (left) or vertical (right). In either case $\partial \mathcal{M}(D, 0, 0)$ is a pentagon, and $\partial' \mathcal{M}(D, 0, 0)$ (red) is a $10$-gon embedded in the interior of $\mathbb{E}_2^d$ with the numbered vertices.}\label{fig11}\end{figure}

Let $r, s, t$ be the signs of the rectangle $R, S, T$, respectively. Following Conventions $\ref{conv1}, \ref{conv2}, \ref{conv3}$ and $\ref{conv4}$, we can decompose $\partial' \mathcal{M}(D, 0, 0)$ into short and long preferred paths as shown in Figure $\ref{fig12}$. 

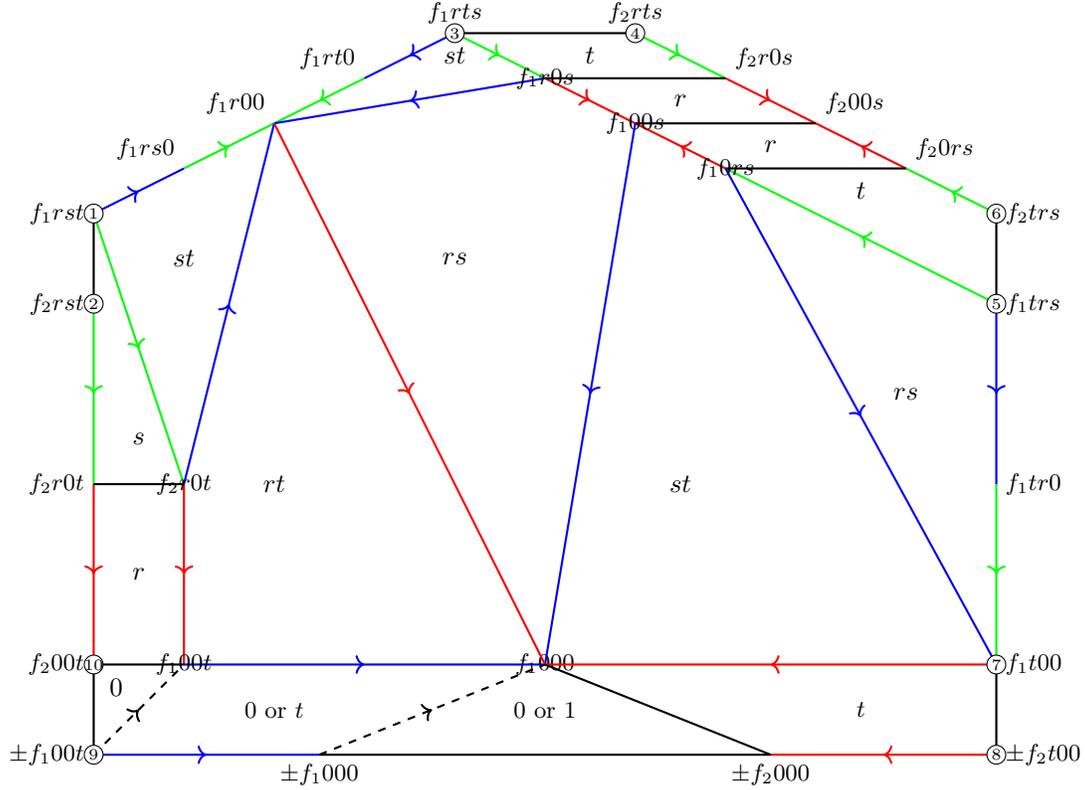
\begin{figure}
\begin{tikzpicture}[scale=6]


\node[anchor=west] at (2, 0.4) {\small $\pm f_2 t00$};
\node[anchor=west] at (2, 0.6) {\small $f_1 t00$};

\node[anchor=west] at (2, 1.4) {\small $f_1 trs$};
\node[anchor=west] at (2, 1.6) {\small $f_2 trs$};

\node[anchor=south] at (0.8, 2) {\small $f_1 rts$};
\node[anchor=south] at (1.2, 2) {\small $f_2 rts$};

\node[anchor=east] at (0, 1.4) {\small $f_2 rst$};
\node[anchor=east] at (0, 1.6) {\small $f_1 rst$};

\node[anchor=east] at (0, 0.4) {\small $\pm f_1 00t$};
\node[anchor=east] at (0, 0.6) {\small $f_2 00t$};

\node at (1, 0.6) {\small $f_1 000$};

\draw[thick] (0.8, 2) -- (1.2, 2);
\draw[thick] (0, 0.4) -- (0, 0.6);
\draw[thick] (0, 1.4) -- (0, 1.6);
\draw[thick] (2, 0.4) -- (2, 0.6);
\draw[thick] (2, 1.4) -- (2, 1.6);

\draw[->-,thick,color=blue] (0, 1.6) -- (0.2, 1.7);
\draw[->-,thick,color=green] (0.2, 1.7) -- (0.4, 1.8);
\draw[->-,thick,color=green] (0.6, 1.9) -- (0.4, 1.8);
\draw[->-,thick,color=blue] (0.8, 2) -- (0.6, 1.9);

\draw[->-,thick,color=green] (1.2, 2) -- (1.4, 1.9);
\draw[->-,thick,color=red] (1.4, 1.9) -- (1.6, 1.8);
\draw[->-,thick,color=red] (1.8, 1.7) -- (1.6, 1.8);
\draw[->-,thick,color=green] (2, 1.6) -- (1.8, 1.7);

\draw[->-,thick,color=green] (0, 1.4) -- (0, 1);
\draw[->-,thick,color=red] (0, 1) -- (0, 0.6);

\draw[->-,thick,color=blue] (2, 1.4) -- (2, 1);
\draw[->-,thick,color=green] (2, 1) -- (2, 0.6);

\draw[->-,thick,color=blue] (0, 0.4) -- (0.5, 0.4);
\draw[->-,thick,color=red] (2, 0.4) -- (1.5, 0.4);
\draw[thick] (0.5, 0.4) -- (1.5, 0.4);


\draw[->-,thick,color=green] (0.8, 2) -- (1, 1.9);
\draw[->-,thick,color=red] (1.4, 1.7) -- (1.2, 1.8);
\draw[->-,thick,color=red] (1, 1.9) -- (1.2, 1.8);
\draw[->-,thick,color=green] (2, 1.4) -- (1.4, 1.7);
\draw[thick] (1, 1.9) -- (1.4, 1.9);
\draw[thick] (1.2, 1.8) -- (1.6, 1.8);
\draw[thick] (1.4, 1.7) -- (1.8, 1.7);

\draw[->-,thick,color=red] (0.2, 1) -- (0.2, 0.6);
\draw[->-,thick,color=green] (0, 1.6) -- (0.2, 1);
\draw[thick] (0, 0.6) -- (0.2, 0.6);
\draw[thick] (0, 1) -- (0.2, 1);
\draw[thick,dashed,->-] (0, 0.4) -- (0.2, 0.6);


\draw[->-,thick,color=blue] (0.2, 1) -- (0.4, 1.8);
\draw[->-,thick,color=blue] (1.2, 1.8) -- (1, 0.6);
\draw[->-,thick,color=blue] (1.4, 1.7) -- (2, 0.6);
\draw[->-,thick,color=blue] (1, 1.9) -- (0.4, 1.8);
\draw[->-,thick,color=blue] (0.2, 0.6) -- (1, 0.6);
\draw[->-,thick,color=red] (0.4, 1.8) -- (1, 0.6);
\draw[->-,thick,color=red] (2, 0.6) -- (1, 0.6);
\draw[thick,dashed,->-] (0.5, 0.4) -- (1, 0.6);
\draw[thick] (1, 0.6) -- (1.5, 0.4);

\node at (0.2, 0.6) {\small $f_1 00t$};
\node at (0.2, 1) {\small $f_2 r0t$};
\node at (1, 1.9) {\small $f_1 r0s$};
\node at (1.2, 1.8) {\small $f_1 00s$};
\node at (1.4, 1.7) {\small $f_1 0rs$};
\node[anchor=east] at (0, 1) {\small $f_2 r0t$};
\node[anchor=west] at (2, 1) {\small $f_1 tr0$};
\node[anchor=south east] at (0.2, 1.7) {\small $f_1 rs0$};
\node[anchor=south east] at (0.4, 1.8) {\small $f_1 r00$};
\node[anchor=south east] at (0.6, 1.9) {\small $f_1 rt0$};
\node[anchor=south west] at (1.4, 1.9) {\small $f_2 r0s$};
\node[anchor=south west] at (1.6, 1.8) {\small $f_2 00s$};
\node[anchor=south west] at (1.8, 1.7) {\small $f_2 0rs$};
\node[anchor=north] at (0.5, 0.4) {\small $\pm f_1 000$};
\node[anchor=north] at (1.5, 0.4) {\small $\pm f_2 000$};


\node at (1.1, 1.95) {$t$};
\node at (1.3, 1.85) {$r$};
\node at (1.5, 1.75) {$r$};
\node at (1.7, 1.65) {$t$};
\node at (0.8, 1.95) {$st$};
\node at (0.2, 1.5) {$st$};
\node at (1.8, 1.2) {$rs$};
\node at (0.1, 1.1) {$s$};
\node at (0.1, 0.8) {$r$};
\node at (0.05, 0.55) {$0$};
\node at (0.4, 1) {$rt$};
\node at (0.8, 1.5) {$rs$};
\node at (1.3, 1) {$st$};
\node at (0.4, 0.5) {\small $0 \text{ or } t$};
\node at (1, 0.5) {\small $0 \text{ or } 1$};
\node at (1.7, 0.5) {\small $t$};

\node at (0.8, 2) {
\begin{tikzpicture}

\draw[fill=white] (0,0) circle (0.125);

\node at (0,0) {\tiny $3$};

\end{tikzpicture}
};
\node at (1.2, 2) {
\begin{tikzpicture}

\draw[fill=white] (0,0) circle (0.125);

\node at (0,0) {\tiny $4$};

\end{tikzpicture}
};
\node at (0, 1.4) {
\begin{tikzpicture}

\draw[fill=white] (0,0) circle (0.125);

\node at (0,0) {\tiny $2$};

\end{tikzpicture}
};
\node at (0, 1.6) {
\begin{tikzpicture}

\draw[fill=white] (0,0) circle (0.125);

\node at (0,0) {\tiny $1$};

\end{tikzpicture}
};
\node at (2, 1.4) {
\begin{tikzpicture}

\draw[fill=white] (0,0) circle (0.125);

\node at (0,0) {\tiny $5$};

\end{tikzpicture}
};
\node at (2, 1.6) {
\begin{tikzpicture}

\draw[fill=white] (0,0) circle (0.125);

\node at (0,0) {\tiny $6$};

\end{tikzpicture}
};
\node at (2, 0.4) {
\begin{tikzpicture}

\draw[fill=white] (0,0) circle (0.125);

\node at (0,0) {\tiny $8$};

\end{tikzpicture}
};
\node at (2, 0.6) {
\begin{tikzpicture}

\draw[fill=white] (0,0) circle (0.125);

\node at (0,0) {\tiny $7$};

\end{tikzpicture}
};
\node at (0, 0.4) {
\begin{tikzpicture}

\draw[fill=white] (0,0) circle (0.125);

\node at (0,0) {\tiny $9$};

\end{tikzpicture}
};
\node at (0, 0.6) {
\begin{tikzpicture}

\draw[fill=white] (0,0) circle (0.125);

\node at (0,0) {\tiny $10$};

\end{tikzpicture}
};

\end{tikzpicture}
\caption{Type 3.3c. The numbered vertices correspond to the numbered vertices of $\partial' \mathcal{M}(D, 0, 0)$ of Figure $\ref{fig9}$.}\label{fig12}\end{figure}

Similarly to Type 3.3a, we use Lemmas $\ref{lem1}, \ref{lem2}$ to compute the homotopy class of most quadrilaterals in Figure $\ref{fig12}$. If $A$ is a vertical annulus, the dashed lines are identity, so the dashed black-blue quadrilateral and both black triangles are clearly nullhomotopic. In this case, $r + s = 1 \pmod{2}$ by the definition of a sign assignment, so that
\begin{align*}\partial f &= 1 + r + s + t + rt + 2rs + 3st + 2(r + t) \\&= 1 + t + (r + s)(t + 1) = 0 \pmod{2}\end{align*}
If $A$ is a horizontal annulus, by Lemma $\ref{lem2}$ the dashed black-blue quadrilateral's homotopy class is given by $t$, while the left and right dashed triangles are nullhomotopic and not nullhomotopic, respectively. (This can be checked with hand motions.) In this case, $r + s = 0 \pmod{2}$ by the definition of a sign assignment, so that
\begin{align*}\partial f &= 1 + 1 + 2t + r + s + rt + 2rs + 3st + 2(r + t) \\&= (r + s)(t + 1) = 0 \pmod{2}\end{align*}

\textbf{Type 3.2a}: In this case, $D$ is an index $2$ domain, $\vec{\lambda}$ has total length $1$, and $D$ is not an annulus. In this case, $D = A_1 * B_1 = A_2 * B_2$. Let $N$ be the nonzero coordinate of $\vec{N}$, and consider the Figure $\ref{fig13}$.

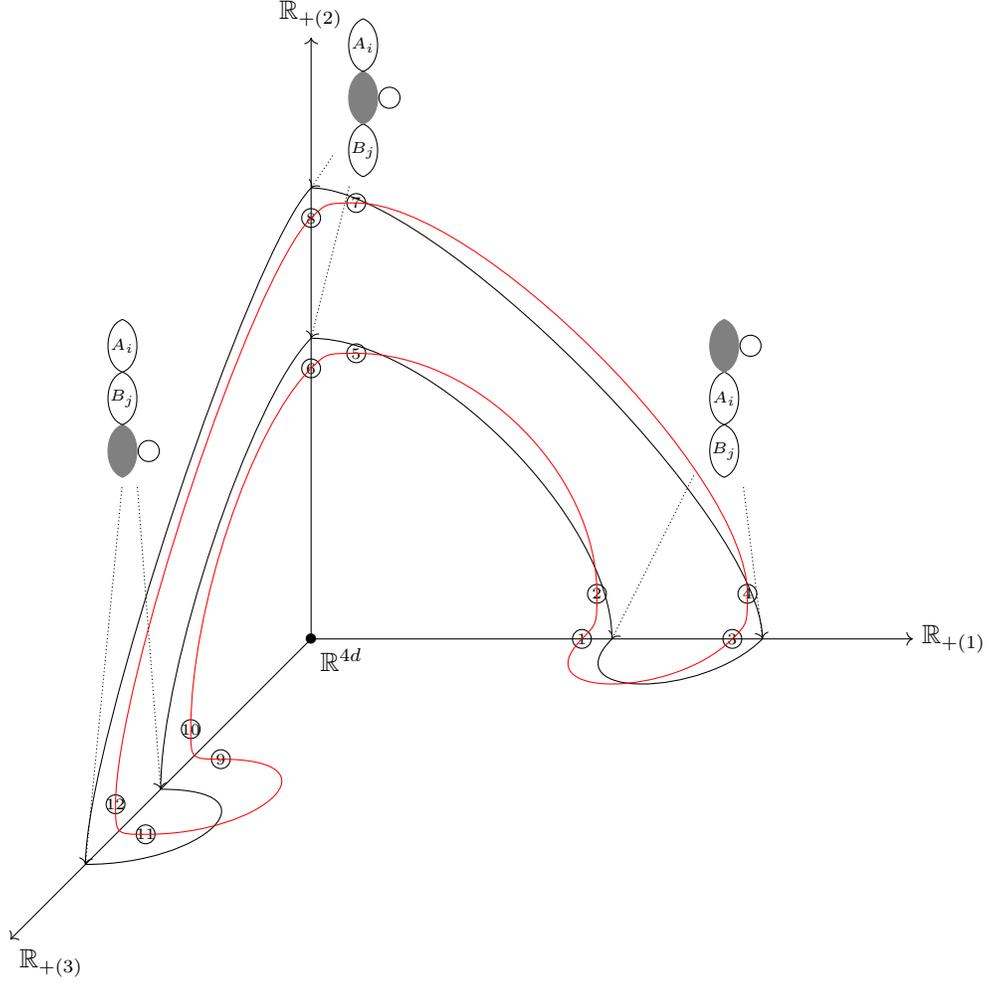
\begin{figure}

\begin{tikzpicture}[x={(1,0)}, y={(0,1)}, z={(-0.5, -0.5)}, scale=8]

\node at (0,0,0) {\textbullet};
\node[anchor=north west] at (0,0,0) {$\R^{4d}$};

\draw[->] (0,0,0) -- (1,0,0) node[anchor=west]{$\R_{+(1)}$};
\draw[->] (0,0,0) -- (0,1,0) node[anchor=south]{$\R_{+(2)}$};
\draw[->] (0,0,0) -- (0,0,1) node[anchor=north west]{$\R_{+(3)}$};

\coordinate (A) at (0.5, 0, 0);
\coordinate (B) at (0.75, 0, 0);
\coordinate (C) at (0, 0.5, 0);
\coordinate (D) at (0, 0.75, 0);
\coordinate (E) at (0, 0, 0.5);
\coordinate (F) at (0, 0, 0.75);

\draw (A) .. controls ($(A) + (0, 0, 0.2)$) and ($(B) + (0, 0, 0.2)$) .. (B);
\draw (E) .. controls ($(E) + (0.2, 0, 0)$) and ($(F) + (0.2, 0, 0)$) .. (F);
\draw (C) .. controls ($(C) + (0.2, 0, 0)$) and ($(A) + (0, 0.2, 0)$) .. (A);
\draw (C) .. controls ($(C) + (0, 0, 0.2)$) and ($(E) + (0, 0.2, 0)$) .. (E);
\draw (D) .. controls ($(D) + (0.2, 0, 0)$) and ($(B) + (0, 0.2, 0)$) .. (B);
\draw (D) .. controls ($(D) + (0, 0, 0.2)$) and ($(F) + (0, 0.2, 0)$) .. (F);

\coordinate (A1) at (0.5, 0.05, 0.1);
\coordinate (A2) at (0.5, 0.1, 0.05);
\coordinate (B1) at (0.75, 0.05, 0.1);
\coordinate (B2) at (0.75, 0.1, 0.05);
\coordinate (C1) at (0.1, 0.5, 0.05);
\coordinate (C2) at (0.05, 0.5, 0.1);
\coordinate (D1) at (0.1, 0.75, 0.05);
\coordinate (D2) at (0.05, 0.75, 0.1);
\coordinate (E1) at (0.1, 0.05, 0.5);
\coordinate (E2) at (0.05, 0.1, 0.5);
\coordinate (F1) at (0.1, 0.05, 0.75);
\coordinate (F2) at (0.05, 0.1, 0.75);

\node at (A1) {
\begin{tikzpicture}

\draw (0,0) circle (0.125);

\node at (0,0) {\tiny $1$};

\end{tikzpicture}
};
\node at (A2) {
\begin{tikzpicture}

\draw (0,0) circle (0.125);

\node at (0,0) {\tiny $2$};

\end{tikzpicture}
};
\node at (B1) {
\begin{tikzpicture}

\draw (0,0) circle (0.125);

\node at (0,0) {\tiny $3$};

\end{tikzpicture}
};
\node at (B2) {
\begin{tikzpicture}

\draw (0,0) circle (0.125);

\node at (0,0) {\tiny $4$};

\end{tikzpicture}
};
\node at (C1) {
\begin{tikzpicture}

\draw (0,0) circle (0.125);

\node at (0,0) {\tiny $5$};

\end{tikzpicture}
};
\node at (C2) {
\begin{tikzpicture}

\draw (0,0) circle (0.125);

\node at (0,0) {\tiny $6$};

\end{tikzpicture}
};
\node at (D1) {
\begin{tikzpicture}

\draw (0,0) circle (0.125);

\node at (0,0) {\tiny $7$};

\end{tikzpicture}
};
\node at (D2) {
\begin{tikzpicture}

\draw (0,0) circle (0.125);

\node at (0,0) {\tiny $8$};

\end{tikzpicture}
};
\node at (E1) {
\begin{tikzpicture}

\draw (0,0) circle (0.125);

\node at (0,0) {\tiny $9$};

\end{tikzpicture}
};
\node at (E2) {
\begin{tikzpicture}

\draw (0,0) circle (0.125);

\node at (0,0) {\tiny $10$};

\end{tikzpicture}
};
\node at (F1) {
\begin{tikzpicture}

\draw (0,0) circle (0.125);

\node at (0,0) {\tiny $11$};

\end{tikzpicture}
};
\node at (F2) {
\begin{tikzpicture}

\draw (0,0) circle (0.125);

\node at (0,0) {\tiny $12$};

\end{tikzpicture}
};

\draw[color=red] (A1) .. controls ($(A1) + (0, 0, 0.2)$) and ($(B1) + (0, 0, 0.2)$) .. (B1);
\draw[color=red] (E1) .. controls ($(E1) + (0.2, 0, 0)$) and ($(F1) + (0.2, 0, 0)$) .. (F1);
\draw[color=red] (C1) .. controls ($(C1) + (0.2, 0, 0)$) and ($(A2) + (0, 0.2, 0)$) .. (A2);
\draw[color=red] (C2) .. controls ($(C2) + (0, 0, 0.2)$) and ($(E2) + (0, 0.2, 0)$) .. (E2);
\draw[color=red] (D1) .. controls ($(D1) + (0.2, 0, 0)$) and ($(B2) + (0, 0.2, 0)$) .. (B2);
\draw[color=red] (D2) .. controls ($(D2) + (0, 0, 0.2)$) and ($(F2) + (0, 0.2, 0)$) .. (F2);

\draw[color=red] (A1) .. controls (0.5, 0.05, 0.05) .. (A2);
\draw[color=red] (B1) .. controls (0.75, 0.05, 0.05) .. (B2);
\draw[color=red] (C1) .. controls (0.05, 0.5, 0.05) .. (C2);
\draw[color=red] (D1) .. controls (0.05, 0.75, 0.05) .. (D2);
\draw[color=red] (E1) .. controls (0.05, 0.05, 0.5) .. (E2);
\draw[color=red] (F1) .. controls (0.05, 0.05, 0.75) .. (F2);

\node (bubblefirst) at (0.7, 0.4) {
\begin{tikzpicture}[scale=0.7]

\draw (0,0) to[out=20,in=-20] ++(0,1) to[out=-160,in=160] ++(0,-1);
\draw (0,1) to[out=20,in=-20] ++(0,1) to[out=-160,in=160] ++(0,-1);
\filldraw[gray] (0,2) to[out=20,in=-20] ++(0,1) to[out=-160,in=160] ++(0,-1);
\draw (0.5, 2.5) circle (0.2);

\node at (0,0.5) {\tiny $B_j$};
\node at (0,1.5) {\tiny $A_i$};

\end{tikzpicture}
};

\node (bubblemiddle) at (0.1, 0.9) {
\begin{tikzpicture}[scale=0.7]

\draw (0,0) to[out=20,in=-20] ++(0,1) to[out=-160,in=160] ++(0,-1);
\filldraw[gray] (0,1) to[out=20,in=-20] ++(0,1) to[out=-160,in=160] ++(0,-1);
\draw (0,2) to[out=20,in=-20] ++(0,1) to[out=-160,in=160] ++(0,-1);
\draw (0.5, 1.5) circle (0.2);

\node at (0,0.5) {\tiny $B_j$};
\node at (0,2.5) {\tiny $A_i$};

\end{tikzpicture}
};

\node (bubblelast) at (-0.3, 0.4) {
\begin{tikzpicture}[scale=0.7]

\filldraw[gray] (0,0) to[out=20,in=-20] ++(0,1) to[out=-160,in=160] ++(0,-1);
\draw (0,1) to[out=20,in=-20] ++(0,1) to[out=-160,in=160] ++(0,-1);
\draw (0,2) to[out=20,in=-20] ++(0,1) to[out=-160,in=160] ++(0,-1);
\draw (0.5, 0.5) circle (0.2);

\node at (0,1.5) {\tiny $B_j$};
\node at (0,2.5) {\tiny $A_i$};

\end{tikzpicture}
};

\draw[thin, densely dotted, ->] (bubblefirst) -- (A);
\draw[thin, densely dotted, ->] (bubblefirst) -- (B);
\draw[thin, densely dotted, ->] (bubblemiddle) -- (C);
\draw[thin, densely dotted, ->] (bubblemiddle) -- (D);
\draw[thin, densely dotted, ->] (bubblelast) -- (E);
\draw[thin, densely dotted, ->] (bubblelast) -- (F);

\end{tikzpicture}\caption{Type 3.2a in the case $l = 1$. $\partial \mathcal{M}(D, l\vec{e_j}, (l))$ is a hexagon. $\partial' \mathcal{M}(D, l\vec{e_j}, (l))$ (red) is a $12$-gon embedded in the interior of $\mathbb{E}_{2l+1}^d$ with the numbered vertices.}\label{fig13}\end{figure}

\begin{remark}Figure $\ref{fig13}$ is slightly misleading. Due to dimension constraints, we are only able to draw the case $N = 1$. In general, the two corners on the second $\R_+$ axis (near vertices $5, 6$ and $7, 8$, respectively) would lie on the second and $2l-1^{st}$ (second to last) $\R_+$ axes, respectively.\end{remark}

The vertices of $\partial' \mathcal{M}(D, N\vec{e_j}, (N)))$ (the red $12$-gon) shown in Figure $\ref{fig13}$ can be written as points in $\SO((2N + 1)(d + 1) - 1)$ that form a loop as above. Similarly to previous cases, most of the framing vectors are identical---the $\R_+$ coordinates and the first, $(d+1)^{st}$, $(2Nd-d+1)^{st}$, and $(2Nd+1)^{st}$ coordinates are the only ones that can differ. We can therefore follow Conventions $\ref{conv1}$ and $\ref{conv3}$ when labelling the decomposition shown in Figure $\ref{fig14}$. For Type 3.2 we use the following color convention:

\begin{conv}\label{conv6}\begin{itemize}

\item All black paths are short preferred.

\item All red paths are long preferred paths changing the first $\R$ coordinate with respect to the first $\R_+$ coordinate.

\item All brown paths are long preferred paths changing the $(d+1)^{st}$ $\R$ coordinate with respect to the first $\R_+$ coordinate.

\item All green paths are long preferred paths changing the $(2Nd-d+1)^{st}$ $\R$ coordinate with respect to the first $\R_+$ coordinate.

\item All blue paths are long preferred paths changing the $(2Nd+1)^{st}$ $\R$ coordinate with respect to the first $\R_+$ coordinate.

\end{itemize}\end{conv}

\begin{figure}
\begin{tikzpicture}[scale=6]


\node[anchor=east] at (0, 0) {\small $f_{2N} f_1 \dots f_{2N-1} 0\dots 0 a_1 b_1$};
\node[anchor=west] at (1, 0) {\small $f_{2N} f_1 \dots f_{2N-1} 0\dots 0 a_2 b_2$};

\node[anchor=east] at (0, 0.1) {\small $f_{2N+1} f_1 \dots f_{2N-1} 0\dots 0 a_1 b_1$};
\node[anchor=west] at (1, 0.1) {\small $f_{2N+1} f_1 \dots f_{2N-1} 0\dots 0 a_2 b_2$};

\node[anchor=east] at (0, 0.8) {\small $f_{2N+1} f_2 \dots f_{2N} a_1 0\dots 0 b_1$};
\node[anchor=west] at (1, 0.8) {\small $f_{2N+1} f_2 \dots f_{2N} a_2 0\dots 0 b_2$};

\node[anchor=east] at (0, 0.9) {\small $f_1 f_2 \dots f_{2N} a_1 0\dots 0 b_1$};
\node[anchor=west] at (1, 0.9) {\small $f_1 f_2 \dots f_{2N} a_2 0\dots 0 b_2$};

\node[anchor=east] at (0, 1.6) {\small $f_1 f_3 \dots f_{2N+1} a_1 b_1 0\dots 0$};
\node[anchor=west] at (1, 1.6) {\small $f_1 f_3 \dots f_{2N+1} a_2 b_2 0\dots 0$};

\node[anchor=east] at (0, 1.7) {\small $f_2 f_3 \dots f_{2N+1} a_1 b_1 0\dots 0$};
\node[anchor=west] at (1, 1.7) {\small $f_2 f_3 \dots f_{2N+1} a_2 b_2 0\dots 0$};

\node at (0, 0.45) {$\dots$};
\node at (0, 1.25) {$\dots$};
\node at (1, 0.45) {$\dots$};
\node at (1, 1.25) {$\dots$};

\draw[thick] (0, 0) -- (0, 0.1);
\draw[thick] (0, 0.2) -- (0, 0.4);
\draw[thick] (0, 0.5) -- (0, 0.7);
\draw[thick] (0, 0.8) -- (0, 0.9);
\draw[thick] (0, 1) -- (0, 1.2);
\draw[thick] (0, 1.3) -- (0, 1.5);
\draw[thick] (0, 1.6) -- (0, 1.7);

\draw[->-,thick,color=green] (0, 0.1) -- (0, 0.2);
\draw[->-,thick,color=red] (0, 0.8) -- (0, 0.7);
\draw[->-,thick,color=blue] (0, 0.9) -- (0, 1);
\draw[->-,thick,color=brown] (0, 1.7) -- (0.25, 1.7);
\draw[->-,thick,color=brown] (0, 1.6) -- (0, 1.5);

\draw[thick] (1, 0) -- (1, 0.1);
\draw[thick] (1, 0.2) -- (1, 0.4);
\draw[thick] (1, 0.5) -- (1, 0.7);
\draw[thick] (1, 0.8) -- (1, 0.9);
\draw[thick] (1, 1) -- (1, 1.2);
\draw[thick] (1, 1.3) -- (1, 1.5);
\draw[thick] (1, 1.6) -- (1, 1.7);

\draw[->-,thick,color=green] (1, 0.1) -- (1, 0.2);
\draw[->-,thick,color=red] (1, 0.8) -- (1, 0.7);
\draw[->-,thick,color=blue] (1, 0.9) -- (1, 1);
\draw[->-,thick,color=brown] (1, 1.7) -- (0.75, 1.7);
\draw[->-,thick,color=brown] (1, 1.6) -- (1, 1.5);

\draw[->-,thick,color=blue] (0, 0) -- (0.25, 0);
\draw[->-,thick,color=green] (0.25, 0) -- (0.5, 0);
\draw[->-,thick,color=green] (0.75, 0) -- (0.5, 0);
\draw[->-,thick,color=blue] (1, 0) -- (0.75, 0);
\draw[->-,thick,color=red] (0.25, 1.7) -- (0.5, 1.7);
\draw[->-,thick,color=red] (0.75, 1.7) -- (0.5, 1.7);


\node at (0.5, 0.45) {$\dots$};
\node at (0.5, 1.25) {$\dots$};

\draw[thick] (0.5, 0) -- (0.5, 0.4);
\draw[thick] (0.5, 0.5) -- (0.5, 1.2);
\draw[thick] (0.5, 1.3) -- (0.5, 1.7);

\draw[thick] (0, 1.5) -- (0.25, 1.7);
\draw[thick] (1, 1.5) -- (0.75, 1.7);

\foreach \i in {1,2,3,4,5,6}{
\draw[->-,thick,color=blue] (0, 0.1 + 0.1 * \i) -- (0.5, 0.1 + 0.1 * \i);
\draw[->-,thick,color=blue] (1, 0.1 + 0.1 * \i) -- (0.5, 0.1 + 0.1 * \i);
\draw[->-,thick,color=red] (0, 0.9 + 0.1 * \i) -- (0.5, 0.9 + 0.1 * \i);
\draw[->-,thick,color=red] (1, 0.9 + 0.1 * \i) -- (0.5, 0.9 + 0.1 * \i);
}
\draw[->-,thick,color=blue] (0, 0.1) -- (0.25, 0.1);
\draw[->-,thick,color=blue] (1, 0.1) -- (0.75, 0.1);
\draw[->-,thick,color=green] (0.25, 0.1) -- (0.5, 0.2);
\draw[->-,thick,color=green] (0.75, 0.1) -- (0.5, 0.2);
\draw[thick] (0.25, 0) -- (0.25, 0.1);
\draw[thick] (0.75, 0) -- (0.75, 0.1);
\draw[->-,thick,color=blue] (0, 0.8) -- (0.25, 0.8);
\draw[->-,thick,color=blue] (1, 0.8) -- (0.75, 0.8);
\draw[->-,thick,color=red] (0.25, 0.8) -- (0.5, 0.7);
\draw[->-,thick,color=red] (0.75, 0.8) -- (0.5, 0.7);
\draw[thick] (0.25, 0.8) -- (0, 1);
\draw[thick] (0.75, 0.8) -- (1, 1);


\node[anchor=east] at (0, 0.2) {\small $f_{2N+1} f_1 \dots f_{2N-1} 0\dots 0 b_1$};
\node[anchor=west] at (1, 0.2) {\small $f_{2N+1} f_1 \dots f_{2N-1} 0\dots 0  b_2$};

\node[anchor=east] at (0, 0.3) {\small $f_{2N+1} (-f_{2N}) f_2 \dots f_{2N-1} 0\dots 0 b_1$};
\node[anchor=west] at (1, 0.3) {\small $f_{2N+1} (-f_{2N}) f_2 \dots f_{2N-1} 0\dots 0 b_2$};

\node[anchor=east] at (0, 0.4) {\small $f_{2N+1} f_2 f_{2N} f_3 \dots f_{2N-1} 0\dots 0 b_1$};
\node[anchor=west] at (1, 0.4) {\small $f_{2N+1} f_2 f_{2N} f_3 \dots f_{2N-1} 0\dots 0 b_2$};

\node[anchor=east] at (0, 0.5) {\small $f_{2N+1} f_2 \dots f_{2N-3} f_{2N} f_{2N-2} f_{2N-1} 0\dots 0 b_1$};
\node[anchor=west] at (1, 0.5) {\small $f_{2N+1} f_2 \dots f_{2N-3} f_{2N} f_{2N-2} f_{2N-1} 0\dots 0 b_2$};

\node[anchor=east] at (0, 0.6) {\small $f_{2N+1} f_2 \dots f_{2N-2} (-f_{2N}) f_{2N-1} 0\dots 0 b_1$};
\node[anchor=west] at (1, 0.6) {\small $f_{2N+1} f_2 \dots f_{2N-2} (-f_{2N}) f_{2N-1} 0\dots 0 b_2$};

\node[anchor=east] at (0, 0.7) {\small $f_{2N+1} f_2 \dots f_{2N} 0\dots 0 b_1$};
\node[anchor=west] at (1, 0.7) {\small $f_{2N+1} f_2 \dots f_{2N} 0\dots 0 b_2$};

\node[anchor=east] at (0, 1) {\small $f_1 f_2 \dots f_{2N} a_1 0\dots 0$};
\node[anchor=west] at (1, 1) {\small $f_1 f_2 \dots f_{2N} a_2 0\dots 0$};

\node[anchor=east] at (0, 1.1) {\small $f_1 f_{2N+1} f_3 \dots f_{2N} a_1 0\dots 0$};
\node[anchor=west] at (1, 1.1) {\small $f_1 f_{2N+1} f_3 \dots f_{2N} a_2 0\dots 0$};

\node[anchor=east] at (0, 1.2) {\small $f_1 f_3 (-f_{2N+1}) f_4 \dots f_{2N} a_1 0\dots 0$};
\node[anchor=west] at (1, 1.2) {\small $f_1 f_3 (-f_{2N+1}) f_4 \dots f_{2N} a_2 0\dots 0$};

\node[anchor=east] at (0, 1.3) {\small $f_1 f_3 \dots f_{2N-2} f_{2N+1} f_{2N-1} f_{2N} a_1 0\dots 0$};
\node[anchor=west] at (1, 1.3) {\small $f_1 f_3 \dots f_{2N-2} f_{2N+1} f_{2N-1} f_{2N} a_2 0\dots 0$};

\node[anchor=east] at (0, 1.4) {\small $f_1 f_3 \dots f_{2N-1} (-f_{2N+1}) f_{2N} a_1 0\dots 0$};
\node[anchor=west] at (1, 1.4) {\small $f_1 f_3 \dots f_{2N-1} (-f_{2N+1}) f_{2N} a_2 0\dots 0$};

\node[anchor=east] at (0, 1.5) {\small $f_1 f_3 \dots f_{2N+1} a_1 0\dots 0$};
\node[anchor=west] at (1, 1.5) {\small $f_1 f_3 \dots f_{2N+1} a_2 0\dots 0$};


\node at (0.1, 1.65) {\small $b_1$};
\node at (0.9, 1.65) {\small $b_2$};
\node at (0.1, 0.85) {\small $b_1$};
\node at (0.9, 0.85) {\small $b_2$};
\node at (0.15, 0.75) {\small $a_1 b_1$};
\node at (0.85, 0.75) {\small $a_2 b_2$};
\node at (0.15, 0.15) {\small $a_1 b_1$};
\node at (0.85, 0.15) {\small $a_2 b_2$};
\node at (0.125, 0.05) {\small $b_1$};
\node at (0.375, 0.05) {\small $a_1$};
\node at (0.625, 0.05) {\small $a_2$};
\node at (0.875, 0.05) {\small $b_2$};
\node at (0.3, 1.6) {\small $a_1$};
\node at (0.7, 1.6) {\small $a_2$};
\node at (0.3, 0.9) {\small $a_1$};
\node at (0.7, 0.9) {\small $a_2$};
\foreach \i in {2,3,5,6}{
\node at (0.25, 0.85 + 0.1 * \i) {\small $a_1$};
\node at (0.75, 0.85 + 0.1 * \i) {\small $a_2$};
\node at (0.25, 0.05 + 0.1 * \i) {\small $b_1$};
\node at (0.75, 0.05 + 0.1 * \i) {\small $b_2$};
}

\node at (0, 0) {
\begin{tikzpicture}

\draw[fill=white] (0,0) circle (0.125);

\node at (0,0) {\tiny $1$};

\end{tikzpicture}
};
\node at (0, 0.1) {
\begin{tikzpicture}

\draw[fill=white] (0,0) circle (0.125);

\node at (0,0) {\tiny $2$};

\end{tikzpicture}
};
\node at (0, 0.8) {
\begin{tikzpicture}

\draw[fill=white] (0,0) circle (0.125);

\node at (0,0) {\tiny $5$};

\end{tikzpicture}
};
\node at (0, 0.9) {
\begin{tikzpicture}

\draw[fill=white] (0,0) circle (0.125);

\node at (0,0) {\tiny $6$};

\end{tikzpicture}
};
\node at (0, 1.6) {
\begin{tikzpicture}

\draw[fill=white] (0,0) circle (0.125);

\node at (0,0) {\tiny $10$};

\end{tikzpicture}
};
\node at (0, 1.7) {
\begin{tikzpicture}

\draw[fill=white] (0,0) circle (0.125);

\node at (0,0) {\tiny $9$};

\end{tikzpicture}
};
\node at (1, 0) {
\begin{tikzpicture}

\draw[fill=white] (0,0) circle (0.125);

\node at (0,0) {\tiny $3$};

\end{tikzpicture}
};
\node at (1, 0.1) {
\begin{tikzpicture}

\draw[fill=white] (0,0) circle (0.125);

\node at (0,0) {\tiny $4$};

\end{tikzpicture}
};
\node at (1, 0.8) {
\begin{tikzpicture}

\draw[fill=white] (0,0) circle (0.125);

\node at (0,0) {\tiny $7$};

\end{tikzpicture}
};
\node at (1, 0.9) {
\begin{tikzpicture}

\draw[fill=white] (0,0) circle (0.125);

\node at (0,0) {\tiny $8$};

\end{tikzpicture}
};
\node at (1, 1.6) {
\begin{tikzpicture}

\draw[fill=white] (0,0) circle (0.125);

\node at (0,0) {\tiny $12$};

\end{tikzpicture}
};
\node at (1, 1.7) {
\begin{tikzpicture}

\draw[fill=white] (0,0) circle (0.125);

\node at (0,0) {\tiny $11$};

\end{tikzpicture}
};

\end{tikzpicture}
\caption{Type 3.2a. Not all vertices are labelled for space constraints, but all labels can be determined from the labelled vertices and Convention $\ref{conv3}$.}\label{fig14}\end{figure}
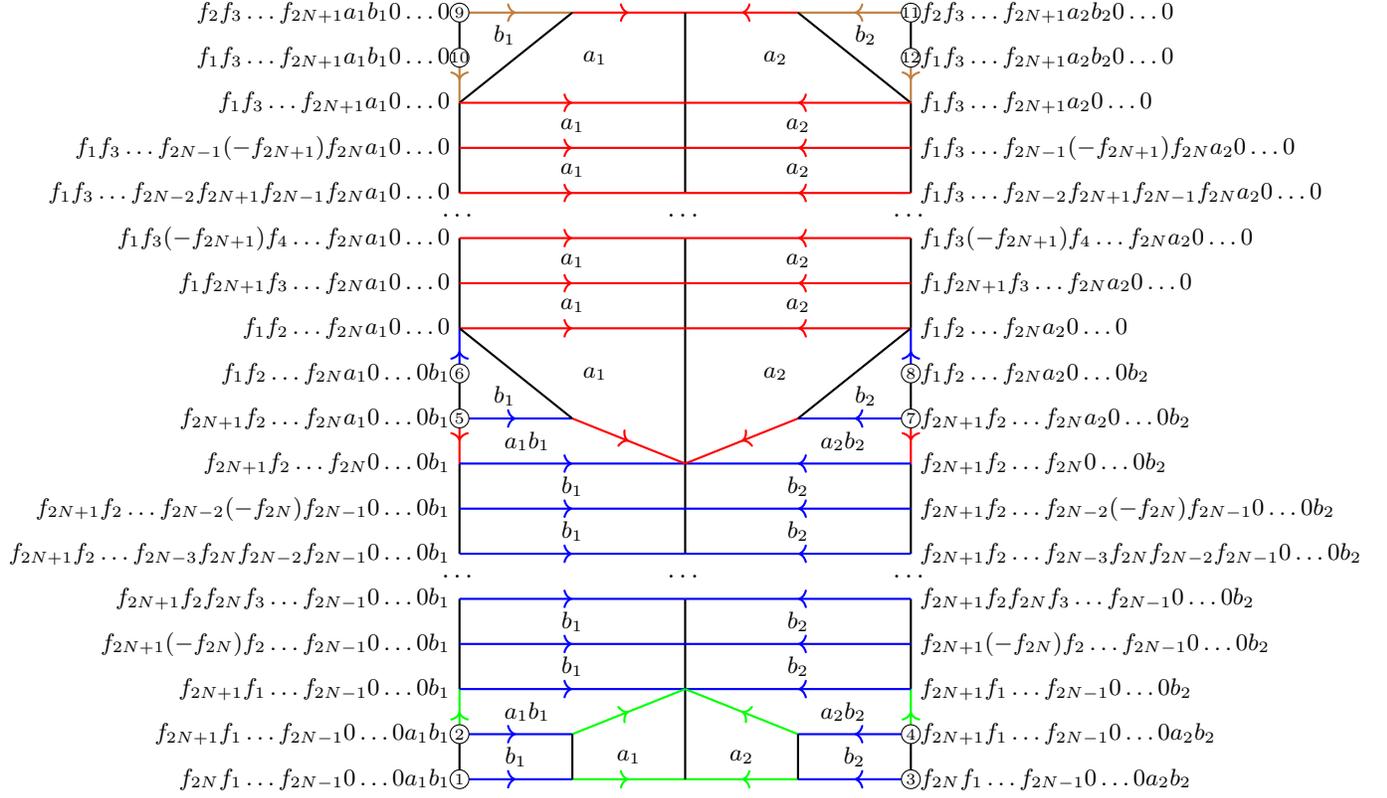

By Lemma $\ref{lem1}$, all black-green, black-blue, black-brown, and black-red quadrilaterals correspond to the element of $\pi_1(\SO((2l + 1)(d + 1) - 1))$ labelled in the above figure, and by Lemma $\ref{lem2}$ all green-blue and red-blue quadrilaterals also correspond to the labelled element of $\pi_1(\SO((2l + 1)(d + 1) - 1))$, so we have
\begin{align*}\delta f(D, l \vec{e}_j, 0) &= 1 + 3(a_1 + a_2 + b_1 + b_2) + 2(a_1 b_1 + a_2 b_2) + (2l-2)(b_1 + b_2) + (2l-2)(a_1 + a_2)\\&= 1 + (2l-1)(a_1 + b_1 + a_2 + b_2) = 0 \pmod{2}\end{align*}

\textbf{Type 3.2b}: $D = V_j$ or $H_j$ is an annulus, but $\lambda_j = \emptyset$. In this case, let $D = R * S$ and $N$ be the nonzero coordinate of $\vec{N}$ (which is also the $k^{th}$ coordinate), and consider the Figure $\ref{fig15}$.

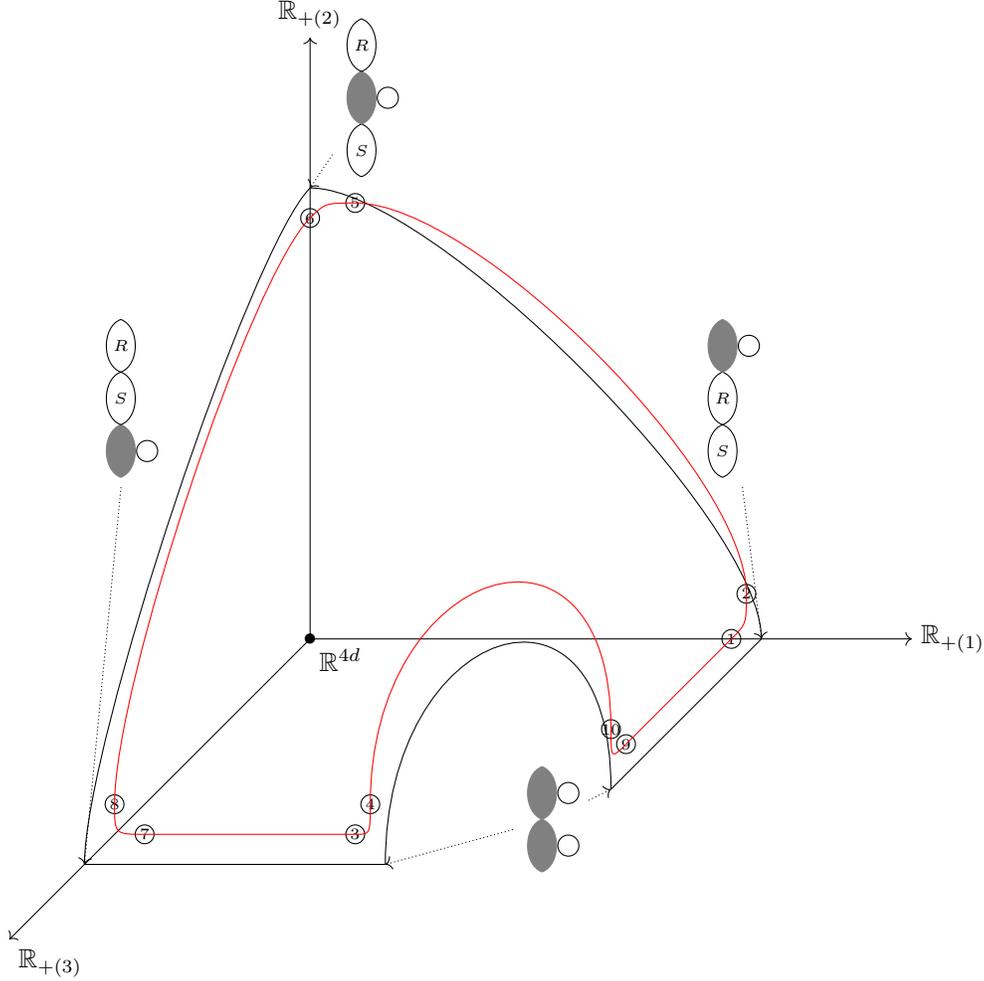
\begin{figure}

\begin{tikzpicture}[x={(1,0)}, y={(0,1)}, z={(-0.5, -0.5)}, scale=8]

\node at (0,0,0) {\textbullet};
\node[anchor=north west] at (0,0,0) {$\R^{4d}$};

\draw[->] (0,0,0) -- (1,0,0) node[anchor=west]{$\R_{+(1)}$};
\draw[->] (0,0,0) -- (0,1,0) node[anchor=south]{$\R_{+(2)}$};
\draw[->] (0,0,0) -- (0,0,1) node[anchor=north west]{$\R_{+(3)}$};

\coordinate (A) at (0.75, 0, 0);
\coordinate (B) at (0.5, 0, 0.75);
\coordinate (C) at (0, 0.75, 0);
\coordinate (D) at (0, 0, 0.75);
\coordinate (E) at (0.75, 0, 0.5);

\draw (A) -- (E);
\draw (D) -- (B);
\draw (A) .. controls ($(A) + (0, 0.2, 0)$) and ($(C) + (0.2, 0, 0)$) .. (C);
\draw (D) .. controls ($(D) + (0, 0.2, 0)$) and ($(C) + (0, 0, 0.2)$) .. (C);
\draw (B) .. controls ($(B) + (0, 0.4, 0)$) and ($(E) + (0, 0.4, 0)$) .. (E);

\coordinate (A1) at (0.75, 0.05, 0.1);
\coordinate (A2) at (0.75, 0.1, 0.05);
\coordinate (B1) at (0.45, 0.05, 0.75);
\coordinate (B2) at (0.475, 0.1, 0.75);
\coordinate (C1) at (0.1, 0.75, 0.05);
\coordinate (C2) at (0.05, 0.75, 0.1);
\coordinate (D1) at (0.1, 0.05, 0.75);
\coordinate (D2) at (0.05, 0.1, 0.75);
\coordinate (E1) at (0.75, 0.05, 0.45);
\coordinate (E2) at (0.75, 0.1, 0.5);

\node at (A1) {
\begin{tikzpicture}

\draw (0,0) circle (0.125);

\node at (0,0) {\tiny $1$};

\end{tikzpicture}
};
\node at (A2) {
\begin{tikzpicture}

\draw (0,0) circle (0.125);

\node at (0,0) {\tiny $2$};

\end{tikzpicture}
};
\node at (B1) {
\begin{tikzpicture}

\draw (0,0) circle (0.125);

\node at (0,0) {\tiny $3$};

\end{tikzpicture}
};
\node at (B2) {
\begin{tikzpicture}

\draw (0,0) circle (0.125);

\node at (0,0) {\tiny $4$};

\end{tikzpicture}
};
\node at (C1) {
\begin{tikzpicture}

\draw (0,0) circle (0.125);

\node at (0,0) {\tiny $5$};

\end{tikzpicture}
};
\node at (C2) {
\begin{tikzpicture}

\draw (0,0) circle (0.125);

\node at (0,0) {\tiny $6$};

\end{tikzpicture}
};
\node at (D1) {
\begin{tikzpicture}

\draw (0,0) circle (0.125);

\node at (0,0) {\tiny $7$};

\end{tikzpicture}
};
\node at (D2) {
\begin{tikzpicture}

\draw (0,0) circle (0.125);

\node at (0,0) {\tiny $8$};

\end{tikzpicture}
};
\node at (E1) {
\begin{tikzpicture}

\draw (0,0) circle (0.125);

\node at (0,0) {\tiny $9$};

\end{tikzpicture}
};
\node at (E2) {\begin{tikzpicture}

\draw (0,0) circle (0.125);

\node at (0,0) {\tiny $10$};

\end{tikzpicture}};

\draw[color=red] (A1) -- (E1);
\draw[color=red] (D1) -- (B1);
\draw[color=red] (C1) .. controls ($(C1) + (0.2, 0, 0)$) and ($(A2) + (0, 0.2, 0)$) .. (A2);
\draw[color=red] (C2) .. controls ($(C2) + (0, 0, 0.2)$) and ($(D2) + (0, 0.2, 0)$) .. (D2);
\draw[color=red] (B2) .. controls ($(B2) + (0, 0.4, 0)$) and ($(E2) + (0, 0.4, 0)$) .. (E2);

\draw[color=red] (A1) .. controls (0.75, 0.05, 0.05) .. (A2);
\draw[color=red] (B1) .. controls (0.475, 0.05, 0.75) .. (B2);
\draw[color=red] (C1) .. controls (0.05, 0.75, 0.05) .. (C2);
\draw[color=red] (D1) .. controls (0.05, 0.05, 0.75) .. (D2);
\draw[color=red] (E1) .. controls (0.75, 0.05, 0.5) .. (E2);

\node (bubblefirst) at (0.7, 0.4) {
\begin{tikzpicture}[scale=0.7]

\draw (0,0) to[out=20,in=-20] ++(0,1) to[out=-160,in=160] ++(0,-1);
\draw (0,1) to[out=20,in=-20] ++(0,1) to[out=-160,in=160] ++(0,-1);
\filldraw[gray] (0,2) to[out=20,in=-20] ++(0,1) to[out=-160,in=160] ++(0,-1);
\draw (0.5, 2.5) circle (0.2);

\node at (0,0.5) {\tiny $S$};
\node at (0,1.5) {\tiny $R$};

\end{tikzpicture}
};

\node (bubblemiddle) at (0.1, 0.9) {
\begin{tikzpicture}[scale=0.7]

\draw (0,0) to[out=20,in=-20] ++(0,1) to[out=-160,in=160] ++(0,-1);
\filldraw[gray] (0,1) to[out=20,in=-20] ++(0,1) to[out=-160,in=160] ++(0,-1);
\draw (0,2) to[out=20,in=-20] ++(0,1) to[out=-160,in=160] ++(0,-1);
\draw (0.5, 1.5) circle (0.2);

\node at (0,0.5) {\tiny $S$};
\node at (0,2.5) {\tiny $R$};

\end{tikzpicture}
};

\node (bubblelast) at (-0.3, 0.4) {
\begin{tikzpicture}[scale=0.7]

\filldraw[gray] (0,0) to[out=20,in=-20] ++(0,1) to[out=-160,in=160] ++(0,-1);
\draw (0,1) to[out=20,in=-20] ++(0,1) to[out=-160,in=160] ++(0,-1);
\draw (0,2) to[out=20,in=-20] ++(0,1) to[out=-160,in=160] ++(0,-1);
\draw (0.5, 0.5) circle (0.2);

\node at (0,1.5) {\tiny $S$};
\node at (0,2.5) {\tiny $R$};

\end{tikzpicture}
};

\node (doublebubble) at (0.4,-0.3) {
\begin{tikzpicture}[scale=0.7]

\filldraw[gray] (0,0) to[out=20,in=-20] ++(0,1) to[out=-160,in=160] ++(0,-1);
\filldraw[gray] (0,1) to[out=20,in=-20] ++(0,1) to[out=-160,in=160] ++(0,-1);
\draw (0.5, 0.5) circle (0.2);
\draw (0.5, 1.5) circle (0.2);

\end{tikzpicture}
};

\draw[thin, densely dotted, ->] (bubblefirst) -- (A);
\draw[thin, densely dotted, ->] (bubblemiddle) -- (C);
\draw[thin, densely dotted, ->] (bubblelast) -- (D);
\draw[thin, densely dotted, ->] (doublebubble) -- (B);
\draw[thin, densely dotted, ->] (doublebubble) -- (E);

\end{tikzpicture}\caption{Type 3.2b in the case where $N = 1$ and $D$ is a horizontal annulus. $\partial \mathcal{M}(D, N\vec{e_k}, (N)_k)$ is a quadrilateral. We will consider $\partial' \mathcal{M}(D, N\vec{e_k}, (N)_k)$ (red), which is a $10$-gon with the indicated vertices.}\label{fig15}\end{figure}

\begin{remark}Again, Figure $\ref{fig15}$ is slightly misleading due to dimension constraints. In general, for larger $l$, the two points in the $\R_{+(1)}-\R_{+(3)}$ plane do not lie on the same plane and the segment between them is the moduli space of a Type 2.0b triple as shown in Figure $\ref{fig2}$.\end{remark}

Let $r, s$ be the signs of the rectangle $R, S$, respectively. Using Conventions $\ref{conv1}, \ref{conv3}, \ref{conv4}$, and $\ref{conv6}$, and assuming $k > j$, we can decompose $\partial' \mathcal{M}(D, N\vec{e_j}, (N)_j)$ into short and long preferred paths as shown in Figure $\ref{fig16}$.

\begin{figure}
\begin{tikzpicture}[scale=6]


\node[anchor=east] at (0, 0) {\small $f_{2N} f_1 \dots f_{2N-1} 0\dots 0 r s$};

\node[anchor=east] at (0, 0.1) {\small $f_{2N + 1} f_1 \dots f_{2N-1} 0\dots 0 r s$};

\node[anchor=east] at (0, 0.8) {\small $f_{2N + 1} f_2 \dots f_{2N} r 0\dots 0 s$};

\node[anchor=east] at (0, 0.9) {\small $f_{1} f_2 \dots f_{2N} r 0\dots 0 s$};

\node[anchor=east] at (0, 1.6) {\small $f_1 f_3 \dots f_{2N + 1} r s 0\dots 0$};

\node[anchor=east] at (0, 1.7) {\small $f_2 f_3 \dots f_{2N + 1} r s 0\dots 0$};

\node[anchor=west] at (1, 0) {\small $\pm f_{2N + 1} f_1 \dots f_{2N-1} 0\dots 0$};

\node[anchor=west] at (1, 1.7) {\small $\pm f_1 f_3 \dots f_{2N + 1} 0\dots 0$};

\node[anchor=south] at (0.5, 1.7) {\small $f_2 f_3 \dots f_{2N + 1} 0\dots 0$};

\node[anchor=north] at (0.5, 0) {\small $f_{2N} f_1\dots f_{2N-1} 0\dots 0$};

\node at (0, 0.45) {$\dots$};
\node at (0, 1.25) {$\dots$};
\node at (1, 0.45) {$\dots$};
\node at (1, 1.25) {$\dots$};

\draw[thick] (0, 0) -- (0, 0.1);
\draw[thick] (0, 0.2) -- (0, 0.4);
\draw[thick] (0, 0.5) -- (0, 0.7);
\draw[thick] (0, 0.8) -- (0, 0.9);
\draw[thick] (0, 1) -- (0, 1.2);
\draw[thick] (0, 1.3) -- (0, 1.5);
\draw[thick] (0, 1.6) -- (0, 1.7);

\draw[->-,thick,color=green] (0, 0.1) -- (0, 0.2);
\draw[->-,thick,color=red] (0, 0.8) -- (0, 0.7);
\draw[->-,thick,color=blue] (0, 0.9) -- (0, 1);
\draw[->-,thick,color=brown] (0, 1.7) -- (0.25, 1.7);
\draw[->-,thick,color=brown] (0, 1.6) -- (0, 1.5);

\draw[thick] (1, 0) -- (1, 0.4);
\draw[thick] (1, 0.5) -- (1, 1.2);
\draw[thick] (1, 1.3) -- (1, 1.7);

\draw[->-,thick,color=blue] (0, 0) -- (0.25, 0);
\draw[->-,thick,color=green] (0.25, 0) -- (0.5, 0);
\draw[->-,thick,color=red] (0.25, 1.7) -- (0.5, 1.7);


\node at (0.5, 0.45) {$\dots$};
\node at (0.5, 1.25) {$\dots$};

\draw[thick] (0.5, 0) -- (0.5, 0.4);
\draw[thick] (0.5, 0.5) -- (0.5, 1.2);
\draw[thick] (0.5, 1.3) -- (0.5, 1.7);

\foreach \i in {1,2,3,4,5,6}{
\draw[->-,thick,color=blue] (0, 0.1 + 0.1 * \i) -- (0.5, 0.1 + 0.1 * \i);
\draw[->-,thick,color=red] (0, 0.9 + 0.1 * \i) -- (0.5, 0.9 + 0.1 * \i);
}
\foreach \i in {0,2,3,4,5,6,7,10,11,12,13,14,15,17}{
\draw[thick] (1, 0.1 * \i) -- (0.5, 0.1 * \i);
}
\draw[->-,thick,color=blue] (0, 0.1) -- (0.25, 0.1);
\draw[->-,thick,color=green] (0.25, 0.1) -- (0.5, 0.2);
\draw[thick] (0.25, 0) -- (0.25, 0.1);
\draw[->-,thick,color=blue] (0, 0.8) -- (0.25, 0.8);
\draw[->-,thick,color=red] (0.25, 0.8) -- (0.5, 0.7);
\draw[thick] (0.25, 0.8) -- (0, 1);
\draw[thick] (0, 1.5) -- (0.25, 1.7);


\node[anchor=east] at (0, 0.2) {\small $f_{2N + 1} f_1 \dots f_{2N-1} 0\dots 0 s$};
\node[anchor=west] at (1, 0.2) {\small $-f_{2N} f_1 \dots f_{2N-1} 0\dots 0$};

\node[anchor=east] at (0, 0.3) {\small $f_{2N+1} f_{2N} f_2 \dots f_{2N-1} 0\dots 0 s$};
\node[anchor=west] at (1, 0.3) {\small $f_1 f_{2N} f_2 \dots f_{2N-1} 0\dots 0$};

\node[anchor=east] at (0, 0.4) {\small $f_{2N+1} f_2 (-f_{2N}) f_3 \dots f_{2N-1} 0\dots 0 s$};
\node[anchor=west] at (1, 0.4) {\small $f_1 f_2 (-f_{2N}) f_3 \dots f_{2N-1} 0\dots 0$};

\node[anchor=east] at (0, 0.5) {\small $f_{2N+1} f_2 \dots f_{2N-3} f_{2N} f_{2N-2} f_{2N-1} 0\dots 0 s$};
\node[anchor=west] at (1, 0.5) {\small $f_1 \dots f_{2N-3} f_{2N} f_{2N-2} f_{2N-1} 0\dots 0$};

\node[anchor=east] at (0, 0.6) {\small $f_{2N+1} f_2 \dots f_{2N-2} (-f_{2N}) f_{2N-1} 0\dots 0 s$};
\node[anchor=west] at (1, 0.6) {\small $f_1 \dots f_{2N-2} (-f_{2N}) f_{2N-1} 0\dots 0$};

\node[anchor=east] at (0, 0.7) {\small $f_{2N + 1} f_2 \dots f_{2N-1} 0\dots 0 s$};
\node[anchor=west] at (1, 0.7) {\small $f_1 \dots f_{2N} 0\dots 0$};

\node[anchor=east] at (0, 1) {\small $f_1 f_2 \dots f_{2N} r 0\dots 0$};
\node[anchor=west] at (1, 1) {\small $-f_{2N + 1} f_2 \dots f_{2N} 0\dots 0$};

\node[anchor=east] at (0, 1.1) {\small $f_1 f_{2N + 1} f_3 \dots f_{2N} r 0\dots 0$};
\node[anchor=west] at (1, 1.1) {\small $f_2 f_{2N + 1} f_3 \dots f_{2N} 0\dots 0$};

\node[anchor=east] at (0, 1.2) {\small $f_1 f_3 (-f_{2N + 1}) f_4 \dots f_{2N} r 0\dots 0$};
\node[anchor=west] at (1, 1.2) {\small $f_2 f_3 (-f_{2N + 1}) f_4 \dots f_{2N} 0\dots 0$};

\node[anchor=east] at (0, 1.3) {\small $f_1 f_3 \dots f_{2N-2} f_{2N + 1} f_{2N-1} f_{2N} r 0\dots 0$};
\node[anchor=west] at (1, 1.3) {\small $f_2 \dots f_{2N-2} f_{2N + 1} f_{2N-1} f_{2N} 0\dots 0$};

\node[anchor=east] at (0, 1.4) {\small $f_1 f_3 \dots f_{2N-1} (-f_{2N + 1}) f_{2N} r 0\dots 0$};
\node[anchor=west] at (1, 1.4) {\small $f_2 \dots f_{2N-1} (-f_{2N + 1}) f_{2N} 0\dots 0$};

\node[anchor=east] at (0, 1.5) {\small $f_1 f_3 \dots f_{2N + 1} r 0\dots 0$};
\node[anchor=west] at (1, 1.5) {\small $f_2 \dots f_{2N + 1} 0\dots 0$};


\node at (0.1, 1.65) {\small $s$};
\node at (0.1, 0.85) {\small $s$};
\node at (0.15, 0.75) {\small $rs$};
\node at (0.15, 0.15) {\small $rs$};
\node at (0.125, 0.05) {\small $s$};
\node at (0.375, 0.05) {\small $r$};
\node at (0.3, 1.6) {\small $r$};
\node at (0.3, 0.9) {\small $r$};
\foreach \i in {2,3,5,6}{
\node at (0.25, 0.85 + 0.1 * \i) {\small $r$};
\node at (0.25, 0.05 + 0.1 * \i) {\small $s$};
}
\foreach \i in {2,3,5,6,10,11,13,14,7,15}{
\node at (0.75, 0.05 + 0.1 * \i) {\small $0$};
}
\foreach \i in {0}{
\node at (0.75, 0.05 + 0.1 * \i) {\small $0 \text{ or } 1$};
}


\node at (0, 0) {
\begin{tikzpicture}

\draw[fill=white] (0,0) circle (0.125);

\node at (0,0) {\tiny $1$};

\end{tikzpicture}
};
\node at (0, 0.1) {
\begin{tikzpicture}

\draw[fill=white] (0,0) circle (0.125);

\node at (0,0) {\tiny $2$};

\end{tikzpicture}
};
\node at (0, 0.8) {
\begin{tikzpicture}

\draw[fill=white] (0,0) circle (0.125);

\node at (0,0) {\tiny $5$};

\end{tikzpicture}
};
\node at (0, 0.9) {
\begin{tikzpicture}

\draw[fill=white] (0,0) circle (0.125);

\node at (0,0) {\tiny $6$};

\end{tikzpicture}
};
\node at (0, 1.4) {
\begin{tikzpicture}

\draw[fill=white] (0,0) circle (0.125);

\node at (0,0) {\tiny $8$};

\end{tikzpicture}
};
\node at (0, 1.7) {
\begin{tikzpicture}

\draw[fill=white] (0,0) circle (0.125);

\node at (0,0) {\tiny $7$};

\end{tikzpicture}
};
\node at (1, 0) {
\begin{tikzpicture}

\draw[fill=white] (0,0) circle (0.125);

\node at (0,0) {\tiny $10$};

\end{tikzpicture}
};
\node at (1, 1.7) {
\begin{tikzpicture}

\draw[fill=white] (0,0) circle (0.125);

\node at (0,0) {\tiny $4$};

\end{tikzpicture}
};
\node at (0.5, 0) {
\begin{tikzpicture}

\draw[fill=white] (0,0) circle (0.125);

\node at (0,0) {\tiny $9$};

\end{tikzpicture}
};
\node at (0.5, 1.7) {
\begin{tikzpicture}

\draw[fill=white] (0,0) circle (0.125);

\node at (0,0) {\tiny $3$};

\end{tikzpicture}
};

\end{tikzpicture}
\caption{Type 3.2b, in the case where $\vec{N} = N \vec{e}_k$ and $k > j$. All $\pm$ on the right-hand side are plus if $D$ is a vertical annulus, and minus if $D$ is a horizontal annulus.}\label{fig16}\end{figure}
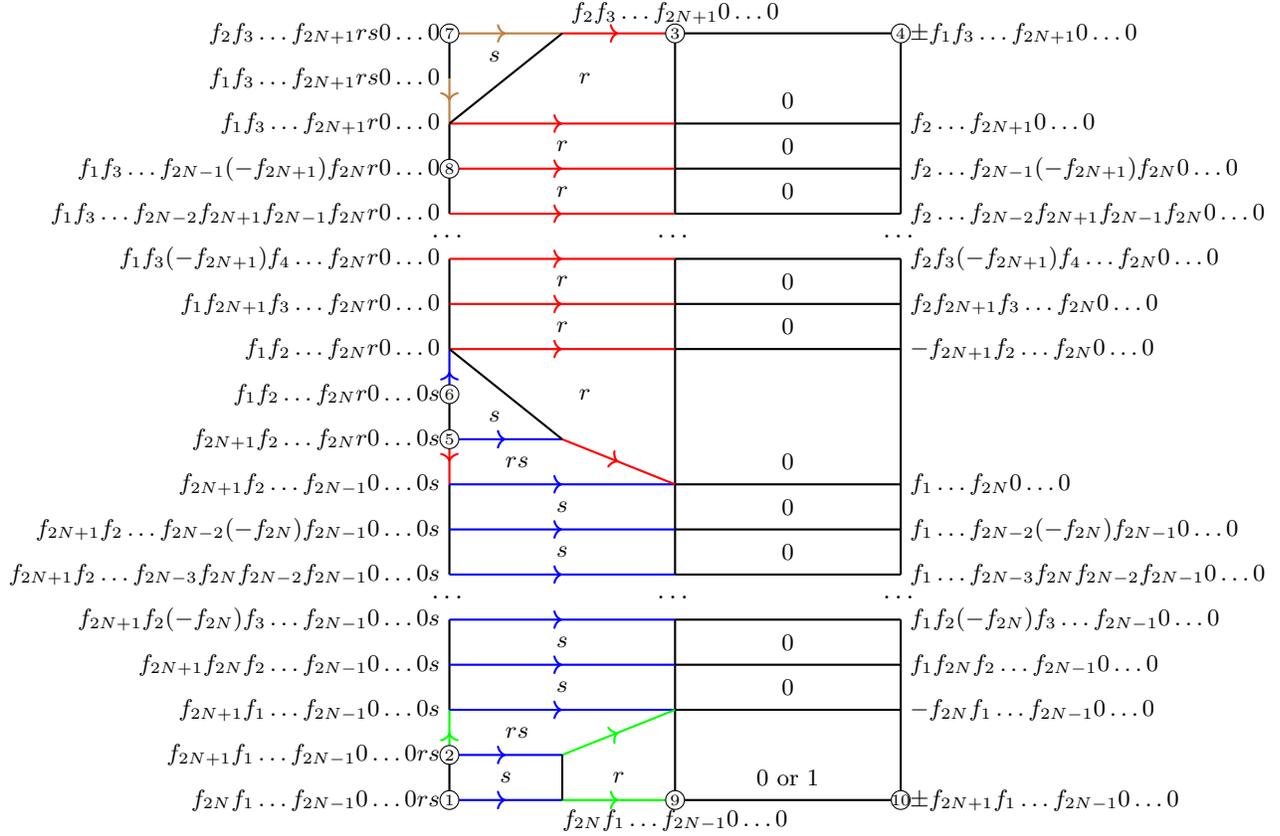

By Lemma $\ref{lem1}$, all black-green, black-blue, and black-red quadrilaterals correspond to the element of $\pi_1(\SO((2N + 1)(d + 1) - 1))$ labelled in Figure $\ref{fig16}$, and by Lemma $\ref{lem2}$ all green-blue and red-green quadrilaterals also correspond to the labelled element of $\pi_1(\SO((2N + 1)(d + 1) - 1))$. Finally, most the black quadrilaterals on the right are $0$ by Lemma $\ref{lem5}$, and the top and bottom black quadrilaterals are $0$ if $D$ is a vertical annulus and $1$ if $D$ is a horizontal annulus by Lemma $\ref{lem3}$. So for a vertical annulus $D$,
\begin{align*}\delta f(D, N\vec{e}_k, 0) &= 1 + (r + s) + rs + (2N-2)s + rs + (r + s) + (2N-2)r + (r + s) \\&= 1 + (2N + 1)(r + s) = 0 \pmod{2} \text{ by the properties of sign assignments}\end{align*}
and for a horizontal annulus $D$,
\begin{align*}\delta f(D, N\vec{e}_k, 0) &= 1 + (r + s) + rs + (2N-2)s + rs + (r + s) + (2N-2)r + (r + s) + 1 \\&= (2N + 1)(r + s) = 0 \pmod{2} \text{ by the properties of sign assignments}\end{align*}
in the case $k > j$. The symmetric computation for $k < j$ gives $\delta f = 0$ in that case as well.

\textbf{Type 3.2c}: $D = V_j$ or $H_j$ is an annulus and $\lambda_j = (1)$ is the nonempty coordinate of $\vec{\lambda}$. In this case, let $D = R * S$ and consider the Figure $\ref{fig17}$.

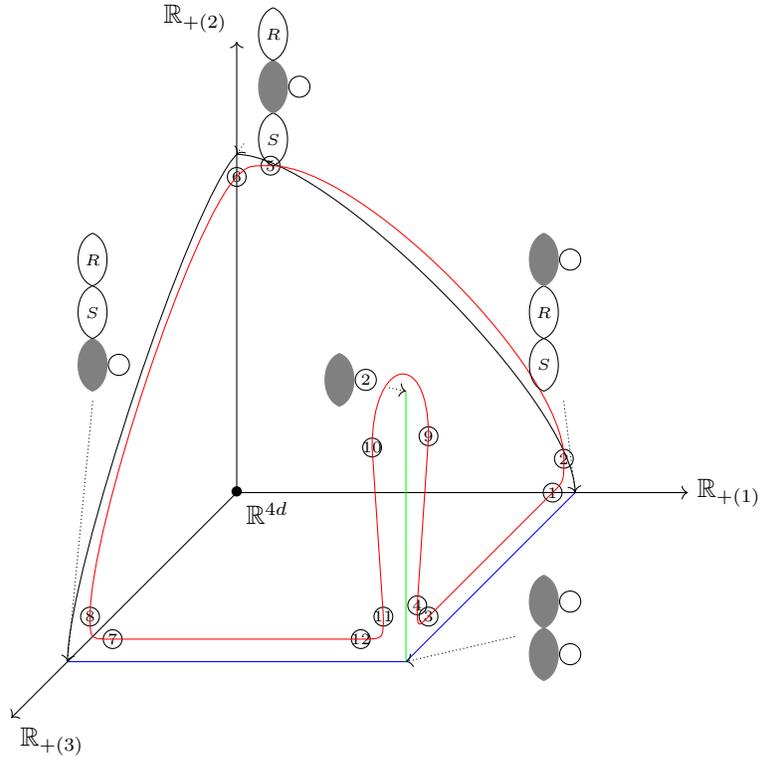
\begin{figure}

\begin{tikzpicture}[x={(1,0)}, y={(0,1)}, z={(-0.5, -0.5)}, scale=6]

\node at (0,0,0) {\textbullet};
\node[anchor=north west] at (0,0,0) {$\R^{4d}$};

\draw[->] (0,0,0) -- (1,0,0) node[anchor=west]{$\R_{+(1)}$};
\draw[->] (0,0,0) -- (0,1,0) node[anchor=south east]{$\R_{+(2)}$};
\draw[->] (0,0,0) -- (0,0,1) node[anchor=north west]{$\R_{+(3)}$};

\coordinate (A) at (0.75, 0, 0);
\coordinate (B) at (0.75, 0, 0.75);
\coordinate (C) at (0, 0.75, 0);
\coordinate (D) at (0, 0, 0.75);
\coordinate (E) at (0.75, 0.6, 0.75);

\draw[color=blue] (A) -- (B);
\draw[color=blue] (D) -- (B);
\draw[color=green] (B) -- (E);
\draw (A) .. controls ($(A) + (0, 0.2, 0)$) and ($(C) + (0.2, 0, 0)$) .. (C);
\draw (D) .. controls ($(D) + (0, 0.2, 0)$) and ($(C) + (0, 0, 0.2)$) .. (C);

\coordinate (A1) at (0.75, 0.05, 0.1);
\coordinate (A2) at (0.75, 0.1, 0.05);
\coordinate (B1) at (0.75, 0.05, 0.65);
\coordinate (B2) at (0.75, 0.1, 0.7);
\coordinate (C1) at (0.1, 0.75, 0.05);
\coordinate (C2) at (0.05, 0.75, 0.1);
\coordinate (D1) at (0.1, 0.05, 0.75);
\coordinate (D2) at (0.05, 0.1, 0.75);
\coordinate (E1) at (0.75, 0.45, 0.65);
\coordinate (E2) at (0.65, 0.45, 0.7);
\coordinate (F1) at (0.7, 0.1, 0.75);
\coordinate (F2) at (0.65, 0.05, 0.75);

\node at (A1) {
\begin{tikzpicture}

\draw (0,0) circle (0.125);

\node at (0,0) {\tiny $1$};

\end{tikzpicture}
};
\node at (A2) {
\begin{tikzpicture}

\draw (0,0) circle (0.125);

\node at (0,0) {\tiny $2$};

\end{tikzpicture}
};
\node at (B1) {
\begin{tikzpicture}

\draw (0,0) circle (0.125);

\node at (0,0) {\tiny $3$};

\end{tikzpicture}
};
\node at (B2) {
\begin{tikzpicture}

\draw (0,0) circle (0.125);

\node at (0,0) {\tiny $4$};

\end{tikzpicture}
};
\node at (C1) {
\begin{tikzpicture}

\draw (0,0) circle (0.125);

\node at (0,0) {\tiny $5$};

\end{tikzpicture}
};
\node at (C2) {
\begin{tikzpicture}

\draw (0,0) circle (0.125);

\node at (0,0) {\tiny $6$};

\end{tikzpicture}
};
\node at (D1) {
\begin{tikzpicture}

\draw (0,0) circle (0.125);

\node at (0,0) {\tiny $7$};

\end{tikzpicture}
};
\node at (D2) {
\begin{tikzpicture}

\draw (0,0) circle (0.125);

\node at (0,0) {\tiny $8$};

\end{tikzpicture}
};
\node at (E1) {
\begin{tikzpicture}

\draw (0,0) circle (0.125);

\node at (0,0) {\tiny $9$};

\end{tikzpicture}
};
\node at (E2) {
\begin{tikzpicture}

\draw (0,0) circle (0.125);

\node at (0,0) {\tiny $10$};

\end{tikzpicture}
};
\node at (F1) {
\begin{tikzpicture}

\draw (0,0) circle (0.125);

\node at (0,0) {\tiny $11$};

\end{tikzpicture}
};
\node at (F2) {
\begin{tikzpicture}

\draw (0,0) circle (0.125);

\node at (0,0) {\tiny $12$};

\end{tikzpicture}
};

\draw[color=red] (A1) -- (B1);
\draw[color=red] (D1) -- (F2);
\draw[color=red] (B2) -- (E1);
\draw[color=red] (F1) -- (E2);
\draw[color=red] (C1) .. controls ($(C1) + (0.2, 0, 0)$) and ($(A2) + (0, 0.2, 0)$) .. (A2);
\draw[color=red] (C2) .. controls ($(C2) + (0, 0, 0.2)$) and ($(D2) + (0, 0.2, 0)$) .. (D2);

\draw[color=red] (A1) .. controls (0.75, 0.05, 0.05) .. (A2);
\draw[color=red] (E1) .. controls ($(E1) + (0, 0.2, 0)$) and ($(E2) + (0, 0.2, 0)$) .. (E2);
\draw[color=red] (C1) .. controls (0.05, 0.75, 0.05) .. (C2);
\draw[color=red] (D1) .. controls (0.05, 0.05, 0.75) .. (D2);
\draw[color=red] (F1) .. controls (0.7, 0.05, 0.75) .. (F2);
\draw[color=red] (B1) .. controls (0.75, 0.05, 0.7) .. (B2);

\node (bubblefirst) at (0.7, 0.4) {
\begin{tikzpicture}[scale=0.7]

\draw (0,0) to[out=20,in=-20] ++(0,1) to[out=-160,in=160] ++(0,-1);
\draw (0,1) to[out=20,in=-20] ++(0,1) to[out=-160,in=160] ++(0,-1);
\filldraw[gray] (0,2) to[out=20,in=-20] ++(0,1) to[out=-160,in=160] ++(0,-1);
\draw (0.5, 2.5) circle (0.2);

\node at (0,0.5) {\tiny $S$};
\node at (0,1.5) {\tiny $R$};

\end{tikzpicture}
};

\node (bubblemiddle) at (0.1, 0.9) {
\begin{tikzpicture}[scale=0.7]

\draw (0,0) to[out=20,in=-20] ++(0,1) to[out=-160,in=160] ++(0,-1);
\filldraw[gray] (0,1) to[out=20,in=-20] ++(0,1) to[out=-160,in=160] ++(0,-1);
\draw (0,2) to[out=20,in=-20] ++(0,1) to[out=-160,in=160] ++(0,-1);
\draw (0.5, 1.5) circle (0.2);

\node at (0,0.5) {\tiny $S$};
\node at (0,2.5) {\tiny $R$};

\end{tikzpicture}
};

\node (bubblelast) at (-0.3, 0.4) {
\begin{tikzpicture}[scale=0.7]

\filldraw[gray] (0,0) to[out=20,in=-20] ++(0,1) to[out=-160,in=160] ++(0,-1);
\draw (0,1) to[out=20,in=-20] ++(0,1) to[out=-160,in=160] ++(0,-1);
\draw (0,2) to[out=20,in=-20] ++(0,1) to[out=-160,in=160] ++(0,-1);
\draw (0.5, 0.5) circle (0.2);

\node at (0,1.5) {\tiny $S$};
\node at (0,2.5) {\tiny $R$};

\end{tikzpicture}
};

\node (doublebubble) at (0.7,-0.3) {
\begin{tikzpicture}[scale=0.7]

\filldraw[gray] (0,0) to[out=20,in=-20] ++(0,1) to[out=-160,in=160] ++(0,-1);
\filldraw[gray] (0,1) to[out=20,in=-20] ++(0,1) to[out=-160,in=160] ++(0,-1);
\draw (0.5, 0.5) circle (0.2);
\draw (0.5, 1.5) circle (0.2);

\end{tikzpicture}
};
\node (mergedbubble) at (0.25,0.25) {
\begin{tikzpicture}[scale=0.7]

\filldraw[gray] (0,0) to[out=20,in=-20] ++(0,1) to[out=-160,in=160] ++(0,-1);
\draw (0.5, 0.5) circle (0.2);
\node at (0.5, 0.5) {\tiny $2$};

\end{tikzpicture}
};

\draw[thin, densely dotted, ->] (bubblefirst) -- (A);
\draw[thin, densely dotted, ->] (bubblemiddle) -- (C);
\draw[thin, densely dotted, ->] (bubblelast) -- (D);
\draw[thin, densely dotted, ->] (doublebubble) -- (B);
\draw[thin, densely dotted, ->] (mergedbubble) -- (E);

\end{tikzpicture}\caption{Type 3.2c in the case where $l = 1$ and $D$ is a horizontal annulus. We will consider $\partial' \mathcal{M}(D, l\vec{e_j}, (l))$ (red), which is a $12$-gon with the indicated vertices.}\label{fig17}\end{figure}

In Figure $\ref{fig17}$, we omit the local model around the green vertical line. Figure $\ref{fig18}$ shows the local model (which is the Whitney umbrella---see \cite{MS} for details). The boundary $\partial M(D, \vec{e}_j, (1)_j)$ lies only on one side of the Whitney umbrella, as the other side (black edges) contains $\partial M(D', \vec{e}_j, (1)_j)$, where $D'$ is the other annulus containing $O_j$.

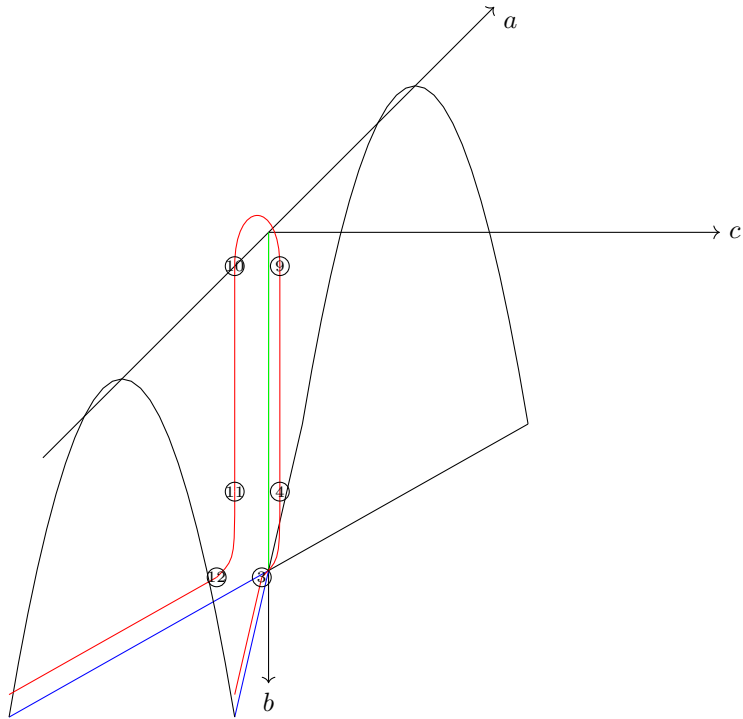
\begin{figure}
\begin{tikzpicture}[x={(1,0)}, y={(0,1)}, z={(-0.5, -0.5)}, scale=6]

\draw[->] (0,0,0) -- (1,0,0) node[anchor=west]{$c$};
\draw[->] (0,0,0) -- (0,-1,0) node[anchor=north]{$b$};
\draw[->] (0,0,1) -- (0,0,-1) node[anchor=north west]{$a$};

\draw[color=green] (0, 0, 0) -- (0, -0.75, 0);
\draw[color=blue] (0.25, -0.75, 0.65) -- (0, -0.75, 0) -- (-0.25, -0.75, 0.65);
\draw (0.25, -0.75, -0.65) -- (0, -0.75, 0) -- (-0.25, -0.75, -0.65);

\draw[scale = 1, domain = -1:1, variable = \x] plot({0.25 * \x}, {-0.75 * \x * \x}, 0.65);

\draw[scale = 1, domain = -1:1, variable = \x] plot({0.25 * \x}, {-0.75 * \x * \x}, -0.65);

\draw[color=red] (0.25, -0.7, 0.65) -- (0.05, -0.7, 0.13);
\draw[color=red] (-0.25, -0.7, 0.65) -- (-0.05, -0.7, 0.13);
\draw[color=red] (0.05, -0.7, 0.13) .. controls (0.05, -0.7, 0.05) .. (0.05, -0.55, 0.05);
\draw[color=red] (-0.05, -0.7, 0.13) .. controls (-0.05, -0.7, 0.05) .. (-0.05, -0.55, 0.05);
\draw[color=red] (0.05, -0.55, 0.05) -- (0.05, -0.05, 0.05);
\draw[color=red] (-0.05, -0.55, 0.05) -- (-0.05, -0.05, 0.05);
\draw[color=red] (0.05, -0.05, 0.05) .. controls (0.05, 0.1, 0.05) and (-0.05, 0.1, 0.05) .. (-0.05, -0.05, 0.05);

\node at (0.05, -0.7, 0.13) {
\begin{tikzpicture}

\draw (0,0) circle (0.125);

\node at (0,0) {\tiny $3$};

\end{tikzpicture}
};

\node at (0.05, -0.55, 0.05) {
\begin{tikzpicture}

\draw (0,0) circle (0.125);

\node at (0,0) {\tiny $4$};

\end{tikzpicture}
};

\node at (0.05, -0.05, 0.05) {
\begin{tikzpicture}

\draw (0,0) circle (0.125);

\node at (0,0) {\tiny $9$};

\end{tikzpicture}
};

\node at (-0.05, -0.05, 0.05) {
\begin{tikzpicture}

\draw (0,0) circle (0.125);

\node at (0,0) {\tiny $10$};

\end{tikzpicture}
};

\node at (-0.05, -0.55, 0.05) {
\begin{tikzpicture}

\draw (0,0) circle (0.125);

\node at (0,0) {\tiny $11$};

\end{tikzpicture}
};

\node at (-0.05, -0.7, 0.13) {
\begin{tikzpicture}

\draw (0,0) circle (0.125);

\node at (0,0) {\tiny $12$};

\end{tikzpicture}
};

\end{tikzpicture}
\caption{The local model for Type 3.2c in the case where $N = 1$ and $D$ is a horizontal annulus near the green line segment of Figure $\ref{fig17}$.}\label{fig18}\end{figure}

Note that vertices $9$ and $10$ have the same frame as vertices $4$ and $11$ respectively. The framed path corresponding to $\partial' \mathcal{M}(D, l\vec{e}_j, (l))$ from vertices $4$ to $3$ to $12$ to $11$ is identical to that for Type 3.2b. Therefore it remains to compare the framed path between vertices $9$ and $10$ to the corresponding path for Type 3.2b.

Consider the frames $f_1, f_2, f_3$, which are the standard frames of the stratum $Z(0, 2, 0; (2)) \subset Z_2$, where the nearby point $z'$ satisfying $y_1 = y_2 = x_1 + x_2 = 0$ and $x_2 - x_1 = 2\epsilon > 0$ has tailored coordinates $(y_1, \Delta_1, y_2, \text{Re}(s_1) = x_1 + x_2)$. In terms of the tailored coordinates, $y_1 = y_2 = \text{Re}(s_1) = 0$ and $\Delta_1 = 2 \epsilon$, so that
\begin{align*}
&\delta (\text{Im}(s_1)) = \delta y_1 + \delta y_2
\\&\delta (\text{Re}(s_2)) = \delta (x_1 x_2 - y_1 y_2) = x_1 \delta x_2 + x_2 \delta x_1 - y_1 \delta y_2 - y_2 \delta y_1 =\epsilon \delta x_1 - \epsilon \delta x_2 = -\epsilon \delta \Delta_1
\\&\delta(\text{Im}(s_2)) = \delta (x_1 y_2 + x_2 y_1) = x_1 \delta y_2 + y_2 \delta x_1 + x_2 \delta y_1 + y_2 \delta x_1 = \epsilon (\delta y_1 - \delta y_2) 
\end{align*}
so that in terms of the usual coordinates $(a, b, c)$ on $\R^3$,
\begin{align*}
&f_1 = \delta y_1 = (1, 0, \epsilon)\\
&f_2 = -\delta \Delta_1 = (0, \epsilon, 0)\\
&f_3 = \delta y_2 = (1, 0, -\epsilon)
\end{align*}

The path between vertices $9$ and $10$ lies in a neighborhood of the stratum $Z(1, 1, 0)$, where we use tailored coordinates $(y_1, x_1, x_2, y_2)$. On $Z(1, 1, 0)$, we have $y_1 = x_1 + x_2 = 0$ and $y_2 < 0$, so $y_2 = a$ and $x_1 = c/a$, and we have
\begin{align*}
&\delta a = \delta y_1 + \delta y_2
\\&\delta b = -y_2 \delta y_1 + \text{ other terms }
\\&\delta c = x_2 \delta y_1 + x_1 \delta y_2
\end{align*}
so since the standard frame here is $\delta y_1$, the standard frame here is $(1, -a, -c/a)$. On the other hand, the distinguished vector $\vec{v}$ is $-f_3 = -\delta y_2 = (-1, 0, -c/a)$ on the side of vertex 9. Changing to the $f_1, f_2, f_3$ coordinates near $Z(0, 2, 0; (2))$, we get
\begin{align}\label{eqnintframe}\begin{pmatrix}1 & 0 & \epsilon \\ 0 & \epsilon & 0 \\ 1 & 0 & -\epsilon\end{pmatrix}^{-1}\begin{pmatrix}1 & -1 \\ -a & 0 \\ -c/a & -c/a\end{pmatrix} = \begin{pmatrix}1/2 - c/2a & -1/2 - c/2a \\ -a/\epsilon & 0 \\ 1/2\epsilon- c/2\epsilon a & -1/2\epsilon -c/2\epsilon a\end{pmatrix}\end{align}

At Vertex $9$, we have $[\vec{v}, f] = [-f_3, f_1]$ where $f$ is the internal frame. At Vertex $10$, $[\vec{v}, f] = [f_1, f_3]$. We see that the right-hand side of $(\ref{eqnintframe})$, for $c/a = \epsilon$ (respectively, $-\epsilon$) and $a \rightarrow 0$, gives the frame at Vertex $9$ (respectively, Vertex $10$) after the Gram-Schmidt process. Hence, to describe the path from Vertex $9$ to Vertex $10$ in $a, b, c$, we move $a$ away from zero (and since $D$ is a horizontal annulus, $a$ is always negative), then change $c/a$ from positive to negative, and then send $a$ back to zero. In the intermediate steps, where $a < 0$, we obtain the frames $[f_2, f_1]$ and $[f_2, f_3]$ from the right-hand side of $(\ref{eqnintframe})$.

Now consider the following loop consisting of both paths 3.2b and 3.2c between Vertices $9$ and $10$ (which are the ends of the dashed line; we have also completed each framing to a positive frame for $\R^3$ with the term in parentheses): \\
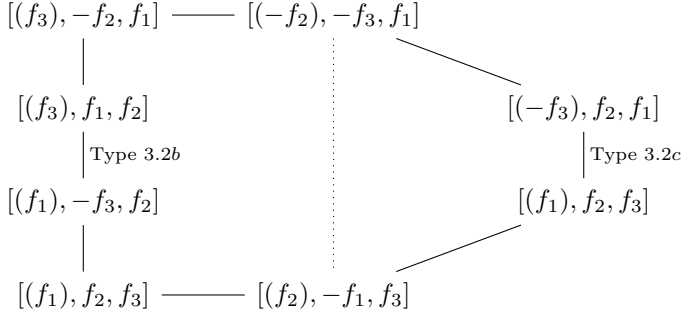
\begin{tikzcd}
{[(f_3), -f_2, f_1]} \arrow[d, no head]                     & {[(-f_2), -f_3, f_1]} \arrow[rd, no head] \arrow[l, no head] \arrow[ddd, no head, dotted] &                                                              \\
{[(f_3), f_1, f_2]} \arrow[d, "\text{Type } 3.2b", no head] &                                                                                           & {[(-f_3), f_2, f_1]} \arrow[d, "\text{Type } 3.2c", no head] \\
{[(f_1), -f_3, f_2]} \arrow[d, no head]                     &                                                                                           & {[(f_1), f_2, f_3]} \arrow[ld, no head]                      \\
{[(f_1), f_2, f_3]} \arrow[r, no head]                      & {[(f_2), -f_1, f_3]}                                                                      &                                                             
\end{tikzcd} \\
The loop lifts to $\Pin(3)$ as
\begin{align*}\frac{1}{\sqrt{2}}(1 - e_2 e_3) \frac{1}{\sqrt{2}}(1 - e_1 e_3) \frac{1}{\sqrt{2}}(1 - e_1 e_2)\frac{1}{\sqrt{2}}(1 - e_2 e_1) \frac{1}{\sqrt{2}}(1 - e_3 e_2) \frac{1}{\sqrt{2}}(1 - e_1 e_3)\frac{1}{\sqrt{2}}(1 - e_2 e_1) \frac{1}{\sqrt{2}}(1 - e_2 e_3)\end{align*}
which we compute to be $1$ by Algorithm $\ref{multalgo}$, so the loop is nullhomotopic. As a result, the path is homotopic to that for Type 3.2b in the case $N = 1$, $D$ a horizontal annulus. A similar computation also gives that the path is homotopic to Type 3.2b in the case $N = 1$, $D$ a vertical annulus.

As a result, the boundary condition for $f$ is identical as for Type 3.2b:
\begin{align*}\delta f(D, \vec{e}_k, (1)) = 0\end{align*}

\textbf{Type 3.1a}: In this case, $D$ is an index $1$ domain, $\vec{\lambda}$ has total length $2$, and we have a triple of the form $(D, 2 \vec{e}_j, (1, 1))$. Let $d$ be the sign of rectangle $D$. The vertices of $\partial' \mathcal{M}(D, 2 \vec{e}_j, (1, 1))$ in the below Figure $\ref{fig24}$ can be written as points in $\SO(5d + 4)$ that form a loop.

\begin{figure}\begin{tikzpicture}[scale=2.5]

\begin{scope}[xshift=0cm, yshift=3cm]
\filldraw[color=gray] (0, 0) .. controls (-0.125, 0.25) .. (0, 0.5) .. controls (0.125, 0.25) .. (0, 0);
\draw (0, 0.5) .. controls (-0.125, 0.75) .. (0, 1) .. controls (0.125, 0.75) .. (0, 0.5);
\filldraw[color=gray] (0, 1) .. controls (-0.125, 1.25) .. (0, 1.5) .. controls (0.125, 1.25) .. (0, 1);

\node at (0, 0.75) {\small $D$};
\draw (0.2, 0.25) circle (0.1);
\node at (0.2, 0.25) {1};
\draw (0.2, 1.25) circle (0.1);
\node at (0.2, 1.25) {1};
\end{scope}

\begin{scope}[xshift=-1cm, yshift=2cm]
\draw (0, 0) .. controls (-0.125, 0.25) .. (0, 0.5) .. controls (0.125, 0.25) .. (0, 0);
\filldraw[color=gray] (0, 0.5) .. controls (-0.125, 0.75) .. (0, 1) .. controls (0.125, 0.75) .. (0, 0.5);
\filldraw[color=gray] (0, 1) .. controls (-0.125, 1.25) .. (0, 1.5) .. controls (0.125, 1.25) .. (0, 1);

\node at (0, 0.25) {\small $D$};
\draw (0.2, 0.75) circle (0.1);
\node at (0.2, 0.75) {1};
\draw (0.2, 1.25) circle (0.1);
\node at (0.2, 1.25) {1};
\end{scope}

\begin{scope}[xshift=1cm, yshift=2cm]
\filldraw[color=gray] (0, 0) .. controls (-0.125, 0.25) .. (0, 0.5) .. controls (0.125, 0.25) .. (0, 0);
\filldraw[color=gray] (0, 0.5) .. controls (-0.125, 0.75) .. (0, 1) .. controls (0.125, 0.75) .. (0, 0.5);
\draw (0, 1) .. controls (-0.125, 1.25) .. (0, 1.5) .. controls (0.125, 1.25) .. (0, 1);

\node at (0, 1.25) {\small $D$};
\draw (0.2, 0.25) circle (0.1);
\node at (0.2, 0.25) {1};
\draw (0.2, 0.75) circle (0.1);
\node at (0.2, 0.75) {1};
\end{scope}

\begin{scope}[xshift=-1.25cm, yshift=0.5cm]
\filldraw[color=gray] (0, 0) .. controls (-0.25, 0.5) .. (0, 1) .. controls (0.25, 0.5) .. (0, 0);
\draw (0, -0.5) .. controls (-0.125, -0.25) .. (0, 0) .. controls (0.125, -0.25) .. (0, -0.5);

\node at (0, -0.25) {\small $D$};
\draw (0.295, 0.5) circle (0.1);
\node at (0.295, 0.5) {2};
\end{scope}

\begin{scope}[xshift=1cm, yshift=0cm]
\filldraw[color=gray] (0, 0) .. controls (-0.25, 0.5) .. (0, 1) .. controls (0.25, 0.5) .. (0, 0);
\draw (0, 1) .. controls (-0.125, 1.25) .. (0, 1.5) .. controls (0.125, 1.25) .. (0, 1);

\node at (0, 1.25) {\small $D$};
\draw (0.295, 0.5) circle (0.1);
\node at (0.295, 0.5) {2};
\end{scope}

\node at (0, 2.9) {\textbullet};
\node at (0.75, 2) {\textbullet};
\node at (0.75, 0) {\textbullet};
\node at (-0.75, 2) {\textbullet};
\node at (-0.75, 0) {\textbullet};

\draw (0, 2.9) -- (0.75, 2) -- (0.75, 0) -- (-0.75, 0) -- (-0.75, 2) -- (0, 2.9);
\draw[color=red] (0.2, 2.5) .. controls (0, 2.7) .. (-0.2, 2.5);
\draw[color=red] (0.6, 2) .. controls (0.65, 2) .. (0.65, 1.9);
\draw[color=red] (-0.6, 2) .. controls (-0.65, 2) .. (-0.65, 1.9);
\draw[color=red] (0.2, 2.5) -- (0.6, 2);
\draw[color=red] (-0.2, 2.5) -- (-0.6, 2);
\draw[color=red] (0.65, 1.9) -- (0.65, 0.2);
\draw[color=red] (-0.65, 1.9) -- (-0.65, 0.2);
\draw[color=red] (-0.55, 0.1) -- (0.55, 0.1);
\draw[color=red] (0.65, 0.2) .. controls (0.65, 0.1) .. (0.55, 0.1);
\draw[color=red] (-0.65, 0.2) .. controls (-0.65, 0.1) .. (-0.55, 0.1);

\node at (0.2, 2.5) {
\begin{tikzpicture}

\draw[fill=white] (0,0) circle (0.125);

\node at (0,0) {\tiny $1$};

\end{tikzpicture}
};

\node at (-0.2, 2.5) {
\begin{tikzpicture}

\draw[fill=white] (0,0) circle (0.125);

\node at (0,0) {\tiny $2$};

\end{tikzpicture}
};

\node at (0.6, 2.05) {
\begin{tikzpicture}

\draw[fill=white] (0,0) circle (0.125);

\node at (0,0) {\tiny $3$};

\end{tikzpicture}
};

\node at (0.65, 1.9) {
\begin{tikzpicture}

\draw[fill=white] (0,0) circle (0.125);

\node at (0,0) {\tiny $4$};

\end{tikzpicture}
};

\node at (-0.6, 2.05) {
\begin{tikzpicture}

\draw[fill=white] (0,0) circle (0.125);

\node at (0,0) {\tiny $5$};

\end{tikzpicture}
};

\node at (-0.65, 1.9) {
\begin{tikzpicture}

\draw[fill=white] (0,0) circle (0.125);

\node at (0,0) {\tiny $6$};

\end{tikzpicture}
};

\node at (-0.65, 0.2) {
\begin{tikzpicture}

\draw[fill=white] (0,0) circle (0.125);

\node at (0,0) {\tiny $7$};

\end{tikzpicture}
};

\node at (0.65, 0.2) {
\begin{tikzpicture}

\draw[fill=white] (0,0) circle (0.125);

\node at (0,0) {\tiny $8$};

\end{tikzpicture}
};

\node at (0.55, 0.1) {
\begin{tikzpicture}

\draw[fill=white] (0,0) circle (0.125);

\node at (0,0) {\tiny $9$};

\end{tikzpicture}
};

\node at (-0.55, 0.1) {
\begin{tikzpicture}

\draw[fill=white] (0,0) circle (0.125);

\node at (0,0) {\tiny $10$};

\end{tikzpicture}
};

\end{tikzpicture}
\caption{Type 3.1a, in the case $N_1 = N_2 = 1$. The embedded picture cannot be drawn due to dimensional constraints.}\label{fig24}\end{figure}
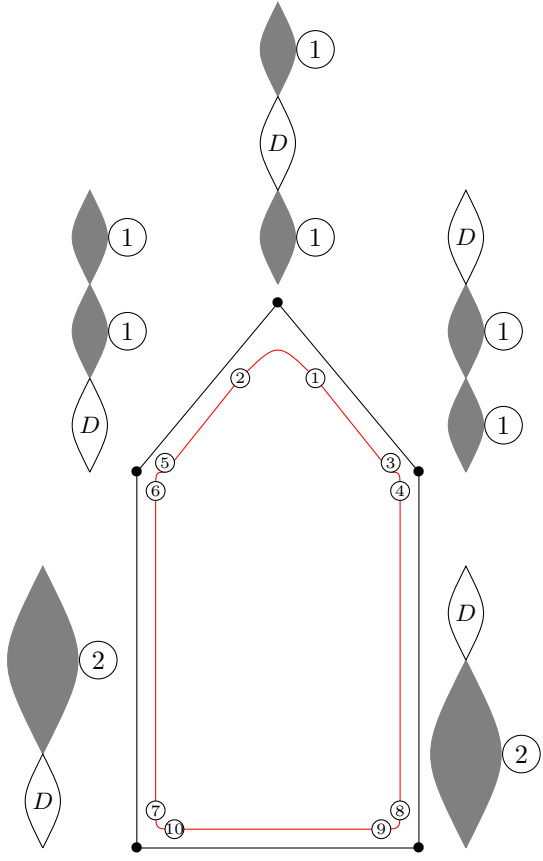

The paths between Vertices 6 and 7, and between Vertices 4 and 8 in $\ref{fig24}$, are each a Type 2.0a moduli space, and therefore are the identity path in the framing. In the following Figure $\ref{fig21}$ they are collapsed together. The path from Vertex 8 to Vertex 9 is similar to the path in the local model from Type 3.2c, except that $f_2, f_3, f_4$ play the role of $f_1, f_2, f_3$, respectively. We refer to the previous calculation to see that the path from Vertex 8 to Vertex 9 is $[f_1 f_2 f_4] \rightarrow [f_1 f_3 f_4] \rightarrow [f_2 f_3 f_4]$, and similarly that the path from Vertex 7 to Vertex 10 is $[f_4 f_1 f_3] \rightarrow [f_4 f_2 f_3] \rightarrow [f_1 f_2 f_3]$. The path between Vertices 9 and 10 is a Type 2.1 moduli space.

Similarly to previous cases, most of the framing vectors are identical---the $\R_+$ coordinates and the first, $(2d + 1))^{st}$, and $(4d + 1)^{st}$ coordinates are the only ones that can differ. We can therefore follow Conventions $\ref{conv1}$ and $\ref{conv3}$ when labelling the decomposition shown in Figure $\ref{fig21}$. For Type 3.1 we use the following color convention:

\begin{conv}\label{conv8}\begin{itemize}

\item All black paths are short preferred.

\item All red paths are long preferred paths changing the first $\R$ coordinate with respect to the first $\R_+$ coordinate.

\item All green paths are long preferred paths changing the $(2 d + 1)^{st}$ $\R$ coordinate with respect to the first $\R_+$ coordinate.

\item All blue paths are long preferred paths changing the $(4d+1)^{st}$ $\R$ coordinate with respect to the first $\R_+$ coordinate.

\end{itemize}\end{conv}

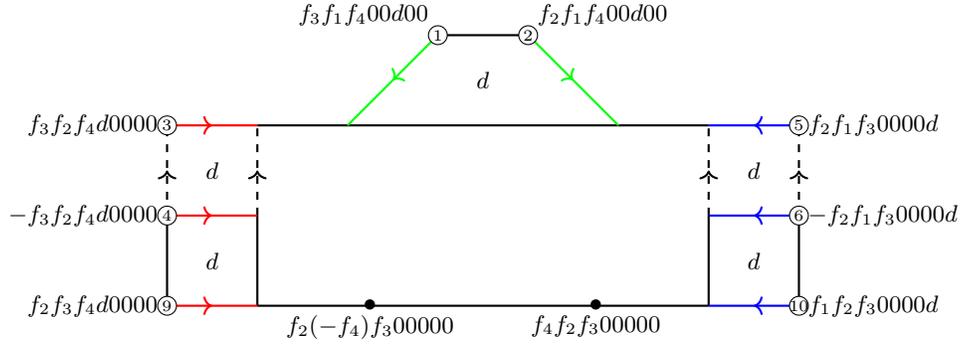
\begin{figure}
\begin{tikzpicture}[scale=6]


\node[anchor=east,align=right] at (-0.2, 1) {\small $f_2 f_3 f_4 d 0 0 0 0$};
\node[anchor=west,align=left] at (1.2, 1) {\small $f_1 f_2 f_3 0 0 0 0 d$};

\node[anchor=east,align=right] at (-0.2, 1.2) {\small $-f_3 f_2 f_4 d 0 0 0 0$};
\node[anchor=west,align=left] at (1.2, 1.2) {\small $-f_2 f_1 f_3 0 0 0 0 d$};

\node[anchor=east,align=right] at (-0.2, 1.4) {\small $f_3 f_2 f_4 d 0 0 0 0$};
\node[anchor=west,align=left] at (1.2, 1.4) {\small $f_2 f_1 f_3 0 0 0 0 d$};

\node[anchor=south east,align=right] at (0.4, 1.6) {\small $f_3 f_1 f_4 00d00$};
\node[anchor=south west,align=left] at (0.6, 1.6) {\small $f_2 f_1 f_4 00d00$};

\draw[thick] (0, 1) -- (1, 1);
\draw[thick] (0, 1.4) -- (1, 1.4);
\draw[thick] (-0.2, 1) -- (-0.2, 1.2);
\draw[thick] (1.2, 1) -- (1.2, 1.2);

\draw[->-,thick,color=green] (0.4, 1.6) -- (0.2, 1.4);
\draw[->-,thick,color=green] (0.6, 1.6) -- (0.8, 1.4);


\draw[thick] (0.4, 1.6) -- (0.6, 1.6);
\draw[thick] (0, 1) -- (0, 1.2);
\draw[thick] (1, 1) -- (1, 1.2);
\draw[thick, dashed, ->-] (-0.2, 1.2) -- (-0.2, 1.4);
\draw[thick, dashed, ->-] (0, 1.2) -- (0, 1.4);
\draw[thick, dashed, ->-] (1, 1.2) -- (1, 1.4);
\draw[thick, dashed, ->-] (1.2, 1.2) -- (1.2, 1.4);

\foreach \i in {5,6,7}{
\draw[->-,thick,color=red] (-0.2, 0.2 * \i) -- (0, 0.2 * \i);
\draw[->-,thick,color=blue] (1.2, 0.2 * \i) -- (1, 0.2 * \i);
}


\node at (0.5, 1.5) {\small $d$};
\foreach \i in {5,6} {
\node at (-0.1, 0.1 + 0.2 * \i) {\small $d$};
\node at (1.1, 0.1 + 0.2 * \i) {\small $d$};
}

\node at (-0.2, 1) {
\begin{tikzpicture}

\draw[fill=white] (0,0) circle (0.125);

\node at (0,0) {\tiny $9$};

\end{tikzpicture}
};
\node at (-0.2, 1.2) {
\begin{tikzpicture}

\draw[fill=white] (0,0) circle (0.125);

\node at (0,0) {\tiny $4$};

\end{tikzpicture}
};
\node at (-0.2, 1.4) {
\begin{tikzpicture}

\draw[fill=white] (0,0) circle (0.125);

\node at (0,0) {\tiny $3$};

\end{tikzpicture}
};
\node at (1.2, 1) {
\begin{tikzpicture}

\draw[fill=white] (0,0) circle (0.125);

\node at (0,0) {\tiny $10$};

\end{tikzpicture}
};
\node at (1.2, 1.2) {
\begin{tikzpicture}

\draw[fill=white] (0,0) circle (0.125);

\node at (0,0) {\tiny $6$};

\end{tikzpicture}
};
\node at (1.2, 1.4) {
\begin{tikzpicture}

\draw[fill=white] (0,0) circle (0.125);

\node at (0,0) {\tiny $5$};

\end{tikzpicture}
};
\node at (0.4, 1.6) {
\begin{tikzpicture}

\draw[fill=white] (0,0) circle (0.125);

\node at (0,0) {\tiny $1$};

\end{tikzpicture}
};
\node at (0.6, 1.6) {
\begin{tikzpicture}

\draw[fill=white] (0,0) circle (0.125);

\node at (0,0) {\tiny $2$};

\end{tikzpicture}
};

\node at (0.25, 1) {\textbullet};
\node[anchor=north] at (0.25, 1) {\small $f_2 (-f_4) f_3 00000$};
\node at (0.75, 1) {\textbullet};
\node[anchor=north] at (0.75, 1) {\small $f_4 f_2 f_3 00000$};

\end{tikzpicture}
\caption{Type 3.1a, in the case $N_1 = N_2 = 1$.}\label{fig21}
\end{figure}

which equals $-1$ by Algorithm $\ref{multalgo}$, so that the central black loop is not nullhomotopic and so
\begin{align*}\delta f(D, 2\vec{e}_j, (1, 1)) &= d\pmod{2}\end{align*}

\textbf{Type 3.1b}: In this case we have triples $(D, \vec{e}_j + \vec{e}_k, (1), (1))$ where $j \neq k$. Let $d$ be the sign of rectangle $D$, and consider the following Figure $\ref{fig25}$.

\begin{figure}\begin{tikzpicture}[scale=2.5]

\begin{scope}[xshift=0cm, yshift=3cm]
\filldraw[color=gray] (0, 0) .. controls (-0.125, 0.25) .. (0, 0.5) .. controls (0.125, 0.25) .. (0, 0);
\draw (0, 0.5) .. controls (-0.125, 0.75) .. (0, 1) .. controls (0.125, 0.75) .. (0, 0.5);
\filldraw[color=gray] (0, 1) .. controls (-0.125, 1.25) .. (0, 1.5) .. controls (0.125, 1.25) .. (0, 1);

\node at (0, 0.75) {\small $D$};
\draw (0.2, 0.25) circle (0.1);
\node at (0.2, 0.25) {\small $1_j$};
\draw (0.2, 1.25) circle (0.1);
\node at (0.2, 1.25) {\small $1_k$};
\end{scope}

\begin{scope}[xshift=-1cm, yshift=2cm]
\draw (0, 0) .. controls (-0.125, 0.25) .. (0, 0.5) .. controls (0.125, 0.25) .. (0, 0);
\filldraw[color=gray] (0, 0.5) .. controls (-0.125, 0.75) .. (0, 1) .. controls (0.125, 0.75) .. (0, 0.5);
\filldraw[color=gray] (0, 1) .. controls (-0.125, 1.25) .. (0, 1.5) .. controls (0.125, 1.25) .. (0, 1);

\node at (0, 0.25) {\small $D$};
\draw (0.2, 0.75) circle (0.1);
\node at (0.2, 0.75) {\small $1_j$};
\draw (0.2, 1.25) circle (0.1);
\node at (0.2, 1.25) {\small $1_k$};
\end{scope}

\begin{scope}[xshift=1cm, yshift=2cm]
\filldraw[color=gray] (0, 0) .. controls (-0.125, 0.25) .. (0, 0.5) .. controls (0.125, 0.25) .. (0, 0);
\filldraw[color=gray] (0, 0.5) .. controls (-0.125, 0.75) .. (0, 1) .. controls (0.125, 0.75) .. (0, 0.5);
\draw (0, 1) .. controls (-0.125, 1.25) .. (0, 1.5) .. controls (0.125, 1.25) .. (0, 1);

\node at (0, 1.25) {\small $D$};
\draw (0.2, 0.25) circle (0.1);
\node at (0.2, 0.25) {\small $1_j$};
\draw (0.2, 0.75) circle (0.1);
\node at (0.2, 0.75) {\small $1_k$};
\end{scope}

\begin{scope}[xshift=0cm, yshift=-1.5cm]
\filldraw[color=gray] (0, 0) .. controls (-0.125, 0.25) .. (0, 0.5) .. controls (0.125, 0.25) .. (0, 0);
\draw (0, 0.5) .. controls (-0.125, 0.75) .. (0, 1) .. controls (0.125, 0.75) .. (0, 0.5);
\filldraw[color=gray] (0, 1) .. controls (-0.125, 1.25) .. (0, 1.5) .. controls (0.125, 1.25) .. (0, 1);

\node at (0, 0.75) {\small $D$};
\draw (0.2, 0.25) circle (0.1);
\node at (0.2, 0.25) {\small $1_k$};
\draw (0.2, 1.25) circle (0.1);
\node at (0.2, 1.25) {\small $1_j$};
\end{scope}

\begin{scope}[xshift=-1.3cm, yshift=0cm]
\draw (0, 0) .. controls (-0.125, 0.25) .. (0, 0.5) .. controls (0.125, 0.25) .. (0, 0);
\filldraw[color=gray] (0, 0.5) .. controls (-0.125, 0.75) .. (0, 1) .. controls (0.125, 0.75) .. (0, 0.5);
\filldraw[color=gray] (0, 1) .. controls (-0.125, 1.25) .. (0, 1.5) .. controls (0.125, 1.25) .. (0, 1);

\node at (0, 0.25) {\small $D$};
\draw (0.2, 0.75) circle (0.1);
\node at (0.2, 0.75) {\small $1_k$};
\draw (0.2, 1.25) circle (0.1);
\node at (0.2, 1.25) {\small $1_j$};
\end{scope}

\begin{scope}[xshift=1cm, yshift=0cm]
\filldraw[color=gray] (0, 0) .. controls (-0.125, 0.25) .. (0, 0.5) .. controls (0.125, 0.25) .. (0, 0);
\filldraw[color=gray] (0, 0.5) .. controls (-0.125, 0.75) .. (0, 1) .. controls (0.125, 0.75) .. (0, 0.5);
\draw (0, 1) .. controls (-0.125, 1.25) .. (0, 1.5) .. controls (0.125, 1.25) .. (0, 1);

\node at (0, 1.25) {\small $D$};
\draw (0.2, 0.25) circle (0.1);
\node at (0.2, 0.25) {\small $1_k$};
\draw (0.2, 0.75) circle (0.1);
\node at (0.2, 0.75) {\small $1_j$};
\end{scope}

\node at (0, 2.9) {\textbullet};
\node at (0.75, 2) {\textbullet};
\node at (0.75, 1) {\textbullet};
\node at (-0.75, 2) {\textbullet};
\node at (-0.75, 1) {\textbullet};
\node at (0, 0.1) {\textbullet};

\draw (0, 2.9) -- (0.75, 2) -- (0.75, 1) -- (0, 0.1) -- (-0.75, 1) -- (-0.75, 2) -- (0, 2.9);
\draw[color=red] (0.2, 2.5) .. controls (0, 2.7) .. (-0.2, 2.5);
\draw[color=red] (0.6, 2) .. controls (0.65, 2) .. (0.65, 1.9);
\draw[color=red] (-0.6, 2) .. controls (-0.65, 2) .. (-0.65, 1.9);
\draw[color=red] (0.2, 2.5) -- (0.6, 2);
\draw[color=red] (-0.2, 2.5) -- (-0.6, 2);
\draw[color=red] (0.65, 1.9) -- (0.65, 1.2);
\draw[color=red] (-0.65, 1.9) -- (-0.65, 1.2);
\draw[color=red] (0.65, 1.2) .. controls (0.65, 1) .. (0.6, 1);
\draw[color=red] (-0.65, 1.2) .. controls (-0.65, 1) .. (-0.6, 1);
\draw[color=red] (0.2, 0.5) -- (0.6, 1);
\draw[color=red] (-0.2, 0.5) -- (-0.6, 1);
\draw[color=red] (0.2, 0.5) .. controls (0, 0.2) .. (-0.2, 0.5);

\node at (0.2, 2.5) {
\begin{tikzpicture}

\draw[fill=white] (0,0) circle (0.125);

\node at (0,0) {\tiny $1$};

\end{tikzpicture}
};

\node at (-0.2, 2.5) {
\begin{tikzpicture}

\draw[fill=white] (0,0) circle (0.125);

\node at (0,0) {\tiny $2$};

\end{tikzpicture}
};

\node at (0.6, 2.05) {
\begin{tikzpicture}

\draw[fill=white] (0,0) circle (0.125);

\node at (0,0) {\tiny $3$};

\end{tikzpicture}
};

\node at (0.65, 1.9) {
\begin{tikzpicture}

\draw[fill=white] (0,0) circle (0.125);

\node at (0,0) {\tiny $4$};

\end{tikzpicture}
};

\node at (-0.6, 2.05) {
\begin{tikzpicture}

\draw[fill=white] (0,0) circle (0.125);

\node at (0,0) {\tiny $5$};

\end{tikzpicture}
};

\node at (-0.65, 1.9) {
\begin{tikzpicture}

\draw[fill=white] (0,0) circle (0.125);

\node at (0,0) {\tiny $6$};

\end{tikzpicture}
};

\node at (-0.65, 1.2) {
\begin{tikzpicture}

\draw[fill=white] (0,0) circle (0.125);

\node at (0,0) {\tiny $7$};

\end{tikzpicture}
};

\node at (0.65, 1.2) {
\begin{tikzpicture}

\draw[fill=white] (0,0) circle (0.125);

\node at (0,0) {\tiny $8$};

\end{tikzpicture}
};

\node at (0.6, 1) {
\begin{tikzpicture}

\draw[fill=white] (0,0) circle (0.125);

\node at (0,0) {\tiny $9$};

\end{tikzpicture}
};

\node at (-0.6, 1) {
\begin{tikzpicture}

\draw[fill=white] (0,0) circle (0.125);

\node at (0,0) {\tiny $10$};

\end{tikzpicture}
};

\node at (0.2, 0.5) {
\begin{tikzpicture}

\draw[fill=white] (0,0) circle (0.125);

\node at (0,0) {\tiny $12$};

\end{tikzpicture}
};

\node at (-0.2, 0.5) {
\begin{tikzpicture}

\draw[fill=white] (0,0) circle (0.125);

\node at (0,0) {\tiny $11$};

\end{tikzpicture}
};

\end{tikzpicture}
\caption{Type 3.1b, in the case $N_1 = N_2 = 1$ and $j > k$. The embedded picture cannot be drawn due to dimensional constraints.}\label{fig25}\end{figure}
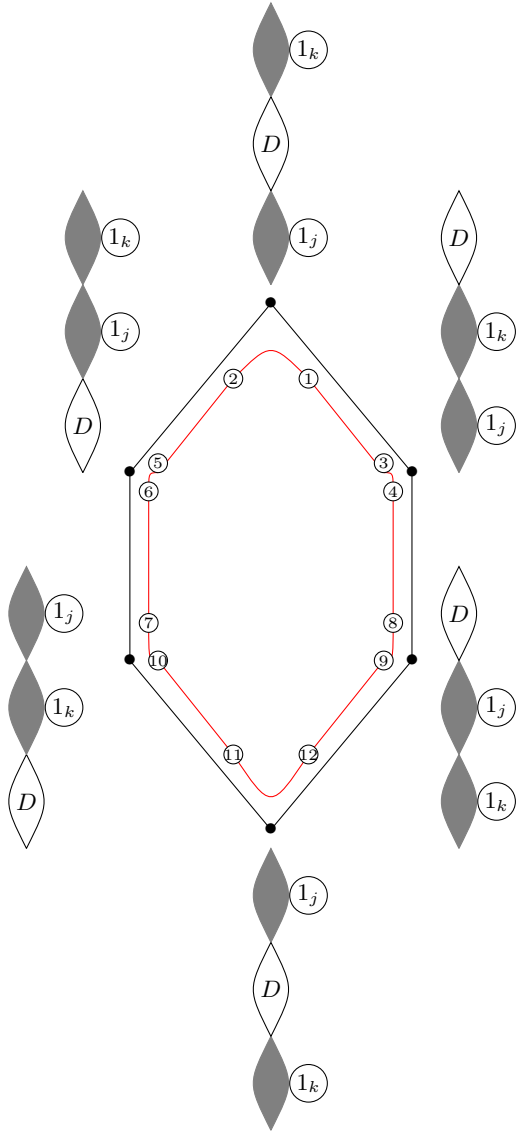

\begin{figure}
\begin{tikzpicture}[scale=6]


\node[anchor=east,align=right] at (-0.2, 0) {\small $f_3 f_4 f_2 d0000$};
\node[anchor=west,align=left] at (1.2, 0) {\small $f_2 f_3 f_1 0000d$};

\node[anchor=east,align=right] at (-0.2, 0.2) {\small $f_1 f_4 f_2 d0000$};
\node[anchor=west,align=left] at (1.2, 0.2) {\small $f_4 f_3 f_1 0000d$};

\node[anchor=east,align=right] at (-0.2, 0.4) {\small $f_1 (-f_3) f_2 d0000$};
\node[anchor=west,align=left] at (1.2, 0.4) {\small $f_4 (-f_2) f_1 0000d$};

\node[anchor=east,align=right] at (-0.2, 0.6) {\small $f_1 f_2 f_3 d0000$};
\node[anchor=west,align=left] at (1.2, 0.6) {\small $f_4 f_1 f_2 0000d$};

\node[anchor=east,align=right] at (-0.2, 0.8) {\small $f_1 (-f_4) f_3 d0000$};
\node[anchor=west,align=left] at (1.2, 0.8) {\small $f_4 (-f_3) f_2 0000d$};

\node[anchor=east,align=right] at (-0.2, 1) {\small $f_1 f_3 f_4 d0000$};
\node[anchor=west,align=left] at (1.2, 1) {\small $f_4 f_2 f_3 0000d$};

\node[anchor=east,align=right] at (-0.2, 1.2) {\small $f_1 f_2 f_4 d0000$};
\node[anchor=west,align=left] at (1.2, 1.2) {\small $f_4 f_1 f_3 0000d$};

\node[anchor=east,align=right] at (-0.2, 1.4) {\small $f_3 f_2 f_4 d0000$};
\node[anchor=west,align=left] at (1.2, 1.4) {\small $f_2 f_1 f_3 0000d$};

\node[anchor=south east,align=right] at (0.4, 1.8) {\small $f_3 f_1 f_4 00d00$};
\node[anchor=south west,align=left] at (0.6, 1.8) {\small $f_2 f_1 f_4 00d00$};

\node[anchor=north east,align=right] at (0.4, -0.4) {\small $f_3 f_4 f_1 00d00$};
\node[anchor=north west,align=left] at (0.6, -0.4) {\small $f_2 f_4 f_1 00d00$};

\draw[thick] (-0.2, 0) -- (-0.2, 1.4);
\draw[thick] (0.4, 1.8) -- (0.6, 1.8);
\draw[thick] (0.4, -0.4) -- (0.6, -0.4);
\draw[thick] (0, 1.6) -- (0.2, 1.6);
\draw[thick] (0, -0.2) -- (0.2, -0.2);

\draw[->-,thick,color=green] (0.4, 1.8) -- (0.2, 1.6);
\draw[->-,thick,color=red] (-0.2, 1.4) -- (0, 1.6);
\draw[->-,thick,color=green] (0.6, 1.8) -- (0.8, 1.6);
\draw[->-,thick,color=blue] (1.2, 1.4) -- (1, 1.6);
\draw[->-,thick,color=green] (0.4, -0.4) -- (0.2, -0.2);
\draw[->-,thick,color=red] (-0.2, 0) -- (0, -0.2);
\draw[->-,thick,color=green] (0.6, -0.4) -- (0.8, -0.2);
\draw[->-,thick,color=blue] (1.2, 0) -- (1, -0.2);

\draw[thick] (1.2, 0) -- (1.2, 1.4);
\draw[thick] (1, 1.6) -- (0.8, 1.6);
\draw[thick] (1, -0.2) -- (0.8, -0.2);


\draw[thick] (0, -0.2) -- (0, 1.6);
\draw[thick] (1, -0.2) -- (1, 1.6);

\foreach \i in {1,2,3,4,5,6}{
\draw[->-,thick,color=red] (-0.2, 0.2 * \i) -- (0, 0.2 * \i);
\draw[->-,thick,color=blue] (1.2, 0.2 * \i) -- (1, 0.2 * \i);
}
\draw[thick] (0.2, 1.6) -- (0.8, 1.6);
\draw[thick] (0.2, -0.2) -- (0.8, -0.2);


\node at (0.5, 1.7) {\small $d$};
\node at (0.5, -0.3) {\small $d$};
\foreach \i in {0,1,2,3,4,5,6}{
\node at (-0.1, 0.1 + 0.2 * \i) {\small $d$};
\node at (1.1, 0.1 + 0.2 * \i) {\small $d$};
}

\node at (0.4, 1.8) {
\begin{tikzpicture}

\draw[fill=white] (0,0) circle (0.125);

\node at (0,0) {\tiny $1$};

\end{tikzpicture}
};

\node at (0.6, 1.8) {
\begin{tikzpicture}

\draw[fill=white] (0,0) circle (0.125);

\node at (0,0) {\tiny $2$};

\end{tikzpicture}
};

\node at (-0.2, 1.4) {
\begin{tikzpicture}

\draw[fill=white] (0,0) circle (0.125);

\node at (0,0) {\tiny $3$};

\end{tikzpicture}
};

\node at (-0.2, 1.2) {
\begin{tikzpicture}

\draw[fill=white] (0,0) circle (0.125);

\node at (0,0) {\tiny $4$};

\end{tikzpicture}
};

\node at (1.2, 1.4) {
\begin{tikzpicture}

\draw[fill=white] (0,0) circle (0.125);

\node at (0,0) {\tiny $5$};

\end{tikzpicture}
};

\node at (1.2, 1.2) {
\begin{tikzpicture}

\draw[fill=white] (0,0) circle (0.125);

\node at (0,0) {\tiny $6$};

\end{tikzpicture}
};

\node at (1.2, 0.2) {
\begin{tikzpicture}

\draw[fill=white] (0,0) circle (0.125);

\node at (0,0) {\tiny $7$};

\end{tikzpicture}
};

\node at (-0.2, 0.2) {
\begin{tikzpicture}

\draw[fill=white] (0,0) circle (0.125);

\node at (0,0) {\tiny $8$};

\end{tikzpicture}
};

\node at (-0.2, 0) {
\begin{tikzpicture}

\draw[fill=white] (0,0) circle (0.125);

\node at (0,0) {\tiny $9$};

\end{tikzpicture}
};

\node at (1.2, 0) {
\begin{tikzpicture}

\draw[fill=white] (0,0) circle (0.125);

\node at (0,0) {\tiny $10$};

\end{tikzpicture}
};

\node at (0.6, -0.4) {
\begin{tikzpicture}

\draw[fill=white] (0,0) circle (0.125);

\node at (0,0) {\tiny $12$};

\end{tikzpicture}
};

\node at (0.4, -0.4) {
\begin{tikzpicture}

\draw[fill=white] (0,0) circle (0.125);

\node at (0,0) {\tiny $11$};

\end{tikzpicture}
};

\end{tikzpicture}
\caption{Type 3.1b, in the case $N_1 = N_2 = 1$ and $j > k$.}\label{fig19}\end{figure}
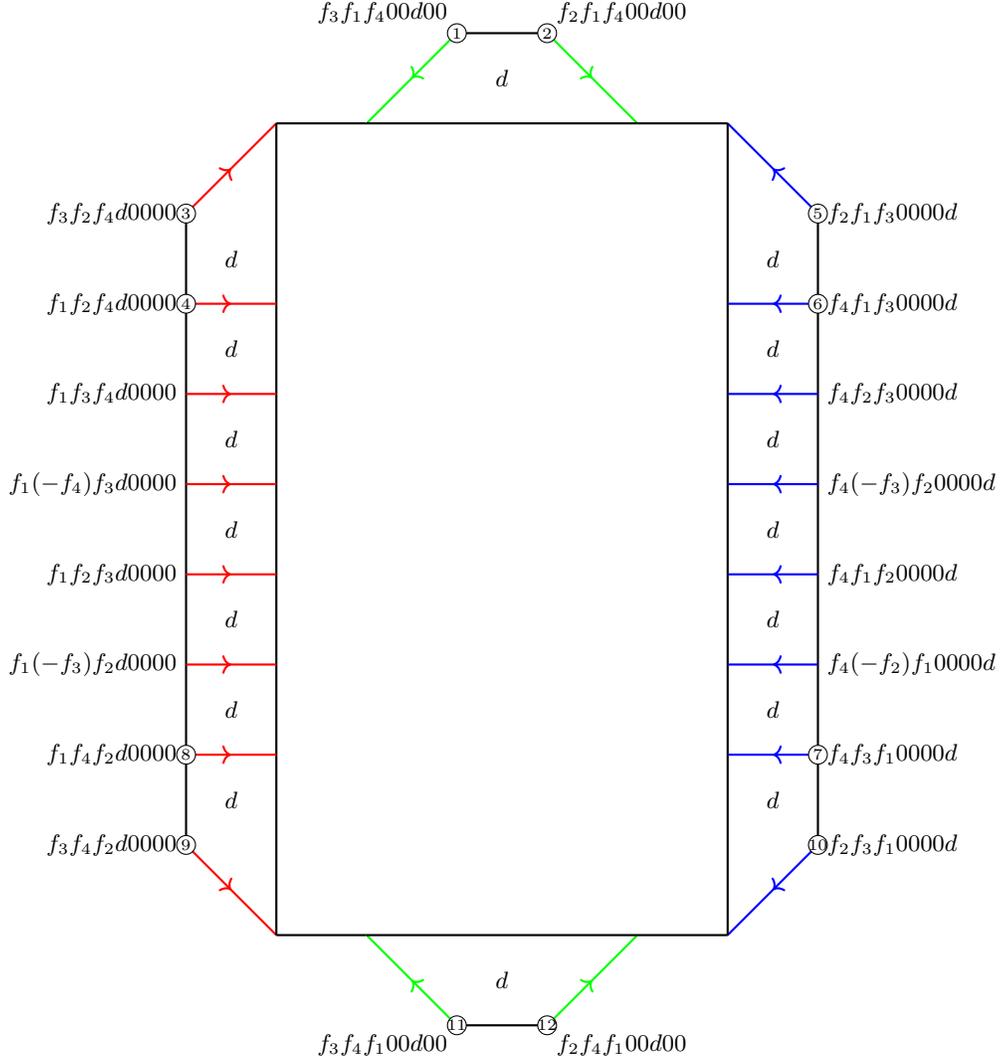

By Lemma $\ref{lem1}$, all black-green, black-blue, and black-red quadrilaterals correspond to the element of $\pi_1(\SO(5d + 4))$ labelled in Figure $\ref{fig19}$. We then compute the central black loop, which lifts to $Pin(4)$ as
\begin{align*}&\frac{1}{\sqrt{2}}(1 - e_3 e_2) \frac{1}{\sqrt{2}}(1 - e_4 e_3) \frac{1}{\sqrt{2}}(1 - e_2 e_4) E_{213} \frac{1}{\sqrt{2}}(1 - e_4 e_2) \frac{1}{\sqrt{2}}(1 - e_3 e_4) \\& \frac{1}{\sqrt{2}}(1 - e_2 e_3) \frac{1}{\sqrt{2}}(1 - e_1 e_2) \frac{1}{\sqrt{2}}(1 - e_3 e_1) E_{342} \frac{1}{\sqrt{2}}(1 - e_1 e_3)\frac{1}{\sqrt{2}}(1 - e_2 e_1)\end{align*}
where 
\begin{align}\label{E-def}
E_{ijk} := \frac{1}{\sqrt{2}}(1 - e_j e_i)\frac{1}{\sqrt{2}}(1 - e_k e_i)\frac{1}{\sqrt{2}}(1 - e_j e_k)\frac{1}{\sqrt{2}}(1 - e_i e_j) \frac{1}{\sqrt{2}}(1 - e_k e_i)\end{align}
represents the preferred Type 2.0b moduli space switching two individual bubbles with internal frames $f_j$ and $f_k$ and $\vec{v} = f_i$.

By Algorithm $\ref{multalgo}$, this equals $-1$, so that the central black loop is not nullhomotopic and so

\begin{align*}\delta f(D, \vec{e}_j + \vec{e}_k, ((1), (1))) &= 0 \pmod{2}\end{align*}

\textbf{Type 3.0c}: In this case we have triples $(c_x, \vec{e}_j + \vec{e}_k + \vec{e}_l, ((1), (1), (1))$ for $j, k, l$ distinct. The vertices of $\partial' \mathcal{M}(c_x, \vec{e}_j + \vec{e}_k + \vec{e}_l, ((1), (1), (1))$ in the below Figure $\ref{fig26}$ can be written as points in $\SO(6d + 5)$ that form a loop.

\begin{figure}\begin{tikzpicture}[scale=2.5]

\begin{scope}[xshift=0cm, yshift=3cm]
\filldraw[color=gray] (0, 0) .. controls (-0.125, 0.25) .. (0, 0.5) .. controls (0.125, 0.25) .. (0, 0);
\filldraw[color=gray] (0, 0.5) .. controls (-0.125, 0.75) .. (0, 1) .. controls (0.125, 0.75) .. (0, 0.5);
\filldraw[color=gray] (0, 1) .. controls (-0.125, 1.25) .. (0, 1.5) .. controls (0.125, 1.25) .. (0, 1);

\draw (0.2, 0.25) circle (0.1);
\node at (0.2, 0.25) {\small $1_j$};
\draw (0.2, 0.75) circle (0.1);
\node at (0.2, 0.75) {\small $1_k$};
\draw (0.2, 1.25) circle (0.1);
\node at (0.2, 1.25) {\small $1_l$};
\end{scope}

\begin{scope}[xshift=1cm, yshift=2cm]
\filldraw[color=gray] (0, 0) .. controls (-0.125, 0.25) .. (0, 0.5) .. controls (0.125, 0.25) .. (0, 0);
\filldraw[color=gray] (0, 0.5) .. controls (-0.125, 0.75) .. (0, 1) .. controls (0.125, 0.75) .. (0, 0.5);
\filldraw[color=gray] (0, 1) .. controls (-0.125, 1.25) .. (0, 1.5) .. controls (0.125, 1.25) .. (0, 1);

\draw (0.2, 0.25) circle (0.1);
\node at (0.2, 0.25) {\small $1_j$};
\draw (0.2, 0.75) circle (0.1);
\node at (0.2, 0.75) {\small $1_l$};
\draw (0.2, 1.25) circle (0.1);
\node at (0.2, 1.25) {\small $1_k$};
\end{scope}

\begin{scope}[xshift=-1cm, yshift=2cm]
\filldraw[color=gray] (0, 0) .. controls (-0.125, 0.25) .. (0, 0.5) .. controls (0.125, 0.25) .. (0, 0);
\filldraw[color=gray] (0, 0.5) .. controls (-0.125, 0.75) .. (0, 1) .. controls (0.125, 0.75) .. (0, 0.5);
\filldraw[color=gray] (0, 1) .. controls (-0.125, 1.25) .. (0, 1.5) .. controls (0.125, 1.25) .. (0, 1);

\draw (0.2, 0.25) circle (0.1);
\node at (0.2, 0.25) {\small $1_k$};
\draw (0.2, 0.75) circle (0.1);
\node at (0.2, 0.75) {\small $1_j$};
\draw (0.2, 1.25) circle (0.1);
\node at (0.2, 1.25) {\small $1_l$};
\end{scope}

\begin{scope}[xshift=0cm, yshift=-1.5cm]
\filldraw[color=gray] (0, 0) .. controls (-0.125, 0.25) .. (0, 0.5) .. controls (0.125, 0.25) .. (0, 0);
\filldraw[color=gray] (0, 0.5) .. controls (-0.125, 0.75) .. (0, 1) .. controls (0.125, 0.75) .. (0, 0.5);
\filldraw[color=gray] (0, 1) .. controls (-0.125, 1.25) .. (0, 1.5) .. controls (0.125, 1.25) .. (0, 1);

\draw (0.2, 0.25) circle (0.1);
\node at (0.2, 0.25) {\small $1_l$};
\draw (0.2, 0.75) circle (0.1);
\node at (0.2, 0.75) {\small $1_k$};
\draw (0.2, 1.25) circle (0.1);
\node at (0.2, 1.25) {\small $1_j$};
\end{scope}

\begin{scope}[xshift=1cm, yshift=0cm]
\filldraw[color=gray] (0, 0) .. controls (-0.125, 0.25) .. (0, 0.5) .. controls (0.125, 0.25) .. (0, 0);
\filldraw[color=gray] (0, 0.5) .. controls (-0.125, 0.75) .. (0, 1) .. controls (0.125, 0.75) .. (0, 0.5);
\filldraw[color=gray] (0, 1) .. controls (-0.125, 1.25) .. (0, 1.5) .. controls (0.125, 1.25) .. (0, 1);

\draw (0.2, 0.25) circle (0.1);
\node at (0.2, 0.25) {\small $1_l$};
\draw (0.2, 0.75) circle (0.1);
\node at (0.2, 0.75) {\small $1_j$};
\draw (0.2, 1.25) circle (0.1);
\node at (0.2, 1.25) {\small $1_k$};
\end{scope}

\begin{scope}[xshift=-1.3cm, yshift=0cm]
\filldraw[color=gray] (0, 0) .. controls (-0.125, 0.25) .. (0, 0.5) .. controls (0.125, 0.25) .. (0, 0);
\filldraw[color=gray] (0, 0.5) .. controls (-0.125, 0.75) .. (0, 1) .. controls (0.125, 0.75) .. (0, 0.5);
\filldraw[color=gray] (0, 1) .. controls (-0.125, 1.25) .. (0, 1.5) .. controls (0.125, 1.25) .. (0, 1);

\draw (0.2, 0.25) circle (0.1);
\node at (0.2, 0.25) {\small $1_k$};
\draw (0.2, 0.75) circle (0.1);
\node at (0.2, 0.75) {\small $1_l$};
\draw (0.2, 1.25) circle (0.1);
\node at (0.2, 1.25) {\small $1_j$};
\end{scope}

\node at (0, 2.9) {\textbullet};
\node at (0.75, 2) {\textbullet};
\node at (0.75, 1) {\textbullet};
\node at (-0.75, 2) {\textbullet};
\node at (-0.75, 1) {\textbullet};
\node at (0, 0.1) {\textbullet};

\draw (0, 2.9) -- (0.75, 2) -- (0.75, 1) -- (0, 0.1) -- (-0.75, 1) -- (-0.75, 2) -- (0, 2.9);
\draw[color=red] (0.2, 2.5) .. controls (0, 2.7) .. (-0.2, 2.5);
\draw[color=red] (0.6, 2) .. controls (0.65, 2) .. (0.65, 1.9);
\draw[color=red] (-0.6, 2) .. controls (-0.65, 2) .. (-0.65, 1.9);
\draw[color=red] (0.2, 2.5) -- (0.6, 2);
\draw[color=red] (-0.2, 2.5) -- (-0.6, 2);
\draw[color=red] (0.65, 1.9) -- (0.65, 1.2);
\draw[color=red] (-0.65, 1.9) -- (-0.65, 1.2);
\draw[color=red] (0.65, 1.2) .. controls (0.65, 1) .. (0.6, 1);
\draw[color=red] (-0.65, 1.2) .. controls (-0.65, 1) .. (-0.6, 1);
\draw[color=red] (0.2, 0.5) -- (0.6, 1);
\draw[color=red] (-0.2, 0.5) -- (-0.6, 1);
\draw[color=red] (0.2, 0.5) .. controls (0, 0.2) .. (-0.2, 0.5);

\node at (0.2, 2.5) {
\begin{tikzpicture}

\draw[fill=white] (0,0) circle (0.125);

\node at (0,0) {\tiny $1$};

\end{tikzpicture}
};

\node at (-0.2, 2.5) {
\begin{tikzpicture}

\draw[fill=white] (0,0) circle (0.125);

\node at (0,0) {\tiny $2$};

\end{tikzpicture}
};

\node at (0.6, 2.05) {
\begin{tikzpicture}

\draw[fill=white] (0,0) circle (0.125);

\node at (0,0) {\tiny $3$};

\end{tikzpicture}
};

\node at (0.65, 1.9) {
\begin{tikzpicture}

\draw[fill=white] (0,0) circle (0.125);

\node at (0,0) {\tiny $4$};

\end{tikzpicture}
};

\node at (-0.6, 2.05) {
\begin{tikzpicture}

\draw[fill=white] (0,0) circle (0.125);

\node at (0,0) {\tiny $5$};

\end{tikzpicture}
};

\node at (-0.65, 1.9) {
\begin{tikzpicture}

\draw[fill=white] (0,0) circle (0.125);

\node at (0,0) {\tiny $6$};

\end{tikzpicture}
};

\node at (-0.65, 1.2) {
\begin{tikzpicture}

\draw[fill=white] (0,0) circle (0.125);

\node at (0,0) {\tiny $7$};

\end{tikzpicture}
};

\node at (0.65, 1.2) {
\begin{tikzpicture}

\draw[fill=white] (0,0) circle (0.125);

\node at (0,0) {\tiny $8$};

\end{tikzpicture}
};

\node at (0.6, 1) {
\begin{tikzpicture}

\draw[fill=white] (0,0) circle (0.125);

\node at (0,0) {\tiny $9$};

\end{tikzpicture}
};

\node at (-0.6, 1) {
\begin{tikzpicture}

\draw[fill=white] (0,0) circle (0.125);

\node at (0,0) {\tiny $10$};

\end{tikzpicture}
};

\node at (0.2, 0.5) {
\begin{tikzpicture}

\draw[fill=white] (0,0) circle (0.125);

\node at (0,0) {\tiny $12$};

\end{tikzpicture}
};

\node at (-0.2, 0.5) {
\begin{tikzpicture}

\draw[fill=white] (0,0) circle (0.125);

\node at (0,0) {\tiny $11$};

\end{tikzpicture}
};

\end{tikzpicture}
\caption{Type 3.0c, in the case $N_1 = N_2 = N_3 = 1$ and $j > k > l$. The embedded picture cannot be drawn due to dimensional constraints.}\label{fig26}\end{figure}

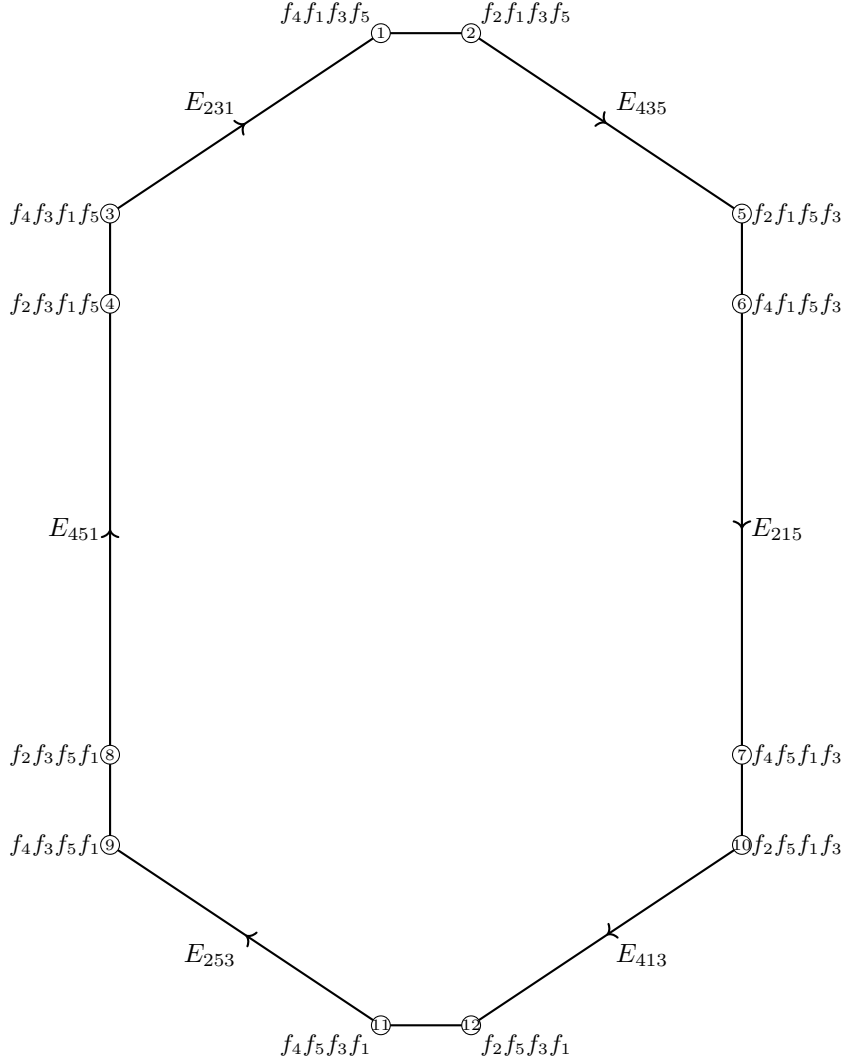
\begin{figure}
\begin{tikzpicture}[scale=6]


\node[anchor=east,align=right] at (-0.2, 0) {\small $f_4 f_3 f_5 f_1$};
\node[anchor=west,align=left] at (1.2, 0) {\small $f_2 f_5 f_1 f_3$};

\node[anchor=east,align=right] at (-0.2, 0.2) {\small $f_2 f_3 f_5 f_1$};
\node[anchor=west,align=left] at (1.2, 0.2) {\small $f_4 f_5 f_1 f_3$};

\node[anchor=east,align=right] at (-0.2, 1.2) {\small $f_2 f_3 f_1 f_5$};
\node[anchor=west,align=left] at (1.2, 1.2) {\small $f_4 f_1 f_5 f_3$};

\node[anchor=east,align=right] at (-0.2, 1.4) {\small $f_4 f_3 f_1 f_5$};
\node[anchor=west,align=left] at (1.2, 1.4) {\small $f_2 f_1 f_5 f_3$};

\node[anchor=south east,align=right] at (0.4, 1.8) {\small $f_4 f_1 f_3 f_5$};
\node[anchor=south west,align=left] at (0.6, 1.8) {\small $f_2 f_1 f_3 f_5$};

\node[anchor=north east,align=right] at (0.4, -0.4) {\small $f_4 f_5 f_3 f_1$};
\node[anchor=north west,align=left] at (0.6, -0.4) {\small $f_2 f_5 f_3 f_1$};

\draw[thick, ->-] (-0.2, 0) -- (-0.2, 1.4);
\draw[thick, ->-] (-0.2, 1.4) -- (0.4, 1.8);
\draw[thick] (0.4, 1.8) -- (0.6, 1.8);
\draw[thick, ->-] (0.6, 1.8) -- (1.2, 1.4);
\draw[thick, ->-] (1.2, 1.4) -- (1.2, 0);
\draw[thick, ->-] (1.2, 0) -- (0.6, -0.4);
\draw[thick] (0.6, -0.4) -- (0.4, -0.4);
\draw[thick, ->-] (0.4, -0.4) -- (-0.2, 0);

\node[anchor=south west] at (0.9, 1.6) {$E_{435}$};
\node[anchor=west] at (1.2, 0.7) {$E_{215}$};
\node[anchor=north west] at (0.9, -0.2) {$E_{413}$};
\node[anchor=north east] at (0.1, -0.2) {$E_{253}$};
\node[anchor=east] at (-0.2, 0.7) {$E_{451}$};
\node[anchor=south east] at (0.1, 1.6) {$E_{231}$};



\node at (0.4, 1.8) {
\begin{tikzpicture}

\draw[fill=white] (0,0) circle (0.125);

\node at (0,0) {\tiny $1$};

\end{tikzpicture}
};

\node at (0.6, 1.8) {
\begin{tikzpicture}

\draw[fill=white] (0,0) circle (0.125);

\node at (0,0) {\tiny $2$};

\end{tikzpicture}
};

\node at (-0.2, 1.4) {
\begin{tikzpicture}

\draw[fill=white] (0,0) circle (0.125);

\node at (0,0) {\tiny $3$};

\end{tikzpicture}
};

\node at (-0.2, 1.2) {
\begin{tikzpicture}

\draw[fill=white] (0,0) circle (0.125);

\node at (0,0) {\tiny $4$};

\end{tikzpicture}
};

\node at (1.2, 1.4) {
\begin{tikzpicture}

\draw[fill=white] (0,0) circle (0.125);

\node at (0,0) {\tiny $5$};

\end{tikzpicture}
};

\node at (1.2, 1.2) {
\begin{tikzpicture}

\draw[fill=white] (0,0) circle (0.125);

\node at (0,0) {\tiny $6$};

\end{tikzpicture}
};

\node at (1.2, 0.2) {
\begin{tikzpicture}

\draw[fill=white] (0,0) circle (0.125);

\node at (0,0) {\tiny $7$};

\end{tikzpicture}
};

\node at (-0.2, 0.2) {
\begin{tikzpicture}

\draw[fill=white] (0,0) circle (0.125);

\node at (0,0) {\tiny $8$};

\end{tikzpicture}
};

\node at (-0.2, 0) {
\begin{tikzpicture}

\draw[fill=white] (0,0) circle (0.125);

\node at (0,0) {\tiny $9$};

\end{tikzpicture}
};

\node at (1.2, 0) {
\begin{tikzpicture}

\draw[fill=white] (0,0) circle (0.125);

\node at (0,0) {\tiny $10$};

\end{tikzpicture}
};

\node at (0.6, -0.4) {
\begin{tikzpicture}

\draw[fill=white] (0,0) circle (0.125);

\node at (0,0) {\tiny $12$};

\end{tikzpicture}
};

\node at (0.4, -0.4) {
\begin{tikzpicture}

\draw[fill=white] (0,0) circle (0.125);

\node at (0,0) {\tiny $11$};

\end{tikzpicture}
};

\end{tikzpicture}
\caption{Type 3.0c, in the case $N_1 = N_2 = N_3 = 1$ and $j > k > l$. We omit external frames since they are throughout the standard positive frame. Long edges are labeled by their corresponding lifts written in terms of $(\ref{E-def})$, with the path oriented clockwise in the figure.}\label{fig27}\end{figure}

The black loop in Figure $\ref{fig27}$ lifts to $Pin(5)$ as
\begin{align*}&\frac{1}{\sqrt{2}}(1 - e_4 e_2)E_{435} \frac{1}{\sqrt{2}}(1 - e_2 e_4)E_{215} \frac{1}{\sqrt{2}}(1 - e_4 e_2)E_{413} \\&\frac{1}{\sqrt{2}}(1 - e_2 e_4)E_{253} \frac{1}{\sqrt{2}}(1 - e_4 e_2)E_{451} \frac{1}{\sqrt{2}}(1 - e_2 e_4)E_{231}\end{align*}

By Algorithm $\ref{multalgo}$, this equals $-1$, so that the black loop is not nullhomotopic and so

\begin{align*}\delta f(c_x, \vec{e}_j + \vec{e}_k + \vec{e}_l, ((1), (1), (1))) &= 0 \pmod{2}\end{align*}\end{proof}

\begin{proof}[Proof of Theorem $\ref{mainthm}$] It will help to recall the generators of $H_*(\CDP_*)$ from \cite[Section 4]{YT}. $H_2(\CDP_*)$ is generated by the domains
\begin{align*}(c_{\text{Id}}, \vec{e}_j + \vec{e}_k, ((1)_j, (1)_k) \text{ for } 1 \leq j < k \leq n\end{align*}
as well as a generator $U'$ which is a linear combination of every annulus in the grid with an even number of Type 2.2a and Type 2.1 triples. $H_3(\CDP_*)$ is generated by the domains
\begin{align*}(c_{\text{Id}}, \vec{e}_j + \vec{e}_k + \vec{e}_l, ((1)_j, (1)_k, (1)_l) \text{ for } 1 \leq j < k < l \leq n\end{align*}
as well as generators $U_j'$ for each $1 \leq j \leq n$, which are unit enlargements of $U'$ (each one adding $\vec{e}_j$ to $\vec{N}$) plus some Type 3.2a triples. $H_2(\CDP_*^0)$ and $H_3(\CDP_*^0)$ are generated by the same triples, but not $U'$ or $U_j'$.

We define a cochain $\mathfrak{o}_2: \CDP_3 \rightarrow \tilde{\Omega}_{fr}^{1} \cong \Z/2$ by $\mathfrak{o}_2(D, \vec{N}, \vec{\lambda}) := [\partial' \overline{\mathcal{M}}(D, \vec{N}, \vec{\lambda})]$; that is, $\mathfrak{o}_2$ measures the obstruction to framing the moduli space given the framing of its boundary. By definition, $f \in \CDP^2(\G; \Z/2)$ is a frame assignment if and only if $\delta f = \mathfrak{o}_2$, so it will suffice to show that $\mathfrak{o}_2$ is a coboundary in $\CDP^3$.

\cite[Section 12]{MS} showed that $\mathfrak{o}_2$ is a coboundary on $\CDP^0_*$. However, their proof only requires blocking an $O$ marking for the homological computation---an identical proof allowing the full grid shows that $\mathfrak{o}_2$ is a cocycle on $\CDP_*$. To show that $\mathfrak{o}_2$ is a coboundary on $\CDP_*$, it then suffices to show that it evaluates to zero on all the generators of $H_3(\CDP_*)$. By Lemma $\ref{framecond}$, we must have that
\begin{align*}\mathfrak{o}_2(c_{\text{Id}}, \vec{e}_j + \vec{e}_k + \vec{e}_l, ((1)_j, (1)_k, (1)_l) = 0\end{align*}
for all $1 \leq j < k < l \leq n$. Finally, since each $U_j'$ comes from a unit enlargement, they consist only of domains of Type 3.2 and 3.1 of the types in Lemma $\ref{framecond}$. Of these, only the Type 3.1a triples $(R_{jk}, 2\vec{e}_j, (1, 1))$ contribute to $\mathfrak{o}_2(U_j')$. The domains $R_{jk}$ come in pairs, and the concatenation of all the pairs $-R_{jk_1} * R_{jk_2}$ form a domain from the generators for the annuli $H_j$ and $V_j$ that comprise $U'$. Decomposing this domain into index 2 domains, we see that the signs of all $R_{jk}$ must cancel with each other. As a result, we have that
\begin{align*}
\mathfrak{o}_2(U_j') = 0
\end{align*}
for all $j$ as well, so that $\mathfrak{o}_2$ is zero in homology, hence a coboundary.\end{proof}

\section{Finding a Frame Assignment}

There are two steps to algorithmically compute a frame assignment $f$. The first is to show some form of uniqueness of $f$, and the second is to find a finite system of linear equations in generators of $\CDP_2$ which $f$ solves. We begin with its uniqueness property.

\begin{prop}\label{unique}A frame assignment $f$ on $\CDP_*^0$ is unique up to the values of
\begin{align*}f_{jk} := f(c_{\text{Id}}, \vec{e}_j + \vec{e}_k, ((1), (1))) \text{ for } 1 \leq j < k \leq n\end{align*}
$f(U')$, and $1$-coboundaries.\end{prop}

\begin{proof}Suppose $f, f'$ are frame assignments. Then $\mathfrak{o}_2 = \delta f = \delta f'$, so that $\delta(f - f') = 0$ and so $f - f'$ is a cocycle. If we also have that $f(U') = f'(U')$ and $f_{jk} = f_{jk}'$ for all $j, k$, $f - f'$ is zero on the generators of $H_2(\CDP_*^0)$, and therefore a coboundary.\end{proof}

As a result of Proposition $\ref{unique}$, we have a $\binom{n}{2}+1$ parameter family of frame assignments. Changing $f_{jk}$ has the same effect as changing the preferred Type 2.0b path from Section 6. We may also fix $f_{jk}$ by comparing our preferred Type 2.0b paths to an existing frame assignment. For example, \cite[Section 13]{MS} constructs an explicit framing of the moduli spaces of triples of the form $(c_{\text{Id}}, \vec{e}_j + \vec{e}_k, ((1), (1)))$, so we may fix $f_{jk}$ by comparing our preferred paths to this framing.

\begin{prop}\label{unique2}Using the framings of \cite[Section 13]{MS}, $f_{jk} = 0$ for all $1 \leq j < k \leq n$.\end{prop}

\begin{proof}Consider the moduli space $\mathcal{M}(c_x, \vec{e}_j + \vec{e}_k, ((1), (1))$, where $j < k$. Our preferred Type 2.0b moduli space has endpoints $[f_1, f_3]$ and $[f_3, f_1]$ and passes through the frame $[f_2, f_3]$. On the other hand, in the Manolescu-Sarkar framing \cite[Example 13.4]{MS}, the endpoints $x = 0, 1$ have the internal framings $[v_{j}(0), v_k(0)] = [f_1, f_3]$ and $[v_{j}(1), v_k(1)] = [f_3, f_1]$, respectively, which is consistent with the corresponding endpoints of our Type 2.0b moduli space. It suffices to show that the Manolescu-Sarkar framing near the midpoint $x = \frac{1}{2}$ also produces $[f_2, f_3]$. Fix a small $\epsilon$. The internal frame $[v_j(\frac{1}{2} + \epsilon), v_k(\frac{1}{2} + \epsilon)]$ at $x = \frac{1}{2} + \epsilon$ is approximately
\begin{align*}[&(k-j)\left(\frac{1}{2} - 2\epsilon^2\right)f_2 + \left(\frac{1}{2} + \epsilon\right)e^{\frac{j+k}{2}} f_1 + \left(\frac{1}{2} - \epsilon\right)e^{\frac{j+k}{2}} f_3, \\&(k-j)\left(-\frac{1}{2} + 2\epsilon^2\right)f_2 + \left(\frac{1}{2} - \epsilon\right)e^{\frac{j+k}{2}} f_1 + \left(\frac{1}{2} + \epsilon\right)e^{\frac{j+k}{2}} f_3]\end{align*}
which, after applying the Gram-Schmidt process, gives the same framing $[f_2, f_3]$. Therefore we must have that $f_{jk} = 0$ for all $1 \leq j < k \leq n$.\end{proof}

As a result, we are able to reduce the $\binom{n}{2}+1$ parameter family of frame assignments to a one-parameter family given by $f(U')$, which gives two (potentially) different frame assignments (and later framed $1$-flow categories). In any case, the following algorithm functions identically for each value of $f_{jk}, f(U')$.

The system of equations will come from the values of $\delta f = \mathfrak{o}_2$, as every equation $b = \mathfrak{o}_2(D, \vec{N}, \vec{\lambda})$ can be rewritten as the linear equation
\begin{align*}b = \sum\limits_{(E, \vec{M}, \vec{\nu}) \text{ is a term of } \partial (D, \vec{N}, \vec{\lambda})} f(E, \vec{M}, \vec{\nu})\end{align*}
In Section 7, we computed $\mathfrak{o}_2(D, \vec{N}, \vec{\lambda})$ for certain triples $(D, \vec{N}, \vec{\lambda})$. We now describe how to compute it for the remaining triples. A direct application of Algorithm $\ref{multalgo}$ suffices for the remaining Type 3.1b and Type 3.0c moduli spaces, as they lie in one stratum in the local model. For Type 3.1a domains, we can decompose their boundaries similarly to Figure $\ref{fig21}$---since every edge is either a Type 2.1 or 2.0a moduli space, we are able to explicitly write each edge as a product of short preferred paths and use Algorithm $\ref{multalgo}$.

The boundaries of Type 3.0a and 3.0b moduli spaces are already nullhomotopic by the definition of the local model, as the internal frames along their local models are constant (for the same type of bubble in Type 3.0b).

Finally, to calculate the remaining Type 3.2c terms, we use induction. Suppose that $D = H_j$ or $V_j$ and $\mathfrak{o}_2(D, N \vec{e}_j, (N)_j)$ is known for $N \leq N_0$. Then, since by \cite[Proposition 12.2]{MS} we have that $\mathfrak{o}_2$ is a cocycle,
\begin{align*}0 &= \delta \mathfrak{o}_2(D, (N_0+1)\vec{e}_j, (N_0, 1)_j) \\&= \mathfrak{o}_2(D, (N_0+1)\vec{e}_j, (N_0+1)_j) + \mathfrak{o}_2(D, \vec{e}_j, (1)_j) + \mathfrak{o}_2(D, N_0\vec{e}_j, (N_0)_j) \\&+ \text{Type 3.0a terms (from Type II differentials)} + \text{Type 3.1a terms (from Type I differentials).}\end{align*}
Since the Type 3.0a and Type 3.1a terms are now understood, we can calculate $\mathfrak{o}_2(D, (N_0+1)\vec{e}_j, (N_0+1)_j)$, which completes the induction since the base case was done in Section $7$.

We now specialize to $\CD_*$ in order to find a finite system of equations. $\CD_*$ is not a subcomplex of $\CDP_*$, so the result of Theorem $\ref{mainthm}$ is not immediately helpful, and we must find a finite subcomplex of $\CDP_*$ containing $\CD_*$.

For triples $(D, \vec{N}, \vec{\lambda}) \in \CDP_*$, let $\omega(D, \vec{N}, \vec{\lambda}) := \mu(D) + \lvert \vec{N} \rvert$. We claim that $\omega$ is a filtration. Indeed, Type I differential terms lower $\mu(D)$ by $1$ and leave $\vec{N}$ unchanged. Type II differential terms lower $\mu(D)$ by $2$ and increase $\lvert \vec{N} \rvert$ by $1$. Type III terms leave $\omega$ unchanged, and type IV differential terms leave $\mu(D)$ unchanged and lower $\lvert \vec{N} \rvert$.

\begin{lemma}\label{lemsubcplx}$\CDP_*^{\omega\leq 3}$ is a finite subcomplex of $\CDP_*$ containing 
\begin{align*}\bigoplus\limits_{j = 0}^{3}\CD_j.\end{align*}
whose second homology is generated by the generators of $H_2(\CDP_*)$.\end{lemma}

\begin{proof}That $\CDP_*^{\omega\leq 3}$ is a subcomplex follows from the fact that $\omega$ is a filtration. That it is finite and contains $\CD_0, \dots, \CD_3$ as well as the generators $U'$ and $(c_{\text{Id}}, \vec{e}_j + \vec{e}_k, ((1), (1)))$ of $H_2(\CDP_*)$ follows from its definition. To compute its second homology, we will construct more filtrations on $\CDP_*^{\omega\leq 3}$.

For $(D, \vec{N}, \vec{\lambda})$, let $A(D, \vec{N}, \vec{\lambda}) \in \N^n$ denote the coefficients of $D$ in the column containing $O_1$, recorded from the $O_1$ marking down, which is a filtration as Type I and II differentials can only decrease $A$ and Type III and IV differentials leave it unchanged. Now consider $(D, \vec{N}, \vec{\lambda})$ in the associated graded $\CDP_*^{\omega\leq 3, a}$ where $A(D, \vec{N}, \vec{\lambda}) = a$. Let $B(D) \in \N^{n-1}$ denote the coefficients of $D$ in the row where $a$ has a minimum. In the associated graded $\CDP_*^{\omega\leq 3, a}$, $B(D)$ is a filtration similarly. Finally, consider the associated graded $\CDP_*^{\omega\leq 3, a, b}$. Since differentials here cannot pass through a certain row and column, there are no Type II differentials, so $\lvert \vec{N} \rvert$ is now a filtration. (See \cite{YT} for the full proofs that these are filtrations.)

Finally, the associated graded complexes $\CDP_*^{\omega\leq 3, a, b, \vec{N}}$ generally have no second homology. Specifically, the proofs of \cite[Proposition 4.4]{YT} and \cite[Proposition 4.8]{MS} show that $H_2(\CDP_*^{\omega\leq 3, a, b, \vec{N}}) = 0$ except when $a = \vec{0}$ and exactly two of the $N_j$ are $1$ (the rest zero), or when $a = (1, 1, \dots, 1)$ and $\vec{N} = \vec{0}$, which correspond to exactly the generators of $H_2(\CDP_*)$.\end{proof}

\begin{proof}[Proof of Theorem $\ref{mainthm2}$.] As a result of Theorem $\ref{mainthm}$ and Proposition $\ref{unique}$, there exists a frame assignment $f$ on the subcomplex $\CDP_*^{\omega \leq 3}$ of $\CDP_*$ which is unique up to the values $f_{jk}$ and coboundaries. Since $\CDP_*^{\omega\leq 3}$ is finite and we can algorithmically determine the value $b$ of every $\mathfrak{o}_2(D, \vec{N}, \vec{\lambda})$ for each $(D, \vec{N}, \vec{\lambda}) \in \CDP_3^{\omega\leq 3}$, we set up the finite system of linear equations
\begin{align*}\{ b = \mathfrak{o}_2(D, \vec{N}, \vec{\lambda}) \}_{(D, \vec{N}, \vec{\lambda}) \in \CDP_3^{\omega \leq 3}}\end{align*}
Solving this system is therefore an algorithmic computation of $f$ on $\CDP_2^{\omega\leq 3}$, and hence on $\CD_*$ by Lemma $\ref{lemsubcplx}$.\end{proof}

\section{Framed 1-Flow Categories}

We have now constructed the $0$- and $1$-dimensional moduli spaces and the boundary of the $2$-dimensional moduli spaces, but it still remains to put them together into a framed $1$-flow category. In the later part of this section, we modify our moduli spaces to be suitable for the Lobb-Orson-Schütz construction, but we now repeat this construction; see \cite{LOS} for details.

\begin{definition}A $1$-flow category $\mathscr{C}$ consists of a finite set of objects $\text{Ob}(\mathscr{C})$, a function $\lvert \cdot \rvert: \text{Ob}(\mathscr{C}) \rightarrow \Z$ called the grading, a moduli space for each pair $a, b \in \text{Ob}(\mathscr{C})$ with $\lvert a \rvert - \lvert b \rvert = 1$ or $2$, which is a compact $(\lvert a \rvert - \lvert b \rvert - 1)$-dimensional manifold with boundary denoted $\mathcal{M}(a, b)$, and a $1$-dimensional ``boundary" $\partial \mathcal{M}(a, b)$ for each pair $\lvert a \rvert - \lvert b \rvert = 3$ satisfying:
\begin{enumerate}

\item For any pair $a, b \in \text{Ob}(\mathscr{C})$ with $\lvert a \rvert - \lvert b \rvert = 2$, the boundary of the moduli space is given by
\begin{align*}\partial \mathcal{M}(a, b) = \coprod\limits_{c \in \text{Ob}(\mathscr{C}), \lvert c \rvert - \lvert b \rvert = 1} \mathcal{M}(c, b) \times \mathcal{M}(a, c)\end{align*}

\item For any pair $a, b \in \text{Ob}(\mathscr{C})$ with $\lvert a \rvert - \lvert b \rvert = 3$, the ``boundary" is given by
\begin{align*}\partial \mathcal{M}(a, b) = \coprod\limits_{c \in \text{Ob}(\mathscr{C}), \lvert c \rvert - \lvert b \rvert = 1} \mathcal{M}(c, b) \times \mathcal{M}(a, c) \: \bigcup \: \coprod\limits_{c \in \text{Ob}(\mathscr{C}), \lvert c \rvert - \lvert b \rvert = 2} \mathcal{M}(c, b) \times \mathcal{M}(a, c)\end{align*}

\end{enumerate}
A framed $1$-flow category is a $1$-flow category together with:
\begin{itemize}

\item A sign assignment $s$, which assigns each point in a $0$-dimensional moduli space a $0$ or $1$ such that whenever $(P_1, Q_1)$ and $(P_2, Q_2)$ are the endpoints of an interval in a $1$-dimensional moduli space,
\begin{align*}s(P_1) + s(P_2) + s(Q_1) + s(Q_2) = 1 \pmod 2\end{align*}

\item A frame assignment $f$, which assigns each interval in a $1$-dimensional moduli space\footnote{In general, one-dimensional moduli spaces may also contain circles. The Manolescu-Sarkar moduli spaces, and similarly our moduli spaces, do not, so we omit the convention for the frame assignment of a circle component. The reader is advised to see \cite{Sch} for this convention.} a $0$ or $1$ such that for any pair $a, b \in \mathit{Ob}(\mathscr{C})$ with $\lvert a \rvert - \lvert b \rvert = 3$,
\begin{align*}\sum\limits_{C}\left( 1 + \sum\limits_{I \times \{ Q \}} f(I) + \sum\limits_{\{ P \} \times I} (1 + s(P) + f(I)) \right) = 0 \pmod 2\end{align*}
where the outer sum is taken over all components of $\partial \mathcal{M}(a, b)$, the first inner sum is taken over all intervals in $C$ of the form $I \times \{Q\}$ for an interval $I$ and point $Q$, and the second inner sum is taken over all intervals in $C$ of the form $\{P\} \times I$ for an interval $I$ and point $P$.

\end{itemize}\label{def8}\end{definition}

The objects of the link Floer $1$-flow category are the generators of the grid chain complex $\mathit{GC}^+$, and the moduli spaces will largely be exactly the moduli spaces of the domains from one generator to another. While most of these moduli spaces simply fit into the definition of a framed $1$-flow category, the moduli spaces corresponding to annuli do not, as the endpoints where the annulus has bubbled is not a products of points as required by Definition $\ref{def8}$. However, these endpoints have an internal frames which represents the direction that the corresponding horizontal annulus bubbles in. Since every horizontal annulus $H_j$ has a corresponding vertical annulus $V_j$ which creates the same type of bubble, the other side of these special boundaries contains a different moduli space corresponding to replacing $H_j$ by $V_j$, and gluing these two moduli spaces eliminates these special boundaries.

Specifically, when the annuli $H_j$ and $V_j$ share the marking $O_j$, as well as their starting generator $x$ (which is also their ending generator), both of their moduli spaces enter a lower stratum $Q = \mathcal{M}(c_x, \vec{e}_j, (1)_j)$, which has an internal frame $f_1$ corresponding to the direction of horizontal bubbling. $\mathcal{M}(H_j, 0, 0)$ enters the point $Q$ in the $f_1$ direction, while $\mathcal{M}(V_j, 0, 0)$ enters the same point $Q$ in the $-f_1$ direction, so the two moduli spaces naturally glue together at $Q$, as shown in Figure $\ref{fig22}$.

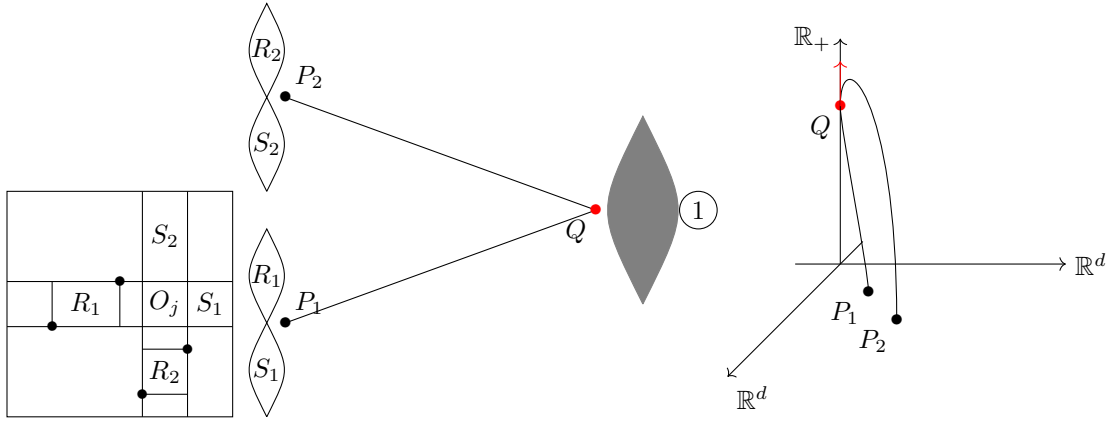
\begin{figure}
\begin{tikzpicture}[scale=3]

\draw (0, 0) rectangle (1, 1);

\node at (0.7, 0.5) {$O_j$};
\draw (0.6, 0) -- (0.6, 1);
\draw (0.8, 0) -- (0.8, 1);
\draw (0, 0.4) -- (1, 0.4);
\draw (0, 0.6) -- (1, 0.6);

\draw (0.6, 0.1) -- (0.8, 0.1);
\draw (0.6, 0.3) -- (0.8, 0.3);
\draw (0.2, 0.4) -- (0.2, 0.6);
\draw (0.5, 0.4) -- (0.5, 0.6);

\node at (0.7, 0.2) {$R_2$};
\node at (0.7, 0.8) {$S_2$};
\node at (0.35, 0.5) {$R_1$};
\node at (0.9, 0.5) {$S_1$};

\node at (0.2, 0.4) {\small \textbullet};
\node at (0.5, 0.6) {\small \textbullet};
\node at (0.6, 0.1) {\small \textbullet};
\node at (0.8, 0.3) {\small \textbullet};

\end{tikzpicture} \begin{tikzpicture}[scale=2.5]

\draw (0, -0.6) .. controls (-0.125, -0.35) .. (0, -0.1) .. controls (0.125, -0.35) .. (0, -0.6);
\draw (0, -0.1) .. controls (-0.125, 0.15) .. (0, 0.4) .. controls (0.125, 0.15) .. (0, -0.1);
\draw (0, 0.6) .. controls (-0.125, 0.85) .. (0, 1.1) .. controls (0.125, 0.85) .. (0, 0.6);
\draw (0, 1.1) .. controls (-0.125, 1.35) .. (0, 1.6) .. controls (0.125, 1.35) .. (0, 1.1);
\filldraw[color=gray] (2, 0) .. controls (1.75, 0.5) .. (2, 1) .. controls (2.25, 0.5) .. (2, 0);

\draw (0.1, -0.1) node {\textbullet};
\draw (0.1, -0.1) node[anchor=south west] {$P_1$};
\draw (0.1, 1.1) node {\textbullet};
\draw (0.1, 1.1) node[anchor=south west] {$P_2$};

\draw (0.1, -0.1) -- (1.75, 0.5);
\draw (0.1, 1.1) -- (1.75, 0.5);

\draw[color=red] (1.75, 0.5) node {\textbullet};
\draw (1.75, 0.5) node[anchor=north east] {$Q$};

\node at (0, 0.15) {$R_1$};
\node at (0, -0.35) {$S_1$};
\node at (0, 1.35) {$R_2$};
\node at (0, 0.85) {$S_2$};

\draw (2.295, 0.5) circle (0.1);
\node at (2.295, 0.5) {1};

\end{tikzpicture}
\begin{tikzpicture}[x={(1,0)}, y={(0,1)}, z={(-0.5, -0.5)}, scale=3]

\draw (0.25, 0, 0.25) node {\textbullet};
\draw (0.5, 0, 0.5) node {\textbullet};
\draw (0.25, 0, 0.25) node[anchor=north east] {$P_1$};
\draw (0.5, 0, 0.5) node[anchor=north east] {$P_2$};

\draw (0.25, 0, 0.25) .. controls (0.25, 0.1, 0.25) and (0, 0.6, 0) .. (0, 0.7, 0);
\draw (0.5, 0, 0.5) .. controls (0.5, 1, 0.5) and (0, 1, 0) .. (0, 0.7, 0);

\draw[color=red] (0, 0.7, 0) node {\textbullet};
\draw (0, 0.7, 0) node[anchor=north east] {$Q$};

\draw[->] (-0.2,0,0) -- (1,0,0) node[anchor=west]{$\R^d$};
\draw[->] (0,0,0) -- (0,1,0) node[anchor=east]{$\R_+$};
\draw[->] (0,0,-0.2) -- (0,0,1) node[anchor=north west]{$\R^d$};
\draw[color=red,->] (0, 0.7, 0) -- (0, 0.9, 0);

\end{tikzpicture}
\caption{(Left) The Type 2.2b domains $H_j$ and $V_j$ from a generator \textbullet \: to itself. (Middle) Moduli spaces corresponding to the domains $H_j$ and $V_j$, glued along the stratum $Z(0, 1, 0; (1))$ (red). (Right) The embedded picture, with the internal frame (red arrow).}\label{fig22}\end{figure}

The glued moduli space has endpoints which are products of points, which each correspond to the non-bubble endpoints of $\mathcal{M}(H_j, \vec{0}, \vec{0})$ and $\mathcal{M}(H_j, \vec{0}, \vec{0})$ ($P_1$ and $P_2$, respectively, in Figure $\ref{fig22}$). It remains to frame these moduli spaces.

\begin{lemma}\label{gluedframe} Given a frame assignment $f$ for $\CDP_*$, its extension (which we interchangeably call $f$) to the moduli spaces of domains in $\CD_*$ with the above gluing which gives the glued moduli space the sum of the frame assignments of the glued pieces is a coherent framing (and thus a frame assignment in the Lobb-Orson-Schütz sense).
\end{lemma}

\begin{proof}Consider a pair of annuli $H_j$ and $V_j$ that we have glued as above, and write $H_j = R_1 * S_1$ and $V_j = R_2 * S_2$. By the properties of sign assignments, $s(R_1) + s(R_2) + s(S_1) + s(S_2) = 1 \pmod 2$, so there is a frameable interval between the non-bubble endpoints $P_1$ and $P_2$. Framing this interval as if it were a Type 2.2a moduli space, the preferred framing is given by changing the external framing near both endpoints with long preferred paths with respect to $f_1$ to the positive framing in the middle.

On the other hand, the preferred framing of each Type 2.2b moduli space also changes the external frame near the endpoints $P_1$ and $P_2$ with respect to $f_1$ into the positive framing at the gluing point $Q$. So we see that the preferred framing of the glued moduli space is the same as the result of gluing the preferred framings of both moduli spaces. The coherence follows from Lemma $\ref{framecond}$ and Theorem $\ref{mainthm}$.\end{proof}

Given a framed $1$-flow category $\mathscr{C}$, \cite{LOS} give an algorithm for computing $Sq^2$, which we recap below. Let $\varphi \in C^{k}(\mathscr{C}; \Z/2)$ be represented by objects $c_1, \dots, c_l$ of $\mathscr{C}$. Let $b_1, \dots, b_m$ be the objects that have some nonempty $0$-dimensional moduli space $\mathcal{M}(b_i, c_j)$ for some $j$. Since $\varphi$ is a cocycle, there are an even number of moduli spaces $\mathcal{M}(b_i, c_j)$ for each fixed $i$.

\begin{definition}A combinatorial boundary matching for $\varphi$ is a partition $\mathcal{C}$ of
\begin{align*}\bigcup\limits_{j = 1}^{l} \mathcal{M}(b_i, c_j)\end{align*}
into ordered pairs for each $i$.\end{definition}

Given a cocycle $\varphi \in C^k(\mathscr{C}; \Z/2)$, a combinatorial boundary matching $\mathcal{C}$ for $\varphi$, and an object $a$ in $\mathscr{C}$ with $\lvert a \rvert = k + 2$, we define the graph $\Gamma_{\mathcal{C}}(a, \varphi)$ as follows:

\begin{itemize}

\item The vertices correspond to nonempty products of zero-dimensional moduli spaces $\mathcal{M}(a, b) \times \mathcal{M}(b, c_j)$ for some $j$.

\item There are edges between these vertices for each interval $I$ of $\mathcal{M}(a, c_j)$, labelled by $f(I)$, where $f$ is the frame assignment. When an edge is labelled, we refer to its label as $f(e)$.

\item For each pair $(b_1, b_2)$ in $\mathcal{C}$, we add an edge between $b_1$ and $b_2$. If $s(b_1) = s(b_2)$, where $s$ is the sign assignment, then the edge is directed towards $b_2$, and otherwise it is undirected.

\end{itemize}

The second Steenrod square is given by the cochain $sq^{\varphi}: C^{k+2}(\mathscr{C}; \Z/2)$ given by
\begin{align*}sq^{\varphi}(a) = \sum\limits_{C} \left(1 + \#(\text{directed edges in } C \text{ pointing a certain direction}) + \sum\limits_{e} f(e)\right)
\end{align*}
where the outer sum is taken over the components $C$ of $\Gamma_C(a, \varphi)$ and the inner sum is taken over the edges $e$ in $C$.

\cite[Section 3]{LOS} show that this sum does not depend on the choice of direction in the sum, nor the choice of boundary matching $\mathcal{C}$. Furthermore, they show that $Sq^2(z) = [sq^{\varphi}]$ for any cocycle $\varphi$ representing $z$.

\begin{proof}[Proof of Theorem $\ref{mainthm3}$.] By the Lobb-Orson-Schütz construction, an algorithm to compute $Sq^2$ for the domains of $\CD_*$ follows from an algorithmic construction of the framed $1$-flow category for the same domains. First, the moduli spaces of each domain in $\CD_1$ are points, and Lemma $\ref{lemsame}$ ensures that the notion of a sign assignment for $\CDP_*$ matches the sign assignment for a framed $1$-flow category, so we may use a sign assignment constructed algorithmically from, say, Definition $\ref{defsa}$, to frame all the $0$-dimensional moduli spaces.

By the construction of the strata, the moduli spaces of each domain in $\CD_2$ have no Type II strata except for the moduli spaces of annuli. In any case, since annuli produce a single bubble in the boundary of their moduli spaces, the moduli spaces have no Type III or IV strata. As a result, the boundaries of all moduli spaces $\mathcal{M}(D, 0, 0)$ for $D \in \CD_2$ consist of products of points, except when $D$ is an annulus. But when $D$ is an annulus, their bubble boundary has been glued with the moduli space of another annulus to produce an interval whose boundary is the product of points as in Figure $\ref{fig22}$. And after gluing, the boundaries of all moduli spaces $\mathcal{M}(D, 0, 0)$ for $D \in \CD_3$ consist of intervals of the form $\{ P \} \times I$ or $I \times \{Q \}$, as in Definition $\ref{def8}$. These intervals can be framed according to the glued version of the frame assignment $f$, constructed previously by $\ref{mainthm2}$, and by Lemma $\ref{gluedframe}$ this is enough to form a framed $1$-flow category.\end{proof}

\section{An Example}

In this section, we compute both $Sq^2$ for the following $2 \times 2$ grid diagram for the unknot $\mathcal{U}$:

\begin{center}\begin{tikzpicture}

\draw (0, 0) grid (2, 2);

\draw (1.4, 0.6) -- (1.6, 0.4);
\draw (1.4, 0.4) -- (1.6, 0.6);

\draw (0.4, 1.6) -- (0.6, 1.4);
\draw (0.4, 1.4) -- (0.6, 1.6);

\node at (0.5, 0.5) {$O_1$};
\node at (1.5, 1.5) {$O_2$};

\end{tikzpicture}\end{center}

The $2 \times 2$ grid complex has two generators:

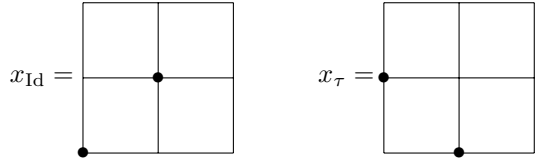
\begin{center}\begin{tikzpicture}

\draw (0, 0) grid (2, 2);

\node at (0, 0) {\textbullet};
\node at (1, 1) {\textbullet};

\node[anchor=east] at (0, 1) {$x_{\Id} = $};

\draw (4, 0) grid (6, 2);

\node at (4, 1) {\textbullet};
\node at (5, 0) {\textbullet};

\node[anchor=east] at (4, 1) {$x_{\tau} = $};

\end{tikzpicture}\end{center}

and four rectangles $R_1, R_3 \in \mathscr{D}^+(x_{\Id}, x_{\tau})$ and $R_2, R_4 \in \mathscr{D}^+(x_{\tau}, x_{\Id})$:

\begin{center}\begin{tikzpicture}

\draw (0, 0) grid (2, 2);

\node at (0.5, 0.5) {$R_1$};
\node at (1.5, 0.5) {$R_2$};
\node at (1.5, 1.5) {$R_3$};
\node at (0.5, 1.5) {$R_4$};

\end{tikzpicture}\end{center}

The index $2$ domains are $R_1 * R_2, R_1 * R_4, R_3 * R_2, R_3 * R_4 \in \mathscr{D}^+(x_{\Id}, x_{\Id})$ and $R_2 * R_1, R_2 * R_3, R_4 * R_1, R_4 * R_3 \in \mathscr{D}^+(x_{\tau}, x_{\tau})$---note that all of them are annuli.

We first find the sign assignment $s$ of each rectangle. From Definition $\ref{defsa}$,
\begin{align*}s(R_1) + s(R_2) &= 0 \\ s(R_2) + s(R_3) &= 1 \\ s(R_3) + s(R_4) &= 0 \\ s(R_1) + s(R_4) &= 1\end{align*}
This system of equations does not have full rank, but we may add an additional equation. By Theorem $\ref{thm1}$, $s$ is only unique up to coboundaries, and since $c_{x_{\tau}} + c_{x_{\Id}} = \partial R_1$, adding a coboundary has the effect of changing the sign of $R_1$ (and every other rectangle). As a result, we may fix
\begin{align*}s(R_1) = 0\end{align*}
so that we obtain $s(R_1) = s(R_2) = 0, s(R_3) = s(R_4) = 1$.

We now find the frame assignment $f$ of each index $2$ domain, as these are the ones needed for $Sq^2$. To write down the equations for $f$, we classify the generators of $\CDP_3^{\omega \leq 3}$:

\textbf{Type 3.3.} As the only index 2 domains are annuli, there are no Type 3.3a domains. From the following Type 3.3b domains, we obtain:
\begin{align*}&R_1 * R_2 * R_1: f(R_1 * R_2, \vec{0}, \vec{0}) + f(R_2 * R_1, \vec{0}, \vec{0}) + f(R_1, \vec{e}_1, (1)) + f(R_1, \vec{e}_1, (1)) = 0
\\&R_2 * R_1 * R_2: f(R_2 * R_1, \vec{0}, \vec{0}) + f(R_1 * R_2, \vec{0}, \vec{0}) + f(R_2, \vec{e}_1, (1)) + f(R_2, \vec{e}_1, (1)) = 0
\\&R_1 * R_4 * R_1: f(R_1 * R_4, \vec{0}, \vec{0}) + f(R_4 * R_1, \vec{0}, \vec{0}) + f(R_1, \vec{e}_1, (1)) + f(R_1, \vec{e}_1, (1)) = 0
\\&R_4 * R_1 * R_4: f(R_4 * R_1, \vec{0}, \vec{0}) + f(R_1 * R_4, \vec{0}, \vec{0}) + f(R_4, \vec{e}_1, (1)) + f(R_4, \vec{e}_1, (1)) = 0
\\&R_2 * R_3 * R_2: f(R_2 * R_3, \vec{0}, \vec{0}) + f(R_3 * R_2, \vec{0}, \vec{0}) + f(R_2, \vec{e}_2, (1)) + f(R_2, \vec{e}_2, (1)) = 0
\\&R_3 * R_2 * R_3: f(R_3 * R_2, \vec{0}, \vec{0}) + f(R_2 * R_3, \vec{0}, \vec{0}) + f(R_3, \vec{e}_2, (1)) + f(R_3, \vec{e}_2, (1)) = 0
\\&R_3 * R_4 * R_3: f(R_3 * R_4, \vec{0}, \vec{0}) + f(R_4 * R_3, \vec{0}, \vec{0}) + f(R_3, \vec{e}_2, (1)) + f(R_3, \vec{e}_2, (1)) = 0
\\&R_4 * R_3 * R_4: f(R_4 * R_3, \vec{0}, \vec{0}) + f(R_3 * R_4, \vec{0}, \vec{0}) + f(R_4, \vec{e}_2, (1)) + f(R_4, \vec{e}_2, (1)) = 0\end{align*}
Many of these different domains give redundant equations. The remaining equations give that
\begin{align*}f(R_i * R_j) = f(R_j * R_i) \text{ for } \lvert i - j \rvert = 1 \pmod 4\end{align*}
For Type 3.3c domains, first note that several different concatenations yield identical domains. For instance, $R_2 * R_1 * R_4 = R_4 * R_1 * R_2$, and their equation is
\begin{align*}&f(R_1 * R_4, \vec{0}, \vec{0}) + f(R_2 * R_1, \vec{0}, \vec{0}) + f(R_1 * R_2, \vec{0}, \vec{0}) + f(R_4 * R_1, \vec{0}, \vec{0}) \\&+ f(R_2, \vec{e}_1, (1)) + f(R_4, \vec{e}_1, (1)) + f(R_4, \vec{e}_1, (1)) + f(R_2, \vec{e}_1, (1)) = 0\end{align*}
which, by the previous equations, is simply $0 = 0$ and gives no new information. It is easily checked that this is the case for all Type 3.3c domains.

\textbf{Other Triples.} Type 3.2 triples only produce domains of the form $(D, 0, 0)$ via Type IV differentials, which cancel each other as the initial and final reductions have the same result. Type 3.1 and 3.0 triples will never produce $(D, 0, 0)$ terms, so their equations are not relevant to solving for the values $f(R_1 * R_2), f(R_2 * R_3), f(R_3 * R_4)$, and $f(R_1 * R_4)$.

The uniqueness of Proposition $\ref{unique2}$ means we need only specify the values of $f$ up to coboundary and $f(U')$. For $f(U')$, note that on the $2 \times 2$ grid we have that $U' = R_3 * R_2 + R_3 * R_4 + R_2 * R_1 + R_4 * R_1$. As there is no canonical choice of $f(U')$, we will obtain two different frame assignments $f_1$ and $f_2$ satisfying $f_1(U') = 0$ and $f_2(U') = 1$. We will handle them separately, but note that $f_1$ and $f_2$ must both satisfy every other equation. Expanding out $f_1(U')$ and $f_2(U')$:
\begin{align*}&f_1(R_1 * R_2) + f_1(R_2 * R_3) + f_1(R_3 * R_4) + f_1(R_1 * R_4) = 0 \\&f_2(R_1 * R_2) + f_2(R_2 * R_3) + f_2(R_3 * R_4) + f_2(R_1 * R_4) = 1\end{align*}

Finally, for the coboundaries, let $h$ be a $1$-cochain and consider $f' = f + \delta h$. If $h(R_1, 0, 0) = 1$, $f'(R_1 * R_2, 0, 0) = 1 + f(R_1 * R_2, 0, 0)$ and $f'(R_1 * R_4, 0, 0) = 1 + f(R_1 * R_4, 0, 0)$, with the other annuli unchanged. Likewise, $h(R_2, 0, 0) = 1$ changes the frame assignments of $R_1 * R_2$ and $R_2 * R_3$, $h(R_3, 0, 0) = 1$ changes the frame assignments of $R_2 * R_3$ and $R_3 * R_4$, and $h(R_4, 0, 0) = 1$ changes the frame assignments of $R_1 * R_4$ and $R_3 * R_4$. Finally, if $h(c_{x_{\Id}}, \vec{e_1}, (1)) = 1$, the frame assignments of $R_1 * R_2$ and $R_1 * R_4$ are changed, and the same holds for $c_{x_{\tau}}$ by symmetry, and $h(c_{x_{\Id}}, \vec{e_1}, (1)) = 1$ (or $x_{\tau}$) changes the frame assignments of $R_2 * R_3$ and $R_3 * R_4$. So the effect of adding a coboundary is to change the frame assignment of zero, two, or four of the annuli.

We now have the unique solutions (up to coboundary) for the two different frame assignments $f_1$ and $f_2$:
\begin{align*}&f_1(R_1 * R_2, \vec{0}, \vec{0}) = f_1(R_2 * R_3, \vec{0}, \vec{0}) = f_1(R_3 * R_4, \vec{0}, \vec{0}) = f_1(R_1 * R_4, \vec{0}, \vec{0}) = 0 \\&f_2(R_1 * R_2, \vec{0}, \vec{0}) = f_2(R_2 * R_3, \vec{0}, \vec{0}) = f_2(R_3 * R_4, \vec{0}, \vec{0}) = 0, f_2(R_1 * R_4, \vec{0}, \vec{0}) = 1\end{align*}
The link Floer homology $\mathit{HFK}^+$ of the unknot is known to be $\mathbb{F}[U]$, since the $U_1$ and $U_2$ actions are chain homotopic as the $O_1$ and $O_2$ markings lie on the same link component. Since $U$ has homological grading $2$, the homology is supported in even gradings, and the cocycle representatives of the generator $\varphi$ of $\mathit{HFK}^+_{2j}$ is $U_1^k U_2^{j-k} x_{\Id}$ for each $0 \leq k \leq j$.

For each $k$, there are two nonempty zero-dimensional moduli spaces $\mathcal{M}(b_i, U_1^{k}U_2^{j-k} x_{\Id})$, which are given by the points $\mathcal{M}(U_1^{k}U_2^{j-k} x_{\tau}, U_1^{k}U_2^{j-k} x_{\Id}) = \mathcal{M}(R_2, 0, 0)$ and $\mathcal{M}(U_1^{k}U_2^{j-k} x_{\tau}, U_1^{k}U_2^{j-k} x_{\Id}) = \mathcal{M}(R_4, 0, 0)$. The only boundary matching for each $b_i$ pairs these two points.

Now for $a$ with $\lvert a \rvert = 2j + 2$, $a$ is represented by some $U_1^{l}U_2^{j+1-l} x_{\Id}$. The moduli space from $a$ to $U_1^{k}U_2^{j-k} x_{\Id}$ is empty unless $k = l$ or $l - 1$. When $k = l$, the moduli space is the glued intervals $\mathcal{M}(R_3 * R_2, 0, 0)$ and $\mathcal{M}(R_3 * R_4, 0, 0)$, and when $k = l-1$ the moduli space is the glued intervals $\mathcal{M}(R_1 * R_2, 0, 0)$ and $\mathcal{M}(R_1 * R_4, 0, 0)$.

For a fixed $l$, there are four products of moduli spaces from $a$ to some representative of $\varphi$:
\begin{align*}&\mathcal{M}(R_1) \times \mathcal{M}(R_2) = \mathcal{M}(U_1^{l}U_2^{j+1-l} x_{\Id}, U_1^{l-1}U_2^{j+1-l} x_{\tau}) \times \mathcal{M}(U_1^{l-1}U_2^{j+1-l} x_{\tau}, U_1^{l-1}U_2^{j+1-l} x_{\Id}) \\&\mathcal{M}(R_1) \times \mathcal{M}(R_4) = \mathcal{M}(U_1^{l}U_2^{j+1-l} x_{\Id}, U_1^{l-1}U_2^{j+1-l} x_{\tau}) \times \mathcal{M}(U_1^{l-1}U_2^{j+1-l} x_{\tau}, U_1^{l-1}U_2^{j+1-l} x_{\Id}) \\& \mathcal{M}(R_3) \times \mathcal{M}(R_2) = \mathcal{M}(U_1^{l}U_2^{j+1-l} x_{\Id}, U_1^{l}U_2^{j-l} x_{\tau}) \times \mathcal{M}(U_1^{l}U_2^{j-l} x_{\tau}, U_1^{l}U_2^{j-l} x_{\Id}) \\&\mathcal{M}(R_3) \times \mathcal{M}(R_4) = \mathcal{M}(U_1^{l}U_2^{j+1-l} x_{\Id}, U_1^{l}U_2^{j-l} x_{\tau}) \times \mathcal{M}(U_1^{l}U_2^{j-l} x_{\tau}, U_1^{l}U_2^{j-l} x_{\Id})\end{align*}
$\mathcal{M}(R_1) \times \mathcal{M}(R_2)$ and $\mathcal{M}(R_1) \times \mathcal{M}(R_4)$ are connected by the interval which is the glued moduli spaces of the annuli $R_1 * R_2$ and $R_1 * R_4$, so it has the frame assignment $f(R_1 * R_2) + f(R_1 * R_4)$. Similarly, $\mathcal{M}(R_3) \times \mathcal{M}(R_2)$ and $\mathcal{M}(R_3) \times \mathcal{M}(R_4)$ are connected by the glued moduli spaces of the annuli $R_3 * R_2$ and $R_3 * R_4$, which has the frame assignment $f(R_3 * R_2) + f(R_3 * R_4)$. The boundary matching connects each $\mathcal{M}(R_1)$ with a corresponding $\mathcal{M}(R_3)$, and since these rectangles have opposite signs, the corresponding edges in the graph $\Gamma_\mathcal{C}(a, \varphi)$ are undirected. So the graph $\Gamma_\mathcal{C}(a, \varphi)$ is the following:
\begin{center}\begin{tikzcd}
\mathcal{M}(R_1) \times \mathcal{M}(R_2) \arrow[dd, "f(R_1 * R_2) + f(R_1 * R_4))"', no head] &  & \mathcal{M}(R_3) \times \mathcal{M}(R_2) \arrow[dd, "f(R_3 * R_2) + f(R_3 * R_4)", no head] \arrow[ll, no head] \\
                                                                                              &  &                                                                                                                 \\
\mathcal{M}(R_1) \times \mathcal{M}(R_4)                                                      &  & \mathcal{M}(R_3) \times \mathcal{M}(R_4) \arrow[ll, no head]                                                   
\end{tikzcd}\end{center}
As a result, we obtain two different Steenrod squares $Sq^2_1$ and $Sq^2_2$ coming from our two framed $1$-flow categories with frame assignments coming from $f_1$ and $f_2$:
\begin{align*}&sq^{\varphi}_1(a) = 1 + f_1(R_1 * R_2, \vec{0}, \vec{0}) + f_1(R_2 * R_3, \vec{0}, \vec{0}) + f_1(R_3 * R_4, \vec{0}, \vec{0}) + f_1(R_1 * R_4, \vec{0}, \vec{0}) = 1 \\&sq^{\varphi}_2(a) = 1 + f_2(R_1 * R_2, \vec{0}, \vec{0}) + f_2(R_2 * R_3, \vec{0}, \vec{0}) + f_2(R_3 * R_4, \vec{0}, \vec{0}) + f_2(R_1 * R_4, \vec{0}, \vec{0}) = 0\end{align*}
so that, since $j$ was arbitrary,
\begin{align*}&Sq_1^2: H^{2j}(\mathit{HFK}^+(\mathcal{U}; \Z/2) \rightarrow H^{2j + 2}(\mathit{HFK}^+(\mathcal{U}); \Z/2) \text{ is the identity map} \\&Sq_2^2: H^{2j}(\mathit{HFK}^+(\mathcal{U}); \Z/2) \rightarrow H^{2j + 2}(\mathit{HFK}^+(\mathcal{U}); \Z/2) \text{ is zero.}\end{align*}

\section{$Sq_2$ for $\widehat{\mathit{CFK}}$}

In this section, we focus on the case where $L$ is a knot. In this case, the Manolescu-Sarkar spectrum is a stable homotopy type of $\widehat{\mathit{CFK}}$.

Given a framed flow category, we can of course ignore all the moduli spaces higher than $1$-dimensional to form a framed $1$-flow category. As \cite{LOS} note, not every framed $1$-flow category comes from a framed flow category in this way. In fact, suppose we have a framed flow category of the link Floer homology of the unknot from a $2 \times 2$ grid diagram, and we compute its $Sq^2$. Since we would get a spectrum via the Cohen-Jones-Segal construction, we must have the Adem relation
\begin{align*}Sq^2 Sq^2 = Sq^1 Sq^2 Sq^1 = 0\end{align*}
since $\mathit{HFK}^+(\mathcal{U})$ is supported in even gradings. In the previous example, $Sq_1^2$ does not satisfy this relation, so our framed $1$-flow category using $f_1$ must not come from a framed flow category, and consequently our framed $1$-flow category using $f_2$ does come from the Manolescu-Sarkar framed flow category. For a general grid diagram, it is not known which choice of $f(U')$ gives the framed $1$-flow category coming from a framed flow category, but it is possible to construct $Sq^2$ for $\widehat{\mathit{CFK}}$ separately and check this value.

$Sq^1$ is found as torsion in $\widehat{\mathit{HFK}}$, which is not known to exist. In the remainder of the section, we give an explanation of why the Steenrod squares for $\widehat{\mathit{HFK}}$ vanish for small knots. We will instead consider the more convenient dual maps on homology, which are well-defined since the framed $1$-flow category of $\widehat{\mathit{CFK}}$ comes from a space (the Manolescu-Sarkar spectrum).
\begin{align*}&Sq_1: \widehat{\mathit{HFK}}_{j}(K) \rightarrow \widehat{\mathit{HFK}}_{j - 1}(K) \\&Sq_2: \widehat{\mathit{HFK}}_{j}(K) \rightarrow \widehat{\mathit{HFK}}_{j - 2}(K)\end{align*}
So far, we have been considering $\mathit{CFK}$ as a filtered chain complex over $\mathbb{F}[U]$, but we now also consider $\mathit{CFK}$ as a chain complex over $\mathbb{F}[U, V]$, where the $U$ variable corresponds to the $O$ markings as before and the $V$ variable corresponds to the $X$ markings. We can write the differential of $\mathit{CFK}$ over $\mathbb{F}[U, V]$ as
\begin{align*}\partial = \sum\limits_{i, j = 0}^{n} U^i V^j\partial_{i, j}\end{align*}
where $\partial_{i, j}$ is given by domains passing through $i$ $O$ markings and $j$ $X$ markings. Note that $\partial_{0, 0}$ is the $\widehat{\mathit{CFK}}$ differential.

\begin{definition}$\partial_U$ and $\partial_V$ are the differentials on $\widehat{\mathit{HFK}}$ given by
\begin{align*}\partial_U = \partial_{1, 0*} \text{ and } \partial_V = \partial_{0, 1*}\end{align*}
\end{definition}

\begin{prop}\label{commute}$Sq_{k}: \widehat{\mathit{HFK}}_j(K) \rightarrow \widehat{\mathit{HFK}}_{j-k}(K)$ commute with both $\partial_U$ and $\partial_V$.\end{prop}

\begin{proof}Let $\widehat{\mathit{HFK}}_{i, j}(K)$ denote the portion of $\widehat{\mathit{HFK}}(K)$ in Alexander grading $i$ and homological (Maslov) grading $j$. Fix some $i_0, j_0$, and let $X$ be the subcomplex of elements with Alexander grading at most $i_0+1$ and $Y$ be the subcomplex of elements with Alexander grading at most $i_0-1$. Then we have the exact triangle
\begin{align*}Y \xrightarrow{\partial_V} X \rightarrow X/Y \rightarrow \Sigma X\end{align*}
since $\partial_V$ decreases both Maslov and Alexander grading by $1$. 

Fix a homological grading $i_0$, and let $M_{i_0, j_0}$ be the portion of $Y/X$ with homological gradings at least $i_0$ and at most $i_0+3$. Then for any two generators of $M_{i_0, j_0}$, the moduli space between them is constructed, embedded, and framed (up to choosing a framing of the two-dimensional moduli spaces). The Cohen-Jones-Segal construction can be used to lift this part of the complex to a spectrum, so by $Sq^2$ for spectra:
\begin{align*}Sq_2 \circ \partial_V = \partial_V \circ Sq_2: \mathit{HFK}_{i_0+1, j_0} \rightarrow \mathit{HFK}_{i_0, j_0-k-1}\end{align*}

For $\partial_U$, let $X$ be the subcomplex of elements with Alexander grading at least $i_0-1$ and $Y$ be the subcomplex of elements with Alexander grading at least $i_0+1$. The same argument holds with the exact triangle
\begin{align*}Y \xrightarrow{\partial_U} X \rightarrow X/Y \rightarrow \Sigma Y\end{align*}
which shows that
\begin{align*}Sq_k \circ \partial_U = \partial_U \circ Sq_k: \mathit{HFK}_{i_0, j_0} \rightarrow \mathit{HFK}_{i_0 + 1, j_0-k+1}\end{align*}\end{proof}

\begin{proof}[Proof of Theorem $\ref{14cross}$]\cite[Figure 36]{MS} shows the knot Floer homology of the knot $14n26580$, along with its $\partial_U$ differentials. We see that there are two candidates for $Sq_2$, starting at grading $(-2, -3)$ and $(2, 1)$, which fail to commute with $\partial_U$ and $\partial_V$, respectively.\end{proof}

We conclude with another example. The knot $15n12300$, as shown in Figure $\ref{15n12300}$ (in grid form), has 41 generators for $\widehat{\mathit{CFK}}$, which are shown in Figure $\ref{15n12300hfk}$ according to their Maslov and Alexander gradings. We see that a potential nontrivial $Sq^2$ exists between the generators of $\widehat{\mathit{CFK}}_{i, i-1}$ and $\widehat{\mathit{CFK}}_{i, i+1}$ for each $i = -3, -2, -1, 1, 2, 3$, and furthermore, if one of these $Sq^2$ is nontrivial, all of them are, by symmetry and Proposition $\ref{commute}$.

\begin{figure}\begin{tikzpicture}
\draw (0, 0) grid (15, 15);

\draw (0.5, 1.5) -- (0.5, 13.5);
\draw (1.5, 0.5) -- (1.5, 2.5);
\draw (2.5, 1.5) -- (2.5, 4.5);
\draw (3.5, 2.5) -- (3.5, 5.5);
\draw (4.5, 6.5) -- (4.5, 14.5);
\draw (5.5, 3.5) -- (5.5, 11.5);
\draw (6.5, 10.5) -- (6.5, 12.5);
\draw (7.5, 5.5) -- (7.5, 7.5);
\draw (8.5, 6.5) -- (8.5, 11.5);
\draw (9.5, 10.5) -- (9.5, 13.5);
\draw (10.5, 8.5) -- (10.5, 12.5);
\draw (11.5, 7.5) -- (11.5, 9.5);
\draw (12.5, 8.5) -- (12.5, 14.5);
\draw (13.5, 0.5) -- (13.5, 4.5);
\draw (14.5, 3.5) -- (14.5, 9.5);

\draw (1.5, 0.5) -- (13.5, 0.5);
\draw (0.5, 1.5) -- (1.4, 1.5);
\draw (1.6, 1.5) -- (2.5, 1.5);
\draw (1.5, 2.5) -- (2.4, 2.5);
\draw (2.6, 2.5) -- (3.5, 2.5);
\draw (5.5, 3.5) -- (13.4, 3.5);
\draw (13.6, 3.5) -- (14.5, 3.5);
\draw (2.5, 4.5) -- (3.4, 4.5);
\draw (3.6, 4.5) -- (5.4, 4.5);
\draw (5.6, 4.5) -- (13.5, 4.5);
\draw (3.5, 5.5) -- (5.4, 5.5);
\draw (5.6, 5.5) -- (7.5, 5.5);
\draw (4.5, 6.5) -- (5.4, 6.5);
\draw (5.6, 6.5) -- (7.4, 6.5);
\draw (7.6, 6.5) -- (8.5, 6.5);
\draw (7.5, 7.5) -- (8.4, 7.5);
\draw (8.6, 7.5) -- (11.5, 7.5);
\draw (10.5, 8.5) -- (11.4, 8.5);
\draw (11.6, 8.5) -- (12.5, 8.5);
\draw (11.5, 9.5) -- (12.4, 9.5);
\draw (12.6, 9.5) -- (14.5, 9.5);
\draw (6.5, 10.5) -- (8.4, 10.5);
\draw (8.6, 10.5) -- (9.5, 10.5);
\draw (5.5, 11.5) -- (6.4, 11.5);
\draw (6.6, 11.5) -- (8.5, 11.5);
\draw (6.5, 12.5) -- (9.4, 12.5);
\draw (9.6, 12.5) -- (10.5, 12.5);
\draw (0.5, 13.5) -- (4.4, 13.5);
\draw (4.6, 13.5) -- (9.5, 13.5);
\draw (4.5, 14.5) -- (12.5, 14.5);

\foreach \i\j in {1.5/0.5, 2.5/1.5, 3.5/2.5, 5.5/3.5, 13.5/4.5, 7.5/5.5, 4.5/6.5, 11.5/7.5, 10.5/8.5, 14.5/9.5, 9.5/10.5, 8.5/11.5, 6.5/12.5, 0.5/13.5, 12.5/14.5}{

\draw (\i - 0.1, \j + 0.1) -- (\i + 0.1, \j - 0.1);
\draw (\i - 0.1, \j - 0.1) -- (\i + 0.1, \j + 0.1);

}

\foreach \i\j in {13.5/0.5, 0.5/1.5, 1.5/2.5, 14.5/3.5, 2.5/4.5, 3.5/5.5, 8.5/6.5, 7.5/7.5, 12.5/8.5, 11.5/9.5, 6.5/10.5, 5.5/11.5, 10.5/12.5, 9.5/13.5, 4.5/14.5}{

\draw (\i, \j) circle (0.1);

}

\end{tikzpicture}\caption{A $15 \times 15$ grid diagram for $15n12300$.}\label{15n12300}\end{figure}

\begin{figure}\begin{tikzpicture}[scale=1.8]

\draw[->] (0, -4.5) -- (0, 4.5) node[anchor=south]{$M$};
\draw[->] (-3.5, 0) -- (3.5, 0) node[anchor=west]{$A$};

\foreach \x in {-3, ..., -1, 1, 2, ..., 3} {
        \draw (\x, 0) -- ++(0, -.2) ++(0, -.15) node [below, outer sep=0pt, inner sep=0pt] {\small\(\x\)};
        \draw (0, \x) -- ++(-.2, 0) ++(-.15, 0) node [left, outer sep=0pt, inner sep=0pt] {\small\(\x\)};
    }
\draw (0, 4) -- ++(-.2, 0) ++(-.15, 0) node [left, outer sep=0pt, inner sep=0pt] {\small\(4\)};
\draw (0, -4) -- ++(-.2, 0) ++(-.15, 0) node [left, outer sep=0pt, inner sep=0pt] {\small\(-4\)};

\node[fill=white] at (0, 0) {$\mathbb{F}^9$};
\node at (1, 1) {$\mathbb{F}^6$};
\node at (-1, -1) {$\mathbb{F}^6$};
\node at (2, 2) {$\mathbb{F}^2$};
\node at (-2, -2) {$\mathbb{F}^2$};
\node at (1, 0) {\textbullet};
\node at (1.9, 1) {\textbullet};
\node at (2.1, 1) {\textbullet};
\node at (3, 2) {\textbullet};
\node at (1, 2) {\textbullet};
\node at (1.9, 3) {\textbullet};
\node at (2.1, 3) {\textbullet};
\node at (3, 4) {\textbullet};
\node at (-1, 0) {\textbullet};
\node at (-2.1, -1) {\textbullet};
\node at (-1.9, -1) {\textbullet};
\node at (-3, -2) {\textbullet};
\node at (-1, -2) {\textbullet};
\node at (-2.1, -3) {\textbullet};
\node at (-1.9, -3) {\textbullet};
\node at (-3, -4) {\textbullet};

\draw[->] (-1.9, -1.9) -- (-1.1, -1.1);
\draw[->] (-0.9, -0.9) -- (-0.1, -0.1);
\draw[->] (0.9, 0.9) -- (0.1, 0.1);
\draw[->] (1.9, 1.9) -- (1.1, 1.1);

\draw[->] (2.95, 3.95) -- (2.15, 3.05);
\draw[->] (1.85, 2.95) -- (1.05, 2.05);
\draw[->] (1.05, 2.05) -- (2.05, 2.95);
\draw[->] (1.95, 3.05) -- (2.95, 3.95);
\draw[->] (2.95, 1.95) -- (2.15, 1.05);
\draw[->] (1.85, 0.95) -- (1.05, 0.05);
\draw[->] (1.05, 0.05) -- (2.05, 0.95);
\draw[->] (1.95, 1.05) -- (2.95, 1.95);
\draw[->] (-2.95, -3.95) -- (-2.15, -3.05);
\draw[->] (-1.85, -2.95) -- (-1.05, -2.05);
\draw[->] (-1.05, -2.05) -- (-2.05, -2.95);
\draw[->] (-1.95, -3.05) -- (-2.95, -3.95);
\draw[->] (-2.95, -1.95) -- (-2.15, -1.05);
\draw[->] (-1.85, -0.95) -- (-1.05, -0.05);
\draw[->] (-1.05, -0.05) -- (-2.05, -0.95);
\draw[->] (-1.95, -1.05) -- (-2.95, -1.95);

\end{tikzpicture}\caption{The knot Floer homology of $15n12300$. Each dot represents a generator, and the dots along the diagonal are condensed for simplicity. Solid arrows going up and down represent the $\partial_U$ and $\partial_V$ differentials, respectively.}\label{15n12300hfk}\end{figure}

Finally, we caution that $\widehat{\mathit{CFK}}(15n12300)$ may have a nontrivial $Sq^2$, we cannot determine this with $\partial_U$ and $\partial_V$ alone. For example, the knot $14n16$, whose knot Floer homology contains a direct summand shown in Figure $\ref{14nhfk}$, appears to have a potential a nontrivial $Sq^1$. But a nontrivial $Sq^1$ for knot Floer homology would require torsion in knot Floer homology, which is not known to exist for any knot (so in particular $14n16$ cannot have a nontrivial $Sq^1$). So $\partial_U$ and $\partial_V$ alone are still not enough to obstruct the existence of every square.

\begin{figure}\begin{tikzpicture}[scale=1.8]

\draw[->] (0, -0.5) -- (0, 4.5) node[anchor=south]{$M$};
\draw[->] (-0.5, 0) -- (3.5, 0) node[anchor=west]{$A$};

\foreach \x in {1, 2, ..., 3} {
        \draw (\x, 0) -- ++(0, -.2) ++(0, -.15) node [below, outer sep=0pt, inner sep=0pt] {\small\(\x\)};
        \draw (0, \x) -- ++(-.2, 0) ++(-.15, 0) node [left, outer sep=0pt, inner sep=0pt] {\small\(\x\)};
    }
\draw (0, 4) -- ++(-.2, 0) ++(-.15, 0) node [left, outer sep=0pt, inner sep=0pt] {\small\(4\)};

\node at (1, 1) {\textbullet};
\node at (1.9, 2) {\textbullet};
\node at (2.1, 2) {\textbullet};
\node at (3, 3) {\textbullet};
\node at (1, 2) {\textbullet};
\node at (1.9, 3) {\textbullet};
\node at (2.1, 3) {\textbullet};
\node at (3, 4) {\textbullet};

\draw[->] (2.95, 3.95) -- (2.15, 3.05);
\draw[->] (1.85, 2.95) -- (1.05, 2.05);
\draw[->] (1.05, 2.05) -- (2.05, 2.95);
\draw[->] (1.95, 3.05) -- (2.95, 3.95);
\draw[->] (2.95, 2.95) -- (2.15, 2.05);
\draw[->] (1.85, 1.95) -- (1.05, 1.05);
\draw[->] (1.05, 1.05) -- (2.05, 1.95);
\draw[->] (1.95, 2.05) -- (2.95, 2.95);

\end{tikzpicture}\caption{A direct summand of the knot Floer homology of $14n16$, which shows a possibility for $Sq^1$.}\label{14nhfk}\end{figure}

\bibliography{1flow}
\bibliographystyle{amsalpha}

\end{document}